\documentclass[11pt]{article}
\usepackage[applemac]{inputenc}
\usepackage{graphicx}
\usepackage{xcolor}
\usepackage{amsmath,amssymb,amsthm,array, amsfonts}
\usepackage[authoryear, round]{natbib}
\usepackage{enumerate}
\usepackage{float}
\usepackage{dsfont}
\usepackage{esint}
\usepackage{tikz}
\usepackage{ulem}
\usepackage{tikz}
\usepackage{ulem}
\usepackage{graphicx}
\usepackage{tikz}
\usetikzlibrary{matrix, automata,arrows,positioning,calc}
\usepackage{hyperref} % for url's
\hypersetup{colorlinks=true, urlcolor= blue, citecolor=magenta}
\usepackage{mathtools}
\usepackage{geometry} % margin auto
\newtheorem{theorem}{Theorem}
\newtheorem{proposition}{Proposition}

\newtheorem{lemma}{Lemma}

\newcommand{\R}{\mathbb{R}}
\def\bZ{\mathbb{Z}}	% integers
\def\bE{\mathbb{E}}	% expectation
\renewcommand{\P}{\mathbb{P}}
\newcommand{\I}{\mathds{1}}	% indicator function

\newcommand{\argmin}{\mathop{\mbox{argmin}}}
\renewcommand{\d}{\mathtt{d}}
\newcommand{\var}{{\rm var}}
\def\cov{{\mbox{cov}}}

\title{\bf 
Gold standard process Markovian poisoning: a~semiparametric approach
}

\author{
Claire Lacour and Pierre Vandekerkhove
}

\date{}%September 20, 2013}

\begin{document}

%\newcount\Comments  % 0 suppresses notes to selves in text
%\Comments=1   % TODO: set to 0 for final version
%% \kibitz{color}{comment} inserts a colored comment in the text
%\newcommand{\kibitz}[2]{\ifnum\Comments=1\textcolor{#1}{#2}\fi}
%% add yourself here:
%\newcommand{\dc}[1]{\kibitz{blue}      {[DC: #1]}}
%\newcommand{\pv}[1]{\kibitz{green}      {[PV: #1]}}

\maketitle

\begin{abstract}
We consider in this paper  a stochastic process that mixes  in time, according to a non-observed stationary Markov selection process, two separate sources of randomness: i)  a stationary process which distribution is accessible (gold standard);  ii)  a pure  i.i.d. sequence which distribution is unknown (poisoning process).  In this framework we propose to estimate, with two different approaches,  the transition of the hidden Markov selection process  along with the distribution, not supposed to belong to any parametric family,  of the unknown i.i.d. sequence, under minimal (identifiability, stationarity and dependence in  time) conditions. We show that 
 both estimators provide consistent  estimations of the  Euclidean  transition parameter, and also prove that one of them, which is  $\sqrt{n}$-consistent, allows to establish a functional central limit theorem about the unknown poisoning sequence cumulative distribution function. The numerical performances of our estimators are illustrated through various challenging examples.
\end{abstract}

\section{Introduction}
Semiparametric latent models have been extensively studied in the last two decades and proved to be adequate tools in some situations where various non-labeled sources 
of randomness are observed and parametric assumption about the output features is not easy to figure out. This field of research arised with the paper of \cite{Hall03} who consider a mixture of two distributions in $\R^d$, $d\geq 3$,  each having independent components. In that seminal work the authors prove, for the first time,  that this model is identifiable and can be consistently estimated without assuming any parametric  assumption about the mixed distributions, see Theorem 4.1 in  \cite{Hall03} regarding the nonparametric 2-order margin distribution no-tensoring condition. Latter \cite{Bordes06a} and \cite{Hunter07} consider jointly, but with two separate approaches, an univariate semiparameric two component mixture model with symmetric components equal up to a shift parameter. The authors prove the identifiability of their model and propose accordingly consistent estimation methods, the $\sqrt{n}$-consistency not being considered yet in these papers.  More recently  \cite{butucea_VDK12} and \cite{Butucea17} investigate the consistency and rates of convergence, including the asymptotic normality, of a new class of Fourier-based estimators for this model. Following a similar  approach, \cite{Werner2020} propose an adaptive estimation in the supremum norm for semiparametric mixtures of regression with symmetric noises.
A generalized semiparametric version of the Expectation-Maximization (EM) algorithm is also proposed in \cite{BORDES20075429} to deal with symmetric shifted $m$-component mixture models, see also the MM-algorithm proposed in \cite{Levine11} for the multivariate case. 

 In \cite{Bordes06b}, the authors consider the classic  contamination model:
\begin{eqnarray}\label{contamination1}
F(x)=pF_0(x)+(1-p)F_1(x),\quad x\in \R,
\end{eqnarray}
where $p\in ]0,1[$ is a proportion parameter, $F$ denotes the cumulative distribution function (cdf) of  independent and identically distributed (i.i.d.) observations, $F_0$ denotes a known cdf modelling  a gold standard, and $F_1$ is an unknown (contaminant)  but symmetric cdf, modelling a departure from $F_0$, without any parametric assumption about it. This model  has been extensively studied over the last decades  and is related to various applications, see  \cite{shen}   for a comprehensive survey on this  topic. This model is of particular interest when considering   generic situations distorted by an unexpected  event: i)  impact of pandemics on mortality,  see \cite{milhaud2}; ii)  the presence of diseased tissues in microarray  analysis, see  \cite{MBJ06}, and \cite{Benjamini},  \cite{Donoho_Jin2004} for the related multiple testing problem, iii)  variables observation,  such as metallicity and radial velocity of stars, in the background of the Milky Way, see  \cite{WMSW09}; iv) trees diameters modeling  in the presence of  extra varieties, see \cite{pod}. 
From a semiparametric perspective let us mention \cite{Patra},  who provide an interesting well-posed identifiability definition for model (\ref{contamination1}) along with the  corresponding estimation approach and rates of convergence.  In \cite{Pommeret_VDK2019}  a testing procedure  is proposed to investigate if the  unknown component $F_1$, involved in (\ref{contamination1}), belongs to a specific parametric family. In \cite{milhaud2}, respectively \cite{MPSV24}, the authors consider the 2-sample, resp. $K$-sample, testing problem about the equality of the unknown $F_1$-component. In  \cite{MPSV24} the authors also provide a model-based clustering procedure to  collect groups of  samples sharing the same contaminant component (within the group).
Motivated by the detection of rare cells in flow cytometry, \cite{Gaucher25} study the test $\{F=F_0\}$ against $\{F=pF_0+(1-p)F_1\}$ from 3 samples with respective distributions $F_0$, $F_1$, $F$.

A chronological   extension of mixture models are the so-called Hidden Markov Models (HMM) which are basically mixtures of randomness picked by a latent Markov chain.  The identifiability for this model, in the  finite discrete case,  has been studied first  in \cite{Gilbert59}  when  the asymptotic behavior  of the Maximum Likelihood Estimator (MLE)  is fully addressed in  \cite{Baum66}. This model has been reconsidered in the early 90's by \cite{Leroux92} where  the latent Markov chain has a finite state space but the emissions of the hidden Markov model  lie in an Euclidean space (not necessarily finite). In \cite{Leroux92} the author  proves  the consistency of  the MLE under minimal conditions when the asymptotic normality  for this  model is established latter by  \cite{Bickel98}. Since then the inference  for  HMMs has been extensively studied in the parametric case and an excellent monography about this topic  is proposed in  \cite{Cappe05}, see also \cite{Bouguila22} for  recent developments  and applications.

More recently \cite{Gassiat16} introduced a nonparametric finite translation HMM. The authors establish identifiability conditions 
and  provide a Fourier based consistent estimator of the number of populations, of the translation parameters along with  the distribution of two consecutive latent variables, which they  prove to be asymptotically normal under mild dependence assumptions. The authors also  propose a nonparametric estimator of the unknown translated density. In case the latent variables form a Markov chain, the authors  prove that their estimator is minimax adaptive over regularity classes of densities. In \cite{Alexandrovich16}  the authors investigate the nonparametric identification and maximum likelihood estimation for finite-state HMMs. In \cite{DGL2016} the authors address the estimation issue for a general HMM with  nonparametric
modeling of the emission distributions. They propose a new penalized least-squares estimator, based on projections of the emission distributions onto nested subspaces of increasing complexity,  for the emission distributions which is statistically optimal and practically tractable. They also  prove a non asymptotic oracle inequality for their  nonparametric estimator of the emission distributions. A consequence is that this new estimator is rate minimax adaptive up to a logarithmic term. For a better overview on semi/nonparametric mixtures or Hidden Markov Models we recommend the two excellent surveys by \cite{gassiat19} and \cite{xiang2019}.
%The popular spectral estimators are unable to achieve the optimal rate but may be used as initial points in our procedure. Simulations are given that show the improvement obtained when applying the least-squares minimization consecutively to the spectral estimation.

In this paper we propose to investigate the semiparametric estimation of a model sharing some common features with the  Hidden Markov Mixture of Markov Models (H4M) introduced in \cite{VDK05}. These latent models are  basically based on $K$ independent Markov processes (having their own dynamic) which observation is picked randomly by a latent Markov chain valued in the (label) state-space $\left\{1,\dots,K \right\}$, see Section \ref{H4Msection} for a short presentation under  $K=2$. Such models are interesting for their  ability to describe discrete time series with: (i) abrupt changes, when the latent variable undergoes a change
of state; (ii) local stationarity, during stages where the latent Markov chain remains in the same state; (iii)
multimodal marginal distributions from mixture structure; and (iv) phase-type feedback
effects (continuation  of trajectories from the past).  In \cite{VDK05} the author  proposes a general tractable approach for estimating parametrically  these models (admitting parametrization of the stationary distribution and identifiability) and checks in detail that the  assumptions are fully satisfied for a Markov mixture of two linear AR(1) models with Gaussian noise. A Monte Carlo method is also proposed to calculate the split data likelihood of the  H4M  when no analytic expression for the invariant probability densities of the independent  Markov processes is known.

The semiparametric model, sharing features with the H4M, we will consider in this paper,   can be defined by three independent processes: i) a stationary stochastic  process which  distribution is accessible (gold standard); ii) a sequence of i.i.d. random variables which distribution is unknown and not supposed to belong to any parametric family (poisoning process); iii) a 2-state Markov chain picking observations in time from the two previous sources  of randomness, let say label 0 for the gold standard process  and label 1 for  the i.i.d. sequence.  Even though this model looks simpler than a general H4M it does still embed the complexity of a classical HMM since part of the time (during phases where  latent Markov chain visits state 0)  the data has an identified stochastic dynamic, when the rest of the time it has an unknown  i.i.d. pattern exactly like a standard  HMM (conditionally on the fact that the latent Markov chain visits state 1). The bridge made by this model between stochastic dynamics (possibly Markovian)   and i.i.d. patterns  through a Markovian state-process could be  of particular interest in the cyber-security field as explained in the two following paragraphs.

In  \cite{Dass2021} an interesting description of HMMs used in cyber security is provided. Briefly,  nowadays HMMs find their applications across several domains such as natural language processing and machine-learning due to the simplicity of adapting the model to predict unknown/hidden state sequences. The prediction is based on the features or observations emitted from each state. HMMs are, as a consequence,  very useful in the cyber security domain as cyber-attacks are often conducted in several phases or steps where these steps may not always be conspicuous as attackers often
try to mask their activities. However, HMM can help in identifying patterns in the data such as network trace spread across time and can evidently help in determining attacks, see \cite{Holgado2020}. 

On the other hand, as  considered  for example in    \cite{Korczynski14},  a popular method to classify internet traffic  activities is the Markovian {\it fingerprinting} of applications based on training data. As explained in  \cite{Korczynski14}, the past research on traffic analysis and classification showed that once we are able to generate a unique signature based on the packet or message payload (e.g. HTTP request headers), we can classify applications with high accuracy. Unfortunately, such approaches fail in case of encrypted traffic. In their work  the authors propose a payload-based method to identify application flows encrypted with the Secure Socket Layer/Transport Layer Security (SSL/TLS) protocol, which is a fundamental cryptographic protocol suite supporting secure communication over the Internet.
 A fingerprint can be considered as any distinctive feature allowing identification of a given traffic class. In  \cite{Korczynski14}, a fingerprint corresponds to a first-order homogeneous Markov chain reflecting the dynamics of an SSL/TLS session.  They can also serve to reveal intrusions (security failure) trying to exploit the SSL/TLS protocol by establishing abnormal communications with a server.
 
To summarize we think that our model could help somehow to analyse inside one single model the intrusion of hackers, which Markovian plan and secret activity is unknown (HMM-type),  attacking a system which Markovian  fingerprint is known. During the hacking attack one can think that the nature of the fingerprint is altered (departure from what it is usually observed) and replaced by another signal supposed to be  i.i.d. (for technical  matters). We will not address in this paper the actual treatment of  such a problem but invite any practitioner to test our method on their data if applicable, see Section \ref{num_perf} and formula  (\ref{monte_carlo_cdf_F0_G0})  for practical implementation based on the knowledge of a training data issued from the fingerprint source (gold standard).

%In that setup we investigate the identifiability of our model and propose a corresponding semiparametric minimum contrast approach based on adequate  1 and 2-order empirical cdf of the observed process  comparison. 
In that setup we investigate the identifiability of our model: interestingly, the dependence overt time of the gold standard process is crucial to ensure it. 
We propose a  semiparametric minimum contrast approach based on a adequate  1 and 2-order empirical cdf of the observed process  comparison and we show that our estimators   are strongly consistent under very mild technical assumptions. Moreover we prove that, under some  mixing properties, our estimation of the latent Markov chain transition matrix along with the cdf of the unknown poisoning  sequence is $\sqrt{n}$-consistent. 
The model is detailed in Section~\ref{sec:model}, when  identifiability and consistency results for the parametric part (transition of the latent Markov chain) are stated in Section~\ref{sec:est}. The asymptotic normality is proved in Section~\ref{sec:asymptoticnormality} and the nonparametric cdf estimation of  the unknown  i.i.d. sequence is addressed in Section~\ref{sec:nonparametric}.
%	Numerical performance in various submodels are assessed in Section~\ref{num_perf}. 
Finally
 Section~\ref{num_perf} presents numerical implementations of our procedure over a collection of challenging submodels. 
Concluding remarks are discussed in Section~\ref{sec:conclusion} when proofs are gathered in Section~\ref{sec:proofs}.

\section{Chronological mixtures of stochastic processes}
\label{sec:model}
In this section we  present successively  the H4M and the so called gold standard process poisoning model (our model of interest). To the best of our knowledge, both models have the unique particularity  to mix in time, according to a non-observed stationary Markov selection process, pieces of independent stochastic processes. Note for example that Hidden Markov Models (HMMs)  mix only observations that are conditionally independent given a latent  Markov chain. That specific  structure  allows  favorably to solve  the three HMM standard problems, which are respectively the likelihood, decoding and parameters learning,  in linear computing time thanks,  respectively,   to the Forward, Viterbi and Forward-Backward algorithms, see \citet{Cappe05} for further details.  Unfortunately no such estimation strategies based on filtering and smoothing equations can be derived when considering chronological mixtures of Markov processes. This is the main challenge of this new class of chronological latent models.

\subsection{Hidden Markov Mixtures of Markov Models (H4M)}\label{H4Msection}
In \citet{VDK05}, the author  introduce a new class of missing data process, so called  Hidden Markov Mixture of Markov processes (H4M). This model is defined, in its simplest version, by considering three independent  stationary Markov processes denoted $X=(X_i)_{i\geq 1}$, $Y^0=(Y^0_i)_{i\geq 1}$ and $Y^1=(Y^1_i)_{i\geq 1}$. The Markov chain $X$ is valued in $\left\{0,1\right\}$ with unknown transition matrix  $\Pi$ and natural parametrization
\begin{eqnarray}\label{transition_matrix}
\Pi_{\theta}=\begin{bmatrix} 
	\Pi_\theta(0,0) & \Pi_\theta(0,1) \\
	\Pi_\theta(1,0) & \Pi_\theta(1,1)
	\end{bmatrix}=\begin{bmatrix} 
	1-\alpha & \alpha \\
	\beta & 1-\beta
	\end{bmatrix},
\end{eqnarray}
parametrized by $$\theta=(\alpha,\beta)\in [\delta,1-\delta]^2,\quad \text{ with }0<\delta<1/2.$$  The true value of the parameter $\theta^*$ is such that $\Pi_{\theta^*}=\Pi$.
The invariant probability vector associated to $\Pi_\theta$ is denoted $(\pi_\theta(0),\pi_\theta(1))=(\frac{\beta}{\alpha+\beta},\frac{\alpha}{\alpha+\beta})$. The Markov process $Y^0$ is valued in a  measurable state space $(E, \mathcal E)$ provided with a finite dominating measure $\lambda$, and has  a transition kernel density, with respect to this reference measure, $q^0(\cdot|\cdot)$  usually unknown (but can be  supposed to be known in a contamination modeling perspective as presented in the next section), when the Markov process $Y^1$, also  valued in $(E,\mathcal E)$, has  an unknown transition kernel $q^1(\cdot|\cdot)$ (which can be reduced to a simple non-conditional density function $f^1$ as considered in the next section). In other words
$\P(Y^j_{n+1}\in dv |Y^j_n=u)=q^j(v|u) d\lambda(v)$, $j=0,1$.
%It is supposed that the  transition kernels $q^{j}(\cdot|\cdot)$, $j=1,2$,  induce respectively a recurrent positive Markov process with unique invariant probability density $f^j(\cdot)$ that satisfy
%\begin{eqnarray}\label{stationarity}
%\int_E f^j(u)q^j(v|u)\lambda(du)=f^j(v),\quad  v \in E.
%\end{eqnarray}
In this setup the observed process $Z=(Z_i)_{i\geq 1}$ is defined as follows:
\begin{eqnarray}\label{model}
Z_i=\I_{\left\{X_i=0\right\}}Y^{0}_i+\I_{\left\{X_i=1\right\}}Y^{1}_i,\qquad i\geq 1.
\end{eqnarray}
We display in Figure \ref{dessin} a simple situation in which  the Markov process $Y^0$ is observed at time $t=1$ and 4 and by-passed by an i.i.d. process $Y^1$ during time $t=2$ and 3.
%\vspace{0.3cm}

\begin{figure}
\begin{center}
\begin{tikzpicture}[->,>=stealth',shorten >=1pt,auto,node distance=3cm,semithick]
  \tikzstyle{every state}=[fill=none,draw=black,text=black]
   \tikzset{insert/.style={draw=none,text=white}}
     
  \node[state]         (A)               {\textcolor{blue}{$Y_1^0$}};
  \node[state]         (B)  [right of=A]  {$Y_2^0$};
  \node[state]         (C)  [right  of=B] {$Y_3^0$};
  \node[state]         (D)  [right  of=C]       {\textcolor{blue}{$Y_4^0$}};
  \node[]              (E)  [right=0.6cm of D]       {...};
  \path (A) edge              node {$q^0(\cdot|Y_1^0)$} (B)
        (B) edge              node {$q^0(\cdot|Y_2^0)$} (C)
        (C) edge              node {$q^0(\cdot|Y_3^0)$} (D)
        (D) edge              node {} (E);
        
   \node[state]         (A1) [below=1cm of A]    {$Y_1^1$};        
    \node[state]         (B1)  [right of=A1]  {\textcolor{blue}{$Y_2^1$}};
  \node[state]         (C1)  [right  of=B1] {\textcolor{blue}{$Y_3^1$}};
  \node[state]         (D1)  [right  of=C1]       {$Y_4^1$};
  \node[]              (E1)  [right=0.6cm of D1]       {...};
%   \path (A1) edge              node {$q^1(\cdot|Y_1^1)$} (B1)
%        (B1) edge              node {$q^1(\cdot|Y_2^1)$} (C1)
%        (C1) edge              node {$q^1(\cdot|Y_3^1)$} (D1)
%        (D1) edge              node {} (E1);
        
         \path (A) edge[blue] node{} (B1)
            (B1) edge[blue] node{} (C1)
            (C1) edge[blue] node{} (D);
\end{tikzpicture}
\end{center}
\caption{Illustration of a Markovian mixture of a Markov process and an i.i.d. sequence.  Here $X_1=0, X_2=1,X_3=1,X_4=0$.}\label{dessin}
\end{figure}
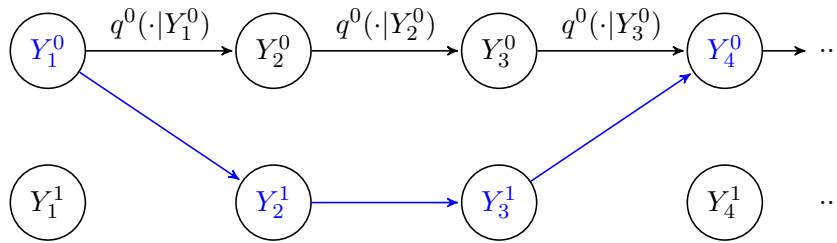

%\vspace{0.3cm}

As studied in \citet[proof of Lemma 1]{VDK05}, model (\ref{model}) is stationary and geometrically $\alpha$-mixing provided that the processes $Y^0$ and $Y^1$ are themselves geometrically $\alpha$-mixing.
Based on  a $n$-trajectory $(Z_1,\dots,Z_n)$ from $Z$, the basic statistical challenge  is to  provide and study, under weak technical conditions, a statistical method to recover the true  transition matrix $\Pi_{{\theta}^*}$ along with the unknown transition kernels $q^j(\cdot|\cdot)$, $j=0,1$, when these ones are supposed to belong to parametric families.  Given the untractability of the H4Ms complete likelihood, see expression (8) in \citet{VDK05}, the author consider an idea introduced by  \cite{Ryden:1994aa} for HMMs,  based on a more tractable $m$-splitted likelihood, $m$ being the length of the splittng.  \cite{Ryden:1994aa} proves that the Maximum Split Data likelihood estimate (MSDLE) is consistent and asymptotically normal under standard  technical conditions (identifiability, ergodicity, regularity, etc.). In   \citet[Section 2]{VDK05}, the author provides a set of conditions (C1--7) under which the MDSLE adapted to the H4Ms is also strongly consistent and asymptotically normal, see \citet[Theorems 1 and 2]{VDK05}. In Section 4 of that paper, a Monte Carlo procedure is also proposed to compute the stationary probability density function $f(z)$ of  a generic transition density kernel $q(\cdot|\cdot)$.
%, see property (\ref{stationarity}). 
Note that this step is crucial in order to compute and maximize  the $m$-splitted likelihood  involving $f^j(z)$'s, generally analytically unknown (except for Gaussian AR(1) processes),  at the beginning of each local $m$-likelihood. 
%\citet[Section 5]{VDK05} is dedicated to an explicit  verification of the technical conditions (C1--7) when considering an H4M based on a Markovian mixture of two real valued AR(1) processes.

%\subsubsection*{Bibliographic notes: }

%\begin{itemize}
%\item \cite{duarte2017}
%\item \cite{dehling2011}
%\item \cite{DGL2016}
%\end{itemize}
\subsection{Gold standard process poisoning}
In this section, for generality purpose, we will not necessarily suppose that the $Y^0$ and $Y^1$ processes have a  probability density function (pdf) with respect to some reference  measure. We will instead consider the cdf of the processes to describe their random features.  Thus, from now on, the space $E$ is supposed to be a subset of $\R$.
This first point being stated, let us define now our gold standard process poisoning model.   We can simply describe it by pointing  two noticeable departures from the H4M: i) the  so called gold standard process $Y^0$ is no longer a Markov process but more generally  a stationary mixing process which 1st and 2nd order distributions are fully accessible;  ii) the   so called poisoning process $Y^1$  is no longer  Markovian but made of a collection of i.i.d. random variables with unknown common cdf the poisoning sequence $F^1$.
For illustration purpose, one can look at Figure \ref{dessin} and consider that the distribution of the  $Y^0$ process is known but  no longer driven by a Markovian dynamic (omit the transitions $q^0(\cdot|\cdot)$).
Our goal is to estimate, given a sample $(Z_1,\dots,Z_n)$ from the above model,  the transition $\Pi_\theta$  of the underlying Markov chain $X$ along with $F^1$ the unknown cdf of $Y^1$.
Our statistical methodology will be basically grounded on the 1st and 2nd order cdfs of the $Y^0$ and $Y^1$ processes. We will thus consider the following distribution functions 
\begin{eqnarray*}
F^j(x)=\P(Y_i^j\leq x),  \quad j=0,1\quad\mbox{and}\quad G^0(x,y)=\P(Y_i^0\leq x, Y_{i+1}^0\leq y),
\end{eqnarray*}
for all  $i\geq 1$ and $(x,y)\in E^2$. Let us now define the 1st and 2nd order distributions of the observed process $Z$ defined in (\ref{model}) with the above specifications.\\

\noindent{\it \bf First order distribution}. Let us denote by $F(\cdot)$ the stationary cumulative distribution  function of process $Z$ given by
\begin{eqnarray}\label{1order}
F(x)&=&\pi(0) F^0(x)+\pi(1)F^1(x)\nonumber\\
&=&\frac{\beta^*}{\alpha^*+\beta^*} F^0(x)  +  \frac{\alpha^*}{\alpha^*+\beta^*}F^1(x),
\end{eqnarray}
for all $x\in E$. In order to simplify future computation, we denote  
$$p^*=\pi(0)=\frac{\beta^*}{\alpha^*+\beta^*} ,\quad\mbox{and}\quad  r^*=\pi(1)=\frac{\alpha^*}{\alpha^*+\beta^*}= 1-p^*.$$
Thus the stationary distribution vector  of $X$ is $(p^*,r^*)^T$ and 
\begin{eqnarray*}
F(x)=p^* F^0(x)  +  r^*F^1(x),
\end{eqnarray*}
for all $x\in E$, where $F$ is the cdf of $Z$. %
From (\ref{1order}) we can build-up a pseudo-parametric cdf family based on the knowledge (observability) of $F(\cdot)$, which is:
\begin{eqnarray}\label{parametricfamily}
\mathcal F=\left\{ F^1_{\theta}, ~\theta= (\alpha,\beta)\in \Theta   \right\},
\end{eqnarray}
where
%\begin{eqnarray*}
%q^1_{\theta}(u)= \frac{\alpha+\beta}{\beta} \left[  f(u)-\frac{\alpha}{\alpha+\beta} q^0(u) \right].
%\end{eqnarray*}
\begin{eqnarray}\label{cdfinversion}
F^1_{\theta}(x)= \frac1{r} \left(  F(x)-p F^0(x) \right), \quad p=\frac{\beta}{\alpha+\beta} ,\quad r=\frac{\alpha}{\alpha+\beta}= 1-p,
\end{eqnarray}
and the parametric space  is denoted and defined by  $\Theta:=[\delta,1-\delta]^2$, for some $\delta \in ]0,1/2[$.
Let  us observe from (\ref{1order}),  that if  we consider $\theta=\theta^*$ in the above parametrized function (\ref{cdfinversion}), we pointly retrieve  that $F^1_{\theta^*}=F^1$.\\

\noindent{\it \bf Second order distribution}. Let us denote by $G(\cdot,\cdot)$ the 2nd order stationary cdf  of process $Z$ given by
%\begin{eqnarray*}\label{2order}
%g (u,v)
%&=&\frac{\beta^*}{\alpha^*+\beta^*}(1-\alpha^*)f^0(u)q^0(v|u)+\frac{\beta^*}{\alpha^*+\beta^*}\alpha^* f^0(u)f^1(v)\nonumber\\
%&&+\frac{\alpha^*}{\alpha^*+\beta^*}\beta^* f^1(u)f^0(v)+\frac{\alpha^*}{\alpha^*+\beta^*}(1-\beta^*)f^1(u)q^1(v|u),
%\end{eqnarray*}
%for all $(u,v)\in E^2$. 
\begin{eqnarray*}\label{2order}
G (x,y)
&=&\pi(0)\Pi(0,0)G^0(x,y)+\pi(0)\Pi(0,1)F^0(x)F^1(y)\nonumber\\
&&+\pi(1)\Pi(1,0) F^1(x)F^0(y)+\pi(1)\Pi(1,1) F^1(x)F^1(y),
\end{eqnarray*}
for all $(x,y)\in E^2$. 
Now let us  denote 
$$\lambda_1=\frac{\beta(1-\alpha)}{\alpha+\beta},\quad 
\lambda_2=\lambda_3=\frac{\alpha \beta}{\alpha+\beta},\quad
\lambda_4=1-\lambda_1-2\lambda_2=\frac{\alpha(1-\beta)}{\alpha+\beta},$$
which leads to naturally consider  the true corresponding versions of these parameters (when $\theta=\theta^*$)
$$\lambda_1^*=\pi(0)\Pi(0,0),\quad 
\lambda_2^*=\pi(0)\Pi(0,1), \quad \lambda_3^*=\pi(1)\Pi(1,0),\quad
\lambda_4^*=\pi(1)\Pi(1,1).$$
Note that $\lambda_1^*+\lambda_2^*=p^*$ and $\lambda_3^*+\lambda_4^*=r^*$. By this reparametrization we obtain %the following
 a more interpretable/tractable representation of the  $Z$ process 2nd order cdf:
\begin{eqnarray}\label{G-decomp}
G=\lambda_1^* G^0+\lambda_2^* F^0\otimes F^1+\lambda_3^* F^1\otimes F^0+\lambda_4^* F^1\otimes F^1
\end{eqnarray}
with the notation $(\phi\otimes\psi)(x,y)=\phi(x)\psi(y)$, for all $(x,y)\in E^2$.

\section{Identifiability and parametric estimation}
\label{sec:est}
%\bigskip

We are now able to propose two discrepancy functions, denoted $ \mathbf{d}(\cdot)$ and $ \mathbf{s}(\cdot)$, for the parametric part of our estimation problem. Indeed consider the so called  {\it $(\theta^*,\theta)$-deviation quantity}, comparing the true 2nd order distribution of $Z$ with its natural (\ref{G-decomp}) based  reconstruction under parameter $\theta$ with $F^1_\theta$ picked in the parametric family $\mathcal F$, see expression (\ref{parametricfamily}),   defined by 
\begin{eqnarray}\label{deviation_quant}
& \Delta(\theta,x,y)=G(x,y)
-\underbrace{\left [\lambda_1 G^0(x,y)+\lambda_2 F^0(x)F^1_\theta(y)
+\lambda_3F^1_\theta(x)F^0(y)+\lambda_4F^1_\theta(x)F^1_\theta(y)\right ]},\\
&~~~~~~~~~~~~~~~~~~~~~~~~~~~~~~\mbox{parametric $\theta$-reconstruction of $G$}\nonumber 
\end{eqnarray}
and,  for  a given finite weight measure $dH$, consider 
\begin{eqnarray}\label{contrasts}
 \mathbf{d}(\theta)=\iint \Delta^2(\theta,x,y) dH(x,y), \quad 
\mathbf{s}(\theta)=\sup_{(x,y)\in E^2} |\Delta(\theta,x,y)|,
% \qquad
% \bar S(\theta)=\sup_{x,y\in E} \Delta^2(\theta,x,y)\right]
\end{eqnarray}
where the $\iint $ symbol stands for $\int_{E\times E}$. For sake of simplicity and without loss of generality  we will suppose that $\iint dH(x,y)=1$.\\
We will focus from now on  our study on the discrepancy function $\mathbf{d}(\cdot)$ but $\mathbf{s}(\cdot)$ will basically inherit the same consistency properties except the $\sqrt{n}$-convergence which requires some smoothness properties that unfortunately do not hold for $\mathbf{s}(\cdot)$. We can observe that  $\mathbf{d}(\cdot)$ is a non-negative function   that satisfies the following property:
\begin{center}
{\it { if} $\theta=\theta^*$ {then}   $ \Delta(\theta,x,y)=0$ for almost all  $(x,y)\in E^2$, and  $ \mathbf{d}(\theta^*)=0=\mathbf{s}(\theta^*)$}. 
\end{center}
The following proposition ensures that the converse is true  under a certain condition.
\begin{proposition} \label{contrast}
The $(\theta^*,\theta)$-deviation quantity defined in (\ref{deviation_quant}) can be expressed  as follows
\begin{equation}\label{Deltabar}
\Delta(\theta,\cdot)=c_1(\theta^*,\theta) (G^0-F^0\otimes F^0)+c_2(\theta^*,\theta) (F^1-F^0)\otimes (F^1-F^0)
\end{equation}
where $c_1$ and $c_2$ are only depending on $\theta$ and $\theta^*$, see  (\ref{c1c2}) for close form expressions. Moreover, we have the implication 
$$\left\{c_1(\theta^*,\theta)=0\quad \text{and}\quad c_2(\theta^*,\theta)=0 \right\}\Longrightarrow \theta=\theta^*.$$
\end{proposition}
This proposition is proved in Section~\ref{proofcontrast}. This leads us to introduce the fundamental identifiability condition.\\

\noindent \textbf{ Assumption LinInd}: 
The family $
\left\{G^0-F^0\otimes F^0,\, (F^1-F^0)\otimes (F^1-F^0)\right\}$ is linearly independent {in the following sense:
for all set $\Gamma$ %\subset E\times E$ 
such that $\iint_{\Gamma} dH=1$, the two functions restricted to $\Gamma$ are linearly independent}.\\

In a mixture model $pF^0+(1-p)F^1$, it is natural to assume that $F^1\neq F^0$ to ensure identifiability. The case where $G^0=F^0\otimes F^0$ corresponds to an independent process $Y^0$ (in this case $Z$ is a hidden Markov chain).  {Note that this independent case is excluded from our study.} Under assumption  {\bf LinInd}, the parameter value $\theta^*$ is the unique minimizer of $\mathbf{d}(\cdot)$, {\it i.e.} we have:
\begin{equation}\label{injectivitecontraste}
\mathbf{d}(\theta)=0\Longleftrightarrow \theta=\theta^*.
\end{equation}   
{When considering the supremum discrepancy $\mathbf{s}(\cdot)$, it is sufficient to assume the linear independence in the basic sense.  We directly have that
$$\mathbf{s}(\theta)=0 \Longleftrightarrow\theta=\theta^*.$$
For the  integral discrepancy $\mathbf{d}(\cdot)$, we have to ensure that $\iint \Delta^2dH=0$ implies $\Delta=0$, hence our assumption  needs to hold over sets of $H$-measure 1.  We can alternatively suppose the linear independence in the basic sense,  assuming in addition that  $F^1$ and $G^0$ are continuous and $dH$ admits a continuous density with respect to the Lebesgue measure}.\\

{
Since $\Delta$ only depends on known quantities $F^0,G^0$ and on the distribution of the observations, the equivalence 
$\Delta(\theta,\cdot)=0\Longleftrightarrow \theta=\theta^*$ allows us to write the following result.
 \begin{proposition}
 Under  {\rm\textbf{LinInd}}, $\theta^*$ and then $F^1$ are identifiable.
\end{proposition}
}

{It is an already observed phenomenon that dependence of latent variables allows identifiability in mixture models  that are not identifiable in the dependent case, see \cite{gassiat19}.  Informally,  dependence allows to store  more information (implicit identifiability constraints)  than independence. Our study is in line with these results, which also require linear independence assumptions.}\\

Empirical versions of the contrast functions   $\mathbf{d}$ and $\mathbf{s}$ can now be proposed
\begin{eqnarray}\label{emp_d_s}
{\mathbf d}_n(\theta)=\iint \Delta_n^2(\theta,x,y) dH(x,y),   \quad \mathbf{s}_n(\theta)=\sup_{(x,y)\in E^2} |\Delta_n(\theta,x,y) |,
\end{eqnarray}
where  for all $(x,y)\in E^2$, the empirical $(\theta^*,\theta)$-deviation version of (\ref{deviation_quant}) is defined by
\begin{eqnarray*}
\Delta_n(\theta,x,y)=\hat G_n(x,y)-\left[\lambda_1G^0(x,y)-\lambda_2 F^0(x)\hat F^1_{n,\theta}(y)-\lambda_3\hat F^1_{n,\theta}(x)F^0(y)-\lambda_4\hat F^1_{n,\theta}(x)\hat F^1_{n,\theta}(y)\right],
\end{eqnarray*}
where 
\begin{equation} \label{inv-plug-emp-cdf}
\hat F^1_{n,\theta}(x)=\frac{1}{r}\left(\hat F_n(x)-pF^0(x)\right),\quad x\in E,
\end{equation} 
with the standard  1st and 2nd order empirical cdfs
\begin{eqnarray}
\hat F_n(x)=\frac{1}{n}\sum_{i=1}^{n}\I_{\left\{Z_i\leq x\right\}}, \quad
\hat G_n(x,y)=\frac{1}{n-1}\sum_{i=1}^{n-1} \I_{\left\{Z_i\leq x,Z_{i+1}\leq y\right\}},
\end{eqnarray}
for all $(x,y)\in \R^2$. We finally consider two possible parametric estimators of $\theta$:
\begin{equation}
\label{defestimateur}
\hat{\theta}_n=\argmin_{\theta\in \Theta} \mathbf{d}_n(\theta), \quad \mbox{and}\quad \tilde{\theta}_n=\argmin_{\theta\in \Theta} \mathbf{s}_n(\theta).
\end{equation}

\begin{theorem} \label{theo:CVestim}
Assume that the sequence $(Z_i)_{i\geq 1}$ is stationary and assumption {\rm\textbf{LinInd}} holds, then our both estimators, defined in (\ref{defestimateur}),  are strongly consistent, {\it i.e.}
  $$\hat \theta_n\stackrel{a.s.}{\longrightarrow} \theta^*\quad  \text{and} \quad  \tilde \theta_n\stackrel{a.s.}{\longrightarrow} \theta^*, \quad \text{ as }\quad n\rightarrow +\infty.$$
\end{theorem}
\begin{proof}
The convergence of $\hat \theta_n$ in Theorem~\ref{theo:CVestim}  is proved using  a classical result about  the minimum contrast estimators theory, see \cite{vandervaart} or \cite{DCD}, using equivalence \eqref{injectivitecontraste}, Lemma~\ref{lipschitz} (regularity of $\mathbf{d}(\cdot)$ and $\mathbf{d}_n(\cdot)$, see Section~\ref{sec:preuveconsistance}) and Propositions~\ref{convergencecontraste} and \ref{GlivenkoCantelli} (uniform convergence of $\mathbf{d}_n(\cdot)$ towards $\mathbf{d}(\cdot)$, see Section~\ref{sec:preuveconsistance}). The same type of proof holds for $\tilde \theta_n$.
\end{proof}

{Note that we could have use characteristic functions to define our  constrasts,  instead of distribution functions. The algebraic computations are the same, as well as Proposition~\ref{contrast}.  Nevertheless the consistency requires a Glivenko-Cantelli theorem,  which does not exist in the same way for characteristic functions.  Denoting $c_n$ the empirical characteristic function and $c=\bE(c_n)$ the true characteristic function, the convergence of $c_n$ toward $c$ does not hold on the whole real line, see \cite{FeuervergerMureika}.  Moreover,  the process $\sqrt{n}(c_n-c)$ converges only under specific assumptions,  see \cite{Csorgo} and this convergence is required pointly  to prove  the asymptotic normality, that we will study in the next section.}

\section{Asymptotic Normality}
\label{sec:asymptoticnormality}
In order to establish the $\sqrt{n}$-consistency of the minimum contrast estimator $\hat \theta_n$ associated with ${\mathbf d}(\cdot)$, we need  to introduce an additional  stationarity/mixing condition about the stochastic processes $X$ and $Y^0$, $Y^1$ being just made of i.i.d random variables,  involved in our model. \\
Let us define, for any generic stationary process $\tilde Y$ and all $(t,k)\in \mathbb{N}^*\times \mathbb{N}$, the sequence of  $\alpha$-mixing coefficients associated  to the stochastic process $\tilde Y$ by:
\begin{eqnarray}\label{alphacoef}
\alpha^{\tilde Y}(k)=\sup_{A\in \mathcal{F}^t_{\tilde Y,1},B\in \mathcal{F}^{\infty}_{\tilde Y,t+k+1}}|\P(A\cap B)-\P(A)\P(B)|,\quad k\geq 1,
\end{eqnarray}
where $\mathcal{F}^{t_2}_{\tilde Y,t_1}$ denotes, for all $t_2>t_1\geq 1$, the $\sigma$-algebra generated by $(\tilde Y_{t_1},\dots,\tilde Y_{t_2})$.

\noindent \textbf{Assumption Mix}: The stochastic processes $X$ and $Y^0$ are  {strictly stationary and $Y^0$ is}  $\alpha$-mixing with $\alpha^{Y^0}(k)=O(k^{-a})$ for  $a>1$.

%(Remarque : pour une chaine a espace d'etats finis, $\alpha_n\to 0$ equivaut a etre irreducible and aperiodic (Bradley 2005))

\begin{theorem} \label{theo:normasymp}
Assume that $F$ is continuous. Assume also {\rm\textbf{LinInd}} and {\rm\textbf{Mix}}. 
Assume that $\theta^*$ is an interior point of $\Theta=[\delta,1-\delta]^2$.
Then we have the following central limit behavior
$$\sqrt{n}\left( \hat \theta_n -\theta^*\right)
\stackrel{d}{\longrightarrow}
%{\rightsquigarrow}
 \mathcal N(0, \Sigma),\quad \mbox{as}\quad n\rightarrow +\infty,$$  
where the covariance matrix  $\Sigma$  is detailed in Section~\ref{preuvecov}.
\end{theorem}

A crucial step in  the proof of this theorem  is the use of  empirical Central Limit Theorem for process $Z$ along with the  bivariate process $(Z_i,Z_{i+1})_{\i\geq 1}$.  The continuity of $F$ is assumed to use the CLT result of \cite{Rio2014}, which also requires $\alpha$-mixing.
Note that the $\alpha$-mixing assumption is the weakest mixing condition insofar as  all other mixing ($\beta,\rho, \phi, $ etc.) imply $\alpha$-mixing. Moreover note that  the required $\alpha$-mixing rate decrease is rather slow (polynomial in $n$).

Let us mention  the more modern notion of weak dependence, which gather dependence conditions that are weaker than mixing.  Empirical Central Limit Theorems have been proved for weak dependent processes, typically for $\alpha$, $\beta$ and $\theta$ weak dependence, see \cite{DDLLLP} where examples of such weak dependent sequences are also given.  Nevertheless the implication $Y^0$ weak-dependent $\Rightarrow$ $(Z_i,Z_{i+1})_{i\geq 1}$ weak-dependent is long to detail,  especially for "two-points-in-the-future" coefficients $\tilde\alpha_2, \tilde\beta_2, \theta_2$.  Moreover the condition of decrease is more restrictive since it is $n^{-a}$ with $a>4$ or $a>8$ according to coefficients (the better condition of  \cite{dedecker2010} with  $a>1$ is published only for univariate sequences).

\section{Nonparametric estimation}
\label{sec:nonparametric}
Once the parametric part of the model is estimated, one can recover the nonparametric part $F^1$ by using a classical inversion formula. Indeed we can estimate 
the nonparametric part by considering, based on expression (\ref{inv-plug-emp-cdf}),  the following  plug-in estimator 
\begin{eqnarray}\label{plug-in}
\hat F_n^1(x)=\hat F^1_{n,\hat \theta_n}(x)=\frac{1}{\hat r_n} \left(\hat F_n(x)-\hat {p}_n F^0(x)  \right),\quad \mbox{$x\in E$},
\end{eqnarray}
with $\hat r_n=1-\hat p_n = \hat\alpha_n/(\hat\alpha_n+\hat\beta_n)$. We can expect, because  $\hat \theta_n\rightarrow \theta^*$ almost surely as $n\rightarrow +\infty$,  that $\hat F_n^1=\hat F^1_{n,\hat \theta_n}$ will converge to $F^1_{\theta^*}=F^1$ with the classical convergence rates. This point is stated in the two following theorems.

\begin{theorem} \label{theo:normasymp}
Assume that $F$ is continuous. Assume also {\rm\textbf{LinInd}} and {\rm\textbf{Mix}} and that $\theta^*$ is an interior point of $\Theta$.
Then we have the convergence in distribution,  in the space $B(\R^2)$ of real-valued and bounded functions over $\R^2$
\begin{eqnarray}
\sqrt{n}\left( \hat F_n^1 -F^1\right)\rightsquigarrow \mathcal{G}, \quad \mbox{as} \quad n\rightarrow +\infty,
\end{eqnarray}  
{where $\mathcal{G}$ is  a zero-mean  Gaussian process.}
\end{theorem}

\begin{theorem} \label{theo:risqueF}
Assume that $F$ is continuous. Assume also {\rm\textbf{LinInd}} and {\rm\textbf{Mix}} with $a>4$ and that $\theta^*$ is an interior point of $\Theta$.
Then there exists a positive constant $C$ depending on $\theta^*,G^0, F^1, \delta $ such that
$$
\bE\|\hat \theta_n -\theta^*\|\leq \frac{C}{\sqrt{n}}, 
$$
where $\|\cdot\|$ is the Euclidean norm, and
$$\bE\|\hat F_n^1 -F^1\|_{\infty}\leq \frac{C}{\sqrt{n}}.$$
\end{theorem}

This result is not a straightforward corollary of the previous one but requires specific processing, see the proof in Section~\ref{preuverisqueF}. Using \cite{DRM14}, the mixing condition can be replaced by: 
$Y^0$ is  $\beta$-mixing with $\beta_n=O(n^{-a})$ for  $a>1$.
We can also prove that, for all integer $q\geq 1$, $\bE(\|\hat \theta_n -\theta^*\|^q)\leq {C(q)}n^{-q/2}$ and $\bE(\|\hat F_n^1 -F^1\|_{\infty}^q)\leq {C(q)}n^{-q/2}$.

\bigskip

One interesting lead of research, which is out of the reach of this paper, would be the nonparametric estimation of the $Y^1$ sequence density , denoted $f^1$. 
The natural candidate to estimate $f^1$, is the following plug-in semiparametric estimator
\begin{eqnarray}\label{inversed_pdf}
\hat f^1_n(x)=\frac{1}{\hat r_n}\left( \hat f_n(x)-\hat p_nf_0(x)\right), \quad \mbox{where}\quad \hat p_n=\frac{\hat \beta_n}{\hat \alpha_n+\hat \beta_n}\quad \mbox{and}\quad  \hat r_n=\frac{\hat \alpha_n}{\hat \alpha_n+\hat \beta_n}
\end{eqnarray}
and $\hat f_n(x)=(nh_n)^{-1}\sum_{i=1}^n K_{h_n}(x-Z_i)$, where $K(\cdot)$ is a probability kernel,  $K_h(\cdot)=h^{-1}K(\cdot/h)$, for all $h>0$, and $h_n$ is a bandwidth  parameter that goes to zero as $n$ goes to infinity.
In that setup we could indeed investigate the local decoding problem which is the prediction of the fact that observations $Z_i$ is poisoned  (not generated from the gold standard) or not.  In fact the conditional distribution of the Markov  latent  couple  $(X_{2i},X_{2i+1})$ associated with the observed couple $(Z_{2i},Z_{2i+1})$, for $i\geq 1$,  is given for all $(k,l)\in \{0,1\}^2$ and all $(z,z')\in E^2$,  by
\begin{eqnarray}\label{decoding}
\P_{\theta,f}(X_{2i}=k,X_{2i+1}=l |Z_{2i}=z,Z_{2i+1}=z'):=\frac{\eta_{\theta,f^1}(k,l;z,z')}{\P_{\theta,f}(Z_{2i}=z,Z_{2i+1}=z')}, \nonumber
%\frac{\P_\theta(\left\{X_{2i}=k,X_{2i+1}=l\right\}\cap \left\{Z_{2i}=z,Z_{2i}=z')\right\})}{\P_\theta(Z_{2i}=z,Z_{2i}=z')} \nonumber
\end{eqnarray}
where 
\begin{eqnarray*}
\P_{\theta,f}(Z_{2i}=z,Z_{2i}=z')=\lambda_1g^0(z,z')+\lambda_2f^0(z)f(z')+\lambda_3f(z)f^0(z')+\lambda_4 f(z)f(z'),
\end{eqnarray*}
and 
\begin{eqnarray*}
&&\eta_{\theta,f}(0,0;z,z')=\lambda_1g^0(z,z'),\quad \eta_{\theta,f}(0,1;z,z')=\lambda_2f^0(z)f(z'),\\
&&\eta_{\theta,f}(1,0;z,z')= \lambda_3f(z)f^0(z'), \quad \eta_{\theta,f}(1,1;z,z')=\lambda_4 f(z)f(z').
\end{eqnarray*}
Suppose now that we run our estimation method and get an estimator $\hat \theta_n$ of $\theta^*$ along with a density estimator $\hat f^1_n$ of $f^1$, we can easily estimate 
, for all $(k,l)\in \{0,1\}^2$ and all $(z,z')\in E^2$, the true decoding probabilities $\P_{\theta^*,f^1}(X_{2i}=k,X_{2i+1}=l|Z_{2i}=z,Z_{2i+1}=z')$ by
\begin{eqnarray}\label{predict}
\mbox{predict}_i((k,l)|z_1^n):=\P_{\hat \theta_n,\hat f^1_n}(X_{2i}=k,X_{2i+1}=l|Z_{2i}=z,Z_{2i+1}=z').
\end{eqnarray}
Note that we could also predict sequences of length greater than two by deriving corresponding conditional probabilities in the spirit of (\ref{decoding}) and the following material (obviously heavier to compute).

\section{Numerical performances}
\label{num_perf}
In this section we propose to investigate the numerical performances of our semiparametric estimation method developed for model (\ref{model})
in various situations enhancing more or less the observability and mixing properties of the involved processes.  For this purpose we propose to consider the following models
\begin{eqnarray}
\quad Z_j=(1-X_j)Y_{j}^0+X_jY_{j}^1,\quad j=1,\dots,n
 \end{eqnarray}
where for $k=1,\dots,n-1$ and an i.i.d. sequence of Gaussian noises $(\varepsilon_k)_{k\geq 1}$  drawn from a ${\mathcal N} (m_0,v_0^2)$ distribution:
\begin{eqnarray*}
X_{1:n}&\sim& \mbox{Markov}(n,\Pi_{\theta},\pi_\theta):~X_1\sim \pi_\theta, \P_\theta(X_{k+1}=\cdot |X_k=\cdot )=\Pi_{\theta}(\cdot,\cdot)\\
                  Y_{1:n}^0&\sim& \mbox{AR(1)}(n,\varphi,m_0,v_0): Y_1^0\sim \mathcal{N}\left(\mu_0,\var_0 \right),~Y_{k+1}^0=\varphi Y_{k}^0+\varepsilon_k\\
                  Y_{1:n}^1&\sim& \mathcal{N}^{{\otimes}n}(m,v^2).
                \end{eqnarray*}
The parameter $\varphi\in ]0,1[$,  is the regression coefficient  of the AR(1) process when $(\mu_0,\var_0)=\left (\frac{m_0}{1-\varphi}, \frac{v_0^2}{1-\varphi^2}\right)$ are respectively the stationary  (marginal) mean and variance of the process $Y_{0}$. All along  our simulations we will take for simplicity $(m_0,v_0^2)=(1,1)$. 
To practically estimate/compute, as we should do when no closed-form  expression is available,  the 1st and 2nd order cdfs $F^0(\cdot)$ and $G^0(\cdot,\cdot)$  associated to the known gold standard  process we generate two independent i.i.d. samples of size $N$
\begin{eqnarray*}
&&Y^0[i]\sim \mathcal{N}(\mu_0,\var_0)\\
&&(Y^0_{1},Y^0_{2})[i]\quad \mbox{where}\quad Y^0_{1}[i]\sim \mathcal{N}(\mu_0,\var_0)\quad \mbox{and }Y^0_{2}[i]=\varphi Y^0_{1}[i]+\varepsilon[i],
\end{eqnarray*}
for $i=1,\dots,N$, where $(\varepsilon[i])_{1\leq i\leq N}$ is an i.i.d. sequence drawn from the ${\mathcal N} (m_0,v_0^2)$ distribution. From these samples we compute 
\begin{eqnarray}\label{monte_carlo_cdf_F0_G0}
\hat F_N^0(x)=\frac{1}{N}\sum_{i=1}^N {\mathbb I}_{\{Y^0[i]\leq x\}},\quad \hat G_N^0(x,y)=\frac{1}{N}\sum_{i=1}^N {\mathbb I}_{\{(Y^0_{1},Y^0_{2})[i]\leq (x,y)\}}, \quad (x,y)\in E^2,
\end{eqnarray}
which are uniformly strongly consistent estimators of the first and second order cdfs of the known AR(1)-process $Y^0$. Note that for $N$  large  enough with respect to $n$  we can achieve a satisfactory level of precision in regard of  the stochastic  fluctuations involved in the $n$-empirical contrasts defined in (\ref{emp_d_s}). In our simulations we considered $N=2n$ with satisfactory results in terms of computing time and accuracy.  The $H$ distribution considered in our simulation is the  $\mathcal U^{\otimes 2}$ where $\mathcal U$ denotes the  uniform distribution over $[\min_{1\leq i\leq n}(Z_i),\max_{1\leq i\leq n}(Z_i)]$ in order to get the most neutral weight function as possible. Note that any heavily tailed distribution over $\R$ would have done the same job.

\subsection{Two  strong observability cases}
We consider first two setups based on a strong observability of the contaminant outputs with expectation $2.5\times\mu_0$ and  standard deviation $0.8\times \sqrt{\var_0}$ (way higher than the AR(1)  process marginal expectation and smaller standard deviation) and weak  or strong  qualitative  mixing  properties of the latent Markov chain.
\begin{itemize}
\item  $\mbox{(S0)}_{strong}$: $\theta=(0.7,0.8)$, $\varphi=0.7$ and $(m,v)=(2.5\times \mu_0,0.8\times \sqrt{\var_0})$.
\item $\mbox{(S0)}_{weak}$: $\theta=(0.3,0.2)$, $\varphi=0.7$ and $(m,v)=(2.5\times \mu_0,0.8\times \sqrt{\var_0})$.
%\item (S5): $\theta=(0.6,0.3)$, $\varphi=0.5$ and $(m,v)=(2\times \mu_0,0.8\times \sqrt{\var_0})$.
\end{itemize}
For illustration purpose we display in red in Figures \ref{Fig:S0strong} and \ref{Fig:S0weak} trajectories corresponding respectively to model $\mbox{(S0)}_{strong}$ and $\mbox{(S0)}_{weak}$. Note that we display  in blue an extra  informative dummy process which is $X1=X\times \bE(Y^1)$, where $X$ denotes the non-observed 2-state Markov chain. Note that $X1$ must not be confused with $X_1$ the first observation of the Markov process $X$.
This way, when $\left\{X1=0\right\} $ we are informed  that we observe the known AR(1)-process $Y^0$, in contrast, when  $\left\{X1=m\neq 0\right\}$ we know that we observe the i.i.d. process $Y^1$ centered at $m=\bE(Y^1)$.
\begin{figure}[!htb]
\begin{minipage}{0.48\textwidth}
     \centering
     \includegraphics[width=7cm,height=4cm]{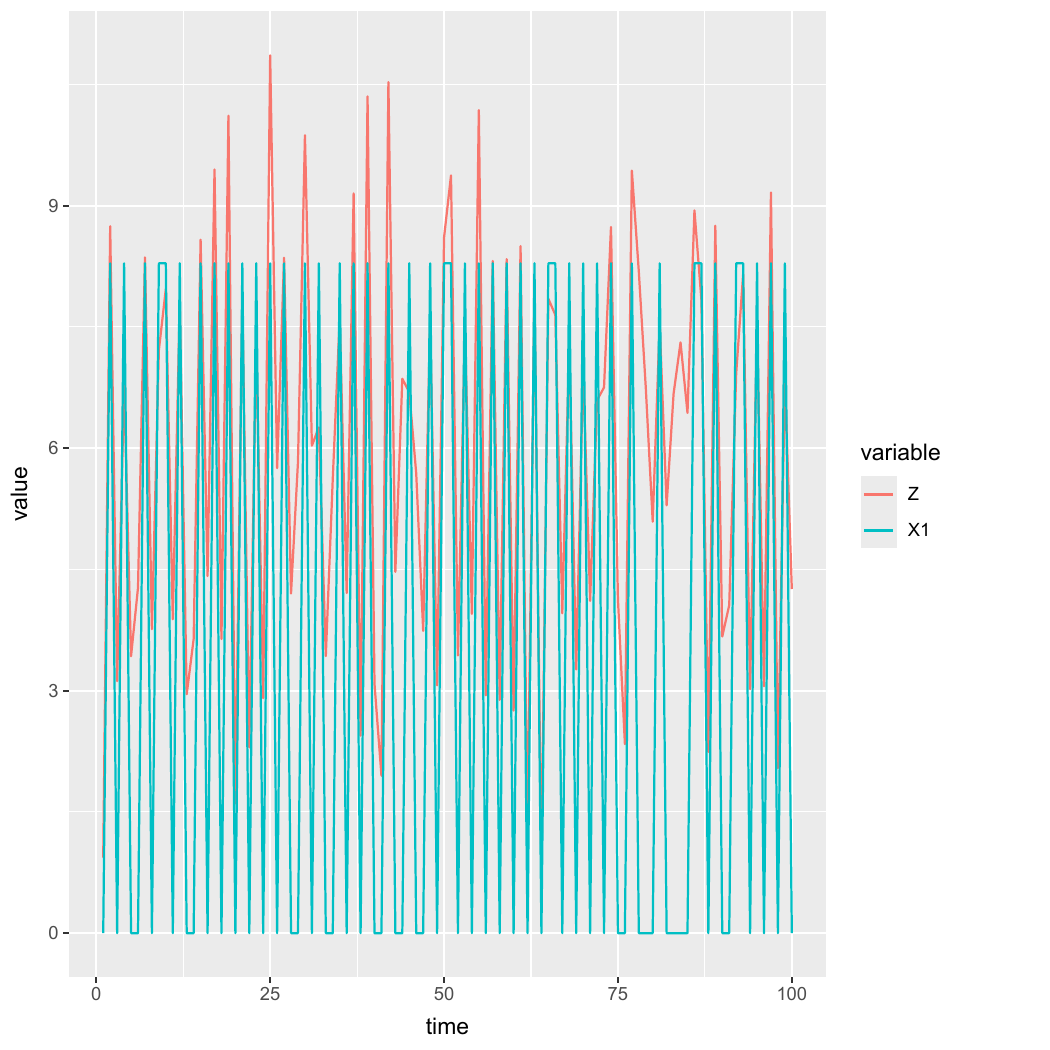}
     \caption{$\mbox{(S0)}_{strong}$ trajectory}\label{Fig:S0strong}
   \end{minipage}\hfill
   \begin{minipage}{0.48\textwidth}
     \centering
     \includegraphics[width=7cm,height=4cm]{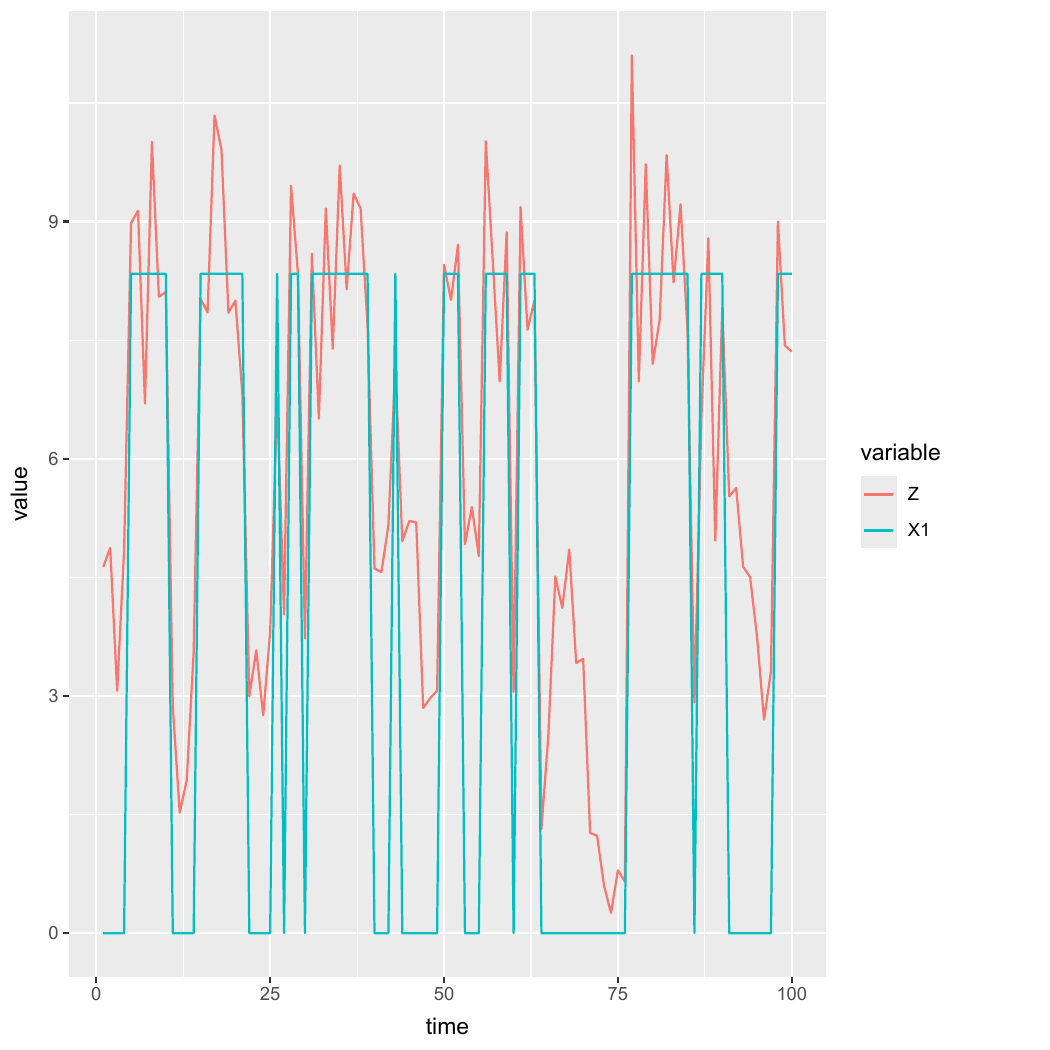}
     \caption{$\mbox{(S0)}_{weak}$ trajectory}\label{Fig:S0weak}
   \end{minipage}
\end{figure}\\
In Table \ref{tab:performanceS0_strong_weak} we provide the bias and standard deviation based performance results collected from  a 100-repetition Monte Carlo scheme (used consistently in this section except if explicitly mentioned),  of our estimators $\hat \theta_n=(\hat \alpha_n,\hat \beta_n)$ and $\tilde \theta_n=(\tilde \alpha_n,\tilde \beta_n)$, respectively associated with  discrepancies $\mathbf{d}(\cdot)$ (integration-based)
and $\mathbf{s}(\cdot)$ (sup-based).
\begin{table}[!b]
		% \usepackage{multirow}
		\centering
	\scalebox{1}{
		\begin{tabular}{ccc|cc}
		\hline
			 $\mbox{(S0)}_{strong}$ & $\hat \alpha_n$& $\hat \beta_n$& $\tilde \alpha_n$&$\tilde \beta_n$ \\
			\hline
			$n=1,000$    & (0.036,0.118)&(-0.004,0.023)& (0.017,0.127)&(-0.007,0.035)\\ 
			$n=3,000$    &  (0.034,0.056)& (-0.010,0.013)& (0.015,0.083)&(-0.005,0.019)\\  
			$n=5,000$    &  (-0.007,0.052)& (0.001,0.010)& (0.028,0.078)&(-0.002,0.013)\\   
			\hline
		 $\mbox{(S0)}_{weak}$ & $\hat \alpha_n$& $\hat \beta_n$& $\tilde \alpha_n$&$\tilde \beta_n$ \\
			\hline
			$n=1,000$    & (-0.218,0.116)      &(-0.134,0.096)            & (-0.048,0.057)&(-0.024,0.111)\\ 
			$n=3,000$   &  (-0.051,0.103)     & (-0.010,0.114)           & (-0.012,0.043)&(0.011,0.069)\\  
			$n=5,000$    &  (-0.051,0.101)        & (-0.011,0.112)           & (0.002,0.033)& (0.015,0.065)\\       
			 \end{tabular}
		}
	\caption{Bias and standard deviation based performances  under model $\mbox{(S0)}_{strong}$ and $\mbox{(S0)}_{weak}$,  for $n=1,000$, $n=3,000$ and $n=5,000$ (in rows) of estimators
	$\hat \theta_n=(\hat \alpha_n,\hat \beta_n)$ and $\tilde \theta_n=(\tilde \alpha_n,\tilde \beta_n)$, respectively associated with  discrepancies
	 $\mathbf{d}(\cdot)$ (integration-based)
and $\mathbf{s}(\cdot)$ (sup-based), are respectively displayed in the last two columns.}\label{tab:performanceS0_strong_weak}
\end{table}\\
Finally to illustrate the bivariate asymptotically Gaussian behavior of our estimator $\hat \theta=(\hat \alpha,\hat \beta)$ we display in Figure \ref{fig:TCL2D} the $\sqrt{n}$-centered sample of estimators we obtain when we run $\ell=10,000$ estimations on simulated $\mbox{(S0)}_{strong}$ models, {\it i.e.}  $(\sqrt{n}[\hat \theta_{i,n}-\frac{1}{n}\sum_{j=1}^\ell \hat \theta_{j,n}])_{1\leq i\leq \ell}$ where  $\hat \theta_{i,n}$ denotes the $i$-th estimation output, $i=1,\dots,\ell$,  based on $n$ observations  with $n=1,000$, $3,000$ and $5,000$. For simplicity matters we only keep 4 digits  in our estimation statistics (bias,  variance or standard deviation).

\begin{figure}[!htb]
\includegraphics[width=5cm,height=4cm]{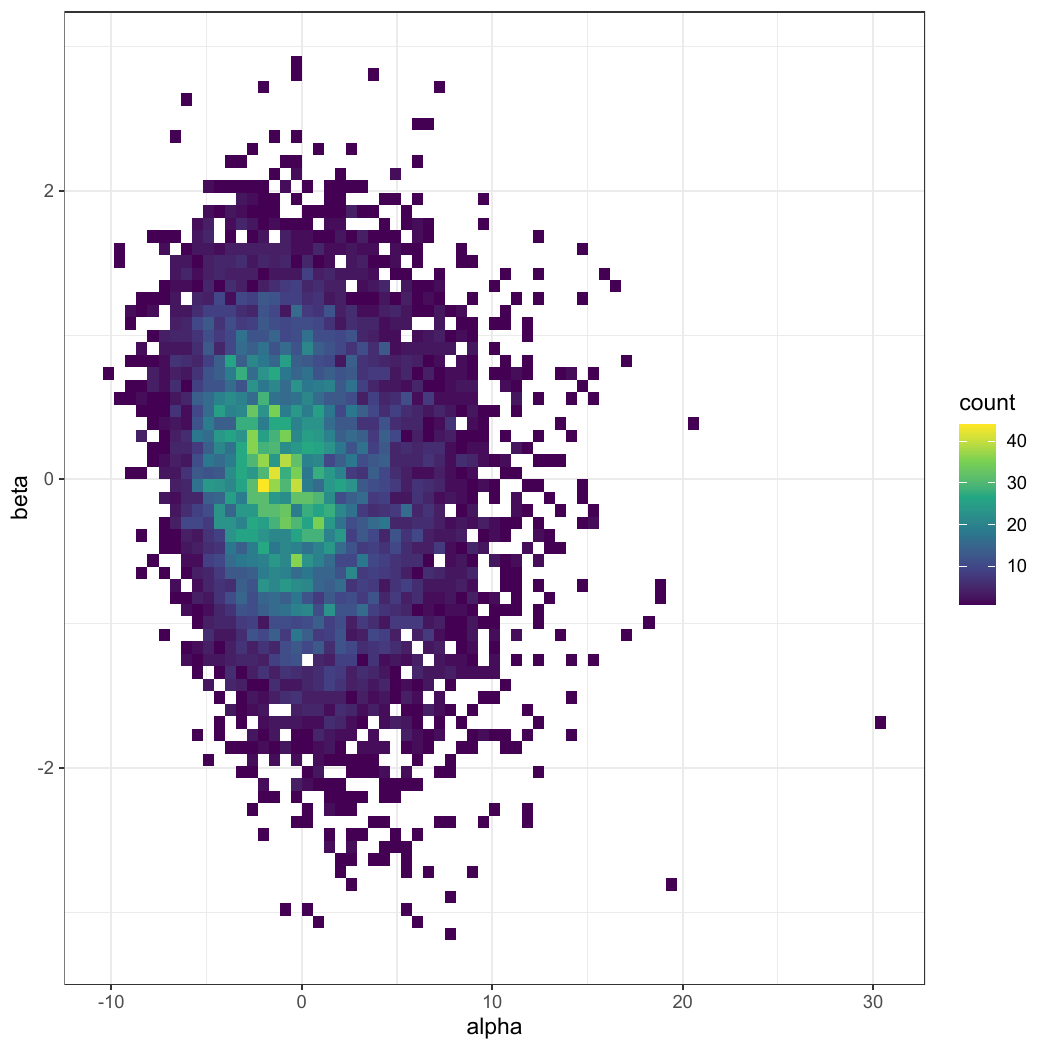}
\includegraphics[width=5cm,height=4cm]{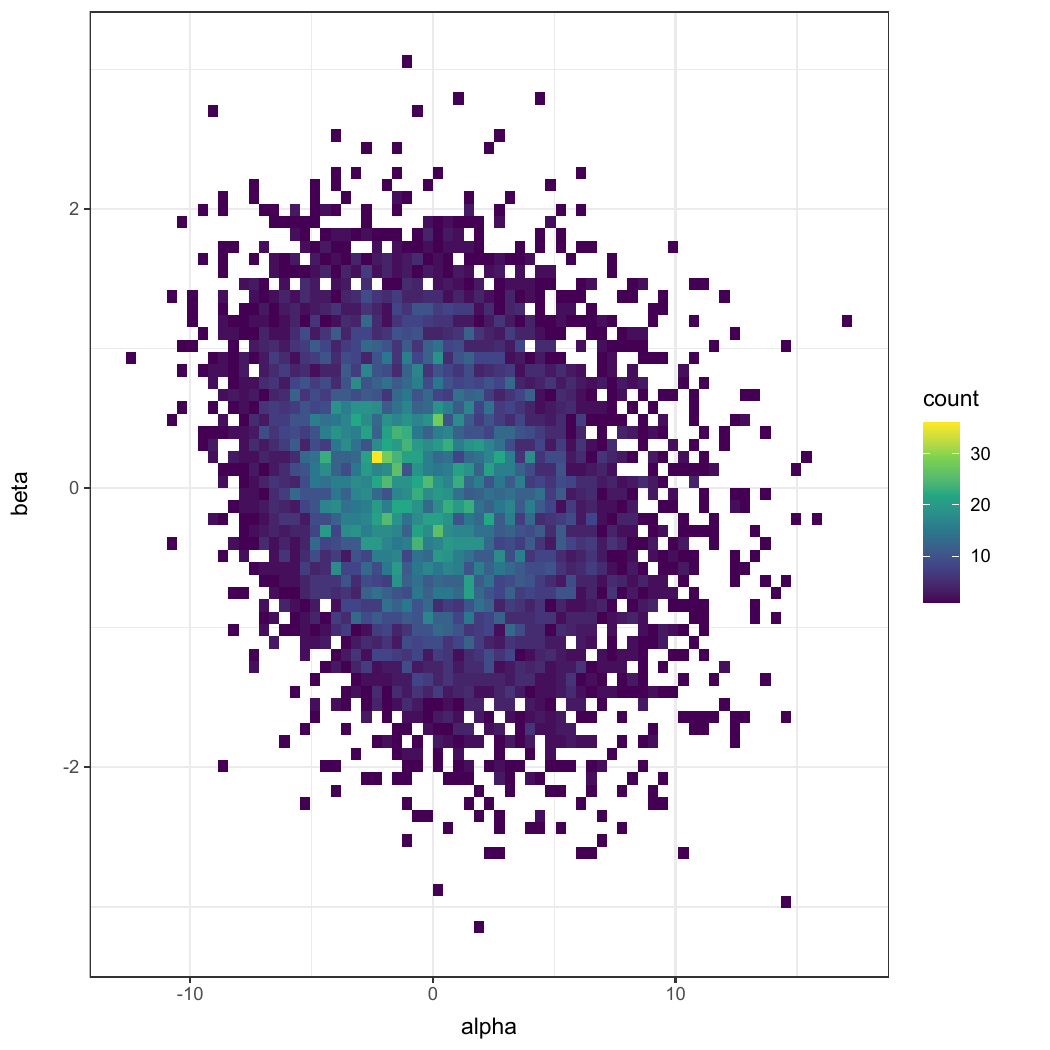}
 \includegraphics[width=5cm,height=4cm]{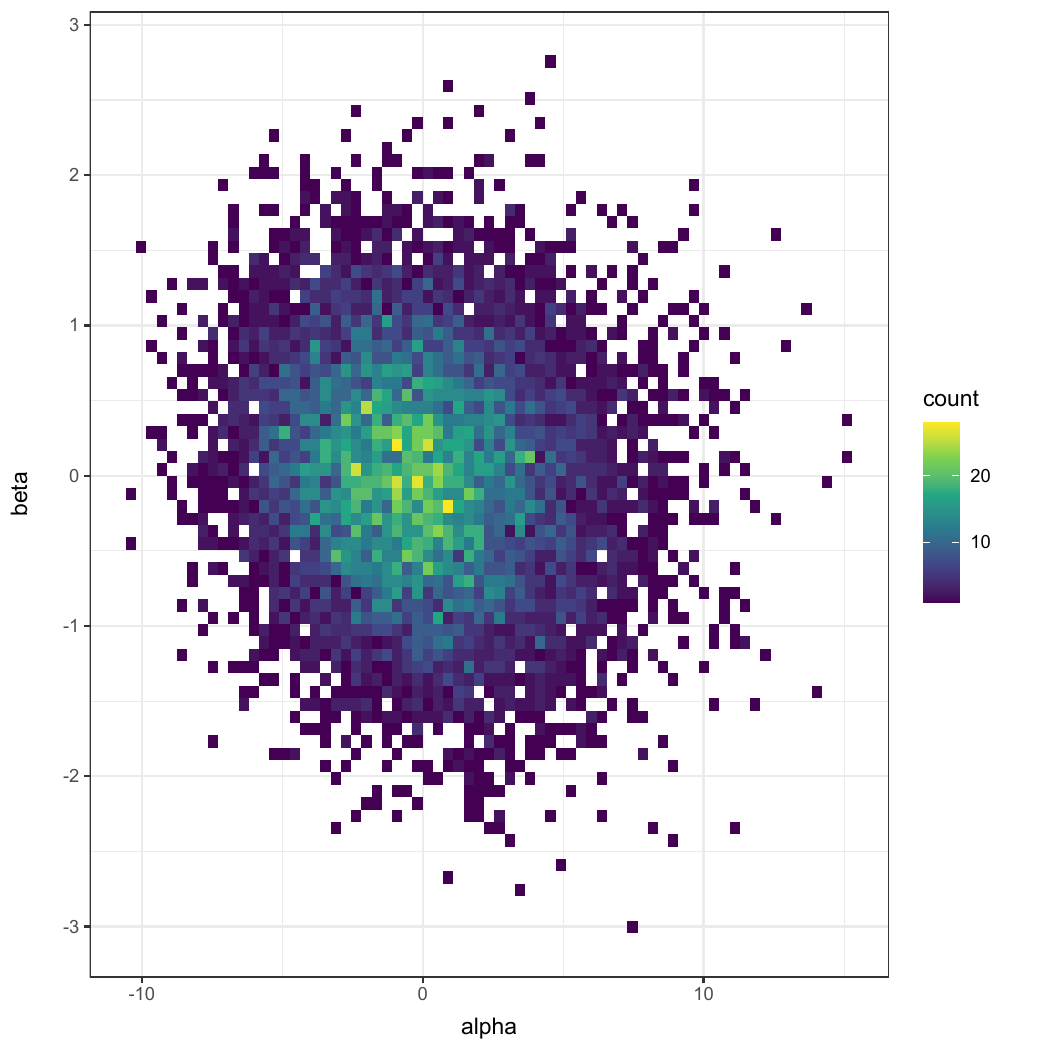}
 \caption{$N$-sample of $\sqrt{n}$-centered $\hat \theta_n$ estimators  under model $\mbox{(S0)}_{strong}$, with $N=10,000$ repetitions  and $n=1,000$, $3,000$ and $5,000$ observations.}\label{fig:TCL2D}
  \end{figure}
The estimated covariance matrices based on $N=10,000$ repetitions are:
\begin{eqnarray*}
\hat \Sigma_{1,000}=\begin{bmatrix}
  14.111& -0.543\\
-0.543& 0.623
\end{bmatrix},\quad 
\hat\Sigma_{3,000}=\begin{bmatrix}
   14.736& -0.698\\
    -0.698&  0.607
\end{bmatrix},
\quad 
\hat\Sigma_{5,000}=\begin{bmatrix}
11.709& -0.334\\
-0.334&  0.536
\end{bmatrix}.
\end{eqnarray*}
\noindent {\it \bf Comments on $\mbox{(S0)}_{strong}$ and $\mbox{(S0)}_{weak}$}. Note first  that the higher $\alpha$ and $\beta$ are the more frequently the Markov chain switches from a state to another. As a consequence, as shown comparatively  in Figures \ref{Fig:S0strong} and \ref{Fig:S0weak}, the $\mbox{(S0)}_{weak}$ model shows longer periods of time  where the hidden Markov states can be almost visually  guessed   when contrarily the $\mbox{(S0)}_{strong}$ model shows a high level of instability which makes the hidden Markov states hard to figure out (and thus probably to estimate).  We show in the performance Table \ref{tab:performanceS0_strong_weak} that this intuitive idea claiming that model $\mbox{(S0)}_{weak}$ would be easier to estimate than $\mbox{(S0)}_{strong}$ is actually not completely accurate. Indeed we can clearly see that the behavior, in terms  of  bias and standard deviation on both $\hat \theta_n$ and $\tilde \theta_n$, is globally  better under $\mbox{(S0)}_{strong}$ than $\mbox{(S0)}_{weak}$ especially when the sample size turns to be large.
If we now  look specifically  at the  results under $\mbox{(S0)}_{strong}$ it happens that  estimator $\hat \theta_n$ performs better than $\tilde \theta_n$ when the opposite happens under model $\mbox{(S0)}_{weak}$.
In order to go deeper into the asymptotic analysis  of our estimators we display in Figure \ref{fig:normalityS0strong}, respectively  Figure \ref{fig:normalityS0weak},  see Appendix section,  the empirical  $\sqrt{n}$-distribution (centered and normalized)  of our estimators  under $\mbox{(S0)}_{strong}$, resp. $\mbox{(S0)}_{weak}$. Clearly under $\mbox{(S0)}_{strong}$ the CLT  ``bell-regime" is reached roughly starting from $n=1,000$ for both estimators when under model $\mbox{(S0)}_{weak}$  we are not even close when considering  $n=5,000$, even if $\tilde\theta$ looks  slightly more Gaussian.

\subsection{Challenging cases with lower observability}
The goal of this section is to provide a  better understanding about  how sensitive the \textbf{ LinInd} condition is. In fact we can see, according to the \textbf{ LinInd} condition,  that the worst scenario  happens when the  stationary distribution of the known  stochastic process $Y^0$ is close to  the distribution of the i.i.d. sequence $Y^1$ and the chronological  dependence of $Y^0$ is weak, {\it i.e.} close to the independence setup ($\varphi\simeq 0$ in AR(1) case).  
For this purpose, we define  four extra  simple setups to challenge (with still reasonable asymptotic  performance results) our estimation methods: 
\begin{itemize}
\item (S1): $\theta=(0.2,0.4)$, $\varphi=0.7$ and $(m,v)=(1.5\times \mu_0,0.8\times \sqrt{\var_0})$.
\item (S2): $\theta=(0.2,0.4)$, $\varphi=0.7$ and $(m,v)=(2\times \mu_0,0.8\times \sqrt{\var_0})$.
\item (S3): $\theta=(0.6,0.3)$, $\varphi=0.7$ and $(m,v)=(1.5\times \mu_0,0.8\times \sqrt{\var_0})$.
\item (S4): $\theta=(0.6,0.3)$, $\varphi=0.5$ and $(m,v)=(2\times \mu_0,0.8\times \sqrt{\var_0})$.
%\item (S5): $\theta=(0.6,0.3)$, $\varphi=0.5$ and $(m,v)=(2\times \mu_0,0.8\times \sqrt{\var_0})$.
\end{itemize}
For illustration purpose we display in Figures \ref{Fig:S1} to  \ref{Fig:S4},  trajectories corresponding respectively to models $\mbox{(S1)}$ and $\mbox{(S4)}$.
%Remark: the case $\theta=(0.6,0.3)$, $\varphi=0.7$ and $(m,v)=(1.2 or 1.5 or 2\times \mu_0,0.8\times \sqrt{\var_0})$ works very well (n=10,000 tested).\\
\begin{figure}[!htb]
\begin{minipage}{0.48\textwidth}
     \centering
     \includegraphics[width=7cm,height=4cm]{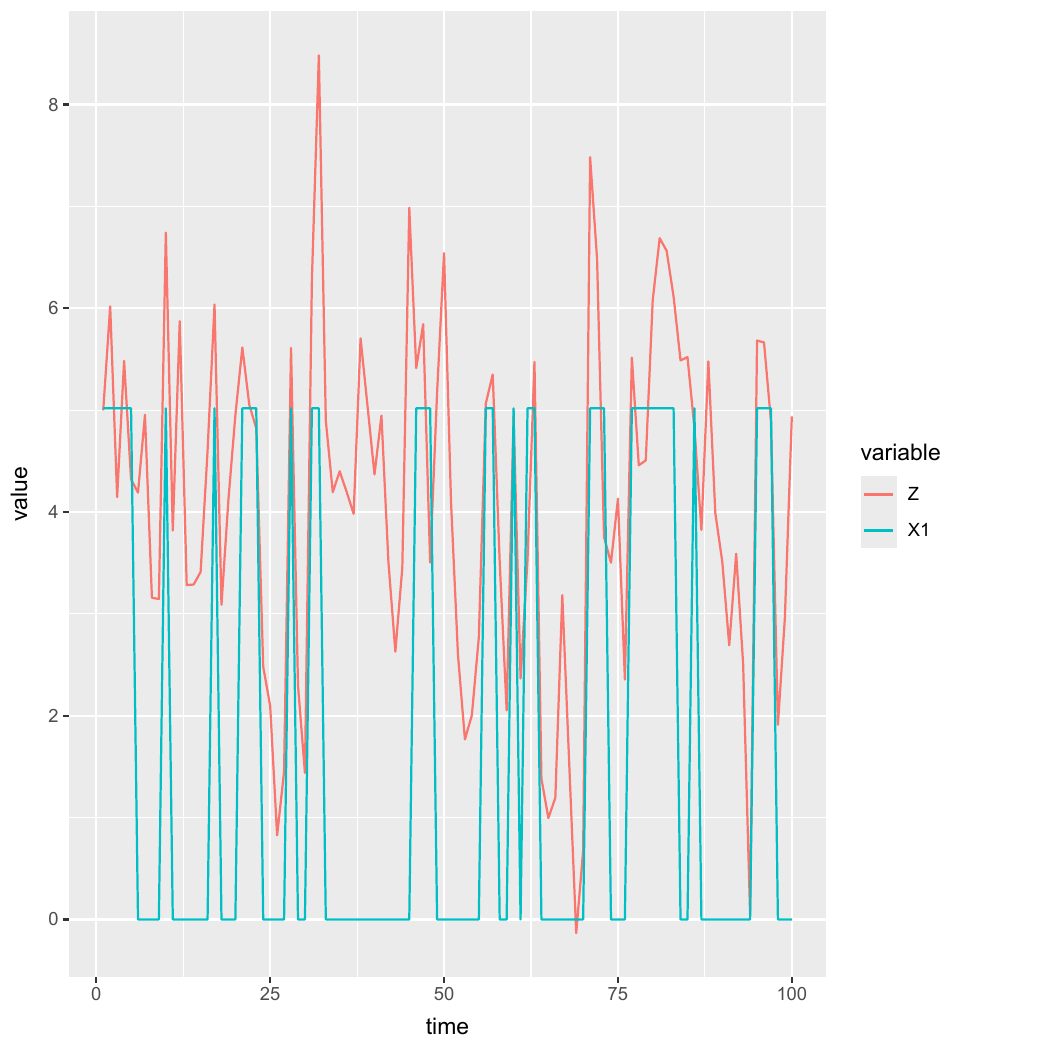}
     \caption{(S1) trajectory}\label{Fig:S1}
   \end{minipage}\hfill
    \begin{minipage}{0.48\textwidth}
     \centering
      \includegraphics[width=7cm,height=4cm]{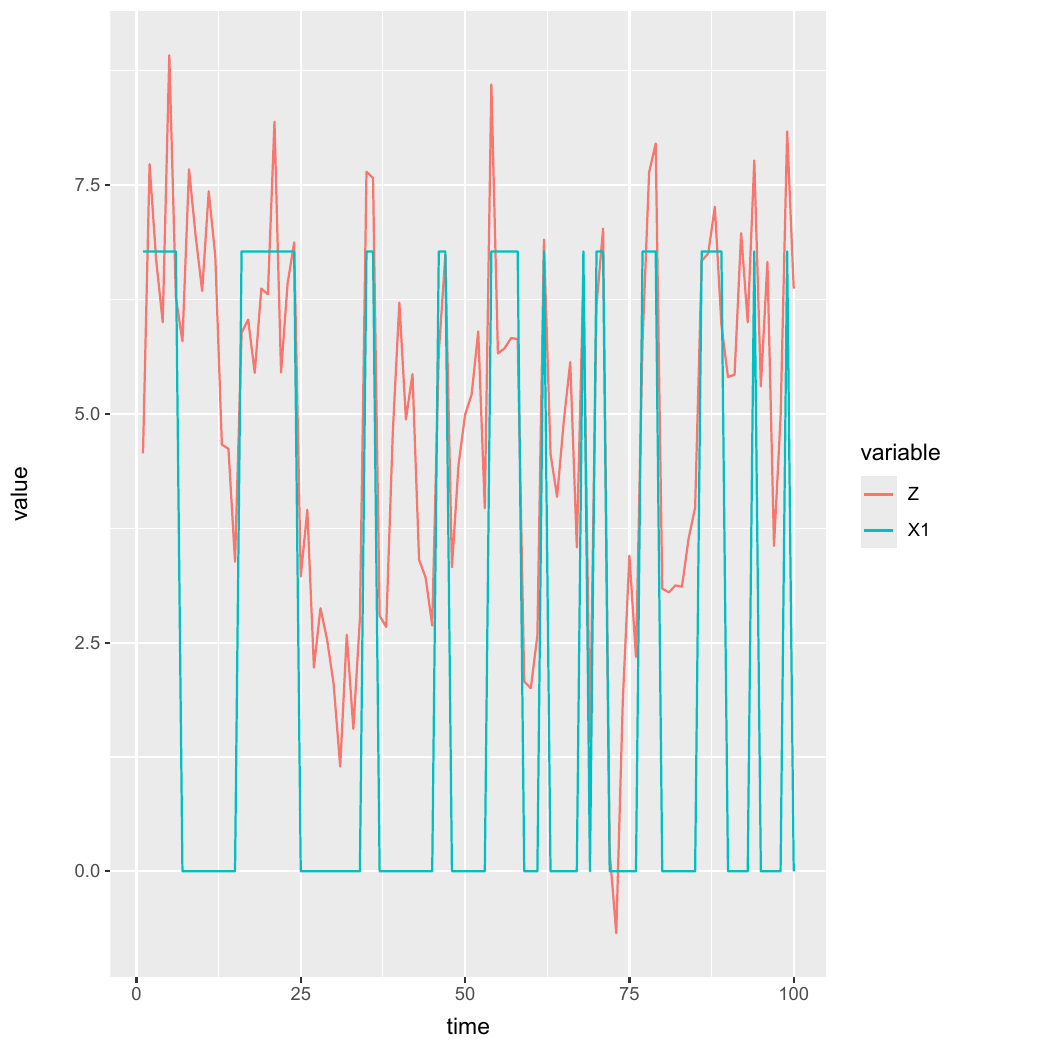}
     \caption{(S2)  trajectory}\label{Fig:S2}
     \end{minipage}\hfill
   \begin{minipage}{0.48\textwidth}
     \centering
     \includegraphics[width=7cm,height=4cm]{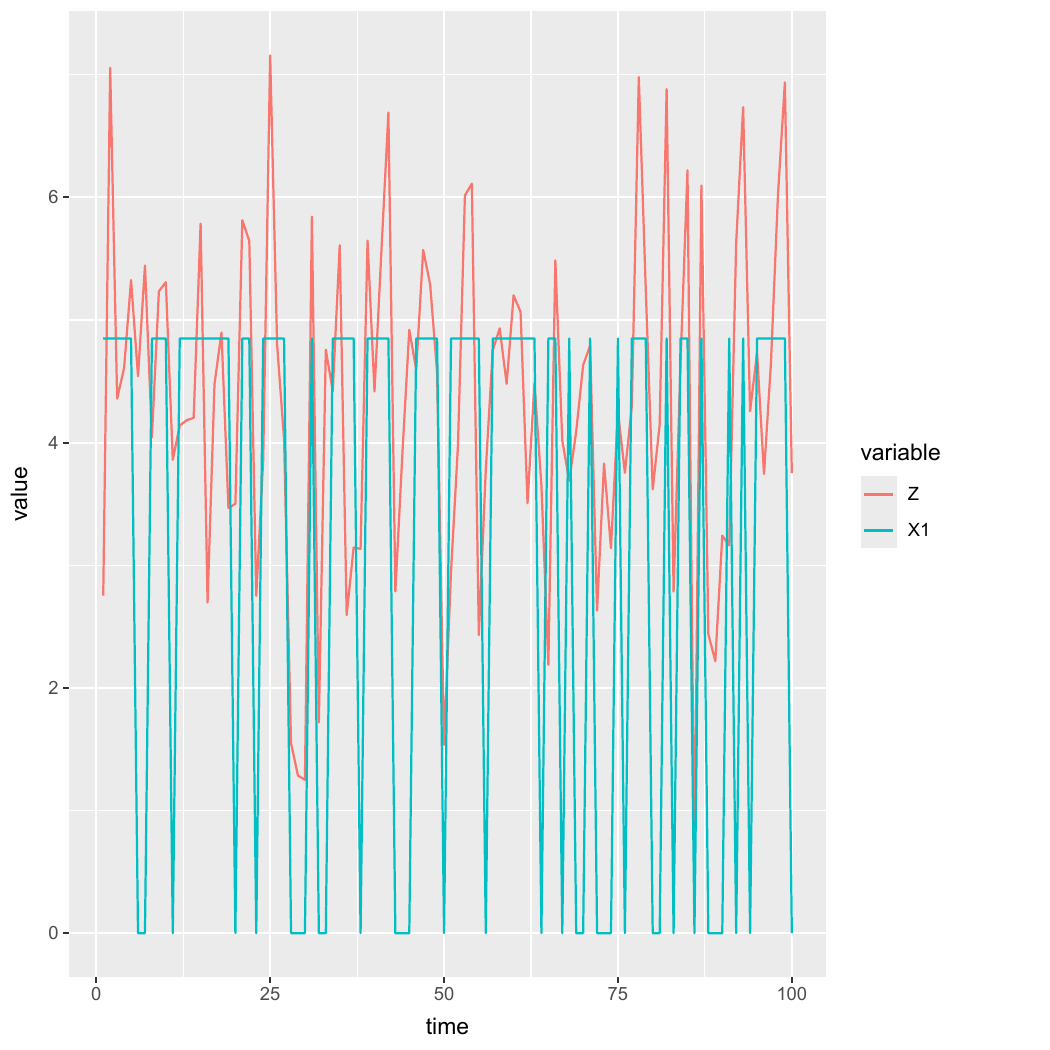}
     \caption{(S3) trajectory}\label{Fig:S3}
   \end{minipage}\hfill
   \begin{minipage}{0.48\textwidth}
     \centering
     \includegraphics[width=7cm,height=4cm]{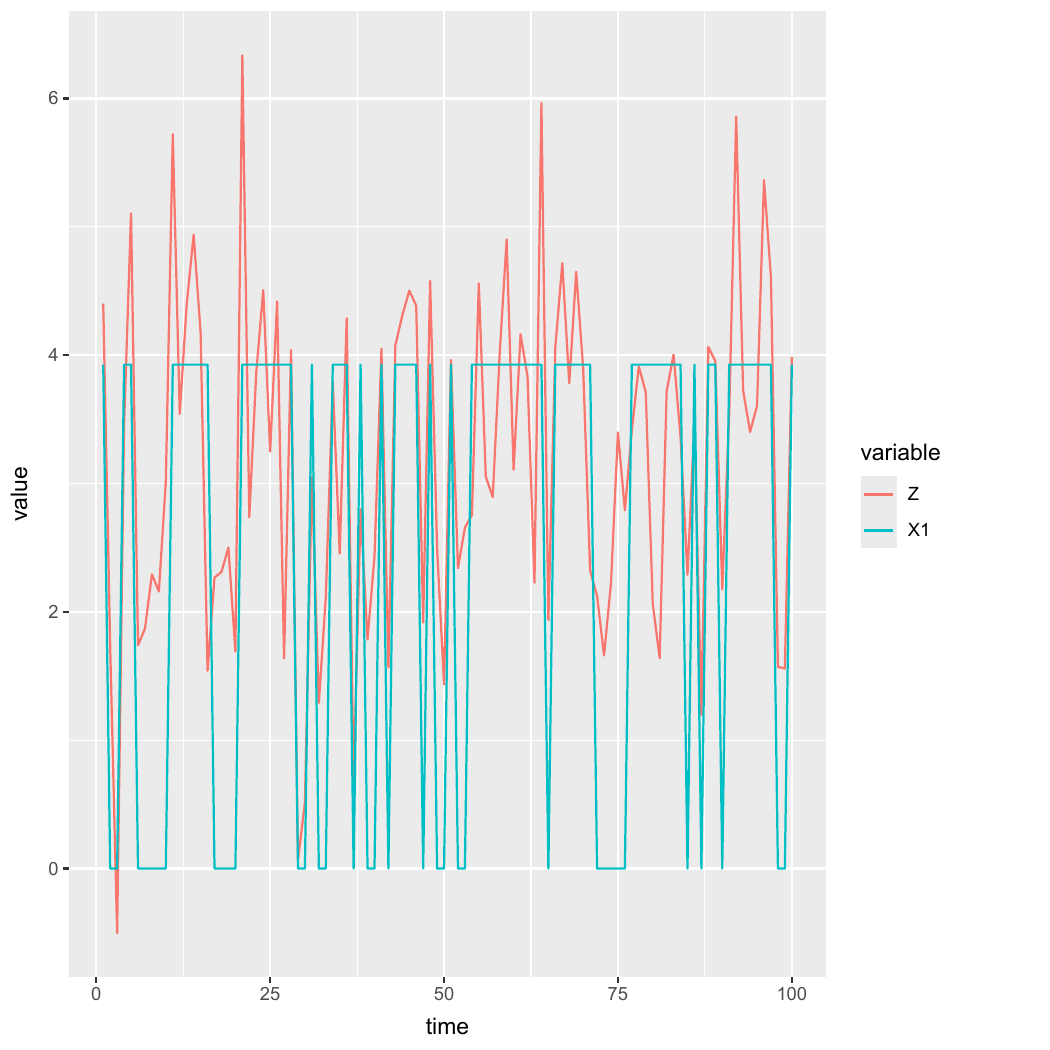}
     \caption{(S4)  trajectory}\label{Fig:S4}
   \end{minipage}
\end{figure}\\

\noindent {\it \bf Comments on models (S1--4)}.  Since models (S1) and (S2) have the  same transition values, taken not too large, we can see in  Figures \ref{Fig:S1} and \ref{Fig:S2} that long periods of time in which  process $Y^0$ and the noise sequence  $Y^1$ are consistently  observed  can happen with sometimes a good separability of the sources $Y^0$ and $Y^1$. Note that we increased the observability of model (S2) versus (S1) by increasing the expectation of $Y^1$ from 5 to a value about $6.66$. More specifically under (S1) we have $(\mu_0,m)\sim (3.33,5)$ when under (S2) we have $(\mu_0,m)\sim (3.33,6.66)$ with  a  positive jump in the expectation-gap, corresponding to $m-\mu_0$, going from about  1.66 to 2.33.
In model (S3) we only increase the jump frequencies by taking higher transition values but exact same AR(1) dynamic and observability as in model (S1).
We can clearly see the impact of this change in Figure \ref{Fig:S3} where the displayed  pattern looks clearly more erratic.  In model (S4) we keep the same transition values as in (S3) but  we drop the value of the regression coefficient from 0.7 to 0.5 which also reduces the expectation of the noise process $Y^1$ from 5 to 4. Let us stress out here that under (S3) we have $(\mu_0,m)\sim (3.33,5)$ when under (S4) we have $(\mu_0,m)\sim (2,4)$ with a jump in the expectation-gap from about  1.66 to 2 which should intuitively make the estimation easier (we will show that this is not necessarily true through the performance simulation results).
These preliminary comments done, let us comment on the performance results collected in Table \ref{tab:performanceS1-4}. 

As expected, the fact that we increase the observability of the jumps between model (S1) and (S2) has  a dramatic impact on  our estimators, reducing almost by half  the variance under (S2) compared to  (S1). We could also expect, similarly to what happened when studying the preliminary models $\mbox{(S0)}_{strong}$ and $\mbox{(S0)}_{weak}$,  that (S3) would  behave significantly  better than (S1). This is indeed the case for $n$ large (not really  noticeable for  $n=5,000$) when considering the estimator $\hat \theta$, when the sup-based estimator $\tilde \theta$ rather struggles in better performing  under (S3) compared to (S1). Finally it is very surprising (at  first glance)  to see that model (S4), which apparently has  a favorable observability gap,  is pretty badly estimated when using $\hat \theta_n$, when estimator $\tilde \theta_n$ achieve performances much closer to what observed in (S3). This could be explained by  the fact  that the regression coefficient  $\varphi$  is weaker under (S4) making the connection/link between two consecutive data-points in time  less obvious (jump from an AR-dynamic or purely independent).   This departure of behavior between $\hat \theta$ and $\tilde \theta$ should come from a too subtle variation (uncertainty management reflecting in the contrast function)  in term  $\Delta_n(\theta,x,y)$, involved in the empirical contrast functions (\ref{emp_d_s}),  along $\R^2$ when $\theta$ moves over $\Theta$. We think that some singularity about $\theta^*$ (slight global extrema) may be easier to capture with a local (singularity oriented contrast) $\tilde \theta$ type estimator  than a smooth but global $\Delta_n$-based contrast  such as $\hat \theta$.
Now to go deeper into the asymptotic analyses of our estimators, we also display in Figures \ref{fig:normalityS1}--\ref{fig:normalityS4}, see Appendix section,  the $\sqrt{n}$-normalized distribution  of our estimators $\hat \theta$ and $\tilde \theta$. In connection with the poor  performances of our estimators  under (S1), we can observe in Figure \ref{fig:normalityS1} that this also translates into the fact that our estimators do not reach their asymptotic normality regime for $n=5,000$ or  $10,000$. Nevertheless for $n=20,000$ we can observe a bell-like distribution for our integral-based estimator $\hat \theta$ illustrating, even in a challenging setup,  the Central Limit Theorem  established in Theorem \ref{theo:normasymp}. In contrast the $\sqrt n$-normalized distribution of both  $\hat \theta$ and $\tilde \theta$ estimators, show  bell-curves  under models (S2) and (S3)  starting from $n=5,000$ (smaller sample size considered in that study) under (S2) and $n=10,000$ under (S3). Finally, again in connection with the poor performances of our estimators  under (S4), the asymptotic normality regime is clearly not achieved for $n=5,000$ or $10,000$ with very different behaviors depending on the estimator: bimodality for $\hat \theta$ (symptom of a spurious minima of the contrast) and flatness for $\tilde \theta$ (lack of precision). However  for  $n=20,000$ the results turn out  to be slightly more  encouraging  especially for $\hat \theta$ which is proved to be asymptotically normal, see  Theorem \ref{theo:normasymp}.

%We can observe in Fig. \label{fig:trajectories} that model (S1) can make happen  long periods of time in which  process $Y^0$ and the noise sequence  $Y^1$ are consistently  observed with sometimes a good separability of the sources $Y^0$ and $Y^1$. Contrarily
%we can observe that in (S2) the design of the process is very confusing since the hidden Markov chain switches very frequently which does not allow to figure out homogeneous periods of time under a model or another. Model (S3) has a $X$-pattern behavior similar to (S1) but since the AR(1) coefficient reduces from 0.8 to 0.5, the $Y^0$ process appears to have a less structured dynamic which makes the identification of homogeneous period of time more challenging to detect visually. In Model (S4) we basically reduce the observability of model (S3) by reducing the regression coefficient $\rho$ from 0.7 to 0.5 which also  lowers the stationary expectation of the $Y^0$ process going from 5 to 4.
%\begin{eqnarray*}
%\gamma^-=\gamma_{X-low-mix}=(0.2,0.3),\quad  &\mbox{and}& \quad \gamma^+= \gamma_{X-high-mix}=(0.7,0.8)\\
%\varphi^-=\varphi_{{Y^0}-low-mix}=0.8, \quad  &\mbox{and}& \quad \varphi^+=\varphi_{{Y^0}-high-mix}=0.2\\
%(m,v^2)^-=(m,v^2)_{Y^1-low-obs}=1.5\times (\mu_0,\var_0) \quad  &\mbox{and}&(m,v^2)^+= (m,v^2)_{Y^1-high-obs}=2\times (\mu_0,\var_0).
%\end{eqnarray*}
%We fix for simplicity matters  in the sequel $m_0=v_0=1$. Combining these  configurations we obtain 8 scenarios illustrated in the next figure:

\begin{table}[!b]
	
		% \usepackage{multirow}
		\centering
	\scalebox{1}{
		\begin{tabular}{ccccc}
		\hline
			Model (S1)& $\hat \alpha_n$& $\hat \beta_n$& $\tilde \alpha_n$&$\tilde \beta_n$ \\
			\hline
			$n=5,000$    & (-0.099,0.075)      &(-0.186,0.208)            & (-0.034, 0.044)&(-0.016,0.186)\\ 
			$n=10,000$   &  (-0.037,0.065)     & (-0.072,0.183)           & (-0.057, 0.044)&(-0.154,0.124)\\  
			$n=20,000$    &  (-0.001,0.018)        & (0.053,0.098)           & (-0.022,0.026)& (-0.058,0.106)\\ 
			\hline
			Model (S2)& $\hat \alpha_n$& $\hat \beta_n$& $\tilde \alpha_n$&$\tilde \beta_n$ \\
			\hline
			$n=5,000$    &  ( -0.011,0.022)     & (0.033,0.105)         & (-0.037,0.029)&(-0.125,0.112) \\ 
			$n=10,000$   &  (-0.010,0.016)         & (-0.041,0.084)             &(-0.013,0.010)&(-0.044,0.085)\\   
			$n=20,000$    &  (-0.005,0.006) & (-0.012,0.053)           &(-0.020,0.012) &(-0.089,0.067)\\ 
			\hline 
			Model (S3)& $\hat \alpha_n$& $\hat \beta_n$& $\tilde \alpha_n$&$\tilde \beta_n$ \\
			\hline
			$n=5,000$    &  (-0.051,0.126)         & (-0.022,0.102)            &(-0.032,0.050)  &(0.010,0.130) \\  
			$n=10,000$   & (-0.002,0.020)         & (0.010,0.061)           & (-0.025,0.039)&(-0.010,0.098)\\  
			$n=20,000$    &  ( -0.000,0.016)        & (-0.003,0.049)             & (-0.007,0.022)&(0.017,0.082)\\
			\hline 
			Model (S4)& $\hat \alpha_n$& $\hat \beta_n$& $\tilde \alpha_n$&$\tilde \beta_n$ \\
			\hline
			$n=5,000$    & (-0.186,0.259)      &(-0.124,0.145)            & (-0.096,0.083)&(-0.052,0.150)\\ 
			$n=10,000$   &  (-0.169,0.233)     & (-0.099,0.147)           & (-0.025,0.041)&(0.044,0.146)\\  
			$n=20,000$    &  (-0.033,0.109)        & (-0.020,0.103)           & (-0.027,0.038)& (0.021,0.120)\\
			        
			 \end{tabular}
		}
	\caption{Bias and standard deviation based performances under model (S1--4) for $n=10,000$, $n=20,0000$ and $n=30,0000$ (in rows) and discrepancies
	 $\mathbf{d}(\cdot)$ (integration based)
$\mathbf{s}(\cdot)$ (sup based)
	}\label{tab:performanceS1-4}.
\end{table}

\subsection{Functional estimator behavior}
In this section we aim to illustrate the asymptotic behavior of our  plug-in inversion based functional estimator defined in (\ref{plug-in}). For this purpose 
we display in Figure \ref{plug-in-curves} some panels of 10 inversed cdfs deduced  from preliminary parametric estimation steps where $(\hat r_n,\hat p_n)=\left(\frac{\hat \alpha_n}{\hat \alpha_n+\hat \beta_n},\frac{\hat \beta_n}{\hat \alpha_n+\hat \beta_n}\right)$. An interesting information  is provided by Figure \ref{emp_Z_cdf_curves}, in which we display a panel of 10 empirical cds of the observed process $Z$ (on which our method is based) compared to the true $Z$-cdf $F$.  Note that Figure  \ref{plug-in-curves} and Figure \ref{emp_Z_cdf_curves} have been generated independently (there is no color-correspondence between the curves).
This figure allows to visualize in particular the type of functional estimation quality we have in input of our semiparametric estimation method and how it deteriorates after the parametric estimation step combined with the plug-in inversion step, see expression (\ref{plug-in}).
For clarity and interpretability  matters  we propose to run our semiparametric inversion based approach on model $(\mbox{S0})_{strong}$ which is identified as a model easy to estimate since reasonably  reliable parametric estimators can be obtained for samples size such as $n=1,000$ and $5,000$. Note that we could have trimmed/regularized our estimator (\ref{plug-in}) to only  keep the positive part of our inversed curve, see expression (\ref{plug-in}), which would  have provided much more satisfactory cdf-like curves. In Figure \ref{plug-in-curves} we preferred exactly to keep the original version of our estimator in order to clearly illustrate  how the  sample size $n$  impacts, because of the convergence results stated in Theorems   \ref{theo:normasymp} and  \ref{theo:risqueF},      the  regularity/cdf-conformity of our plug-in inversion based functional estimator (\ref{plug-in}). In fact we can see in Figure \ref{plug-in-curves} that the left side of the target $F^1$ curve is  in general pretty badly estimated (with obvious consequences over the  whole curve)  when the sample size is low ($n=1,000$) this drawback being almost solved 
without any trick when $n$ turns to be large ($n=5,000$). To explain this bad left-side behavior, one can go back to the following $F^1$-error decomposition: 
\begin{eqnarray}
F^1(x)-\hat F^1_n(x)&=&\frac{1}{r^*}\left(F(x)-p^*F^0(x)\right)-\frac{1}{\hat r_n}\left(\hat F_n(x)-\hat p_nF^0(x)\right)\nonumber\\
&=&\frac{1}{r^*}\left(F(x)-\hat F_n(x) \right)+\left(\frac{1}{r^*}-\frac{1}{\hat r_n}\right)\hat F_n(x)-\left(\frac{p^*}{r^*}-\frac{\hat p_n}{\hat r_n}\right) F^0(x)\nonumber\\
&\simeq&\frac{1}{r^*}F(x)-\left(\frac{p^*}{r^*}-\frac{\hat p_n}{\hat r_n}\right) F^0(x), \quad \mbox{in the left-tail of $F$}\label{left-tail}\\
&\simeq& \left(\frac{1}{r^*}-\frac{1}{\hat r_n}\right)\hat F_n(x)-\left(\frac{p^*}{r^*}-\frac{\hat p_n}{\hat r_n}\right) F^0(x),\quad \mbox{in the right-tail of $F$}\label{right-tail}.
\end{eqnarray}
Approximation (\ref{left-tail}) is based on the fact that in the left-tail of $F$, especially for $n$ (very) small, we have very few observations which makes $\hat F_n$ negligible compared to $F$ and $F_0$ (exact quantities). This phenomenon is still reminiscent for $n=1,000$ when we look closely at Figure \ref{emp_Z_cdf_curves} for design  values between -2 and 3 (negative bias). Note that the left component of our 1-order model (\ref{1order}),  under $(\mbox{S0})_{strong}$,   is $F^0$ (located  remotely on the left side of $F^1$)  which is weighted by $\beta^*/(\alpha^*+\beta^*)\simeq 0.533$. As a consequence the left-tail of $F$ is essentially estimated on a sample of size  just a bit larger than 500 when $n=1,000$.
On the other hand, approximation  (\ref{right-tail}) is based on the fact that in the right-tail of $F$, the difference $F-\hat F_n$ is very close to 0, even for small values of $n$, since we aggregate all the data
in $\hat F_n(x)$ to approximate $F(x)$ (both $F$ and $\hat F_n$ are very close to one). We then see that in the left-tail of $F$ there is a sort of mixed impact of the poor $F$ estimation combined with the fitting quality of the parameter estimates, when in the right-tail of $F$ the bias mainly depends on the parameter estimates  fitting. This explains, according to us, why the fitting of the curves displayed in Figure \ref{plug-in-curves}  looks way better on the right-tail  compared to the   left-tail  especially when $n$ increases and the Euclidean parameters fitting improves with the $\sqrt{n}$-regime, see Theorem \ref{theo:normasymp}.
\begin{figure}[!htb]
 \includegraphics[width=7cm,height=6cm]{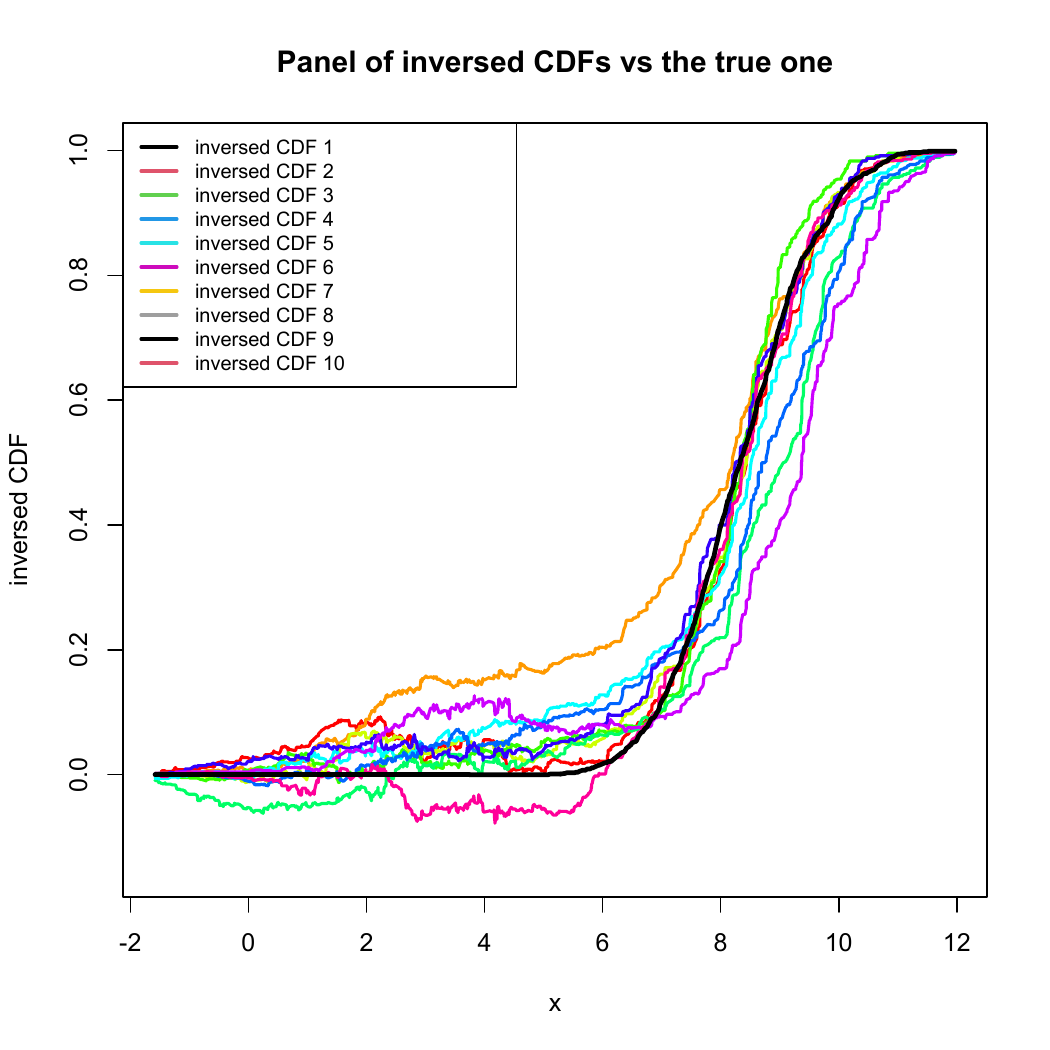}
 \includegraphics[width=7cm,height=6cm]{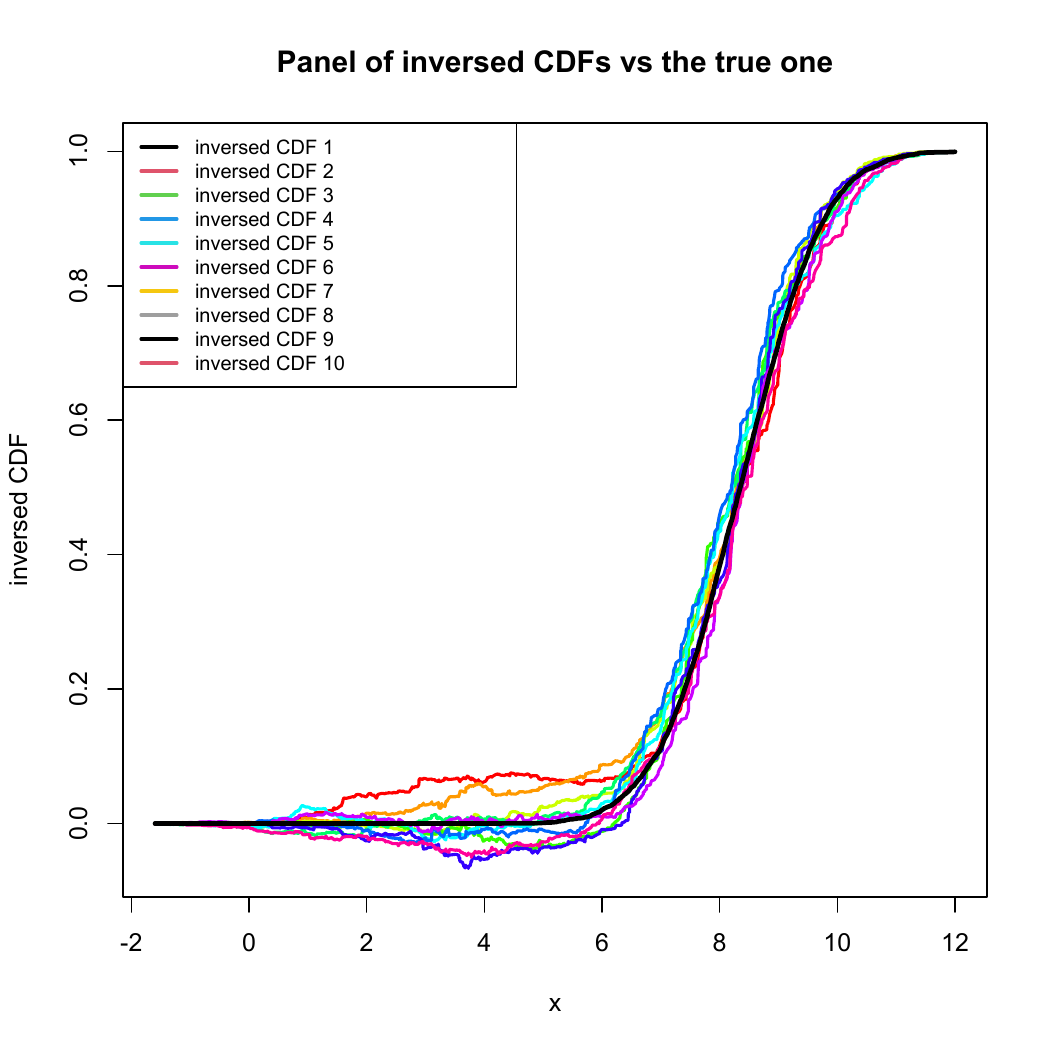}
 \caption{Panels of 10 plug-in inversion based functional estimator (\ref{plug-in}) of $F^1$ (displayed in black bold) under model $(\mbox{S0})_{strong}$  respectively for  $n=1,000$ and $5,000$.}\label{plug-in-curves}
\end{figure}
\begin{figure}[!htb]
 \includegraphics[width=7cm,height=6cm]{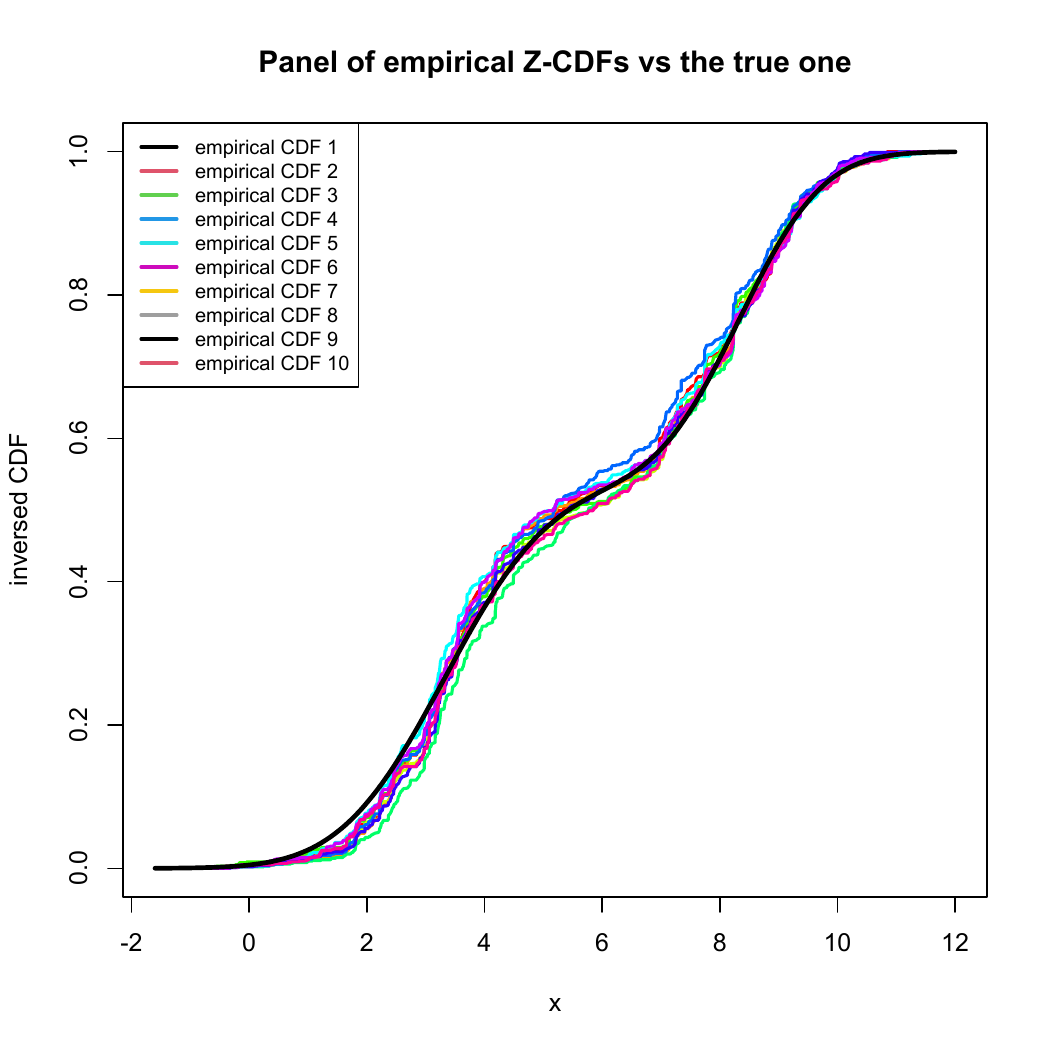}
 \includegraphics[width=7cm,height=6cm]{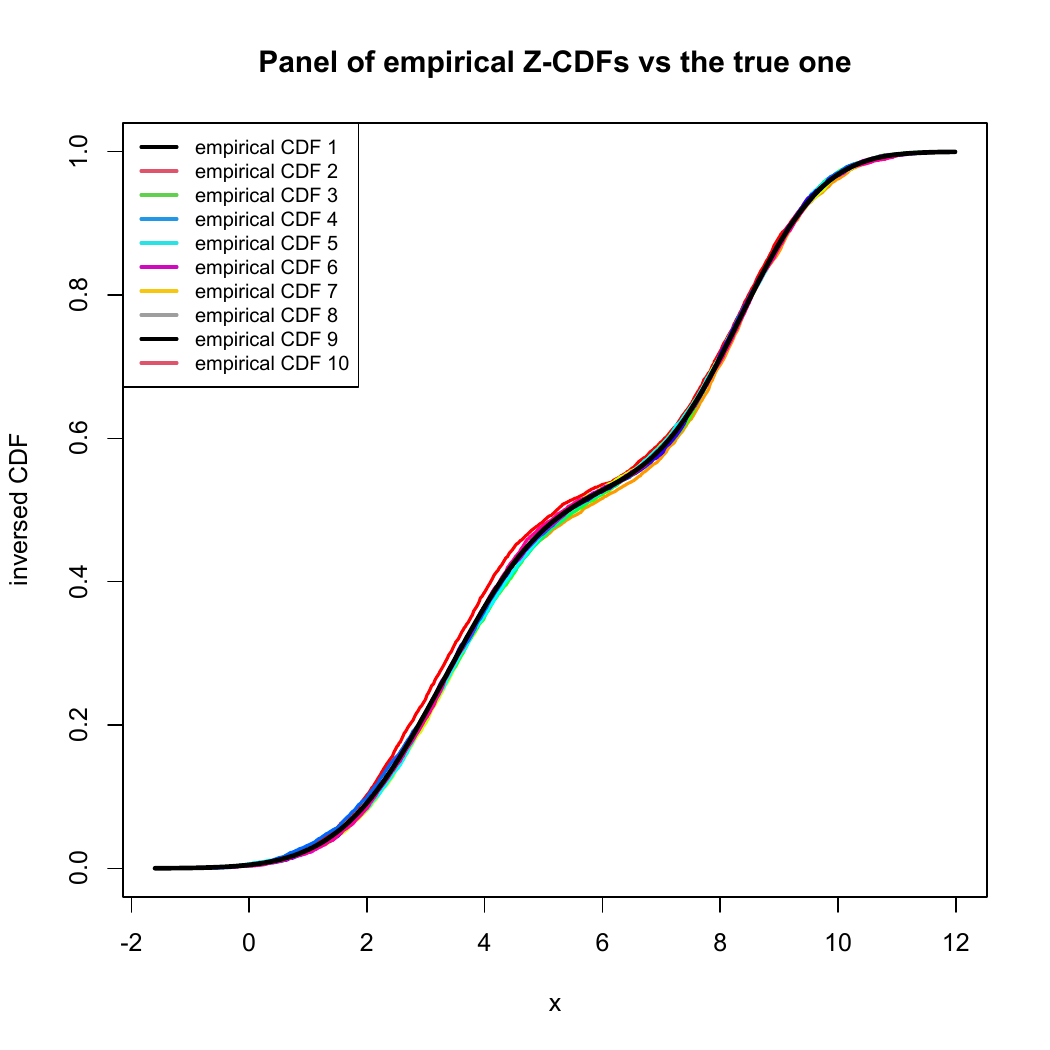}
 \caption{Panels of 10 empirical $Z$-cdfs of $F$ (displayed in black bold) under model $(\mbox{S0})_{strong}$  respectively for  $n=1,000$ and $5,000$.}\label{functional_CLT}\label{emp_Z_cdf_curves}
\end{figure}
%%%%%%%%%%%%%%%%%%%%%%%%%%%%%
%%%%%%%%%%%%%%%%%%%%%%%%%%%%%%%%%%%%%

\section{Concluding remarks}\label{sec:conclusion}

In this paper we introduced a new class of semiparametric chronological  mixture models: the so-called gold standard Markovian poisoning  model.
This model is an extension of  the classical semiparametric contamination model which has been extensively studied in the last two decades, see for example \cite{Patra} or \cite{MPSV24} and references therein.
We have in particular  tackled the semiparametric estimation for our model,  by proposing two minimum contrast estimators
of the latent Markov chain  transition matrix, which we prove, for one of them,  the $\sqrt{n}$-consistency and the strong consistency for the other.  We also stated  a functional central limit theorem along with a ``finite sample size" rate of convergence for the plug-in estimator of  $F^1$ (the unknown i.i.d. poisoning sequence cdf).
\noindent The simulation setups we investigated  interestingly  highlighted the role of the distance between the $Y^1$ and $Y^0$  distributions, see assumption {\bf LinInd}, but also the role of the amount of
 time-dependence carried by the  known  stochastic process $Y^0$, on the estimation performances. The proofs of our theoretical results (see Section~\ref{sec:preuvenormasymp}) also shown the crucial role played by:
 \begin{itemize}
\item $T_1=G^0-F^0\otimes F^0$, which translates the departure from independence over time  about  the gold standard process $Y^0$,
\item $T_2=r^{*2}(F^1-F^0)\otimes (F^1-F^0)$, which detects  the departure  between the mixed processes $Y^1$ and $Y^0$  stationary distributions. 
\end{itemize}

\noindent These quantities are found in particular in the Hessian of the contrast function and then   in the  $\sqrt{n}$-limit covariance matrix $\Sigma$ of our estimator $\hat \theta_n$, see  Theorem \ref{theo:normasymp} and its proof in  Section \ref{preuvecov}. The main tools for the proofs are a judicious  change of variable  (to make pointly appear $T_1$ and $T_2$), see definition of the $g$ and $h$ homeorphisms in  Section~\ref{sec:preuvenormasymp}, and the joint convergence of  the empirical processes $\mathbb F_n:=\sqrt{n}(\hat F_n-F)$  and $\mathbb{G}_n:=\sqrt{n}(\hat G_n-G)$  in the case of  non-independent observations. 
We also proved a non-asymptotic bound for our parametric and non-parametric estimators, using Kiefer process and strong approximation of the empirical distribution function for mixing sequences in $\R^2$, see the proof of Lemma~\ref{supprocessusempiriques} in Section \ref{preuverisqueF}.
Our method can be extended to the case where the distribution of the gold standard process in unknown but separately observed (training data). Indeed, we can then replace $G^0(x,y)$ by 
$$\tilde G^0_N(x,y)=\frac{1}{N-1}\sum_{k=1}^{N-1}\I_{\{Y^0[k]\leq x, Y^0[k+1]\leq y\}},\quad (x,y)\in E^2,$$ 
supposing we have at our disposal additional observations $(Y^0[k])_{1\leq k\leq N}$, at the cost of an additional  error term of order $1/\sqrt{N}$ (negligible if $N$ is large with respect to $n$).
The extension to time-dependent poisoning processes, {\it i.e.} $Y^1$ supposed to be stationary mixing or Markovian,  is the next challenge. It will require the use of the 3-order distribution of $(Z_i,Z_{i+1},Z_{i+2})$. Preliminary computations indicate that it seems possible to apply our previous method but with cumbersome technicalities (numerous intricate terms).
Finally we discussed, at the end of Section \ref{sec:nonparametric}, a decoding strategy, see expression (\ref{predict}), to recover the  states of the latent  picking Markov chain involved in  our model. This method is based on an inversed-kernel density estimate of the $Y^1$ pdf, see expression (\ref{inversed_pdf}), which will be studied in a latter work.

%%%%%%%%%%%%%%%%%%%%%%%%%%%%%%%
\section{Proofs}
\label{sec:proofs}

\subsection{Proof of Proposition~\ref{contrast}}
\label{proofcontrast}

Recall that, since  $\alpha\neq 0$, we have  $r\neq 0$, where $r=\alpha/(\alpha+\beta)$.
Denote now
$$G_\theta=\lambda_1 G^0+\lambda_2 F^0\otimes F^1_\theta
+\lambda_3F^1_\theta \otimes F^0+\lambda_4F^1_\theta \otimes F^1_\theta.$$
Replacing $F^1_\theta$ by its value $(F-pF^0)/r$ and using $\lambda_3=\lambda_2$, we obtain
\begin{eqnarray*}
G_\theta=\lambda_1 G^0+\left(\frac{\lambda_2}{r}-\frac{p\lambda_4}{r^2}\right) (F^0\otimes F+F\otimes F^0)
+\frac{\lambda_4}{r^2}F\otimes F+\left(-2\frac{\lambda_2p}{r}+\frac{p^2\lambda_4}{r^2}\right)F^0\otimes F^0.
\end{eqnarray*}
We then replace the $\lambda_i$'s by their value
$$\lambda_1=p(1-\alpha),\quad 
\lambda_2=\lambda_3=p\alpha=r\beta, \quad
\lambda_4=r(1-\beta),$$
and  use the notation $b=1-\beta$ to obtain the following decomposition
\begin{eqnarray*}
G_\theta=p(1-\alpha) G^0+\left(\beta-\frac{pb}{r}\right) (F^0\otimes F+F\otimes F^0)
+\frac{b}{r}F\otimes F+\left(-2\beta p+\frac{p^2b}{r}\right)F^0\otimes F^0.
\end{eqnarray*}
Next, we only keep variables $r$ and $b$, which gives
\begin{eqnarray}
\label{Ggamma}
G_\theta&=&(1-2r+rb) G^0+\left(1-\frac{b}{r}\right) (F^0\otimes F+F\otimes F^0)
+\frac{b}{r}F\otimes F+\left(-2+2r-rb+\frac{b}{r}\right)F^0\otimes F^0
\nonumber\\ 
&=&(1-2r+rb) \left(G^0-F^0\otimes F^0\right)+\left(1-\frac{b}{r}\right) (F^0\otimes F+F\otimes F^0-F^0\otimes F^0)
+\frac{b}{r}F\otimes F\nonumber\\
&=&(1-2r+rb) \left(G^0-F^0\otimes F^0\right)+\left(\frac{b}{r}-1\right) (F-F^0)\otimes (F-F^0)
+F\otimes F.
\end{eqnarray}
Since $F-F^0={r^*}(F^1-F^0)$, we can also write
\begin{eqnarray*}
G_\theta
&=&(1-2r+rb) \left(G^0-F^0\otimes F^0\right)+{\frac{r^{*2}}{r^2}}\left(rb-r^2\right) (F^1-F^0)\otimes (F^1-F^0)
+F\otimes F.
\end{eqnarray*}
Now, since $G=G_{\theta^*}$, we can write the following $\Delta$ decomposition
\begin{eqnarray*}
\Delta(\theta, \cdot)&=&G_{\theta^*}-G_\theta
=[(1-2r^*+r^*b^*)-(1-2r+rb)] \left(G^0-F^0\otimes F^0\right)\\
&&+[(r^*b^*-r^{*2})-{\frac{r^{*2}}{r^2}}(rb-r^2)] (F^1-F^0)\otimes (F^1-F^0).
\end{eqnarray*}
The equality \eqref{Deltabar} is now proved by taking  
\begin{equation}\label{c1c2}
c_1(\theta^*,\theta)=-2r^*+r^*b^*+2r-rb, \quad \mbox{and}\quad c_2(\theta^*,\theta)={\frac{r^*}{r}(b^*r-r^*b)}.
%r^*b^*-r^{*2}-rb+r^2
\end{equation}
Moreover, if $c_1(\theta^*,\theta)=0$ and $c_2(\theta^*,\theta)=0$ then 
\begin{eqnarray*}
0=c_1(\theta^*,\theta)-c_2(\theta^*,\theta)r/r^*=-2r^*+r^*b^*+2r-rb-b^*r+r^*b=(r^*-r)(-2+b^*+b).
\end{eqnarray*}
But recall that $b$ and $b^*$ belong to $[\delta,1-\delta]$ thus $b+b^*<2$ and
necessarily $r=r^*$. 
Using again $c_1(\theta^*,\theta)=0$, this gives $b=b^*$
and finally $(\alpha, \beta)=(\alpha^*, \beta^*)$. \qed
%%%%%%%%%%%%%%%%%%%%%%%%%%%

\subsection{Consistency proofs}
\label{sec:preuveconsistance}

\begin{lemma}\label{lipschitz}
For all $(\theta, \theta')\in \Theta^2$,  where $\Theta=[\delta,1-\delta]$ with $\delta\in]0,1/2[$, we have
\begin{eqnarray*}
|{\mathbf d}(\theta)-{\mathbf d}(\theta')|\leq  C(\delta) \|\theta-\theta'\|_1,\quad 
\mbox{and} \quad |{\mathbf d}_n(\theta)-{\mathbf d}_n(\theta')|  \leq  C(\delta) \|\theta-\theta'\|_1,
\end{eqnarray*}
where $\| \cdot\|_1$  is the norm defined for any generic vector $u=(u_1,u_2)^T\in \R^2$ by $\|u\|_1=|u_1|+|u_2|$. As a consequence  ${\mathbf d}(\cdot)$ and ${\mathbf d}_n(\cdot)$ are  uniformly continuous mappings on $\Theta$.
The same results  hold also for ${\mathbf s}(\cdot) $ and ${\mathbf s}_n(\cdot)$.
\end{lemma}

\begin{proof}
Denote 
$$\hat G_{n,\theta}=\lambda_1 G^0+\lambda_2 F^0\otimes \hat {F}_{n,\theta}^1
+\lambda_3\hat {F}_{n,\theta} ^1\otimes F^0+\lambda_4\hat {F}_{n,\theta} ^1\otimes \hat {F}_{n,\theta}^1,$$
where we recall that $\hat{F}_{n,\theta}^1=(\hat {F}_{n}-pF^0)/r.$
Reasoning similarly to the proof of Proposition~\ref{contrast}, analogously to equality \eqref{Ggamma}, we have:
$$\hat G_{n,\theta}=(1-2r+rb) \left(G^0-F^0\otimes F^0\right)+\left(\frac{b}{r}-1\right) ( \hat {F}_{n}-F^0)\otimes ( \hat {F}_{n}-F^0)
+\hat {F}_{n}\otimes \hat {F}_{n}.$$ 
Moreover $1-2r+rb=pa$ where we denote $a=1-\alpha$.
Using now the definition of $\Delta_n$,  see below Equation (\ref{emp_d_s}), we obtain:
\begin{eqnarray*}
\Delta_n(\theta,\cdot )-\Delta_n(\theta',\cdot )
=(p'a'-pa) \left(G^0-F^0\otimes F^0\right)+\left(\frac{b'}{r'}-\frac{b}{r}\right) ( \hat {F}_{n}-F^0)\otimes ( \hat {F}_{n}-F^0).
\end{eqnarray*}
Now let observe that $G^0,F^0$, and $\hat F_n$ are all  bounded by 1. This allows to obtain the following majorization
\begin{eqnarray*}
|\Delta_n(\theta,\cdot)-\Delta_n(\theta',\cdot )|
&\leq &|p'a'-pa|+\left|\frac{b'}{r'}-\frac{b}{r}\right|\\
&\leq &|p'(a'-a)+a(p'-p)|+\left|\frac{b'(r-r')+r'(b'-b)}{rr'}\right|.
\end{eqnarray*}
Next we use the fact that $\alpha$ and $\beta$ belong to $[\delta, 1-\delta] $ and then $1/r=1+\beta/\alpha\in [1,1/\delta]$.  This leads to
\begin{eqnarray*}
|\Delta_n(\theta,\cdot )-\Delta_n(\theta',\cdot )|
\leq 
|a'-a|+|p'-p|+\frac{1}{\delta^2}|r-r'|+\frac{1}{\delta}|b-b'|.
\end{eqnarray*}
Moreover 
$(r'-r)(\alpha'+\beta')=p(\alpha'-\alpha)+r(\beta-\beta')$ which provides
$$|r'-r|\leq \frac{1}{2\delta}\left(|\alpha'-\alpha|+|\beta'-\beta|\right).$$
Thus, since $p=1-r$, $a=1-\alpha$, and $b=1-\beta$,  we  have 
\begin{eqnarray*}
|\Delta_n(\theta,\cdot)-\Delta_n(\theta',\cdot)|
\leq \left(\frac{1}{2\delta^3}+
\frac1{2\delta}+1 \right)|\alpha'-\alpha|+
\left(\frac{1}{2\delta^3}+
\frac{3}{2\delta}\right)|\beta'-\beta|,
\end{eqnarray*}
and using $\delta\leq 1$, we finally obtain
$$
|\Delta_n(\theta,\cdot )-\Delta_n(\theta',\cdot )|
\leq \frac1{\delta^3}\left(2|\alpha'-\alpha|+
2|\beta'-\beta|\right)\leq 2 \delta^{-3}\|\theta-\theta'\|_1.
$$
Moreover, since  
$|\Delta_n(\theta,\cdot)|\leq 1+|pa|+\left|\frac{b}{r}-1\right|\leq (3+\delta^{-1})\leq 4\delta^{-1}$,  we obtain  
$$
|\Delta_n(\theta,\cdot )-\Delta_n(\theta',\cdot)|^2\leq 
 (4\delta^{-1}+4\delta^{-1})2\delta^{-3} \|\theta-\theta'\|_1,
$$
and $|{\mathbf d}_n(\theta)-{\mathbf d}_n(\theta')|\leq 
 16 \delta^{-4}\|\theta-\theta'\|_1.$
Similarly, we have  
$|{\mathbf d}(\theta)-{\mathbf d}(\theta')|\leq 
 16 \delta^{-4}\|\theta-\theta'\|_1.$
\end{proof}

\begin{proposition}\label{convergencecontraste}
If  $\Theta=[\delta,1-\delta]^2$, with $\delta\in ]0,1/2[$, we have: 
\begin{eqnarray*}
\sup_{\theta\in \Theta}
|{\mathbf d}_n(\theta)-{\mathbf d}(\theta) |
\leq (6+2\delta^{-1} )\left(\|\hat G_n-G\|_{\infty}
+(2+4\delta^{-1}) \|\hat F_{n}-F\|_{\infty}\right).
\end{eqnarray*}
A similar bound is valid replacing ${\mathbf d}_n(\cdot)-{\mathbf d}(\cdot)$ by ${\mathbf s}_n(\cdot)-{\mathbf s}(\cdot)$.
\end{proposition}

\begin{proof}
Recalling the previous proof, we have the decomposition
$$ \Delta_{n}(\theta,\cdot)=\left(\hat G_n-\hat {F}_{n}\otimes \hat {F}_{n}\right)-pa \left(G^0-F^0\otimes F^0\right)-\left(\frac{b}{r}-1\right) ( \hat {F}_{n}-F^0)\otimes ( \hat {F}_{n}-F^0),$$
 which brings 
$|\Delta_n(\theta,\cdot)|\leq  1+|pa|+\left|\frac{b}{r}-1\right|\leq 3+\delta^{-1}$.
In the same way $|\Delta(\theta,\cdot)|\leq   3+\delta^{-1}$.
%
%First, since $G^0, F^0$, and $F^1$ are bounded by 1,  by using equality \eqref{Deltabar},  we obtain that
%\begin{eqnarray*}
%|\Delta(\theta,x,y)|\leq |c_1(\theta^*,\theta)|+|c_2(\theta^*,\theta)|,\end{eqnarray*}
%for any $(x,y)\in E^2$. Recall also, see \eqref{c1c2},  that 
%\begin{equation*}
%c_1(\theta^*,\theta)=2(r-r^*)+(r^*b^*-rb),\quad c_2(\theta^*,\theta)=\frac{r^*}{r}(b^*r-r^*b),
%%(r^2-r^{*2})+(r^*b^*-rb)
%\end{equation*}
%with $r,r^*,b,b^*$ lying in $[\delta,1]$. This leads to $|c_1(\theta^*,\theta)|\leq 3, |c_2(\theta^*,\theta)|\leq \delta^{-1}$ and 
%we can finally use the bound
%$|\Delta(\theta,\cdot)|\leq 4\delta^{-1}$. Moreover, see previous proof, we have the decomposition
%$$ \Delta_{n}(\theta,\cdot)=\left(\hat G_n-\hat {F}_{n}\otimes \hat {F}_{n}\right)-pa \left(G^0-F^0\otimes F^0\right)-\left(\frac{b}{r}-1\right) ( \hat {F}_{n}-F^0)\otimes ( \hat {F}_{n}-F^0),$$
% which brings 
%$|\Delta_n(\theta,\cdot)|\leq  1+|pa|+\left|\frac{b}{r}-1\right|\leq 4\delta^{-1}$.
Hence we have the following majorizations: 
\begin{eqnarray*}
\left|\Delta_n^2-\Delta^2\right|
\leq  |\Delta_n-\Delta|
(|\Delta_n|+|\Delta|)
\leq  (6+2\delta^{-2})|\Delta_n-\Delta|,
\end{eqnarray*}
along with
\begin{eqnarray*}
|{\mathbf d}_n(\theta)-{\mathbf d}(\theta) |
\leq (6+2\delta^{-1} )\iint |\Delta_n-\Delta| (\theta,x,y)dH(x,y).
\end{eqnarray*}
Now
\begin{eqnarray*}
\Delta_n(\theta,\cdot)-\Delta(\theta,\cdot)&=&\left(\hat G_n-G\right)-\lambda_2 F^0\otimes \left(\hat F_{n,\theta}^1-F_\theta^1\right)
-\lambda_3\left(\hat F_{n,\theta}^1-F_\theta^1\right )\otimes F^0\\
&&
-\lambda_4\left(\hat F_{n,\theta}^1\otimes \hat F_{n,\theta}^1-F_\theta^1\otimes F_\theta^1\right),
\end{eqnarray*}
and using the definition of $\hat F_{\theta}^1$ and $\hat F_{n,\theta}^1$, given respectively in (\ref{cdfinversion}) and (\ref{inv-plug-emp-cdf}), it comes 
\begin{eqnarray*}
\Delta_n(\theta,\cdot)-\Delta(\theta,\cdot)&=&\left(\hat G_n-G\right)-\frac{\lambda_2}{r} F^0\otimes \left(\hat F_{n}-F\right)
-\frac{\lambda_3}{r}\left(\hat F_{n}-F\right)\otimes F^0\\
&&
-\frac{\lambda_4}{r^2}\left(\left(\hat F_{n}-pF^0\right)\otimes \left(\hat F_{n}-pF^0\right)-\left(F-pF^0\right)\otimes \left(F-pF^0\right)\right),
\end{eqnarray*}
with $\lambda_2/r=\lambda_3/r=\beta$ and $\lambda_4/r=1-\beta=b$. Then 
\begin{eqnarray*}
\Delta_n(\theta,\cdot)-\Delta(\theta,\cdot)
&=&\left(\hat G_n-G\right )+\left(\frac{b}{r}-1\right) \left[F^0\otimes \left(\hat F_{n}-F\right)+\left(\hat F_{n}-F\right )\otimes F^0\right]
\\&&
-\frac{b}{r}\left(
\hat F_{n}\otimes \hat F_{n}-F \otimes F\right).
\end{eqnarray*}
Regarding  the last term we can note that
\begin{eqnarray*}
|\hat F_{n}(x) \hat F_{n}(y)-F (x) F(y) |&=&
|\hat F_{n}(x) (\hat F_{n}(y)-F(y))+F(y)(\hat F_n(x)-F (x)) |\\
&\leq &|\hat F_{n}(y)-F(y)|+|\hat F_n(x)-F (x)|,
\end{eqnarray*}
for all $(x,y)\in E^2$. This leads to 
\begin{eqnarray*}
|\Delta_n(\theta,x,y)-\Delta(\theta,x,y) |
\leq
|\hat G_n-G|(x,y)
+\left(1+\frac{2b}{r}\right ) \left[|\hat F_{n}-F |(x)+ |\hat F_{n}-F |(y)\right],
\end{eqnarray*}
with $\sup_{\theta\in \Theta} \left(1+2b/r\right)\leq 1+2\delta^{-1}$. We finally obtain
\begin{eqnarray*}
\sup_{\theta\in \Theta}
|{\mathbf d}_n(\theta)-{\mathbf d}(\theta) |
\leq (6+2\delta^{-1} )\left(\|\hat G_n-G\|_{\infty}
+(2+4\delta^{-1}) \|\hat F_{n}-F\|_{\infty}\right),
\end{eqnarray*}
by recalling that  we assumed $\iint dH=1$. 
\end{proof}

\begin{proposition}\label{GlivenkoCantelli}
If the process $(Z_i)_{i\geq 1}$ is stationary, we  have 
\begin{eqnarray*}
\|\hat G_n-G\|_{\infty}\stackrel{a.s} {\longrightarrow} 0, \quad \text{ and } \quad 
 \|\hat F_{n}-F\|_{\infty}\stackrel{a.s} {\longrightarrow} 0, \quad \mbox{as} \quad n\rightarrow +\infty.
\end{eqnarray*}
\end{proposition}

\begin{proof}

%Denote $p=\lfloor n/2 \rfloor$ and define, for all $i\geq 1$, $S_i=(Z_{2i},Z_{2i+1})$, resp. $T_i=(Z_{2i+1},Z_{2i+2})$, along with $\hat G_{S,n}$, $\hat G_{T,n}$ their respective empirical distribution functions.
%Let us notice first that
%\begin{eqnarray*}
%(n-1) \hat G_n(x,y)&=&\sum_{i=1}^{n-1} \I_{\left\{Z_i\leq x,Z_{i+1}\leq y\right\}}=\sum_{i\text{ even}} \I_{\left\{Z_i\leq x,Z_{i+1}\leq y\right\}}
%+\sum_{i\text{ odd}} \I_{\left\{Z_i\leq x,Z_{i+1}\leq y\right\}}\\
%&=&\sum_{k=1}^{p} \I_{\left\{Z_{2k}\leq x,Z_{2k+1}\leq y\right\}}
%+\sum_{k=0}^{p-1} \I_{\left\{Z_{2k+1}\leq x,Z_{2k+2}\leq y\right\}}\\
%&=&\sum_{k=1}^{p} \I_{\left\{S_k\leq (x,y)\right\}}
%+\sum_{k=0}^{p-1} \I_{\left\{T_k\leq (x,y)\right\}}\\
%&=&p\hat G_{S,p}(x,y)+p\hat G_{T,p}(x,y).
%\end{eqnarray*}
%Then we can   write 
%\begin{eqnarray*}
% \hat G_n(x,y)-G(x,y)&=&
%\frac{1}{n-1}\left(p\hat G_{S,p}(x,y)+p\hat G_{T,p}(x,y)\right)-G(x,y)\\
%&=&\frac{p}{n-1}\left(\hat G_{S,p}(x,y)+\hat G_{T,p}(x,y)-2G(x,y)\right )+G(x,y)\left(\frac{2p}{n-1}-1\right),
%\end{eqnarray*}
%and 
%\begin{eqnarray*}
% \|G_n-G\|_{\infty}\leq
%\frac{\lfloor n/2 \rfloor}{n-1}\left(\|\hat G_{S,p}-G\|_{\infty}+\|\hat G_{T,p}-G\|_{\infty}\right )+\|G\|_{\infty}\left(\frac{2\lfloor n/2 \rfloor}{n-1}-1\right ).\\
%\end{eqnarray*}

{Denote $R_i=(Z_{i},Z_{i+1})$, $i\geq 1$,  then $\hat G_n$ turns out to be  the  empirical distribution function of $R=(R_i)_{i\geq 1}$. Remark  also that $Z$ and $R$ are stationary sequences}.
Then, using \cite{athreya2016} Theorem 2 and Remark 3, it comes that the Glivenko-Cantelli result holds  for $\hat F_n$, as well as {for $\hat G_n$},
%$\hat G_{S,n}$,  and  $\hat G_{T,n}$, 
which  concludes the proof. 
\end{proof}

\subsection{Asymptotic normality proof}
\label{sec:preuvenormasymp}

Let us recall, see expression \eqref{Ggamma}, that 
$$\Delta(\theta,\cdot)=G-F\otimes F-(1-2r+rb) \left(G^0-F^0\otimes F^0\right)-\left(\frac{b}{r}-1\right) (F-F^0)\otimes \left(F-F^0\right),$$
where we recall that $b=1-\beta$, $p=\beta/(\alpha+\beta),$  and $r=\alpha/(\alpha+\beta)$.
Then we denote 
\begin{eqnarray*}
  v_1:=-1+2r-rb, \quad \mbox{and} \quad v_2:=-\frac{b}{r}+1.
  \end{eqnarray*}
 We also use a set of new notations:
\begin{eqnarray*}
T_1& :=& G^0-F^0\otimes F^0\\
T_2& :=&( F-F^0)\otimes ( F-F^0)\\%=r^{*2} T_4\\
%T_4& =&(F^1-F^0)\otimes (F^1-F^0)\\
T_3& :=&G-F\otimes F\\
\hat T_2& :=&(\hat F_n-F^0)\otimes (\hat F_n-F^0)\\
\hat T_3& :=&\hat G_n- \hat F_n\otimes \hat F_n.
\end{eqnarray*}
Thus using \eqref{Ggamma} again, we obtain the following re-shaped decompositions for $\Delta$ and $\Delta_n$
%\textcolor{blue}{en fait on a $T_3=-v_1^*T_1-v_2^* T_2 $ }
\begin{eqnarray*}
\Delta(\theta,\cdot)&=&T_3-(1-2r+rb)T_1-(b/r-1)T_2=v_1T_1+v_2T_2+T_3,\\
\Delta_n(\theta,\cdot)&=&\hat T_3-(1-2r+rb)T_1-(b/r-1)\hat T_2=v_1T_1+v_2\hat T_2+\hat T_3.
\end{eqnarray*}
We  transform now our natural parameter $\theta=(\alpha,\beta)$ into a new parameter $v=(v_1,v_2)$ by using  the following {homeomorphisms} $h$ and $g=h^{-1}$:
  $$h : \begin{array}{rcl}
%\stackrel{o}{  \Theta }& \to  &  h( \stackrel{o}{  \Theta })=V \\ 
{  \Theta }& \to  &  h( {  \Theta })=V \\ 
  \theta & \mapsto  & v=(-1+2r-rb,-\frac{b}{r}+1)= \left(-\frac{\beta(1-\alpha)}{\alpha+\beta},\frac{\beta}{\alpha}(\alpha+\beta-1)\right),
  \end{array} $$
  and 
  $$g : \begin{array}{rcl}
% h( \stackrel{o}{  \Theta })=V  & \to  & \stackrel{o}{\Theta} \\ 
V  & \to  & {\Theta}\\
  v & \mapsto  & \theta=\left(\frac{\beta(1-\beta)}{\beta-v_2},\beta\right),\quad\text{ where } \beta=\sqrt{v_1v_2+v_2-v_1}.
  \end{array} $$
 Let us denote $v^*=h(\theta^*)=g^{-1}(\theta^*)$ and $\hat v_n=g^{-1}(\hat \theta_n)$. Actually, since $\theta^*$  is the (unique)  argmin of $\mathbf d$ over $ {\Theta }$, the new true parameter  $v^*$  is  also  the (unique) argmin of ${\mathbf d}\circ g$ on $V$. In the same way $\hat v_n$ is the argmin of ${\mathbf d}_n\circ g$. 
 Let us  denote from now on 
$$\mathtt{d}:={\mathbf d}\circ g,\quad \mbox{and}\quad  \mathtt{d}_n:={\mathbf d}_n\circ g.$$
%\textcolor{blue}{Definir plutot $\mathtt{d}:=\frac12{\mathbf d}\circ g$ pour eviter les facteurs 2 partout ?}
Hence, since ${\mathbf d}=\iint \Delta^2dH$, we have in particular 
 \begin{eqnarray*}
\mathtt{d}(v)=\iint (v_1T_1+v_2T_2+T_3)^2 dH,\quad \mbox{and}\quad \mathtt{d}_n(v)=\iint (v_1T_1+v_2\hat T_2+\hat T_3)^2dH.
\end{eqnarray*}
Note that these functions are  merely quadratic. With this notation $v^*=\argmin_V \mathtt{d}$ and $\hat v_n=\argmin_V\mathtt{d}_n$. 
We have %previously 
assumed that $\theta^*$ belongs to the interior of $\Theta$, denoted $\mathrm{Int}(\Theta)$. This implies that $\hat \theta_n$ almost surely belongs to the interior of $\Theta$ for $n$ large enough. Since 
$h( \mathrm{Int}(\Theta))\subset  \mathrm{Int}(h( \Theta))$, $v^*$ and $\hat v_n$ belongs to $\mathrm{Int}(V)$ too  for $n$ large enough. This ensures that we have both
\begin{equation}\label{gradientnul}
\dot{\mathtt{d}}( v^*)=0, \quad \mbox{and} \quad \dot{\mathtt{d}}_n(\hat v_n)=0,
\end{equation}
where $\dot{\mathtt{d}}$, respectively $\dot{\mathtt{d}}_n$,  stands for  the gradient of $\mathtt{d}$, resp. of $\mathtt{d}_n$.
For convenience matters we will denote $[\mathtt x]_i$ the $i$-th coordinate  of any vector in $\mathtt x$ and $[M]_{i,j}$ the $(i,j)$-th component of any matrix $M$.
Now if we are able to show  that
$$\sqrt{n}
%\left( 
\begin{pmatrix}[\hat v_n]_1-v_1^*\\
 [\hat v_n]_2- v_2^*\end{pmatrix} 
% -\begin{pmatrix}v_1^* \\ v_2^*\end{pmatrix}
%\right)
\stackrel{d}{\longrightarrow} W\sim \mathcal{N}(0,\Gamma),\quad \mbox{as}\quad n\rightarrow +\infty,$$
then the delta-method leads directly to
$$\sqrt{n}\left( \hat \theta_n -
\theta^*\right)=\sqrt{n}\left( g(\hat v_n) -
g(v^*)\right)\stackrel{d}{\longrightarrow} Dg(v^*)W, \quad \mbox{as}\quad n\rightarrow +\infty,$$
with $Dg(v)$ the differential at point $v=(v_1,v_2)$ of the function $g:(v_1,v_2)\mapsto (\alpha,\beta)$. 
Since $g=h^{-1}$ we have  $Dg(h(\theta))D{h}(\theta)=D(g\circ h)(\theta)=I$, which leads to
\begin{equation}\label{DgDh}
Dh({\theta})=\begin{pmatrix}
\frac{\beta^2+\beta}{(\alpha+\beta)^2} & 
\frac{\alpha^2-\alpha}{(\alpha+\beta)^2}\\ 
\frac{\beta-\beta^2}{\alpha^2} & \frac{\alpha+2\beta-1}{\alpha}\end{pmatrix},
\quad \text{and} \quad Dg(v)=[Dh(\theta)]^{-1}=
%\begin{pmatrix}
%* & *\\ 
%\frac{v_2-1}{2\beta} &\frac{v_1+1}{2\beta}  
%\end{pmatrix}.
\begin{pmatrix}
\frac{(\alpha+\beta)(\alpha+2\beta-1)}{2\beta^2}  & 
\frac{\alpha^2(1-\alpha)}{2\beta^2 (\alpha+\beta)}  \\ 
\frac{(\alpha+\beta)(\beta-1)}{2\alpha\beta} &\frac{\alpha(\beta+1)}{2\beta (\alpha+\beta)}  
\end{pmatrix}.
\end{equation}
Denoting $D^*=Dg(v^*)=[Dh({\theta^*})]^{-1}$, we obtain  $\Sigma=D^* \Gamma (D^*)^T$. 
\noindent We need  to study now  the limit of $\sqrt{n}(\hat v_n-v^*)$ as $n\rightarrow +\infty$. The function $\mathtt{d}_n$ is quadratic with constant Hessian matrix 
$$\ddot{\mathtt{d}}_n=2\begin{pmatrix}
\iint T_1^2dH  & \iint T_1\hat T_2dH \\ 
\iint T_1\hat T_2dH & \iint \hat T_2^2dH
\end{pmatrix} ,$$ 
and we have for any $v\in V$:
$$\dot{\mathtt{d}}_n(v)=\dot{\mathtt{d}}_n(v^*)+\ddot{\mathtt{d}}_n(v-v^*).$$
Then, applying this at  point $v=\hat v_n$,  and using \eqref{gradientnul}, we obtain 
\begin{eqnarray}\label{eq:taylorExp}
&&\sqrt{n}\ddot{\mathtt{d}_n}(\hat v_n-v^*)=-\sqrt{n}\Dot {\mathtt{d}_n}(v^*)=\sqrt{n}(\dot{\mathtt{d}}- \dot{\mathtt{d}_n})(v^*).
\end{eqnarray}
It remains to  show that 
$\sqrt{n}(\Dot{\mathtt{d}}- \dot{\mathtt{d}_n})(v^*)$ tends to a centered Gaussian variable (see Section \ref{sec:convdpoint}) and that  $\ddot{\mathtt{d}_n}$ tends to  an invertible matrix (see Section \ref{sec:hessian}) as $n\rightarrow +\infty$. To do this we need the following lemma.

\begin{lemma} \label{convergenceprocessusempiriques}
Denoting, for all $n\geq 1$,  $\mathbb{F}_n:=\sqrt{n}(\hat F_n-F)$ and
$\mathbb{G}_n:=\sqrt{n}(\hat G_n-G)$, we have the following asymptotic behaviors:
\begin{enumerate}[i)]
\item The empirical process $(\mathbb{F}_n)_{n\geq 1}$ converges in distribution to a Gaussian process $\mathcal{B}_F$
with covariance function $$\cov(\mathcal{B}_F(x),\mathcal{B}_F(y))=\sum_{k\in \bZ}\cov(\I_{\{Z_0\leq x},\I_{Z_k\leq y\}}).$$
\item The empirical process $(\mathbb{G}_n)_{n\geq 1}$ converges in distribution to a Gaussian process $\mathcal{B}_G$
with covariance function $$\cov(\mathcal{B}_G(x,y),\mathcal{B}_G(z,t))=\sum_{k\in \bZ}\cov(\I_{\{Z_0\leq x,Z_1\leq y\}},\I_{\{Z_k\leq z,Z_{k+1}\leq t\}}).$$
\item 
The joint process $(\mathbb{F}_n,\mathbb{G}_n)_{n\geq 1}$  converges in distribution to a Gaussian process $(\mathcal{B}_F,\mathcal{B}_G)$.
\end{enumerate}

\end{lemma}

\begin{proof} 

\begin{enumerate}[i)]
\item { Note that the Markov chain $X$ is geometrically $\alpha$-mixing. Indeed the transition matrix converges to the stationary distribution matrix with rate $(1-\alpha-\beta)^n$ where  $|1-\alpha-\beta|<1-2\delta<1$.}
Using our mixing assumptions, the process $(Z_i)_{i\geq 1}$ is $\alpha$-mixing with $\alpha^Z(n)\leq C n^{-a}$: see \citep{VDK05} Lemma 1(i) and its proof, or \cite{bradley2005} Theorem 5.2. The convergence of $(\mathbb F_n)_{n\geq 1}$ is  then a consequence of Theorem 7.2 of \cite{Rio2014}.

\item
Let us denote $R_i:=(Z_{i},Z_{i+1})$, for all $i\geq 1$. We can now observe that
\begin{eqnarray*}
 \hat G_n(x,y)=\frac{1}{n-1}\sum_{i=1}^{n-1} \I_{\left\{Z_i\leq x,Z_{i+1}\leq y\right\}}=\frac{1}{n-1}\sum_{i=1}^{n-1} \I_{\left\{R_i\leq (x, y)\right\}}
 \end{eqnarray*}
so that  $\hat G_n$ actually is the empirical distribution function of $(R_i)_{i\geq 1}$.
We thus  have 
$$\alpha^R(n)\leq \alpha^Z(n-1)\leq C(n/2)^{-a},$$
with $a>1$, where the $\alpha$-mixing coefficient is defined in (\ref{alphacoef}).  Denote by $B(T)$ the space of real-valued and bounded functions over $T$. 
Using Theorem 7.3 of \cite{Rio2014} and the continuity of $F$, we obtain that 
$\mathbb{G}_n$ converges in distribution to a Gaussian process $\mathcal{B}_G$ in the space $B(\R^2)$.

\item Now assume temporarily  that $Y^0, Y^1$ and then $Z$ are supported on $[0,1]$. We denote $\tilde{\mathbb{G}}_n=\sqrt{(n-1)/n} \mathbb{G}_n$. Then, for all $x\in E$,  $G(x,1)=F(x)$ and 
\begin{eqnarray*}
\tilde{\mathbb{G}}_n(x,1)&=&\sqrt{n-1} \left(\frac{1}{n-1}\sum_{i=1}^{n-1} \I_{\left\{Z_i\leq x,Z_{i+1}\leq 1\right\}}-G(x,1)\right)\\
&=&\sqrt{n-1} \left(\hat F_{n-1}(x)-F(x)\right)=\mathbb{F}_{n-1}(x).
\end{eqnarray*}
Let $$\phi : 
\begin{array}{ccc}
B(\R^2) & \to & B(\R)\times B(\R^2) \\ 
w & \mapsto & ( w(.,1),w).
\end{array} 
$$
Then $\phi$ is  linear and 
$\|w(.,1)\|_\infty \leq \|w\|_\infty$ so $\phi$ is a continuous linear operator. We also have $\phi(\tilde{\mathbb{G}}_n)=(\mathbb{F}_{n-1},\tilde{\mathbb{G}}_n)$.
Since $\tilde{\mathbb{G}}_n$ converges in distribution  to a Gaussian process $\mathcal{B}$  in  the space $B(\R^2)$, the continuous mapping theorem implies that $(\mathbb{F}_{n-1},\tilde{\mathbb{G}}_n)$ (and then $(\mathbb{F}_n,\mathbb{G}_n)$) converges in distribution  to the Gaussian process $\phi(\mathcal{B})$.

If now we do not assume anymore that $Z$ is supported on  $[0,1]$, it is sufficient to consider $U_i=F(Z_i)$. It is well-known that $U_i$ follows an uniform distribution (recall that $F$ is continuous). We denote by $\hat F_n^{U}$ the empirical distribution function of $(U_i)_{1\leq i\leq n}$ and by $\hat G_n^U$ the one of $(U_i,U_{i+1})_{1\leq i\leq n-1}$. Analogously, we use notations $F^{U}, G^{U}, \mathbb{F}_n^{U}, \mathbb{G}_n^{U}$ for the true distribution functions and for the empirical processes. Observe that, for all $x\in E$, we have 
$$\hat F_n(x)=\frac{1}{n}\sum_{i=1}^n\I_{\left\{Z_i\leq x\right\}}=
\frac{1}{n}\sum_{i=1}^n\I_{\left\{F(Z_i)\leq F(x)\right\}}=\hat F_n^{U}(F(x)).$$
In the same way $F(x)=F^{U}(F(x))$, $\hat G_n(x,y)=\hat G^{U}_n(F(x),F(y))$, $G(x,y)=G^U(F(x),F(y))$ and finally
$$\mathbb{F}_n(x)=\mathbb{F}_n^U(F(x)), \qquad \mathbb{G}_n(x,y)=\mathbb{G}_n^U(F(x),F(y)).$$
Since $U$ has its support on  $[0,1]$ we can apply the previous reasoning to $(\mathbb{F}_n^{U},\mathbb{G}_n^{U})_{n\geq1}$. This ensures the convergence in distribution of $(\mathbb{F}_n,\mathbb{G}_n)_{n\geq1}$.  
\end{enumerate}

\end{proof}

\subsubsection{Hessian convergence}
\label{sec:hessian}

We shall show that $\ddot{\mathtt{d}}_n$ converges to $\ddot{\mathtt{d}}$ in probability as $n\rightarrow +\infty$.
 The function $\mathtt{d}$ is quadratic with constant Hessian matrix 
 \begin{equation}\label{dpoinpoint}
 \ddot{\mathtt{d}}=2\begin{pmatrix}
\iint T_1^2dH  & \iint T_1T_2dH \\ 
\iint T_1 T_2dH & \iint T_2^2dH
\end{pmatrix}.
 \end{equation}
Moreover $\ddot{\mathtt{d}}$ is invertible if and only if $(\iint T_1T_2dH)^2\neq \iint T_1^2dH \iint T_2^2dH$. But the equality case in the Cauchy-Schwarz inequality happens  if $T_1$ is proportional to $T_2$.  Assumption\textbf{ LinInd} ensures that $T_1$ is not proportional to $(F^1-F^0)\otimes (F^1-F^0)=T_2/r^{*2}$. Thus $\ddot{\mathtt{d}}$ is invertible. It remains to show that $(\ddot{\mathtt d}_n-\ddot{\mathtt d})\to 0$ in probability as $n\rightarrow +\infty$.  
First $[\ddot{\mathtt{d}}_n-\mathtt{\ddot d}]_{1,1}=0. $ Moreover we have 
$$[\ddot{\mathtt{d}}_n-\mathtt{\ddot d}]_{1,2}=2\iint T_1(\hat T_2-T_2)dH
=\frac2{\sqrt{n}}\iint T_1 \sqrt{n} (\hat T_2-T_2)dH,$$ and %
\begin{eqnarray*}
[\ddot{\mathtt{d}}_n-\mathtt{\ddot d}]_{2,2}&=&2\iint ({\hat T}_2^2-T_2^2)dH
=2\iint ({\hat T}_2-T_2)^2dH
+4\iint T_2 ({\hat T}_2-T_2)dH\\
&=&\frac{2}{n}\iint \left(\sqrt{n}({\hat T}_2-T_2)\right)^2dH
+\frac4{\sqrt{n}}\iint T_2 \sqrt{n} ({\hat T}_2-T_2)dH.
\end{eqnarray*}
Note that for any functions $\phi$ and $\hat \phi$ we have the identity
$$\phi\otimes \phi-\hat \phi\otimes \hat \phi=\phi\otimes ( \phi-\hat \phi)+
(\phi-\hat \phi)\otimes \hat \phi,$$
which gives us 
\begin{eqnarray*}
T_2-\hat T_2 =(F-F^0)\otimes ( F-\hat F_n)+(F-\hat F_n)\otimes (\hat F_n-F^0).
%T_3-\hat T_3 =(G-\hat G)-F\otimes ( F-\hat F)-(F-\hat F)\otimes \hat F
\end{eqnarray*}
Recalling the notation $\mathbb{F}_n:=\sqrt{n}(\hat F-F)$, this leads to the following decomposition
\begin{eqnarray*}
\sqrt{n}(T_2-\hat T_2)=(F^0-F)\otimes \mathbb{F}_n+\mathbb{F}_n\otimes (F^0-F)-\frac{1}{\sqrt{n}} \mathbb{F}_n\otimes \mathbb{F}_n.\\
\end{eqnarray*}
Then the convergence of $\mathbb{F}_n$ induces the one of 
$\iint T_1 \sqrt{n} ({\hat T}_2-T_2)dH$, $\iint T_2 \sqrt{n} ({\hat T}_2-T_2)dH$ and $\iint \left(\sqrt{n}({\hat T}_2-T_2)\right)^2dH$, for this  we use 
 that  $F,F^0,T_1,T_2$ are bounded  and $\iint dH<\infty$. 
A final use of Slutsky's theorem gives then the convergence in probability of $[\ddot{\mathtt{d}}_n-\mathtt{\ddot d}]_{1,2}$ and $[\ddot{\mathtt{d}}_n-\mathtt{\ddot d}]_{2,2}$ to 0. Thus $\mathtt{\ddot d}_n$ converges in probability to $\mathtt{\ddot d}$, which is invertible.

\subsubsection{Convergence of  $\sqrt{n}(\dot{\mathtt{d}}-\dot{\mathtt{d_n}})(v^*)$}
\label{sec:convdpoint}

%In this section, we sometimes omit $dH$ in the integrals for the sake of simplicity.
Recall first  that
\begin{eqnarray*}
\mathtt{d}(v)=\iint (v_1T_1+v_2T_2+T_3)^2 dH, \quad\mbox{and}\quad  \mathtt{d}_n(v)=\iint (v_1T_1+v_2\hat T_2+\hat T_3)^2dH.
\end{eqnarray*}
Then the gradients are 
$$\Dot{\mathtt{d}}(v)=2\begin{pmatrix}
\iint  T_1(v_1T_1+v_2T_2+T_3) \\ 
\iint  T_2(v_1T_1+v_2T_2+T_3)
\end{pmatrix},
\quad\mbox{and}  \quad
\Dot{\mathtt{d}}_n(v)=2\begin{pmatrix}
\iint  T_1(v_1T_1+v_2\hat T_2+\hat T_3) \\ 
\iint  \hat T_2(v_1T_1+v_2\hat T_2+\hat T_3)
\end{pmatrix}, 
$$
and then
\begin{eqnarray*}
[\Dot{\d}-\Dot{\d}_{n}]_1(v)=2v_2\iint  T_1(T_2-\hat T_2)dH+2\iint  T_1(T_3-\hat T_3)dH,
\end{eqnarray*}
and
\begin{eqnarray*}
[\Dot{\d}-\Dot{\d}_{n}]_2(v)&=&2\iint (v_1 T_1+v_2T_2+T_3)( T_2-\hat T_2) dH
+2v_2\iint   \hat T_2(T_2-\hat T_2)dH
\\&&
+2\iint   \hat T_2(T_3-\hat T_3)dH.
\end{eqnarray*}
We use again  that 
$\phi\otimes \phi-\hat \phi\otimes \hat \phi=\phi\otimes ( \phi-\hat \phi)+
(\phi-\hat \phi)\otimes \hat \phi$ for any function $\phi$ and $\hat \phi$.
Thus
\begin{eqnarray*}
T_2-\hat T_2 &=&(F-F^0)\otimes ( F-\hat F_n)+(F-\hat F_n)\otimes (\hat F_n-F^0)\\
T_3-\hat T_3 &=&(G-\hat G_n)-F\otimes ( F-\hat F_n)-(F-\hat F_n)\otimes \hat F_n,
\end{eqnarray*}
%
%Denoting $$\mathbb{F}_n:=\sqrt{n}(\hat F-F)\text{ and }
%\mathbb{G}_n=\sqrt{n}(\hat G-G)$$
%we write
and then
\begin{eqnarray}
\sqrt{n}(T_2-\hat T_2)&=&(F^0-F)\otimes \mathbb{F}_n+\mathbb{F}_n\otimes (F^0-F)-\frac{1}{\sqrt{n}} \mathbb{F}_n\otimes \mathbb{F}_n \label{diffT2}\\
\sqrt{n}(T_3-\hat T_3) &=&-\mathbb{G}_n+F\otimes \mathbb{F}_n+\mathbb{F}_n\otimes F+ \frac{1}{\sqrt{n}} \mathbb{F}_n\otimes \mathbb{F}_n\label{diffT3}.
\end{eqnarray}
We have 
\begin{eqnarray*}
\sqrt{n} [\Dot{\d}-\Dot{\d}_{n}]_1(v^*)
 &=&2v_2^*\iint  T_1\left((F^0-F)\otimes \mathbb{F}_n+\mathbb{F}_n\otimes (F^0-F)-\frac{1}{\sqrt{n}} \mathbb{F}_n\otimes \mathbb{F}_n\right)dH\\
 &&+2\iint  T_1\left(-\mathbb{G}_n+F\otimes \mathbb{F}_n+\mathbb{F}_n\otimes F+ \frac{1}{\sqrt{n}} \mathbb{F}_n\otimes \mathbb{F}_n\right)dH\\
 &=& -2\iint  T_1\mathbb{G}_ndH+2\iint T_1 ( A\otimes \mathbb{F}_n+\mathbb{F}_n\otimes A)dH\\
 &&+ \frac{2(1-v_2^*)}{\sqrt{n}} \iint T_1  (\mathbb{F}_n\otimes \mathbb{F}_n) dH,
\end{eqnarray*}
with $A=v_2^*(F^0-F)+F=(1-b^*)F^0+b^*F^1$. 
In other words
\begin{eqnarray*}
\sqrt{n} [\Dot{\d}_1-\Dot{\d}_{n}]_1(v^*)
 = 2\iint  \left(-T_1\mathbb{G}_n+T_1\overline{A\otimes \mathbb{F}_n}\right)dH+ \frac{1}{\sqrt{n}} \chi_1(\mathbb{F}_n)
\end{eqnarray*}
where $\chi_1(\phi)=2(1-v_2^*)\iint T_1  (\phi\otimes \phi)dH$ and $\overline{a\otimes b}=a\otimes b+b\otimes a$.
% \textcolor{magenta}{A remplacer par $a\otimes_{sym} b$ ?}
%
%
%
In the same way we can compute (see Section~\ref{sec:detailsdiffgradient})
\begin{eqnarray*}
\sqrt{n}[ \Dot{\d}-\Dot{\d}_{n}]_2 (v^*)=
2\iint \left(-T_2 \mathbb{G}_n+
 B\overline{(F^0-F)\otimes \mathbb{F}_n} +T_2 \overline{F\otimes \mathbb{F}_n}\right)dH+\frac{1}{\sqrt{n}}\chi_2(\mathbb{F}_n,\mathbb{G}_n),
 \end{eqnarray*}
where $B=v_1^*T_1+2v_2^*T_2+T_3$ and $\chi_2$ is a continuous mapping.
We denote
\begin{equation}\label{psiVW}
\Psi(\mathbb{F}_n,\mathbb{G}_n)=2\begin{pmatrix}
\iint  \left(-T_1\mathbb{G}_n+T_1\overline{A\otimes \mathbb{F}_n}\right)dH\\ 
\iint \left(-T_2 \mathbb{G}_n+
 B\overline{(F^0-F)\otimes \mathbb{F}_n} +T_2 \overline{F\otimes \mathbb{F}_n}dH\right)
\end{pmatrix} ,
\end{equation}
so that 
\begin{eqnarray}\label{psichi}
\sqrt{n}(\dot \d-\dot \d_n)(v^*)=\Psi(\mathbb{F}_n,\mathbb{G}_n)+\frac{1}{\sqrt{n}}\chi(\mathbb{F}_n,\mathbb{G}_n),
\end{eqnarray}
where $\Psi$ and $\chi$ are continuous mappings. The continuity from  $B(\R^2)\times B(\R)$ to $\R^2$ comes from the boundness of $F^0,F^1,T_i$ and $\iint dH<\infty$. Using Slutsky's theorem, we have 
$$\frac{1}{\sqrt{n}}\chi(\mathbb{F}_n,\mathbb{G}_n)\overset{\P}{\longrightarrow}0,\quad \mbox{as}\quad n\rightarrow +\infty,$$ and finally, Lemma~\ref{convergenceprocessusempiriques} gives us
$$\sqrt{n}(\dot \d-\dot \d_n)(v^*)\overset{d}{\longrightarrow}\Psi(\mathcal{B}_F,\mathcal{B}_G), \quad \mbox{as}\quad n\rightarrow +\infty.$$
Note that $\Psi$ is a linear mapping depending on $F^0,F^1,G^0$, and $\theta^*$. Then $\Psi(\mathcal{B}_F,\mathcal{B}_G)$ is a centered Gaussian variable 
$\mathcal{N}(0, \Lambda)$.

\subsubsection{Covariance matrix}
\label{preuvecov}

In the previous sections, we have shown that $\ddot \d_n$ converges in probability to the invertible matrix $\ddot \d$ and that
$$\sqrt{n}(\dot \d-\dot \d_n)(v^*)\overset{d}{\longrightarrow}\Psi(\mathcal{B}_F,\mathcal{B}_G)=
\mathcal{N}(0, \Lambda),\quad \mbox{as}\quad n\rightarrow +\infty.$$ 
Let us make the matrix $\Lambda$ explicit. Using \eqref{psiVW}, we can write $$\Psi(\mathcal{B}_F,\mathcal{B}_G)=\iint M(x,y) \begin{pmatrix}
\mathcal{B}_F(x) \\ 
\mathcal{B}_F(y) \\ 
\mathcal{B}_G(x,y)
\end{pmatrix} dH(x,y),$$
with 
{\small $$M(x,y)=2\begin{pmatrix}
T_1(x,y)A(y) & T_1(x,y)A(x) & -T_1(x,y) \\ 
T_2(x,y)F(y) +B(x,y) (F^0-F)(y)& T_2(x,y)F(x)+B(x,y) (F^0-F)(x) & -T_2(x,y)
\end{pmatrix} ,$$}
where  $A=v_2^*(F^0-F)+F$ and $B=v_1^*T_1+2v_2^*T_2+T_3$.
Let us denote 
$$
C(x,y,z,t):=\cov(\mathcal{B}_G(x,y),\mathcal{B}_G(z,t))=\sum_{k\in \bZ}\cov\left(\I_{\{Z_0\leq x, Z_1\leq y\}}, \I_{\{Z_k\leq z,Z_{k+1}\leq t\}}\right),$$
and observe that $\mathcal{B}_F(x)=\lim_{y\to \infty} \mathcal{B}_G(x,y)$.
Then the covariance matrix of the vector $\mathcal{T}(x,y):=(\mathcal{B}_F(x), \mathcal{B}_F(y),
\mathcal{B}_G(x,y))^T$ is 
\begin{eqnarray*}
C_{\mathcal{T}}(x,y):=\begin{pmatrix}
C(x,\infty,x,\infty) & C(x,\infty,y,\infty) & C(x,\infty,x,y)\\ 
C(x,\infty,y,\infty) & C(y,\infty,y,\infty) & C(y,\infty,x,y) \\ 
C(x,\infty,x,y) & C(y,\infty,x,y) & C(x,y,x,y)
\end{pmatrix}.
\end{eqnarray*} 
Then we obtain the following representation of $\Lambda$:
$$\Lambda=\iint M(x,y)C_{\mathcal{T}}(x,y)M^T(x,y) dH(x,y).$$
Moreover, for $n$ large enough,
$$\sqrt{n}(\hat v_n-v^*)=\ddot{\mathtt{d}_n}^{-1}\left[\sqrt{n}(\dot{\mathtt{d}}- \dot{\mathtt{d}_n})(v^*)\right].
$$
Thus $\sqrt{n}(\hat v_n-v^*)\overset{d}{\longrightarrow}
\mathcal{N}(0, \Gamma)$ as $n\rightarrow +\infty$, with an asymptotic covariance matrix given by
$$\Gamma=\ddot{\mathtt{d}}^{-1}\Lambda\ddot{\mathtt{d}}^{-1}.$$
From \eqref{dpoinpoint}, we can write
$$\ddot{\mathtt{d}}^{-1}=\frac{1/2}{\iint T_1^2dH \iint T_2^2dH-\left(\iint T_1T_2dH\right)^2} \begin{pmatrix}
\iint T_2^2dH  & -\iint T_1T_2dH \\ 
-\iint T_1 T_2dH & \iint T_1^2dH
\end{pmatrix},$$
where we recall that 
\begin{eqnarray*}
T_1& =& G^0-F^0\otimes F^0,\\
T_2& =&( F-F^0)\otimes ( F-F^0)=r^{*2}( F^1-F^0)\otimes ( F^1-F^0),\\
T_3& =&G-F\otimes F=-v_1^*T_1-v_2^*T_2.
\end{eqnarray*}
Finally, denoting $D^*=Dg(v^*)=[Dh({\theta^*})]^{-1}$, we have $\Sigma=D^* \Gamma (D^*)^T$
and thus 
$$\Sigma=\iint \left(D^*\ddot{\mathtt{d}}^{-1}M(x,y) \right)C_{\mathcal{T}}(x,y) \left(D^*\ddot{\mathtt{d}}^{-1}M(x,y) \right)^TdH(x,y),$$
with $\ddot{\mathtt{d}}^{-1}, M, C_{\mathcal{T}} $ defined just above in this section   and $D^*$ detailed in \eqref{DgDh}.\qed

%%%%%%%%%%%%%%%%%%%%%%%%%%%%%%%%%%%%%
\subsection{Proof of Theorem~\ref{theo:normasymp} }

Recall that \begin{eqnarray*}
\hat F_n^1(x)=\frac{1}{\hat r_n} \left(\hat F_n(x)-\hat {p}_n F^0(x)  \right),\quad x\in E.
\end{eqnarray*}
Then, dropping for simplicity matters the dependence on $x$, we get
\begin{eqnarray*}
\hat F_n^1-F^1&=&\frac{1}{ r^*} (\hat F_n-F)+\left(\frac{1}{\hat r_n}-\frac{1}{r^*}\right)\hat F_n+\left(\frac{p^*}{r^*}-\frac{\hat p_n}{\hat r_n}\right) F^0\\
&=&\frac{1}{ r^*} (\hat F_n-F)+\left(\frac{\hat \beta_n}{\hat \alpha_n}-\frac{\beta^*}{\alpha^*}\right)(\hat F_n-F^0),
\end{eqnarray*}
since $1/r=1+p/r=1+\beta/\alpha$.
Thus, recalling that $\mathbb{F}_n=\sqrt{n} \left(\hat F_n-F\right)$,
\begin{eqnarray}\label{aa}
\sqrt{n}\left(\hat F_n^1-F^1\right)
=\frac{1}{ r^*}\mathbb{F}_n+\sqrt{n}\left(\frac{\hat \beta_n}{\hat \alpha_n}-\frac{\beta^*}{\alpha^*}\right)\left(F-F^0+\frac{\mathbb{F}_n}{\sqrt{n}}\right).
\end{eqnarray}

We shall use the  change of variable  already used for the proof of asymptotic normality (see Section~\ref{sec:preuvenormasymp}). Recall that parameter $\theta=(\alpha,\beta)$ is transformed into a new parameter $v=(v_1,v_2)$ by using  the following {homeomorphisms} $h$ and $g=h^{-1}$:
  $$h : \begin{array}{rcl}
%\stackrel{o}{  \Theta }& \to  &  h( \stackrel{o}{  \Theta })=V \\ 
{  \Theta }& \to  &  h( {  \Theta })=V \\ 
  \theta & \mapsto  & v=(-1+2r-rb,-\frac{b}{r}+1)= \left(-\frac{\beta(1-\alpha)}{\alpha+\beta},\frac{\beta}{\alpha}(\alpha+\beta-1)\right),
  \end{array} $$
  and 
  $$g : \begin{array}{rcl}
% h( \stackrel{o}{  \Theta })=V  & \to  & \stackrel{o}{\Theta} \\ 
V  & \to  & {\Theta}\\
  v & \mapsto  & \theta=\left(\frac{\beta(1-\beta)}{\beta-v_2},\beta\right),\quad\text{ where } \beta=\sqrt{v_1v_2+v_2-v_1}.
  \end{array} $$
Define now the ratio function $\varphi_1: \Theta\rightarrow \R^{+*}$ that does the  $(\alpha,\beta)\mapsto \frac{\beta}{\alpha}$ mapping.  Define also $\varphi_2=\varphi_1\circ g:V\rightarrow \R^{+*}$. With this notation we can write 
$$\frac{\beta}{\alpha}=\varphi_1(\theta)=\varphi_1\circ g(v)=\varphi_2(v).$$
Coming back to \eqref{aa}, we obtain 
\begin{eqnarray}\label{bb}
\sqrt{n}(\hat F_n^1-F^1)
=\frac{1}{ r^*}\mathbb{F}_n+
\left[\sqrt{n}\left(\varphi_2(\hat v_n)-\varphi_2(v^*)\right)\right]
\left(F-F^0+\frac{\mathbb{F}_n}{\sqrt{n}}\right).
\end{eqnarray}
We know that $(\mathbb{F}_n,\mathbb{G}_n)$ tends  in distribution to $(\mathcal{B}_F,\mathcal{B}_G)$ as $n\rightarrow +\infty$, and 
we have already proved in Section~\ref{sec:preuvenormasymp} that, for $n$ large enough (as soon as $\ddot{\mathtt{d}}_n$ is invertible), 
\begin{eqnarray*}
\sqrt{n}(\hat v_n-v^*)&=&(\ddot{\mathtt{d}_n}^{-1})\left[\sqrt{n}(\dot{\mathtt{d}}- \dot{\mathtt{d}_n})(v^*)\right]\\
&=&(\ddot{\mathtt{d}_n}^{-1})\left[\Psi(\mathbb{F}_n,\mathbb{G}_n)+\frac{1}{\sqrt{n}}\chi(\mathbb{F}_n,\mathbb{G}_n)\right],
\end{eqnarray*}
which tends in distribution to $\ddot{\mathtt{d}}^{-1}\Psi(\mathcal{B}_F,\mathcal{B}_G)$ as $n\rightarrow +\infty$. 
Since $v^*$ is an interior point of $V$, the line segment $[\hat v_n, v^*]$ belongs to $V$ for $n$ large enough, and the Taylor theorem for real-valued functions provides:
$$\varphi_2(\hat v_n)-\varphi_2(v^*)=D\varphi_2(\tilde v_n)
(\hat v_n -v^*),$$
for some $\tilde v_n \in [\hat v_n, v^*]$. Moreover the smoothness (continuity)  of $\varphi_2$ entails the almost sure convergence of $D\varphi_2(\tilde v_n)$ towards $D\varphi_2( v^*)$.
Thus, for $n$ large enough, we have  
\begin{eqnarray*}
\sqrt{n}(\varphi_2(\hat v_n)-\varphi_2(v^*))
=(D\varphi_2(\tilde v_n)\ddot{\mathtt{d}_n}^{-1})\left[\Psi(\mathbb{F}_n,\mathbb{G}_n)+\frac{1}{\sqrt{n}}\chi(\mathbb{F}_n,\mathbb{G}_n)\right],
\end{eqnarray*}
which tends in distribution to $D\varphi_2( v^*)\ddot{\mathtt{d}}^{-1}\Psi(\mathcal{B}_F,\mathcal{B}_G)$ as $n\rightarrow +\infty$.
Then,   according to the  Slutsky's theorem, we have
$$
R_n:=\frac{1}{\sqrt{n}}\left[\sqrt{n}\left(\varphi_2(\hat v_n)-\varphi_2(v^*)\right)\right]\mathbb{F}_n\overset{d}{\longrightarrow} 0,\quad \mbox{as}\quad n\rightarrow+\infty.$$
Hence expression \eqref{bb} becomes
\begin{eqnarray*}
\sqrt{n}(\hat F_n^1-F^1)
&=&\frac{1}{ r^*}\mathbb{F}_n+
\left[\sqrt{n}\left(\varphi_2(\hat v_n)-\varphi_2(v^*)\right)\right](F-F^0)+R_n\\
&=&\frac{1}{ r^*}\mathbb{F}_n+
(D\varphi_2(\tilde v_n)\ddot{\mathtt{d}_n}^{-1})\left[\Psi(\mathbb{F}_n,\mathbb{G}_n)+\frac{1}{\sqrt{n}}\chi(\mathbb{F}_n,\mathbb{G}_n)\right](F-F^0)+R_n,
\end{eqnarray*}
which gives us the wanted  convergence  in distribution 
\begin{eqnarray*}
&&\sqrt{n}(\hat F_n^1-F^1)
\stackrel{d}{\longrightarrow} \underbrace{\frac{1}{ r^*}\mathcal{B}_F+D\varphi_2( v^*)\ddot{\mathtt{d}}^{-1}\Psi(\mathcal{B}_F,\mathcal{B}_G)(F-F^0)}, \quad \mbox{as}\quad n\rightarrow +\infty.\\
&&~~~~~~~~~~~~~~~~~~~~~~~~~~~~~~~~~~~~~~~~~~~~~~~~~\mathcal{G} 
\end{eqnarray*}
\qed

\subsection{Proof of Theorem~\ref{theo:risqueF} }
\label{preuverisqueF}

$\bullet$ \underline{Proof of Theorem \ref{theo:risqueF} about $\theta$. }
First let us bound $\|\hat \theta_n-\theta^*\|$ by $\|\hat v_n -v^*\|$ up to a multiplicative constant. Since $v^*\in \mathrm{Int}(V)$ there exists an open ball $B$ with center $v^*$ such that: for all $v\in B$, the segment $[v,v^*]$ is included in $V$.  Then the mean value inequality provides
$$\|g(v)-g(v^*)\|\leq \|Dg\|_{\infty}
\|v -v^*\|\leq \frac{2}{\delta^2}\|v -v^*\|.$$
Moreover, on the compact set $V\backslash B$, the function 
$\|g(v)-g(v^*)\|/\|v -v^*\|$ is continuous and  then bounded. Finally there exists a constant $C(\delta)>0$ such that for all $v\in V$
$$\|g(v)-g(v^*)\|\leq C(\delta)\|v -v^*\|.$$
Thus $\|\hat \theta_n-\theta^*\|=\|g(\hat v_n)-g(v^*)\|\leq C(\delta)\|\hat v_n -v^*\|$
  and it is sufficient to bound the sequence $(\sqrt{n}\bE\|\hat v_n -v^*\|)_{n\geq 1}$. 
But \eqref{eq:taylorExp} gives
$$
\ddot{\mathtt{d}_n}(\hat v_n-v^*)=(\dot{\mathtt{d}}- \dot{\mathtt{d}_n})(v^*),
$$
and we know from \eqref{psichi} that
$$\sqrt{n}(\dot \d-\dot \d_n)(v^*)=\Psi(\mathbb{F}_n,\mathbb{G}_n)+\frac{1}{\sqrt{n}}\chi(\mathbb{F}_n,\mathbb{G}_n).$$
Now we use the following lemma on the empirical processes $\mathbb{F}_n$ and $\mathbb{G}_n$.

\begin{lemma} \label{supprocessusempiriques}
Recall that $\mathbb{F}_n:=\sqrt{n}(\hat F_n-F)$ and
$\mathbb{G}_n:=\sqrt{n}(\hat G_n-G)$.  Assume $(Z_i)_{i\geq 0}$ and $(Z_i,Z_{i+1})_{i\geq 0}$ are $\alpha$-mixing with $\alpha(n)=O(n^{-a})$ for $a>4$ (see  (\ref{alphacoef}) for the $\alpha$ coefficient definition). Then we have the following properties:
\begin{enumerate}[i)]
\item With probability 1 : $\forall n\geq 1 ,\quad  \|\mathbb{F}_n-\mathcal{B}\|_{\infty}\leq C(\log n)^{-\lambda}$, for some  $\lambda>0$, where  
%for all $n\geq 2$, $K_n$ 
$\mathcal{B}$ is a Brownian bridge.
\item For all positive integer $q$, there exists $C_F(q)>0$ such that : $\forall n\geq 2 ,\: \bE\left(\|\mathbb{F}_n\|_{\infty}^q\right)\leq C_F(q)$.
\item For all positive integer $q$, there exists $C_G(q)>0$ such that: $\forall n\geq 2 ,\: \bE\|\mathbb{G}_n\|_{\infty}\leq C_G(q)$.
\end{enumerate}

\end{lemma}

\begin{proof} 
\begin{enumerate}[i)]
\item Denote by $\mathbb{K}(.,.)$ the  Kiefer process. Theorem 3 of \cite{dhompongsa84} ensures that with probability 1, for all $n\geq 2 $
$$\sup_{s\in \R} |\sqrt{n} \mathbb{F}_n(s)-\mathbb{K}(s,n)|\leq C\sqrt{n} (\log n)^{-\lambda},$$
for some $\lambda >0$. To conclude it is sufficient to remind that 
%$K_n
$\mathcal{B}:=\mathbb{K}(.,n)/\sqrt{n}$ is a Brownian bridge, whose distribution does not depend on $n$.

\item It is known that %$\|K_n\|_{\infty}$ 
$\|\mathcal{B}\|_{\infty}$ follows the Kolmogorov-Smirnov distribution whose cumulative distribution function is given by  $1-2\sum_{k=1}^\infty (-1)^{k-1}\exp(-2k^2x^2)$ (see for instance \cite{shorackwellner}). This distribution has finite moments : 
$\bE(\|\mathcal{B}\|_{\infty}^q)=q2^{-q/2}\Gamma(q/2)\eta(q)<\infty$
. Then, for all $n\geq 2$, 
\begin{eqnarray*}
\bE(\|\mathbb{F}_n\|_{\infty}^q)&\leq &2^{q-1}\bE(\|\mathcal{B}\|_{\infty}^q)+ 2^{q-1}C^q(\log n)^{-\lambda q}\\
&\leq & \underbrace{q2^{q/2-1}\Gamma(q/2)\eta(q)+ 2^{q-1}C^q(\log 2)^{-\lambda q}}.\\
&&~~~~~~~~~~~~~~~~~~~~~~C_F(q)
\end{eqnarray*}

\item The proof is the same considering the process $R_i=(Z_{i},Z_{i+1})$. Indeed we have used a theorem valid for processes valued in $\R^d$ as soon as the mixing rate  verifies $a>d+2$. The definition of the 2-dimensional Kiefer process can be found in \cite{PhilippPinzur80}. Again the distribution of $\mathcal{B}:=\mathbb{K}(.,n)/\sqrt{n}$  does not depend on $n$ and is a 2-dimensional Brownian bridge. The distribution of the supremum of a Brownian bridge in dimension two is less known than in dimension one, but \cite{AdlerBrown86} have shown that 
$\P(\|\mathcal{B}\|_{\infty}>\lambda)\leq C \lambda ^2 e^{-2\lambda^2}$, which is sufficient to ensure finite moments, {\it i.e.}
$\bE(\|\mathcal{B}\|_{\infty}^q)<\infty.$ 
\end{enumerate}
\end{proof}
Now, we are in position to bound $\Psi_1(\mathbb{F}_n,\mathbb{G}_n)$. We compute
\begin{eqnarray*}
\bE|\Psi_1(\mathbb{F}_n)|&= &2\bE\left|\iint  \left(-T_1\mathbb{G}_n+T_1\overline{A\otimes \mathbb{F}_n}\right)dH\right|\\
&\leq & 2\bE\|\mathbb{G}_n\|_{\infty}\iint  |T_1|dH+ 2\bE\|\mathbb{F}_n\|_{\infty}\iint |T_1(x,y)|(|A(x)|+|A(y)|)dH(x,y)\\
&\leq & 2C_G(1)+ 8C_F(1).
\end{eqnarray*}
Using the same approach for $\Psi_2$ and $\chi_1,\chi_2$, we obtain
$$\bE\|\sqrt{n}(\dot \d-\dot \d_n)(v^*)\|\leq C,$$
and then 
\begin{equation}\label{majodnv}
\sqrt{n}\bE\|\ddot{\mathtt{d}_n}(\hat v_n-v^*)\|\leq C.
\end{equation}
Now $\hat v_n-v^*=\ddot\d^{-1}\ddot \d (\hat v_n-v^*) $, since $\ddot\d$ is invertible under \textbf{LinInd}. For the sake of simplicity, we denote by $x_n$ the random sequence $x_n=(\hat v_n-v^*)$.  That gives us
$$\hat v_n-v^*=x_n=\ddot\d^{-1}\ddot \d x_n=\ddot\d^{-1}\left(\ddot \d_n x_n+(\ddot \d -\ddot \d_n) x_n\right),$$
and
$$\bE\|\hat v_n-v^*\|\leq \|\ddot\d^{-1}\|\left(\bE\|\ddot \d_n x_n\|+\bE(\|\ddot \d -\ddot \d_n\| \|x_n\|)\right),$$
where we denote by $\|M\|$ the spectral norm of any matrix $M$. 
%For convenience matters we will denote $[\mathtt x]_i$ the $i$-th coordinate  of any vector in $\mathtt x$ and $[M]_{i,j}$ the $(i,j)$-th component of any matrix $M$.
Note that, for all $\theta \in \Theta$, $v=h(\theta)\in [-1,0]\times [-\frac{1}{\delta}, 1]$, then with probability 1, for all $n\geq 1$: 
$|[x_n]_1|\leq 1$ and $|[x_n]_2|\leq 1+\delta^{-1}.$  Using this and \eqref{majodnv}, we obtain
\begin{eqnarray}\label{v_diff}
\bE\|\hat v_n-v^*\|\leq \|\ddot\d^{-1}\|\left(\frac{C}{\sqrt{n}}+c\delta^{-1}\bE(\|\ddot \d -\ddot \d_n\|)\right).
\end{eqnarray}
From Section~\ref{sec:hessian} we know that $\ddot \d_n-\ddot \d=\frac{1}{\sqrt{n}}\Phi(\mathbb{F}_n)$ with $[\Phi(\mathbb{F}_n)]_{1,1}=0$ and

\begin{eqnarray*}
[\Phi(\mathbb{F}_n)]_{1,2}&=&[\Phi(\mathbb{F}_n)]_{2,1}=2\iint T_1 \sqrt{n} (\hat T_2-T_2)dH, \\
\left[\Phi(\mathbb{F}_n)\right]_{2,2}&=&\frac{2}{\sqrt{n}}\iint \left(\sqrt{n}({\hat T}_2-T_2)\right)^2dH
+4\iint T_2 \sqrt{n} ({\hat T}_2-T_2)dH,
\end{eqnarray*}
where
$
\sqrt{n}(T_2-\hat T_2)=\overline{\mathbb{F}_n\otimes (F^0-F)}-\frac{1}{\sqrt{n}} \mathbb{F}_n\otimes \mathbb{F}_n.
$
Thus
\begin{eqnarray*}
[|\Phi(\mathbb{F}_n)|]_{1,2}&\leq &2\iint |T_1| \left|\overline{\mathbb{F}_n\otimes (F^0-F)}-\frac{1}{\sqrt{n}} \mathbb{F}_n\otimes \mathbb{F}_n\right|dH\\
&\leq &2\|\mathbb{F}_n\|_\infty+\frac{2}{\sqrt{n}}\|\mathbb{F}_n\|_\infty^2.
\end{eqnarray*}
so that  $\bE[|\Phi(\mathbb{F}_n)|]_{1,2}\leq 2 C_F(1)+2C_F(2)$, see Lemma \ref{supprocessusempiriques}. In the same way we can prove that  the sequence $\bE[|\Phi(\mathbb{F}_n)|]_{2,2}$ is bounded by a positive finite constant. This ensures that
$\bE(\|\ddot \d -\ddot \d_n\|)\leq C'/\sqrt{n}$ and, from \eqref{v_diff},  we finally obtain 
$\bE\|\hat v_n-v^*\|=O(1/\sqrt{n})$.

\bigskip

\noindent $\bullet$ \underline{Proof of Theorem \ref{theo:risqueF} about  $F^1$.}
We know that  for all $n\geq 1$
\begin{eqnarray*}
\hat F_n^1-F^1=\frac{1}{ r^*} (\hat F_n-F)+\left(\frac{\hat \beta_n}{\hat \alpha_n}-\frac{\beta^*}{\alpha^*}\right)(\hat F_n-F^0).
\end{eqnarray*}
Using that  $\mathbb{F}_n=\sqrt{n}(\hat F_n -F)$ and $\|\hat F_n-F^0\|_{\infty}\leq 1$,  it comes
\begin{eqnarray*}
\|\hat F_n^1-F^1\|_{\infty}&\leq &\frac{1}{ r^*} \frac{\|\mathbb{F}_n\|_{\infty}}{\sqrt{n}}+\left|\frac{\hat \beta_n}{\hat \alpha_n}-\frac{\beta^*}{\alpha^*}\right|\\
&\leq &\frac{1}{ r^*} \frac{\|\mathbb{F}_n\|_{\infty}}{\sqrt{n}}+\frac{\sqrt{2}}{\delta^2}\|\hat \theta_n-\theta\|.
\end{eqnarray*}
Thus, using again Lemma~\ref{supprocessusempiriques}, we have
\begin{eqnarray*}
\bE\|\hat F_n^1-F^1\|_{\infty}&\leq &\frac{1}{ r^*} \frac{1}{\sqrt{n}}\bE\|\mathbb{F}_n\|_{\infty}+\frac{\sqrt2}{\delta^2}\bE\|\hat \theta_n-\theta\|\\
&\leq &\left(\frac{C_F(1)}{ r^*}+\frac{C''}{\delta^2 }\right)\frac{1}{\sqrt{n}},
\end{eqnarray*}
which concludes the proof.
\qed

%%%%%%%%%%%%%%%%%%%%%
\section*{Aknowledgements}

The authors would like to express their  sincere gratitude  to Professors  Florence Merlevede and   J\'er\^{o}me Dedecker,   for their  insightful comments  about the convergence of $(\mathbb{F}_n,\mathbb{G}_n)$ and $\bE\|\mathbb{F}_n\|_{\infty}$,  resp.  the weak dependence of processes.
%%%%%%%%%%%%%%%%%%%%%%

\section*{Fundings and grants}
Claire  Lacour and Pierre Vandekerkhove would like to aknowledge that this work was supported by the CNES under the BIOSWOT-AdAC project and the MIO Axes Transverses. \\
Pierre  Vandekerkhove would also like to acknowledge the support received from the Research Chair ACTIONS under the aegis of the Risk Foundation, an initiative by BNP Paribas Cardif and the French Institute of Actuaries.

\bibliographystyle{abbrvnat}
\bibliography{biblio}

\appendix
\section{Appendix}
\subsection{Details for the computation of $\dot{\mathtt{d}}-\dot{\mathtt{d_n}}$}
\label{sec:detailsdiffgradient}

We start from 
\begin{eqnarray*}
[ \Dot{\d}-\Dot{\d}_{n}]_2(v) &=&2\iint (v_1 T_1+v_2T_2+T_3)( T_2-\hat T_2) dH
+2v_2\iint   \hat T_2(T_2-\hat T_2)dH
 \\&&
 +2\iint   \hat T_2(T_3-\hat T_3)dH.
\end{eqnarray*}
We can write 
\begin{eqnarray*}
[\Dot{\d}-\Dot{\d}_{n}]_2(v^*) &=&2\iint \Delta(\theta^*,.)( T_2-\hat T_2) dH
+2v_2^*\iint   (\hat T_2-T_2+T_2)(T_2-\hat T_2)dH
 \\&&
 +2\iint   (\hat T_2-T_2+T_2)(T_3-\hat T_3)dH\\
 &=&2\iint \left[\Delta(\theta^*,.)+v_2^*T_2\right]( T_2-\hat T_2) dH
 +2\iint  T_2(T_3-\hat T_3)dH\\
 \\&&
+2v_2^*\iint   (\hat T_2-T_2)(T_2-\hat T_2)dH
 +2\iint   (\hat T_2-T_2)(T_3-\hat T_3)dH\\
 &=:&2I_1+2I_2+2I_3+2I_4.
\end{eqnarray*}
Denoting $\overline{a\otimes b}=a\otimes b+b\otimes a$, equalities \eqref{diffT2} and \eqref{diffT3} become
\begin{eqnarray}
\sqrt{n}(T_2-\hat T_2)&=&\overline{(F^0-F)\otimes \mathbb{F}_n}-\frac{1}{\sqrt{n}} \mathbb{F}_n\otimes \mathbb{F}_n,\label{T2decomp}\\
\sqrt{n}(T_3-\hat T_3) &=&-\mathbb{G}_n+\overline{F\otimes \mathbb{F}_n}+ \frac{1}{\sqrt{n}} \mathbb{F}_n\otimes \mathbb{F}_n.\label{T3decomp}
\end{eqnarray}
This gives according to (\ref{T2decomp})
\begin{eqnarray*}
\sqrt{n}I_1
 &=&\iint \left[\Delta(\theta^*,.)+v_2^*T_2\right]( T_2-\hat T_2)dH\\
 &=& \iint \left[\Delta(\theta^*,.)+v_2^*T_2\right]\left(\overline{(F^0-F)\otimes \mathbb{F}_n}-\frac{1}{\sqrt{n}} \mathbb{F}_n\otimes \mathbb{F}_n\right)dH\\
  &=& \iint B\left(\overline{(F^0-F)\otimes \mathbb{F}_n}-\frac{1}{\sqrt{n}} \mathbb{F}_n\otimes \mathbb{F}_n\right)dH,
\end{eqnarray*}
with $B=v_1^*T_1+2v_2^*T_2+T_3$.
Now according to (\ref{T3decomp})
\begin{eqnarray*}
\sqrt{n}I_2=\iint  T_2(T_3-\hat T_3)dH
 =\iint T_2 \left(-\mathbb{G}_n+\overline{F\otimes \mathbb{F}_n}+ \frac{1}{\sqrt{n}} \mathbb{F}_n\otimes \mathbb{F}_n\right)dH.
\end{eqnarray*}
Regarding the third term, we write thanks to (\ref{T2decomp})
\begin{eqnarray*}
 \sqrt{n}I_3=
-v_2^*\iint  (T_2-\hat T_2)^2dH=-\frac{v_2^*}{\sqrt{n}}\iint \left(\overline{(F^0-F)\otimes \mathbb{F}_n}-\frac{1}{\sqrt{n}} \mathbb{F}_n\otimes \mathbb{F}_n\right)^2dH.
\end{eqnarray*}
Finally, according to (\ref{T3decomp}) we have 
\begin{eqnarray*}
 \sqrt{n}I_4 &=&
\iint   (\hat T_2-T_2)(T_3-\hat T_3)dH
\\&=&
\frac{1}{\sqrt{n}}\iint \left(\overline{(F^0-F)\otimes \mathbb{F}_n}-\frac{1}{\sqrt{n}} \mathbb{F}_n\otimes \mathbb{F}_n\right)\left(\mathbb{G}_n-\overline{F\otimes \mathbb{F}_n}- \frac{1}{\sqrt{n}} \mathbb{F}_n\otimes \mathbb{F}_n\right)dH.
\end{eqnarray*}
Gathering all the terms we obtain  
\begin{eqnarray*}
\frac{\sqrt{n}}2[ \Dot{\d}-\Dot{\d}_{n}]_2 (v^*)&=&
\iint B\overline{(F^0-F)\otimes \mathbb{F}_n}dH+\iint T_2 \left(-\mathbb{G}_n+\overline{F\otimes \mathbb{F}_n}\right)dH+\frac{1}{\sqrt{n}}\chi_2(\mathbb{F}_n,\mathbb{G}_n)
\\&=&\iint \left(-T_2 \mathbb{G}_n+
 B\overline{(F^0-F)\otimes \mathbb{F}_n} +T_2 \overline{F\otimes \mathbb{F}_n}\right)dH+\frac{1}{\sqrt{n}}\chi_2(\mathbb{F}_n,\mathbb{G}_n),
 \end{eqnarray*}
 which concludes our $\Dot{\d}-\Dot{\d}_{n}$ analysis.
%%%%%%%%%%%%%%%%%%%%%%
\section{Appendix}
\subsection{Bell-like regime of $\hat \theta_n$ and $\tilde \theta_n$ for models $\mbox{(S0)}_{strong}$ and  $\mbox{(S0)}_{weak}$}

\begin{figure}[!h]
	\begin{center}
	        \includegraphics[scale=0.2]{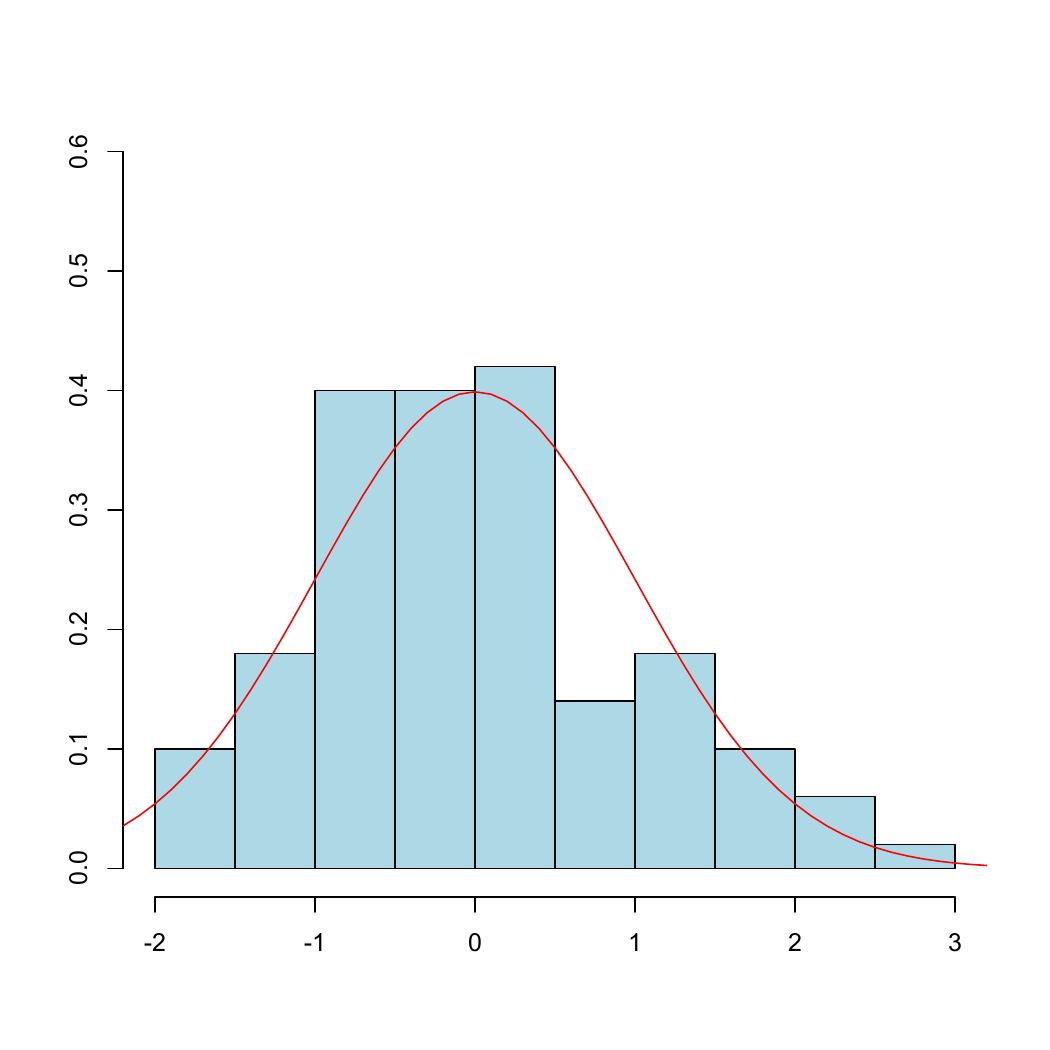}
		\includegraphics[scale=0.2]{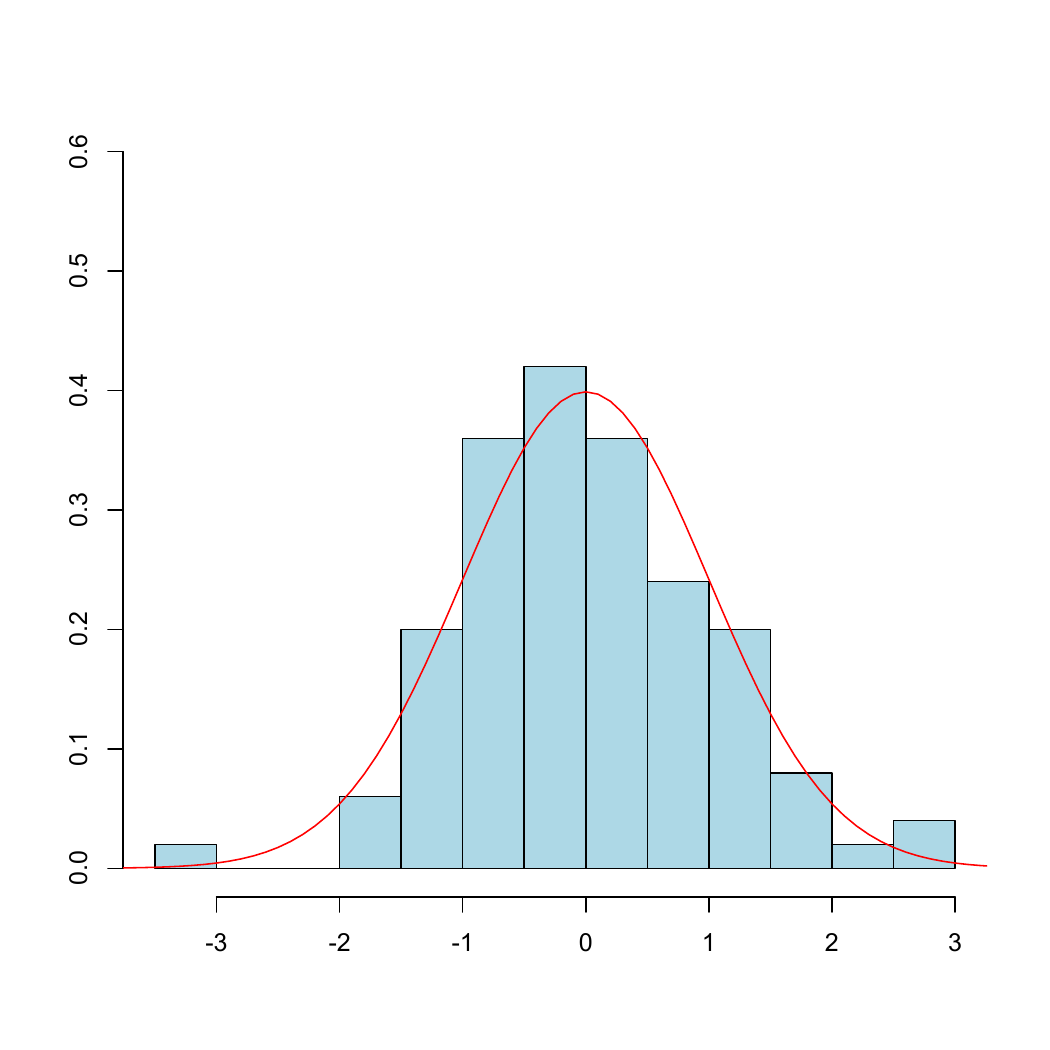} 
		 \includegraphics[scale=0.2]{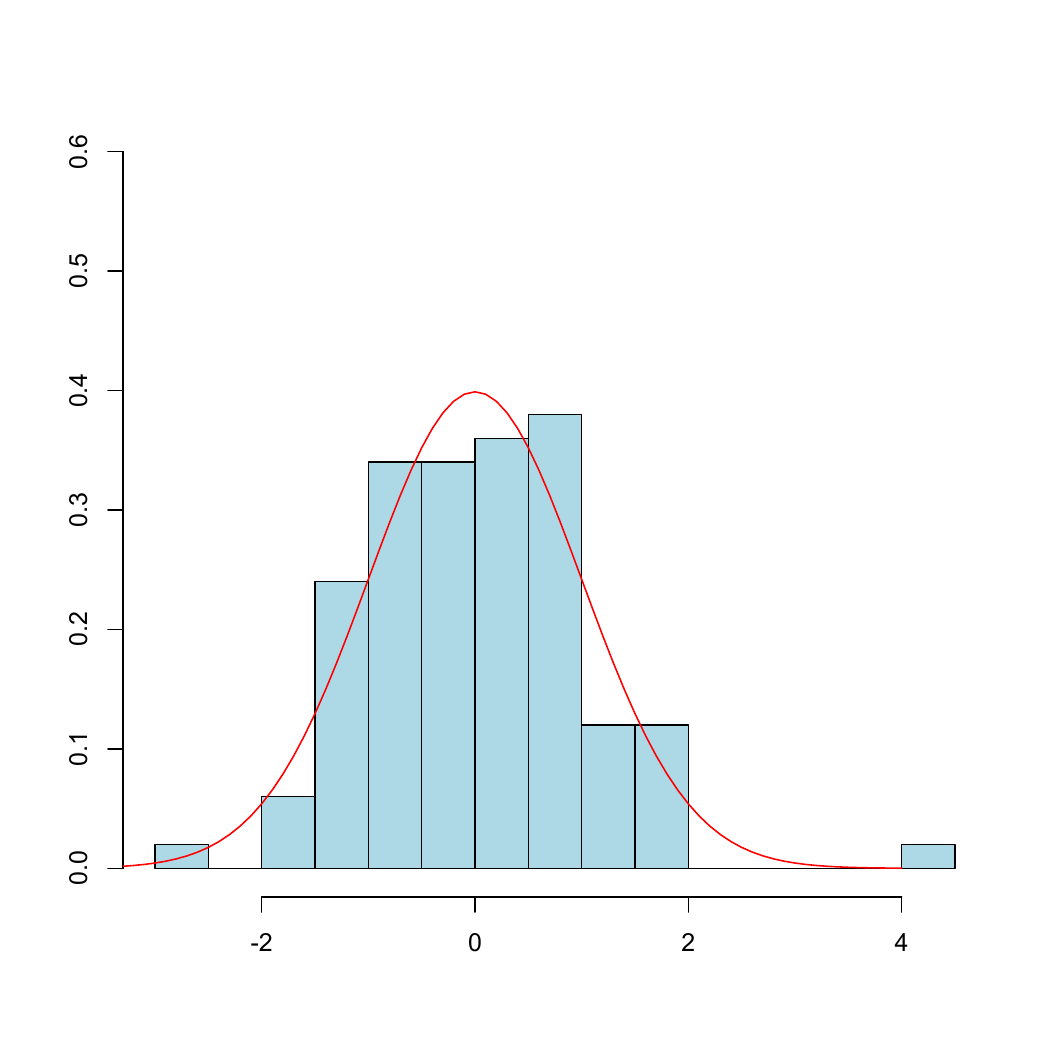}
		\includegraphics[scale=0.2]{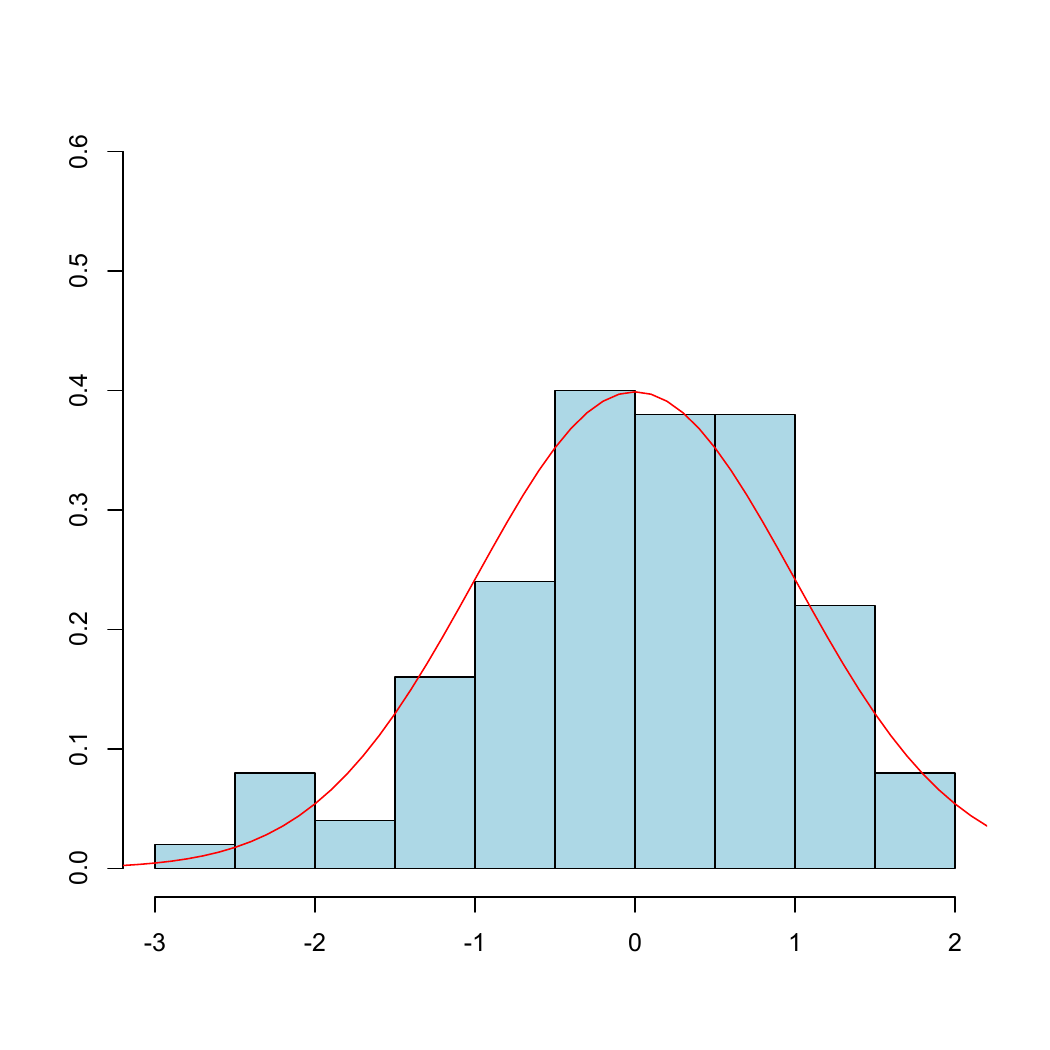} 
		\\
	        \includegraphics[scale=0.2]{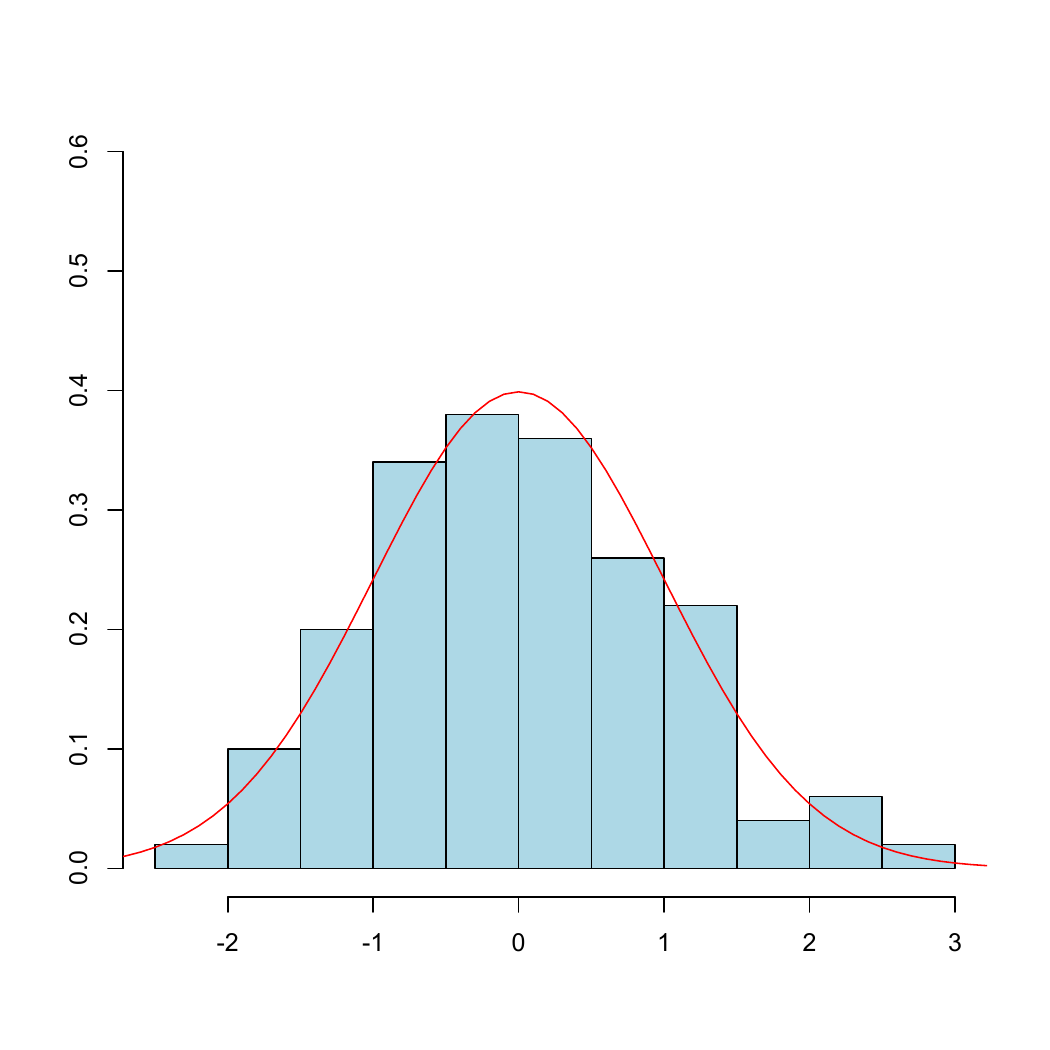}
		\includegraphics[scale=0.2]{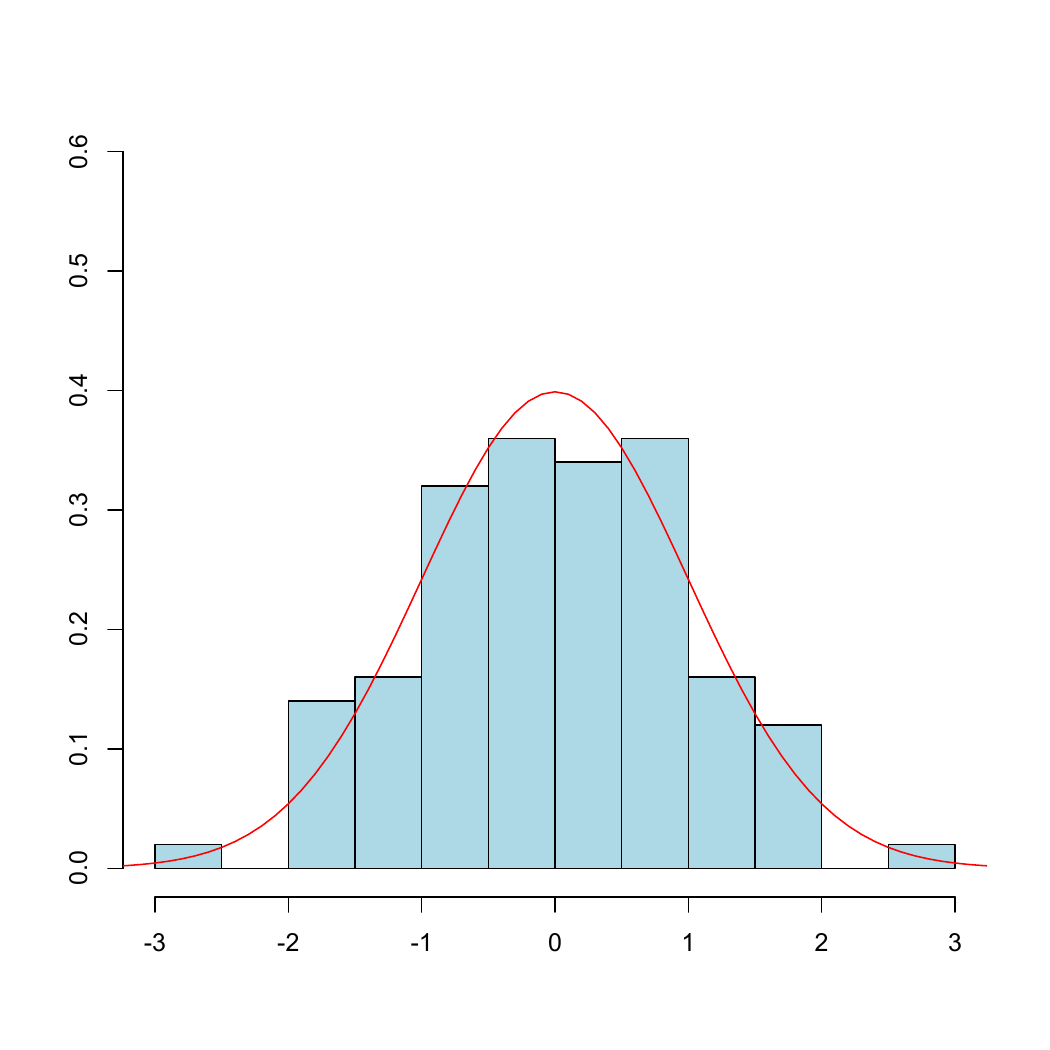} 
		\includegraphics[scale=0.2]{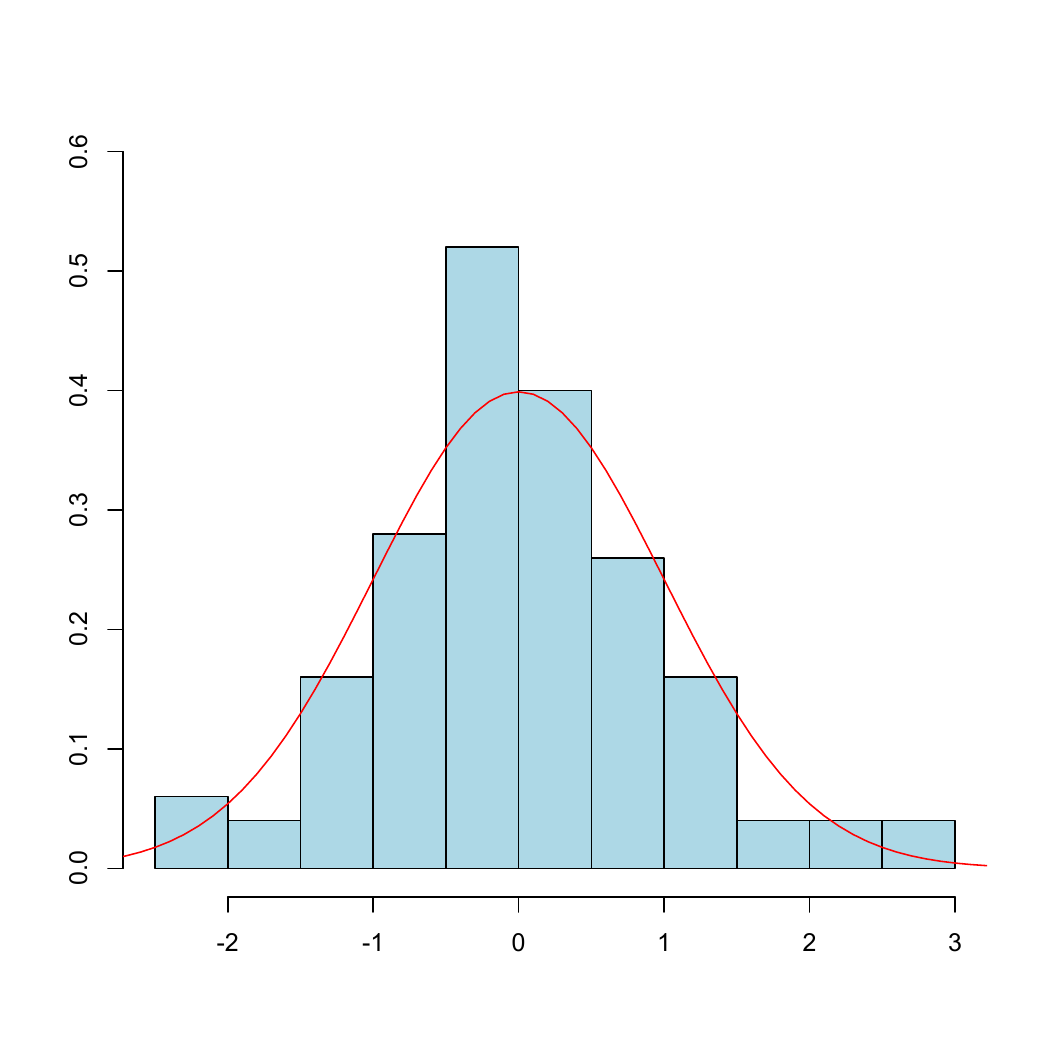}
		\includegraphics[scale=0.2]{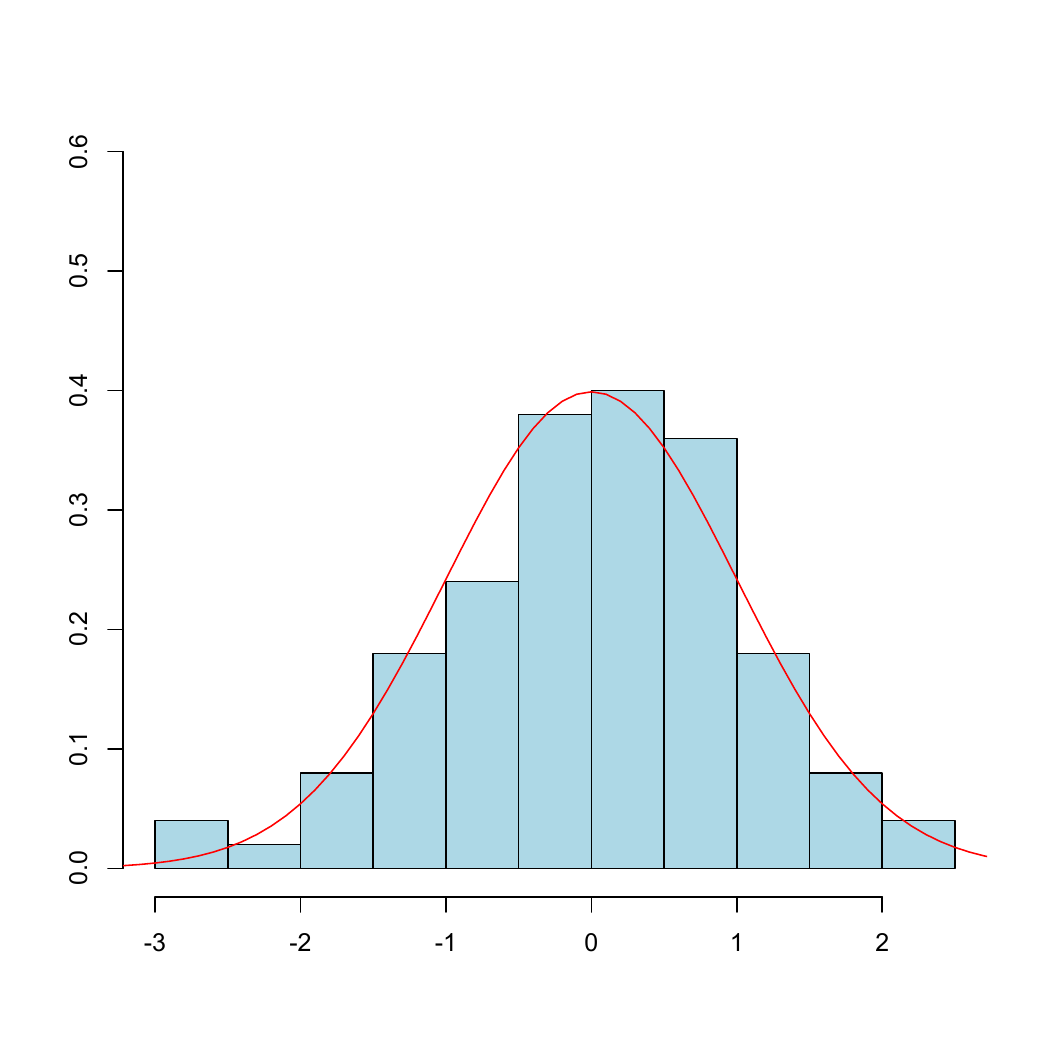}
		\\
		\includegraphics[scale=0.2]{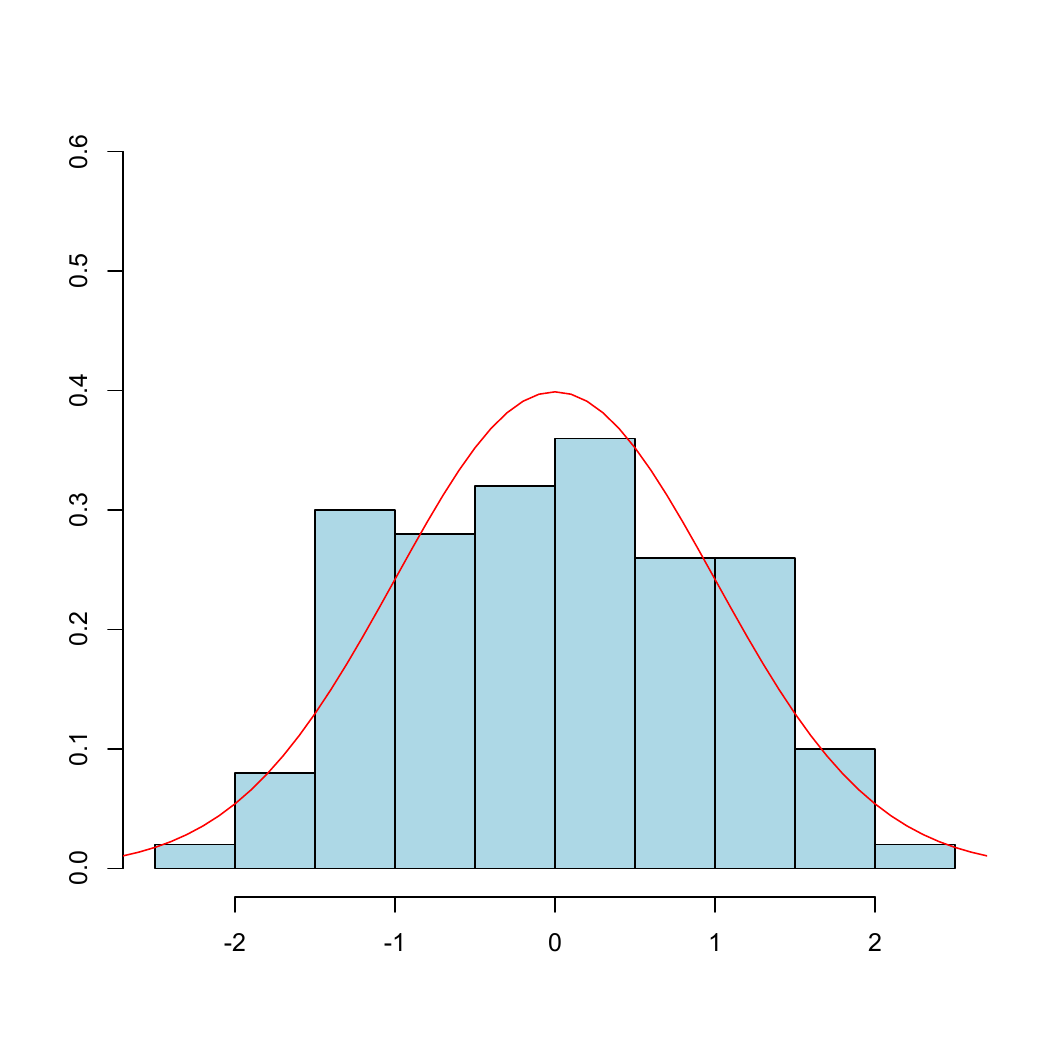}
		\includegraphics[scale=0.2]{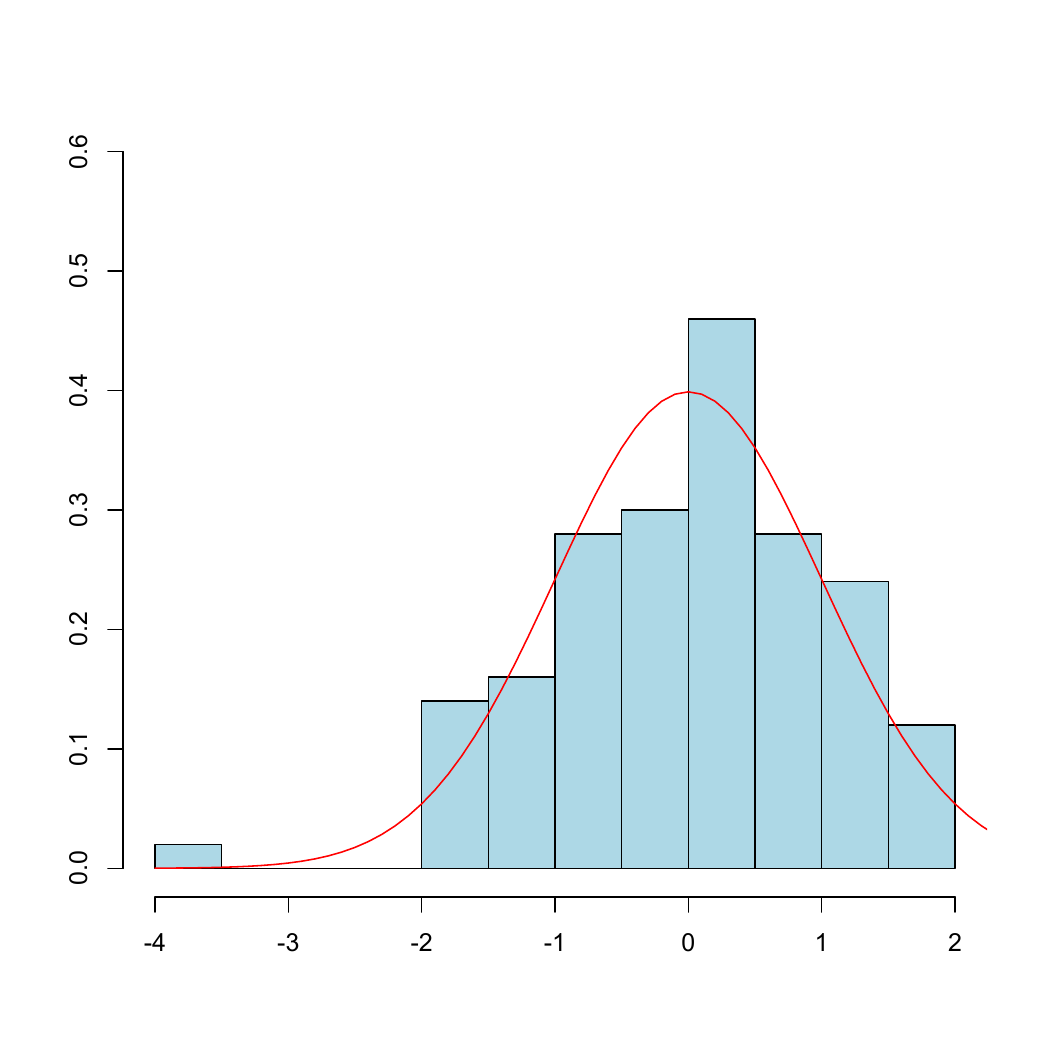} 
		\includegraphics[scale=0.2]{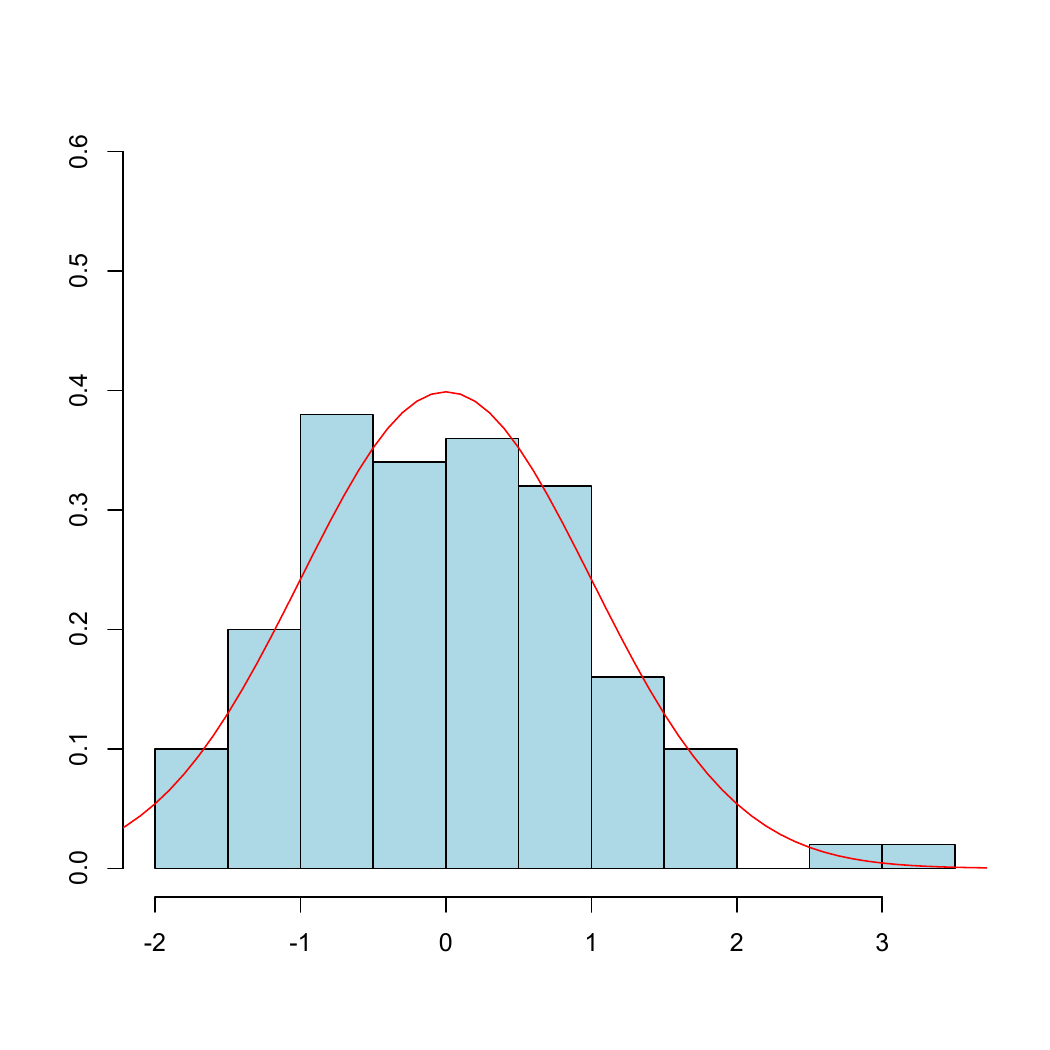}
		\includegraphics[scale=0.2]{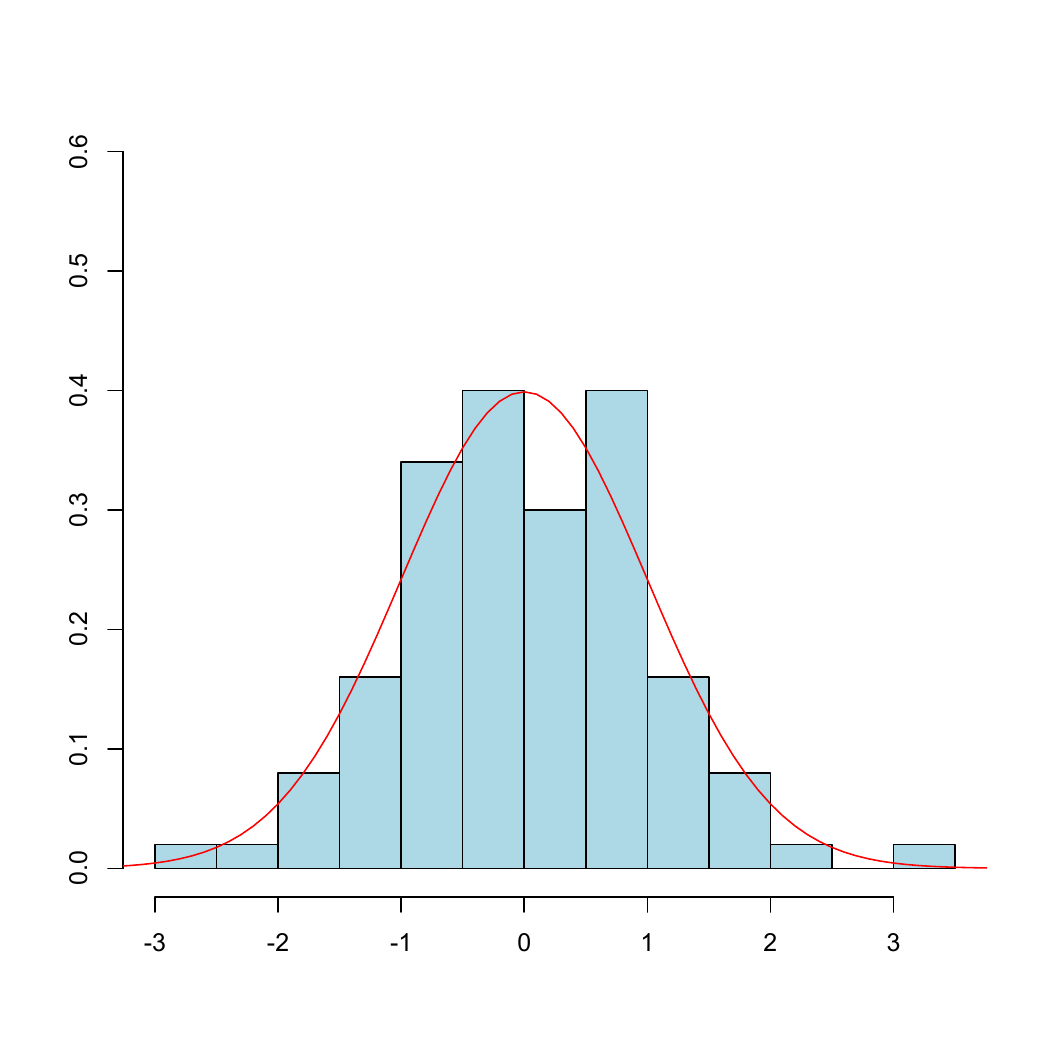} 
	\caption{Histogram of  $(\alpha,\beta)$ $\sqrt{n}$-normalised distribution under model $\mbox{(S0)}_{strong}$. First row $n=1,000$, second row $n=3,000$, third row $n=5,000$. First column $\hat \alpha_n$, second column $\hat \beta_n$, third column  $\tilde \alpha_n$, fourth column $\tilde \beta_n$.}
		\label{fig:normalityS0strong}
	\end{center}
\end{figure}
\newpage
\begin{figure}[!h]
	\begin{center}
	        \includegraphics[scale=0.2]{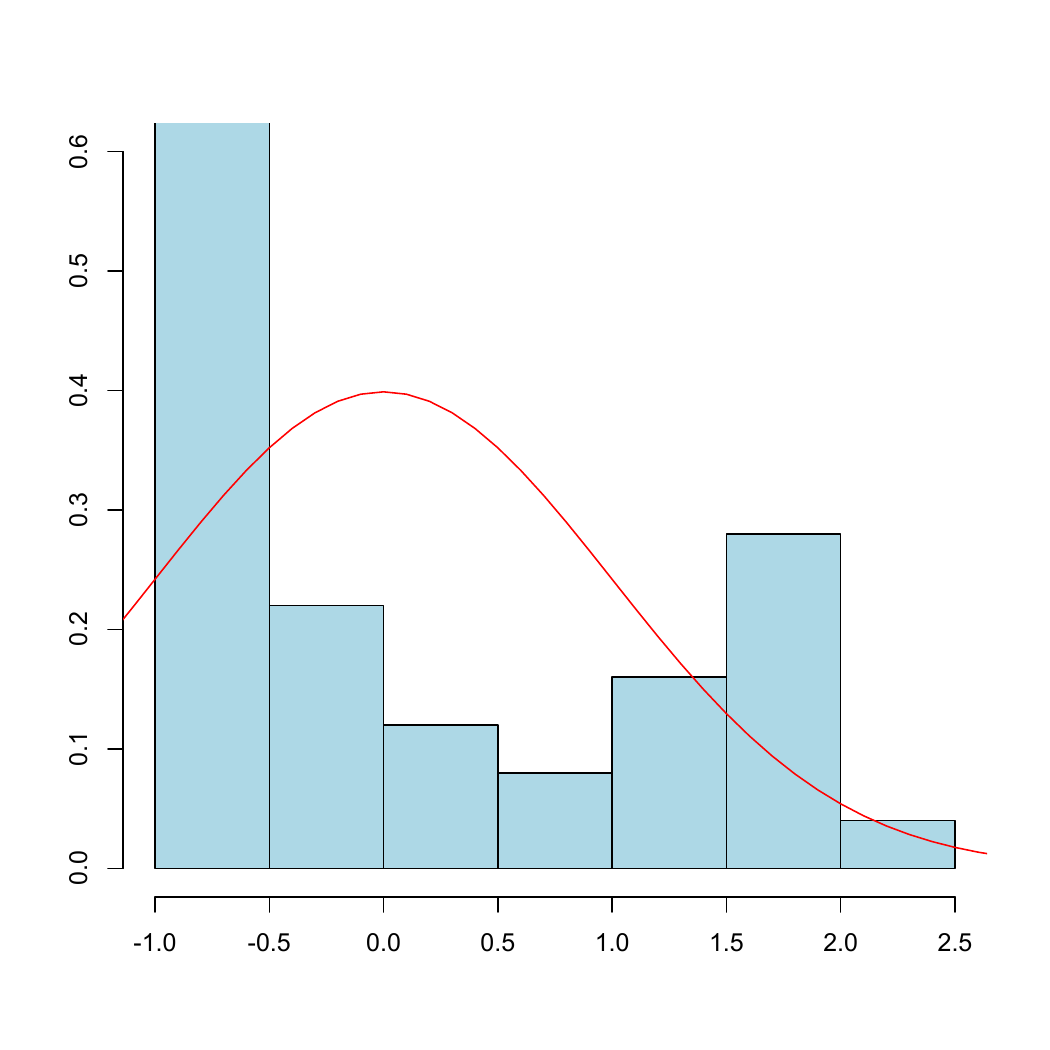}
		\includegraphics[scale=0.2]{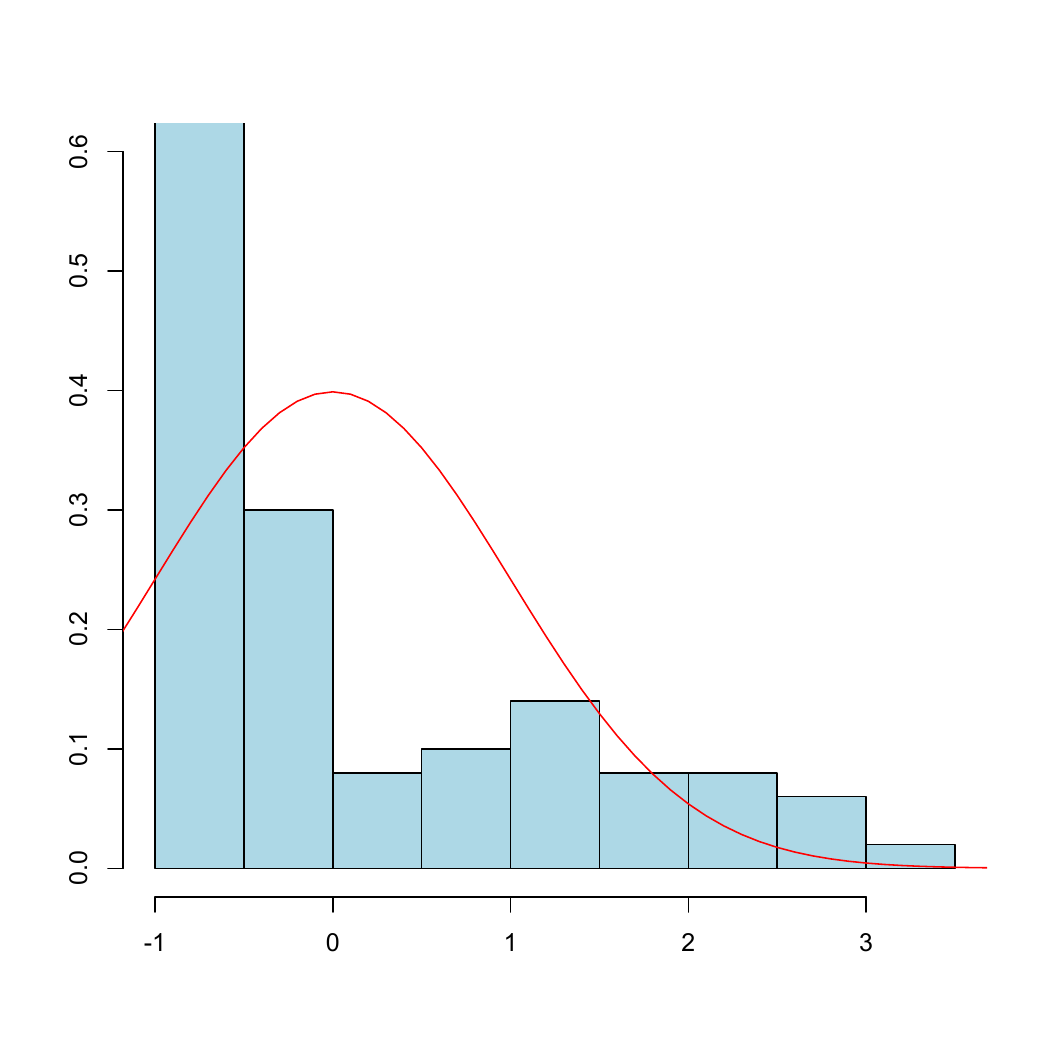} 
		 \includegraphics[scale=0.2]{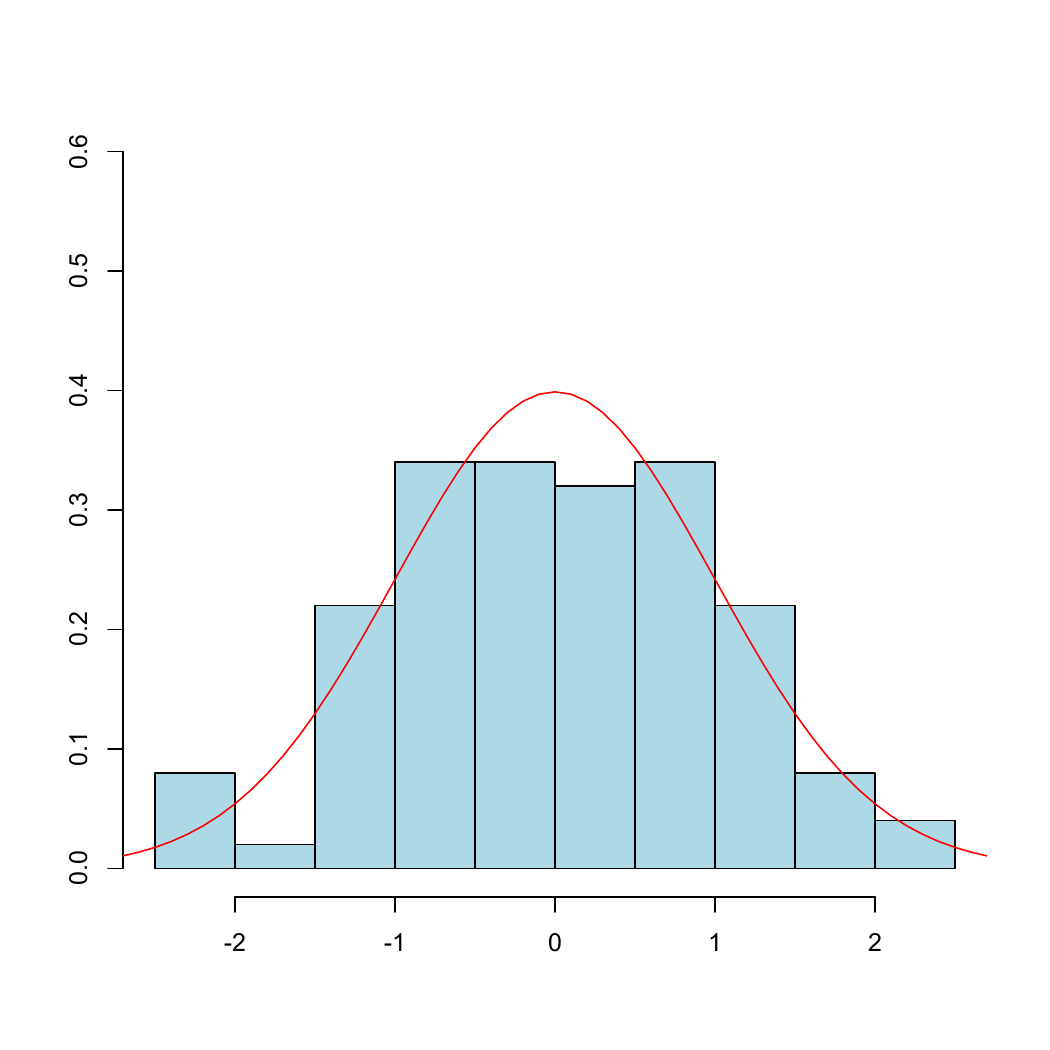}
		\includegraphics[scale=0.2]{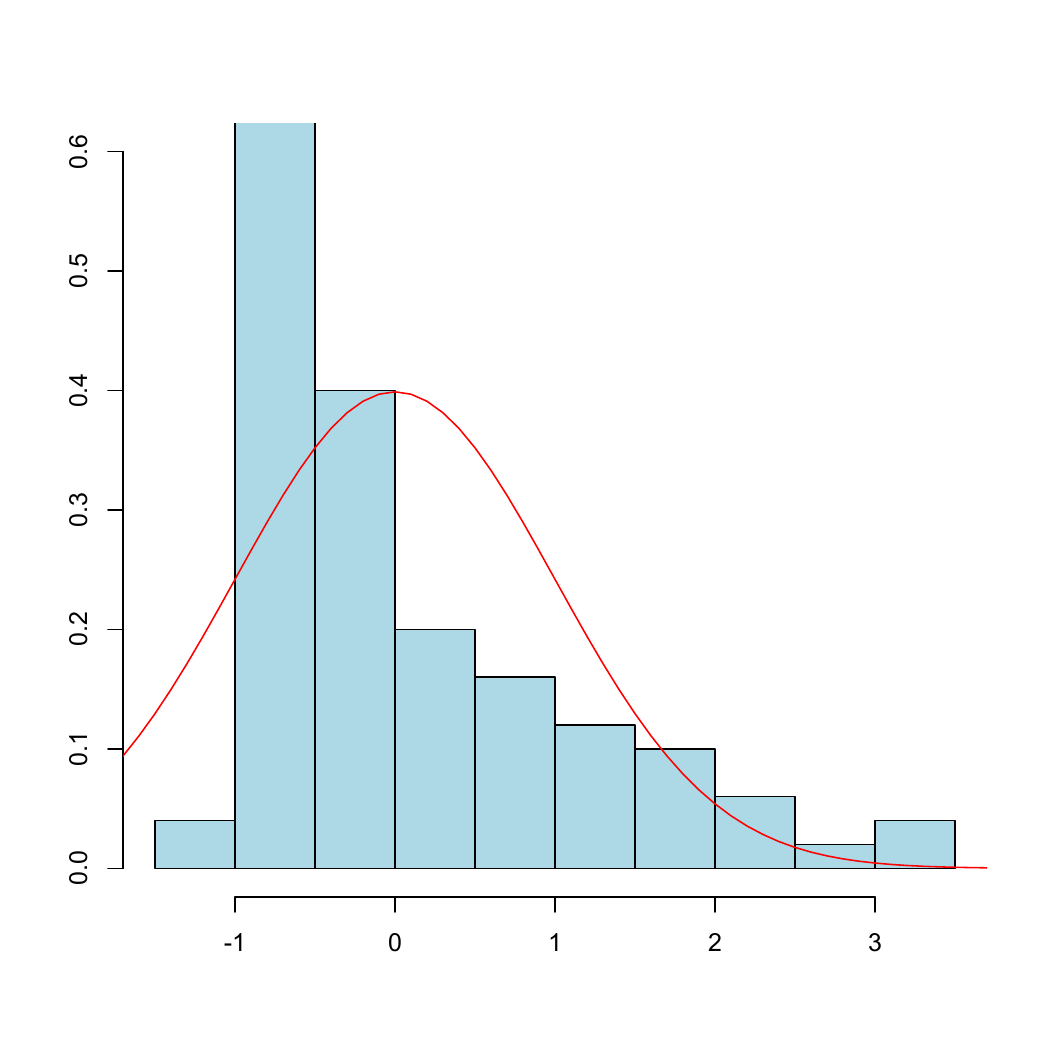} 
		\\
	        \includegraphics[scale=0.2]{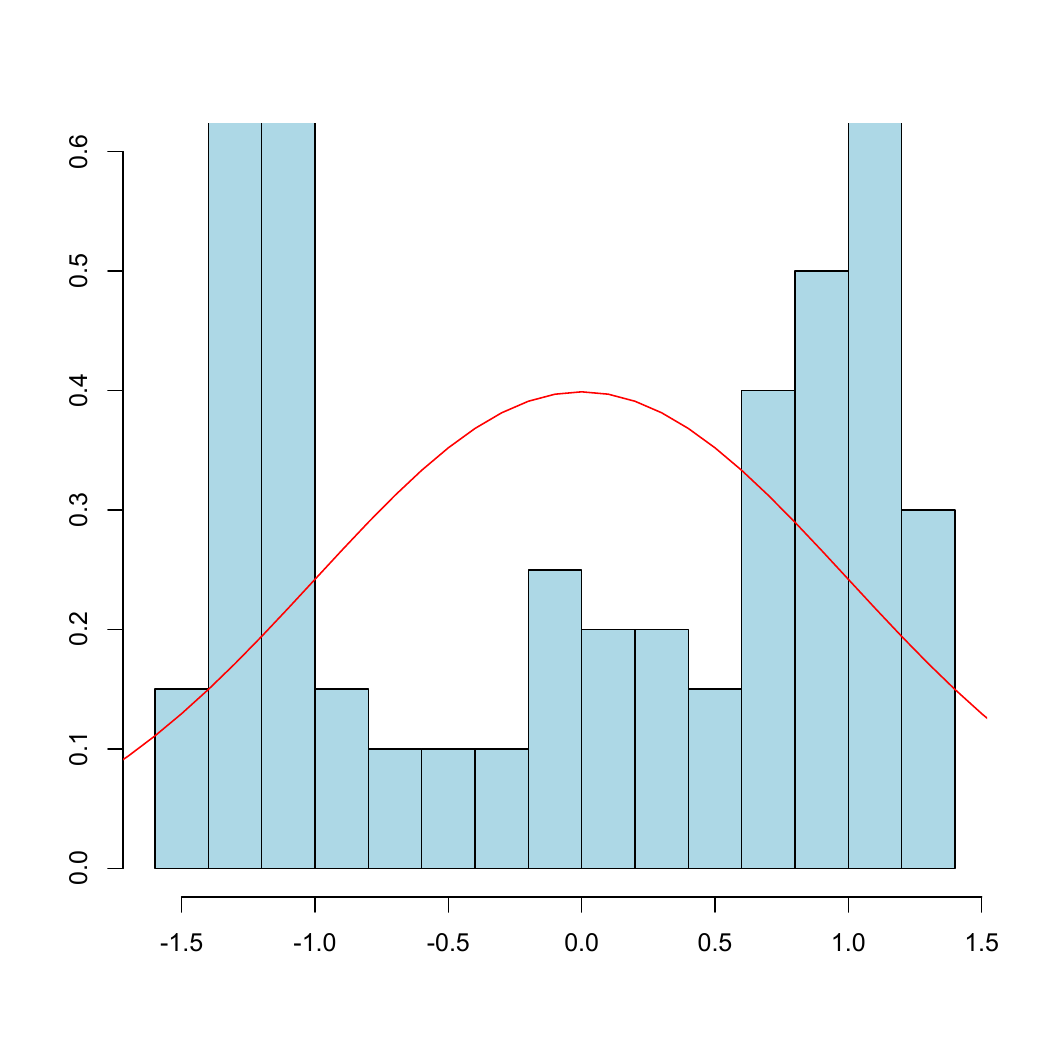}
		\includegraphics[scale=0.2]{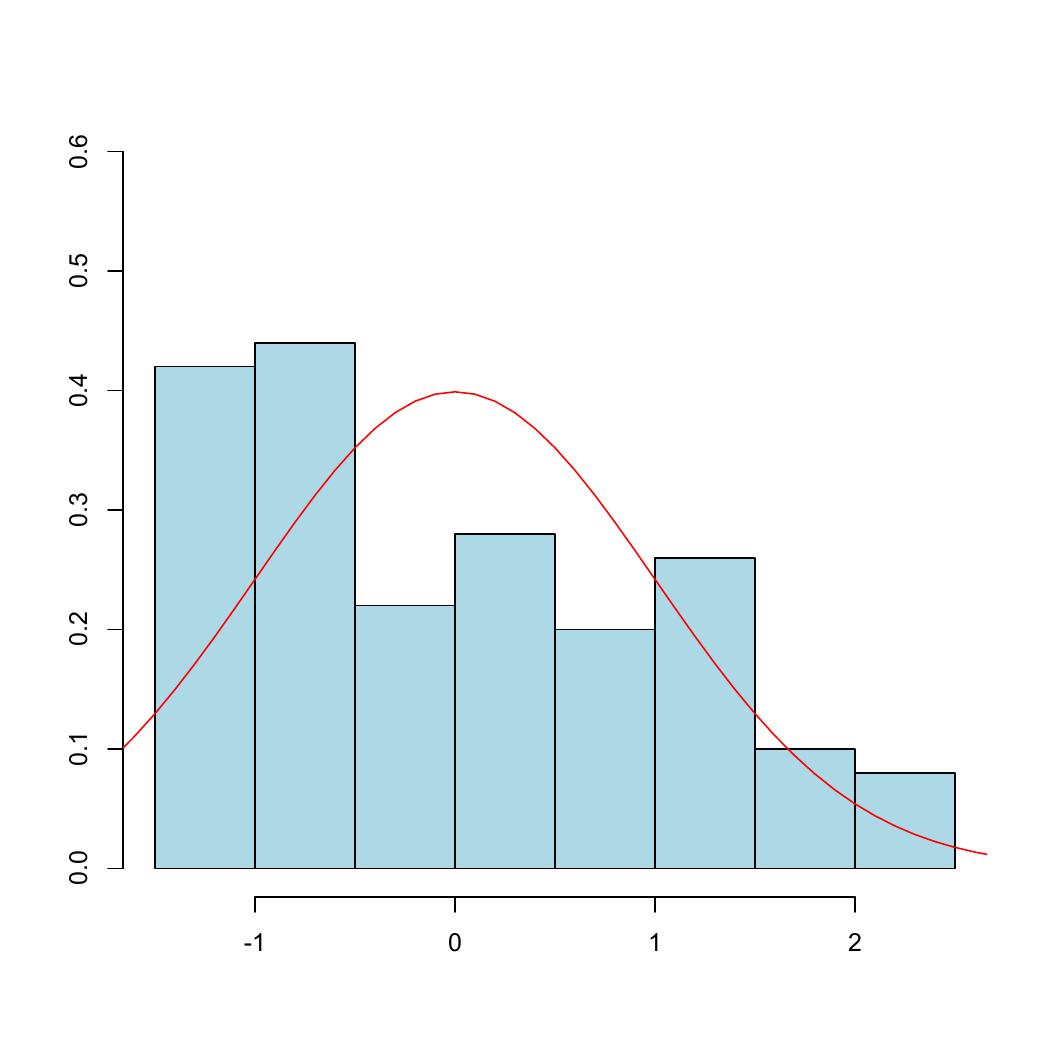} 
		\includegraphics[scale=0.2]{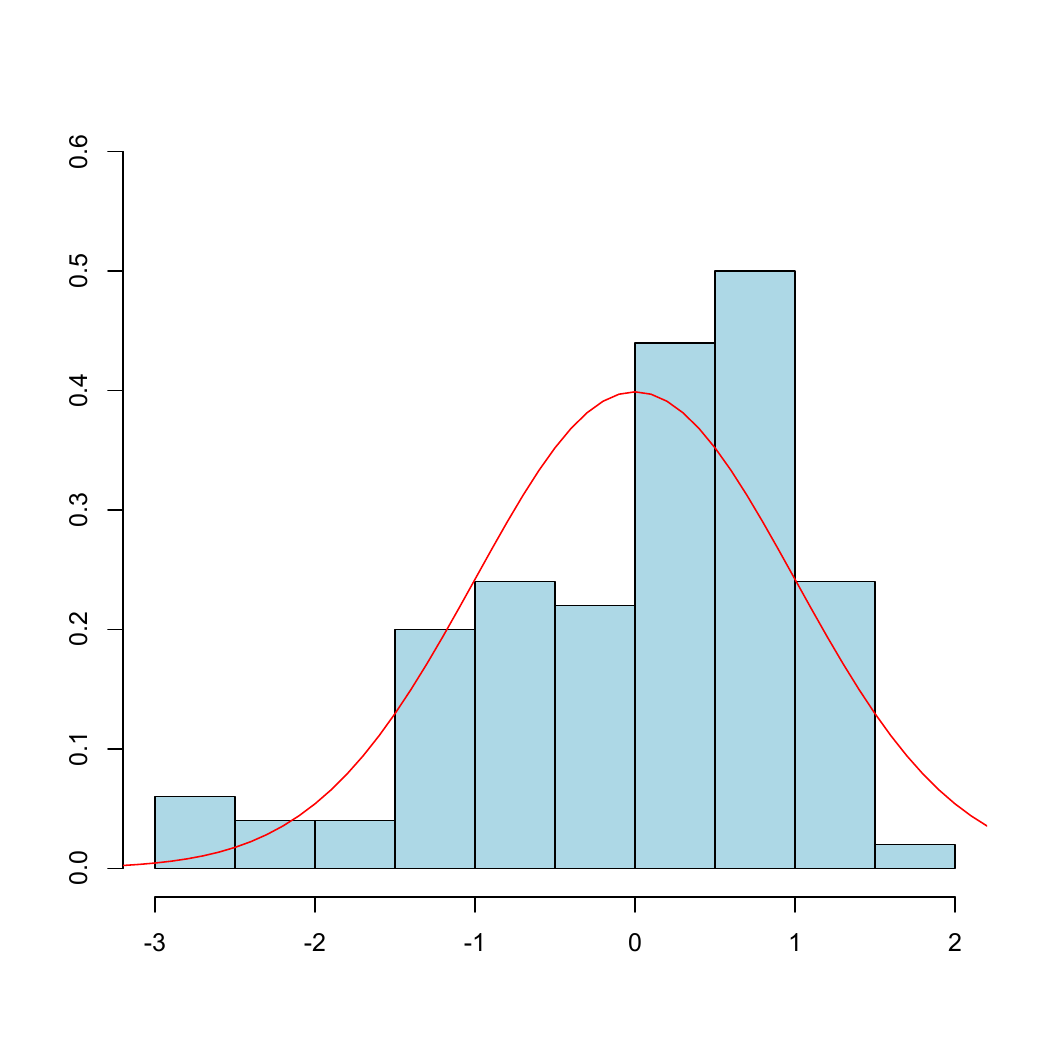}
		\includegraphics[scale=0.2]{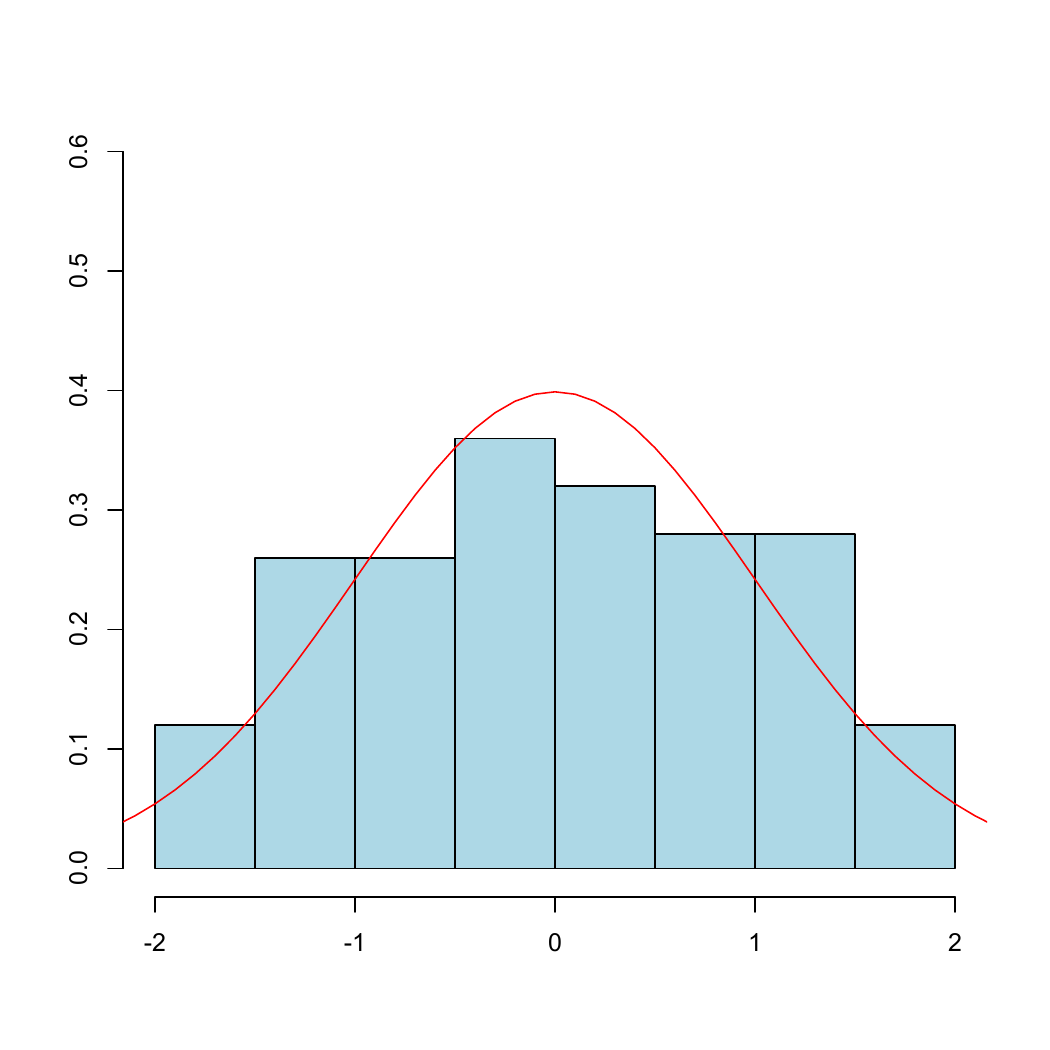}
		\\
		\includegraphics[scale=0.2]{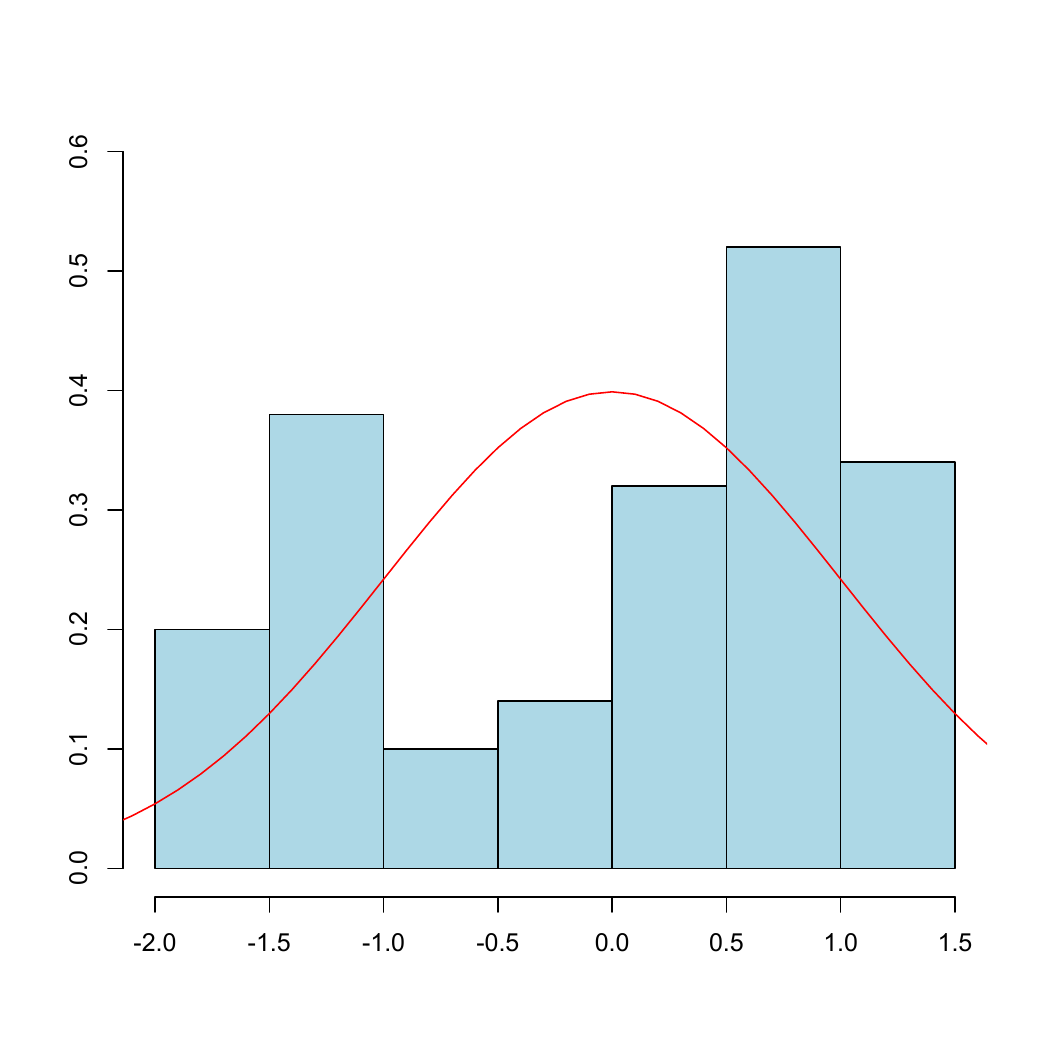}
		\includegraphics[scale=0.2]{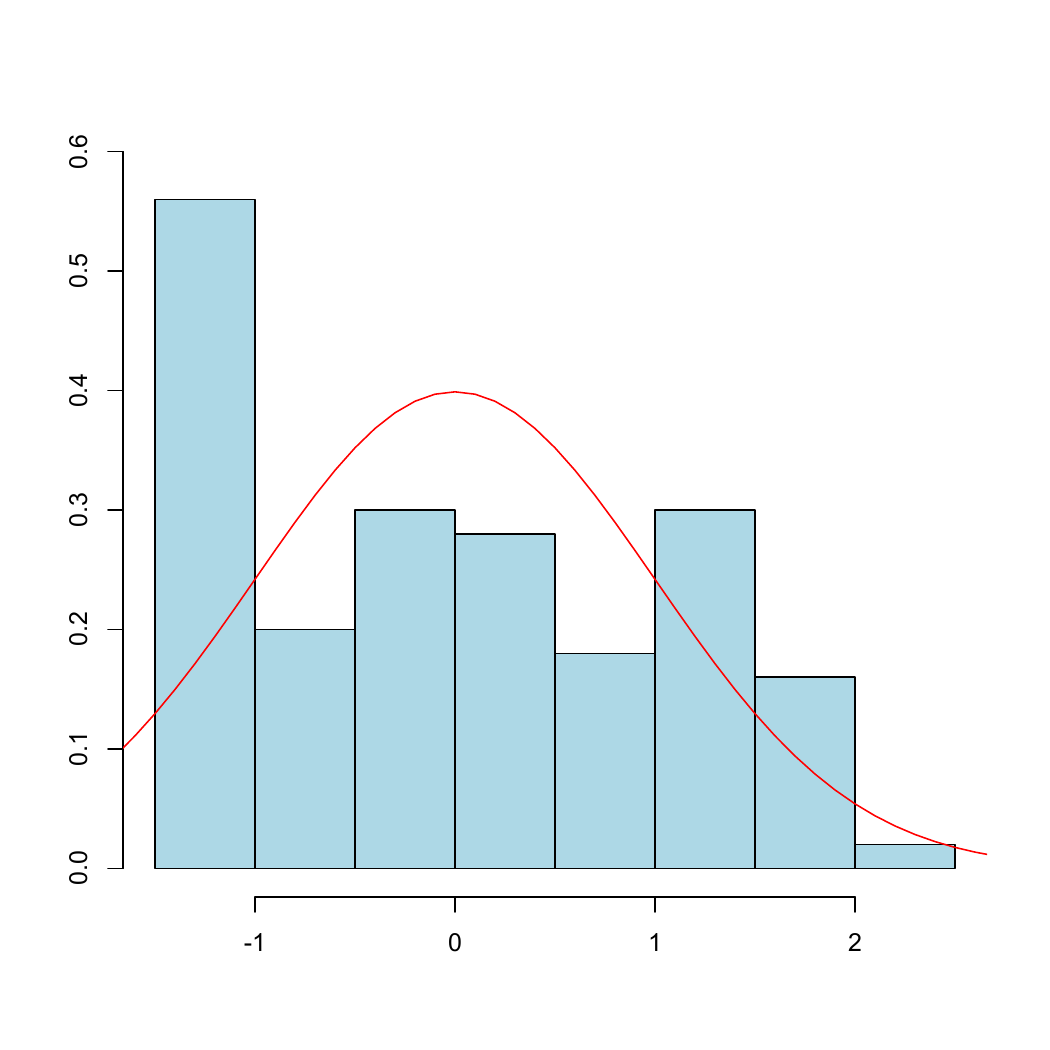} 
		\includegraphics[scale=0.2]{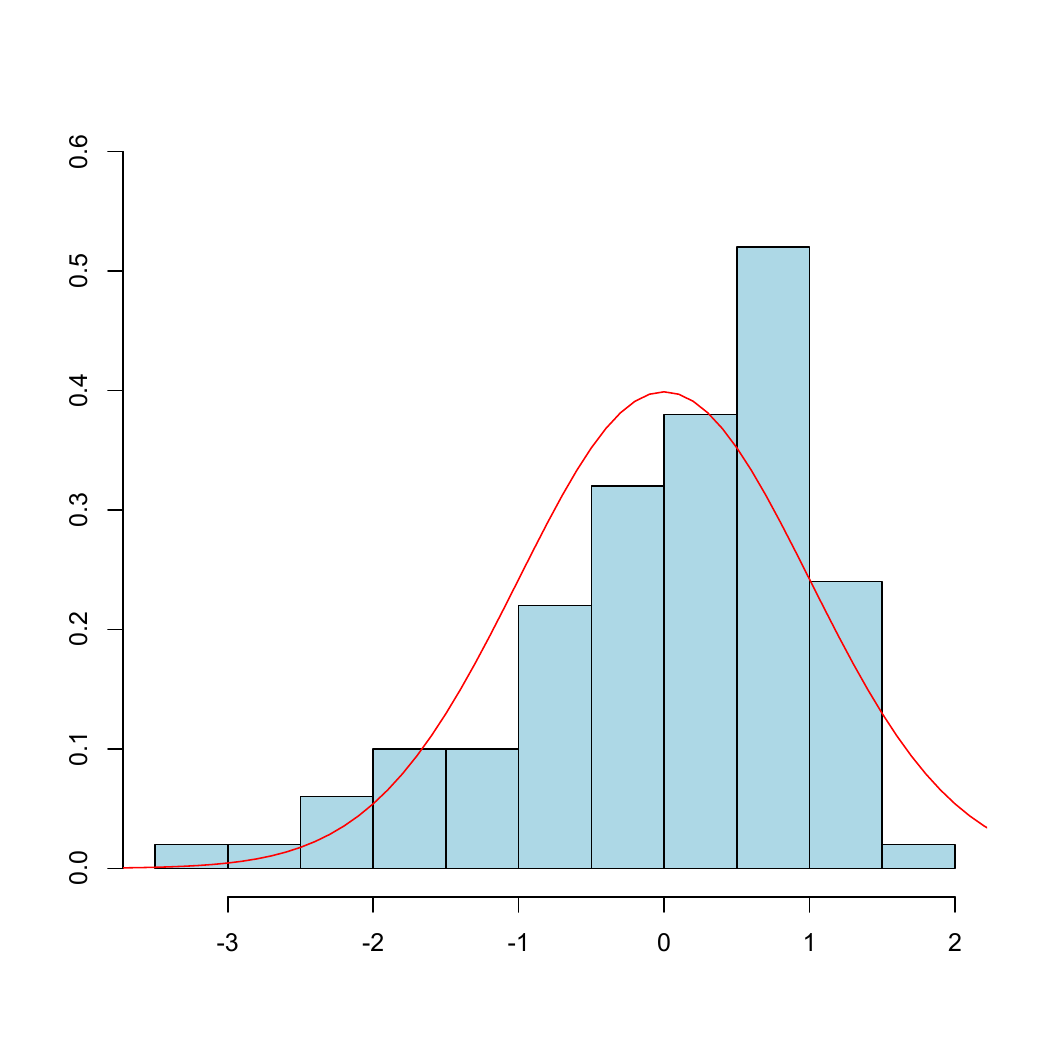}
		\includegraphics[scale=0.2]{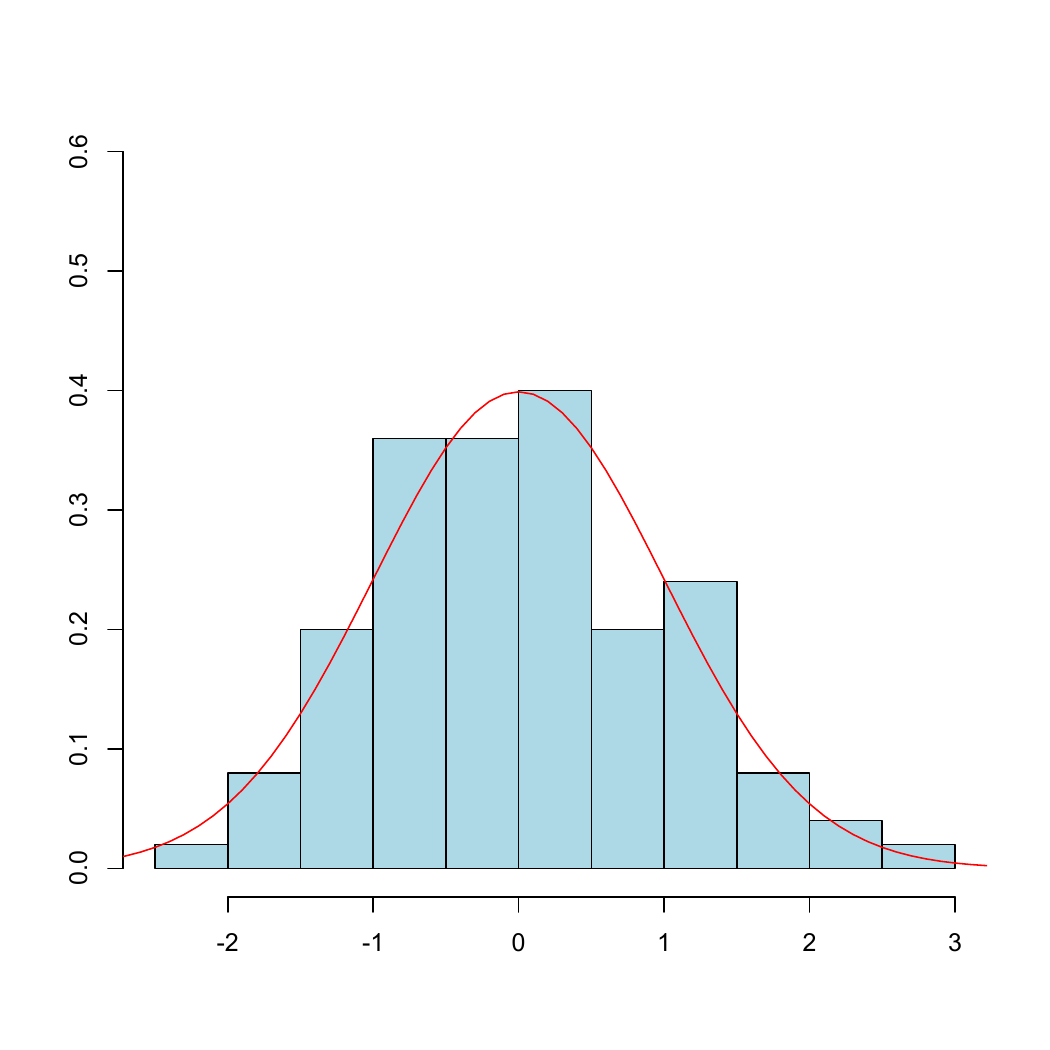} 
	\caption{Histogram of  $(\alpha,\beta)$ $\sqrt{n}$-normalised estimators distribution under model $\mbox{(S0)}_{weak}$. First row $n=1,000$, second row $n=3,000$, third row $n=5,000$. First column $\hat \alpha_n$, second column $\hat \beta_n$, third column  $\tilde \alpha_n$, fourth column $\tilde \beta_n$.}
		\label{fig:normalityS0weak}
	\end{center}
\end{figure}
\newpage
\subsection{Bell-like regime of $\hat \theta_n$ and $\tilde \theta_n$ for models (S1-4)}
\begin{figure}[!h]
	\begin{center}
	        \includegraphics[scale=0.2]{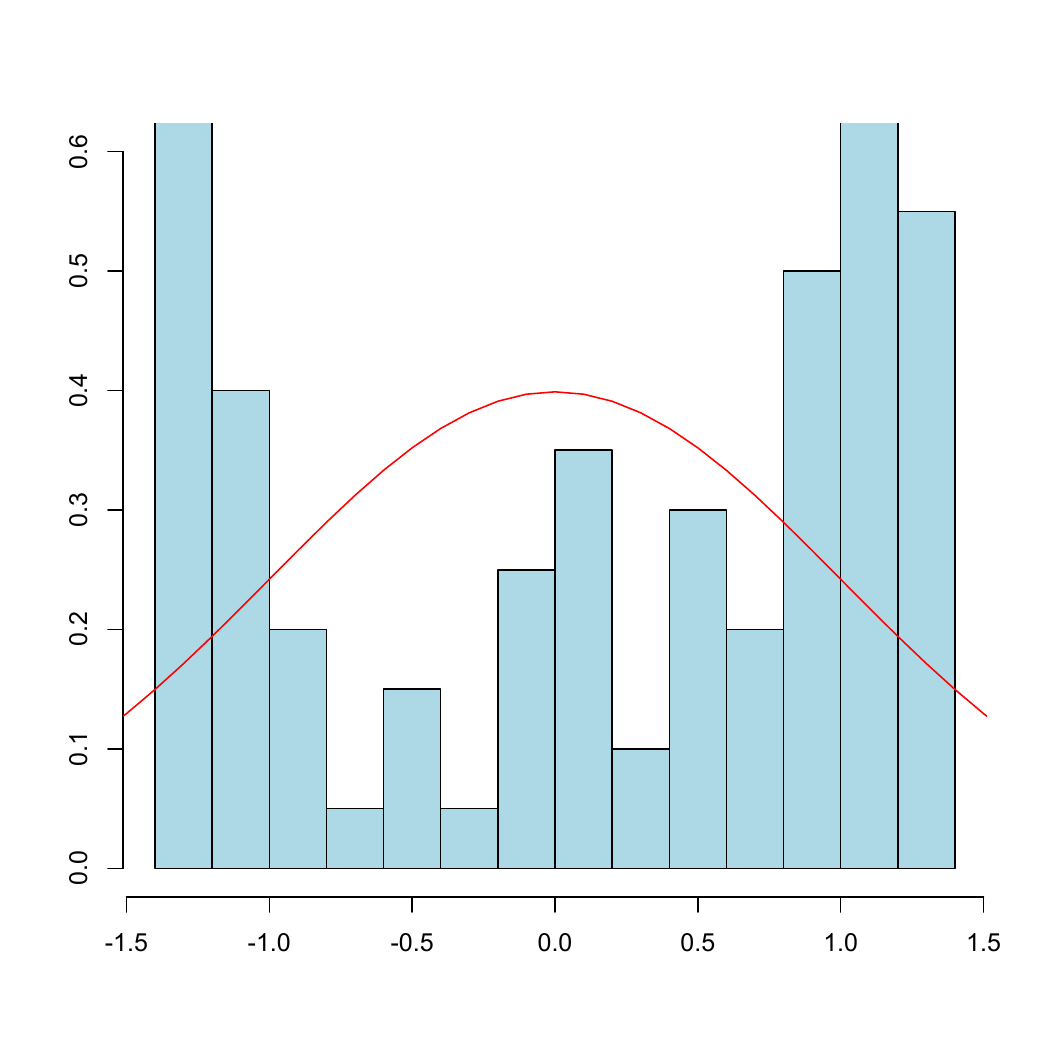}
		\includegraphics[scale=0.2]{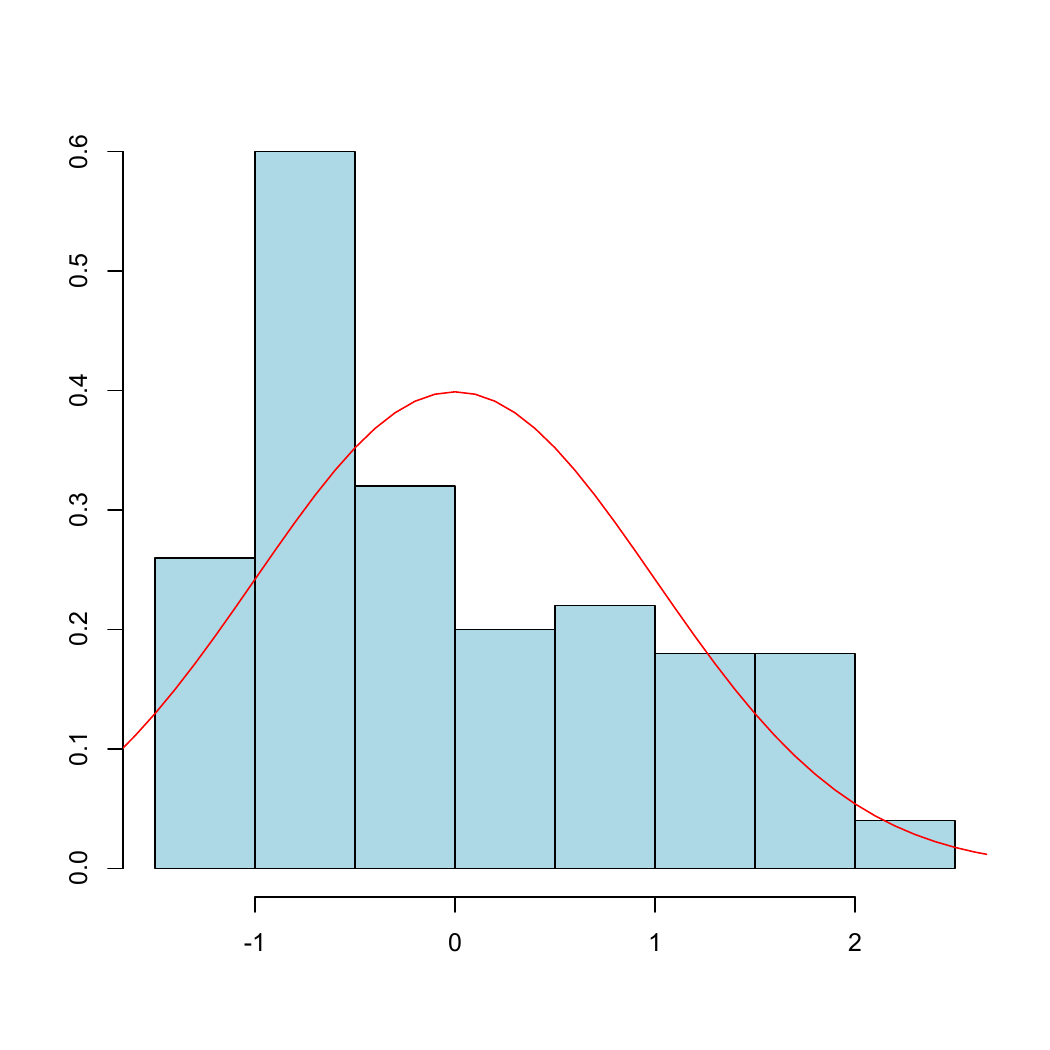} 
		 \includegraphics[scale=0.2]{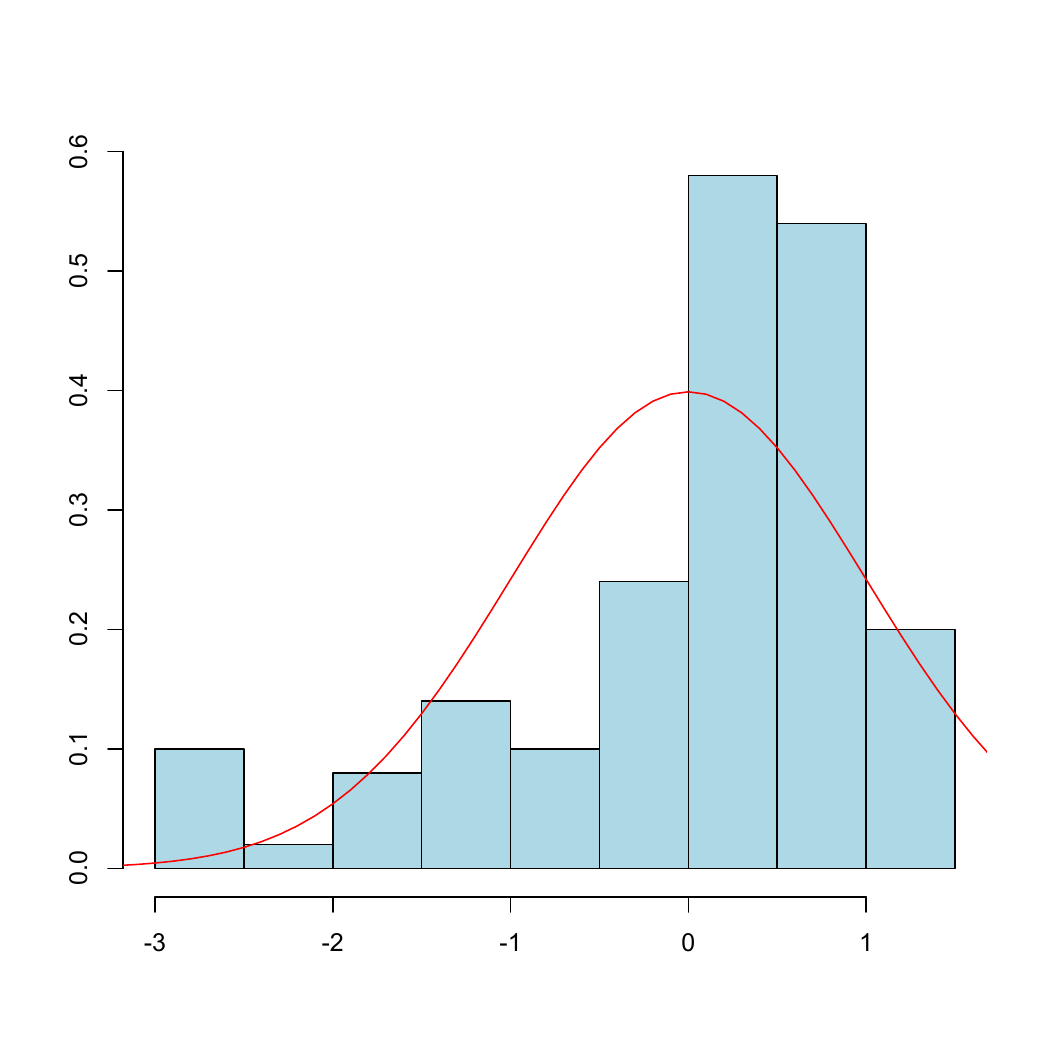}
		\includegraphics[scale=0.2]{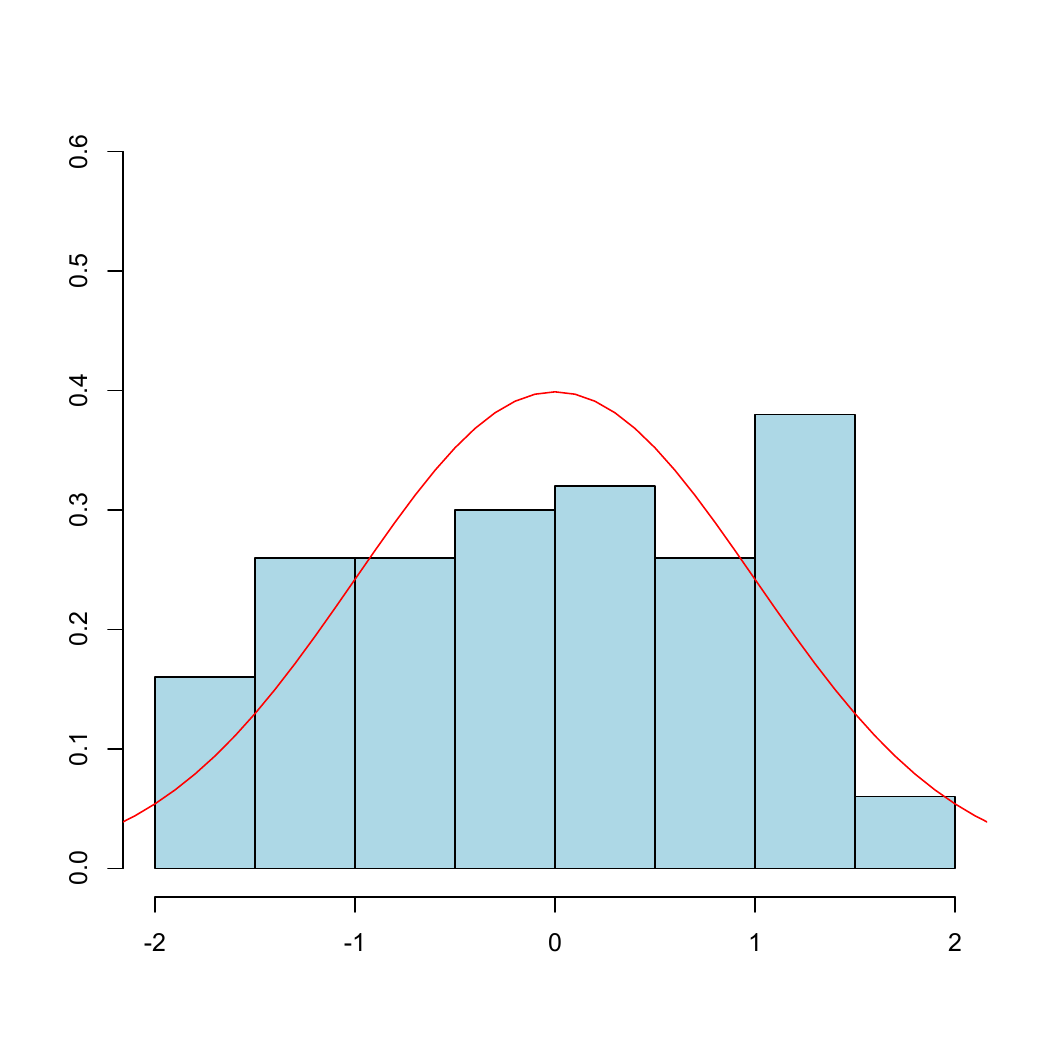} 
		\\
	        \includegraphics[scale=0.2]{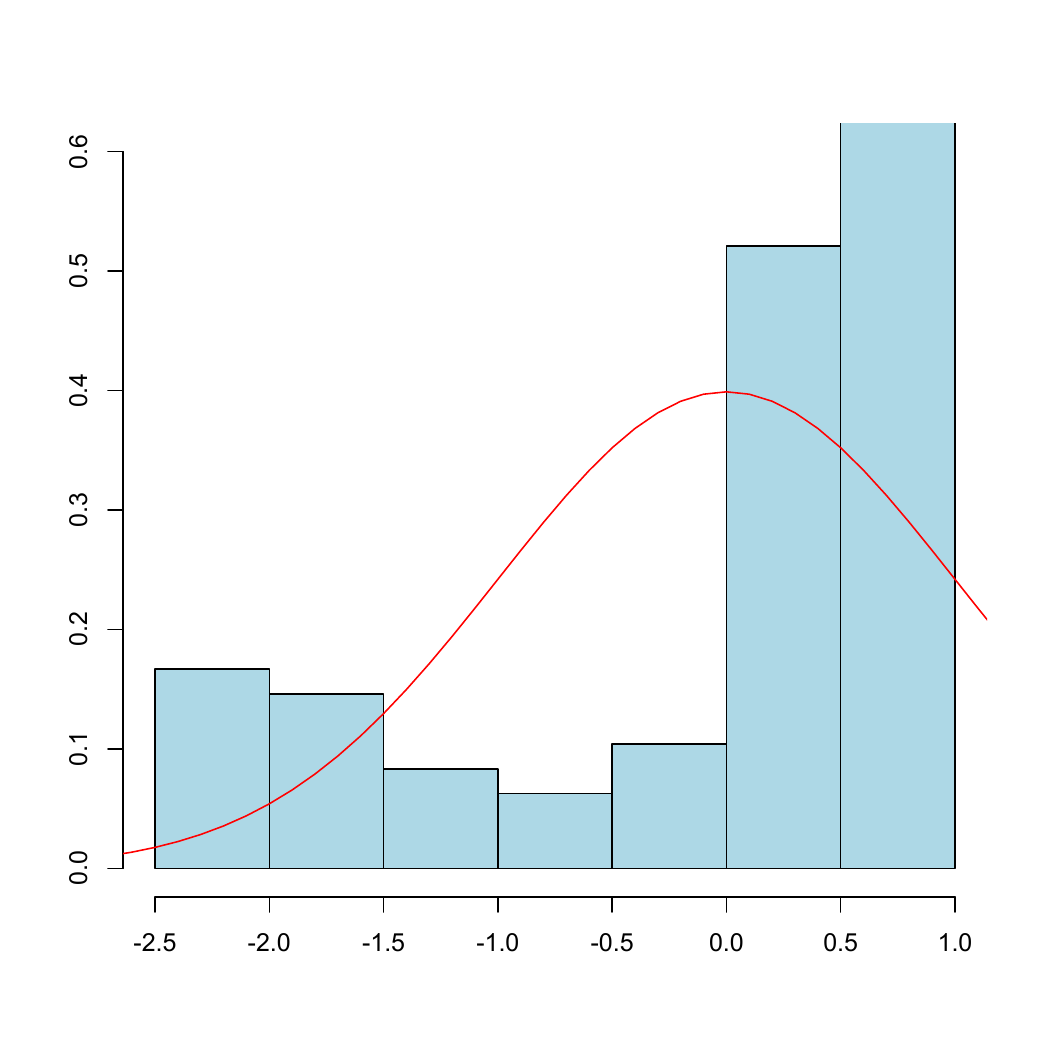}
		\includegraphics[scale=0.2]{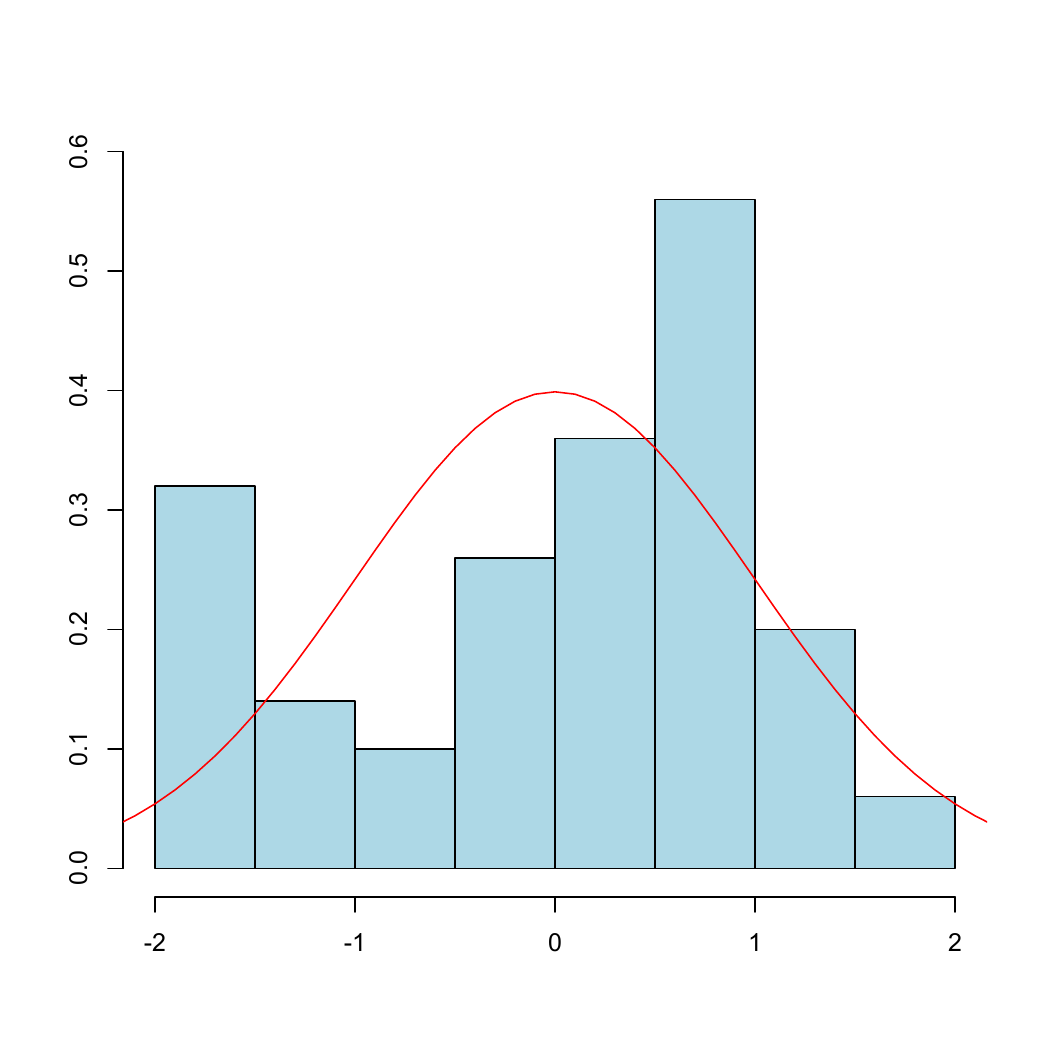} 
		\includegraphics[scale=0.2]{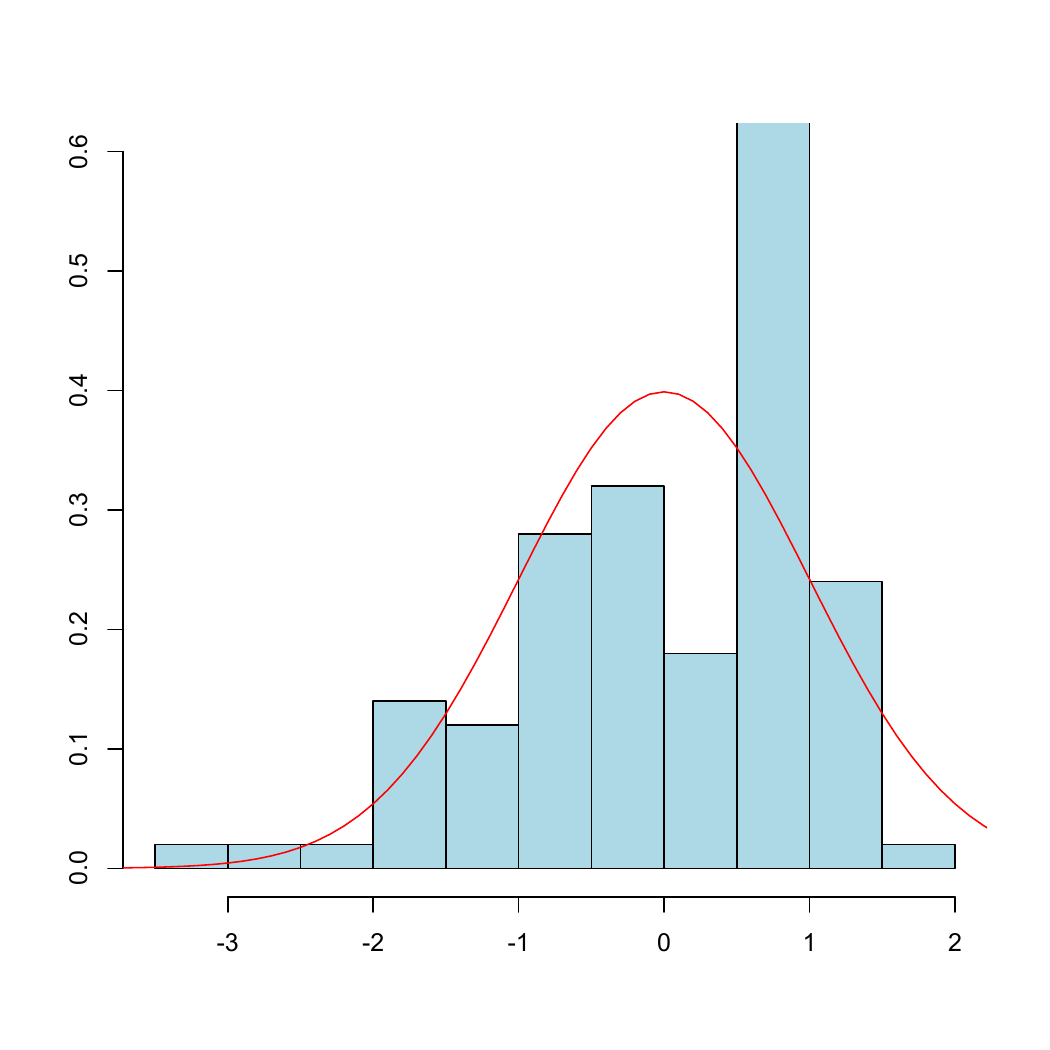}
		\includegraphics[scale=0.2]{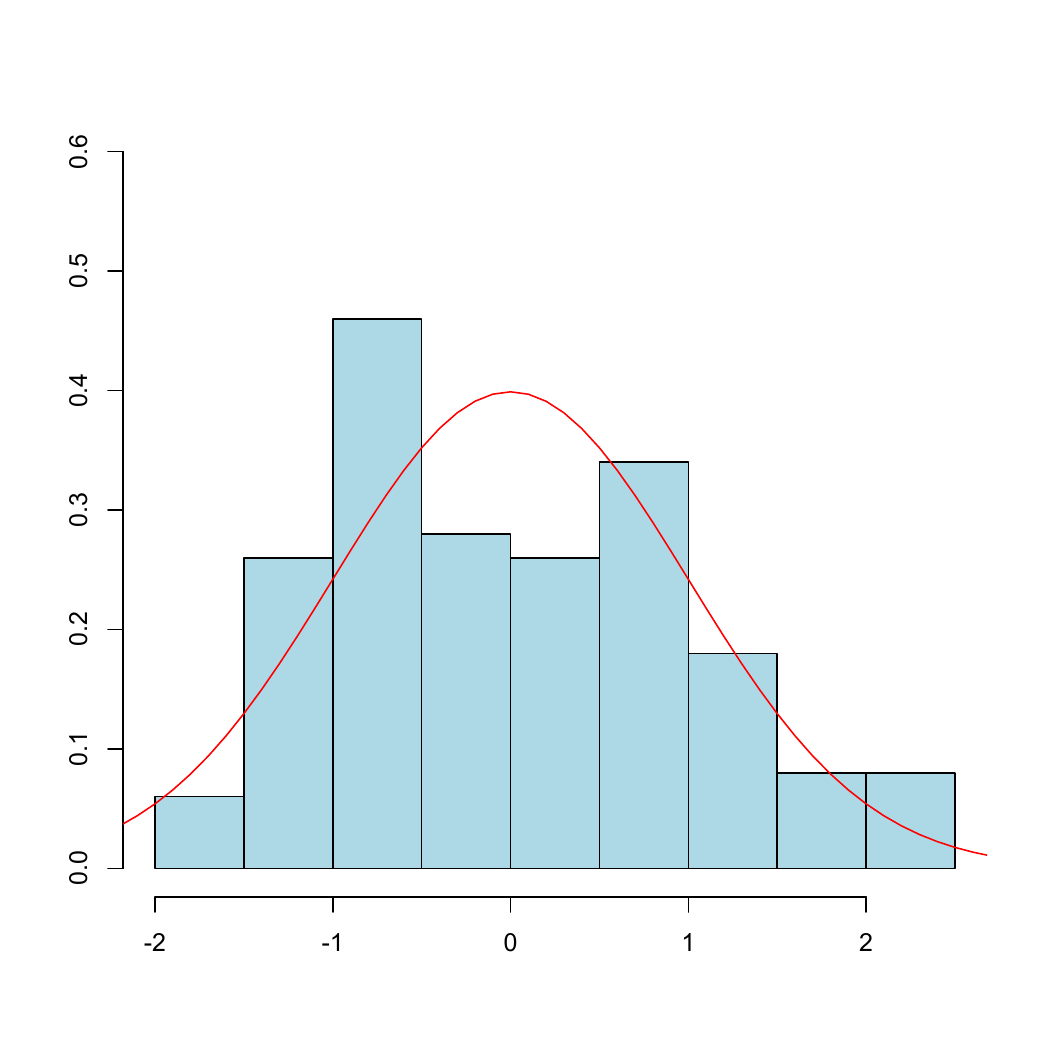}
		\\
		\includegraphics[scale=0.2]{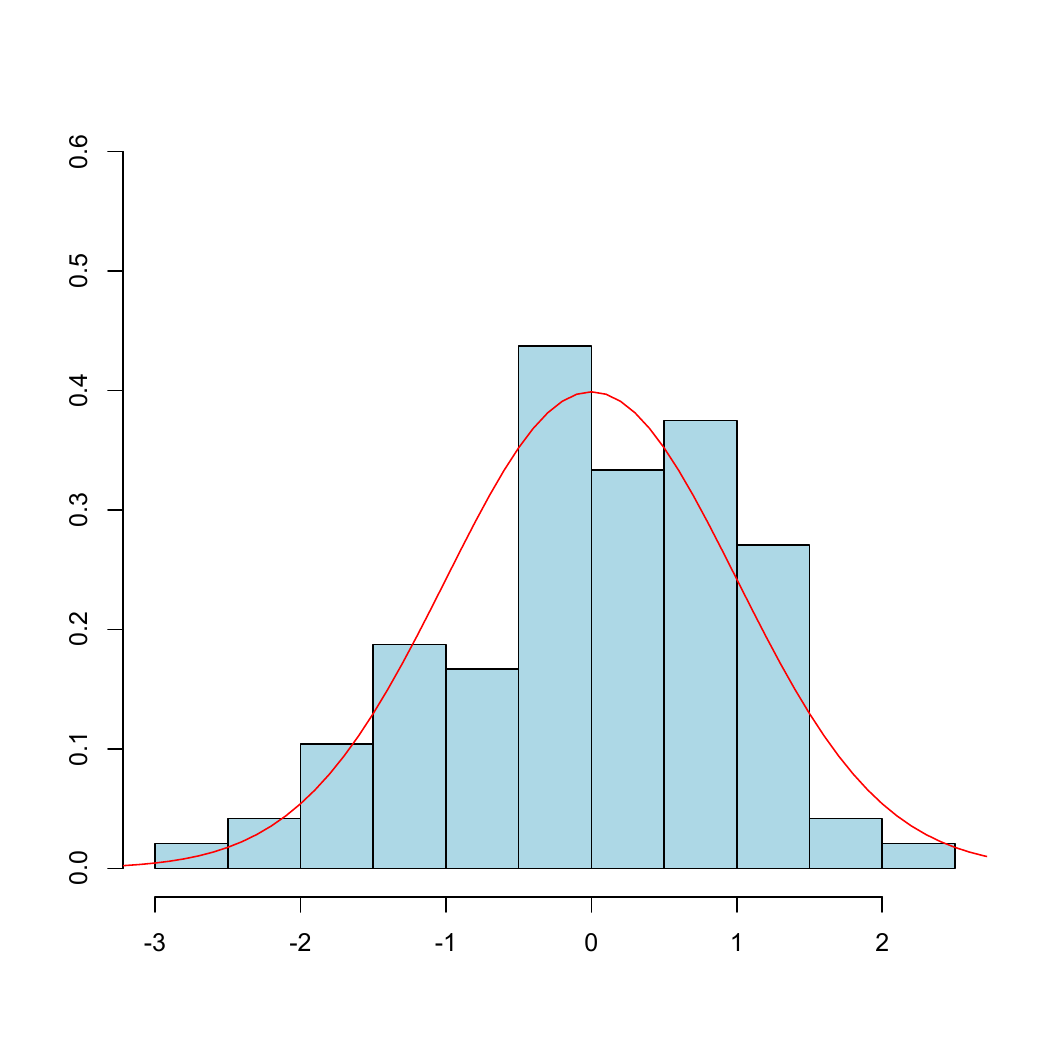}
		\includegraphics[scale=0.2]{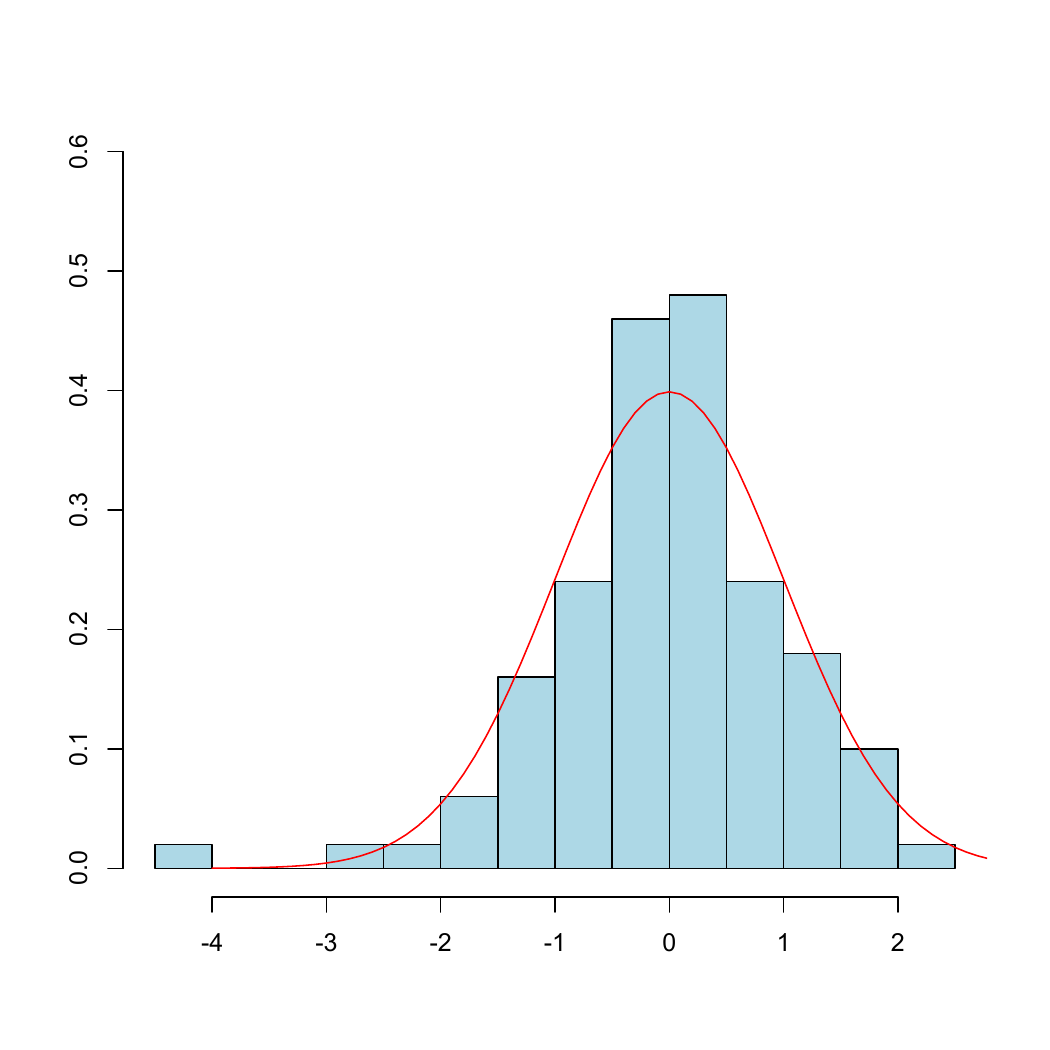} 
		\includegraphics[scale=0.2]{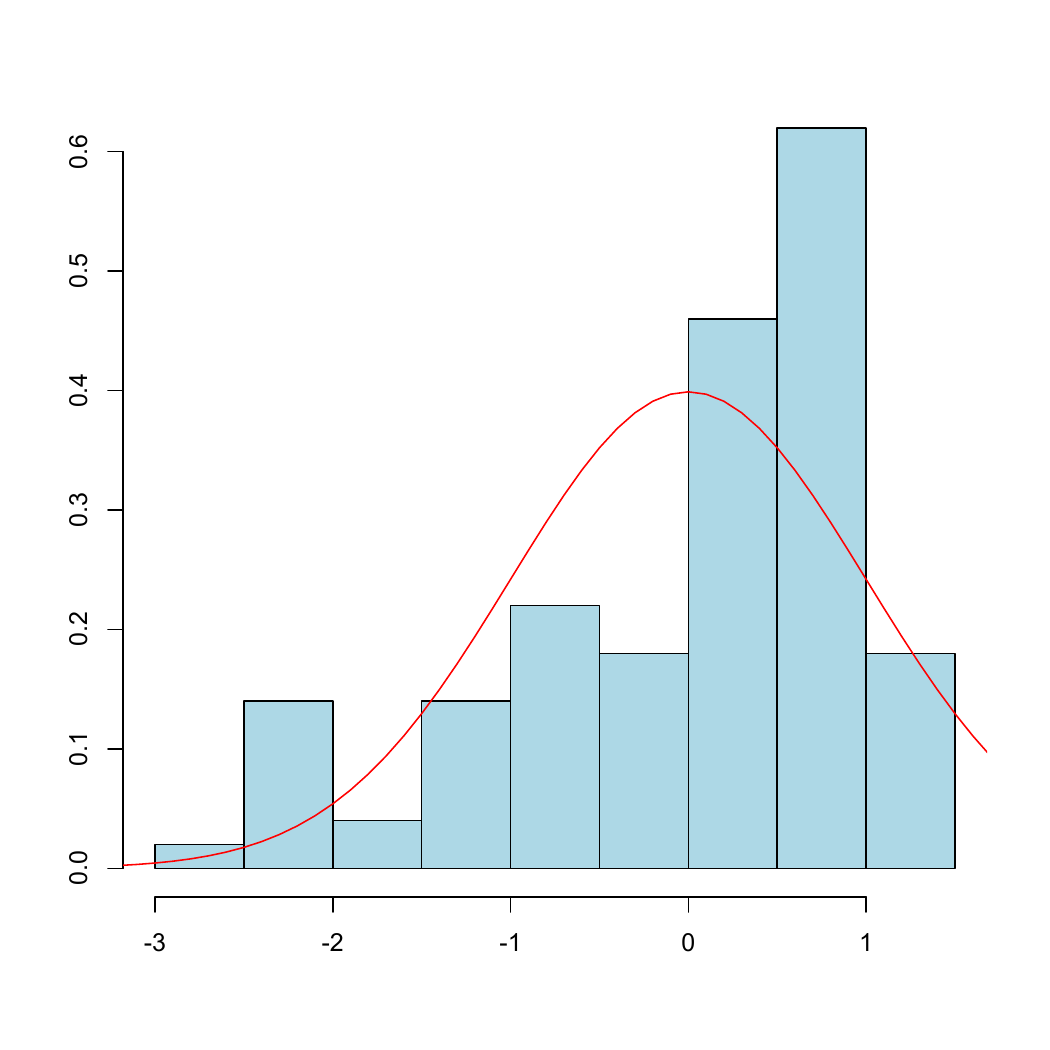}
		\includegraphics[scale=0.2]{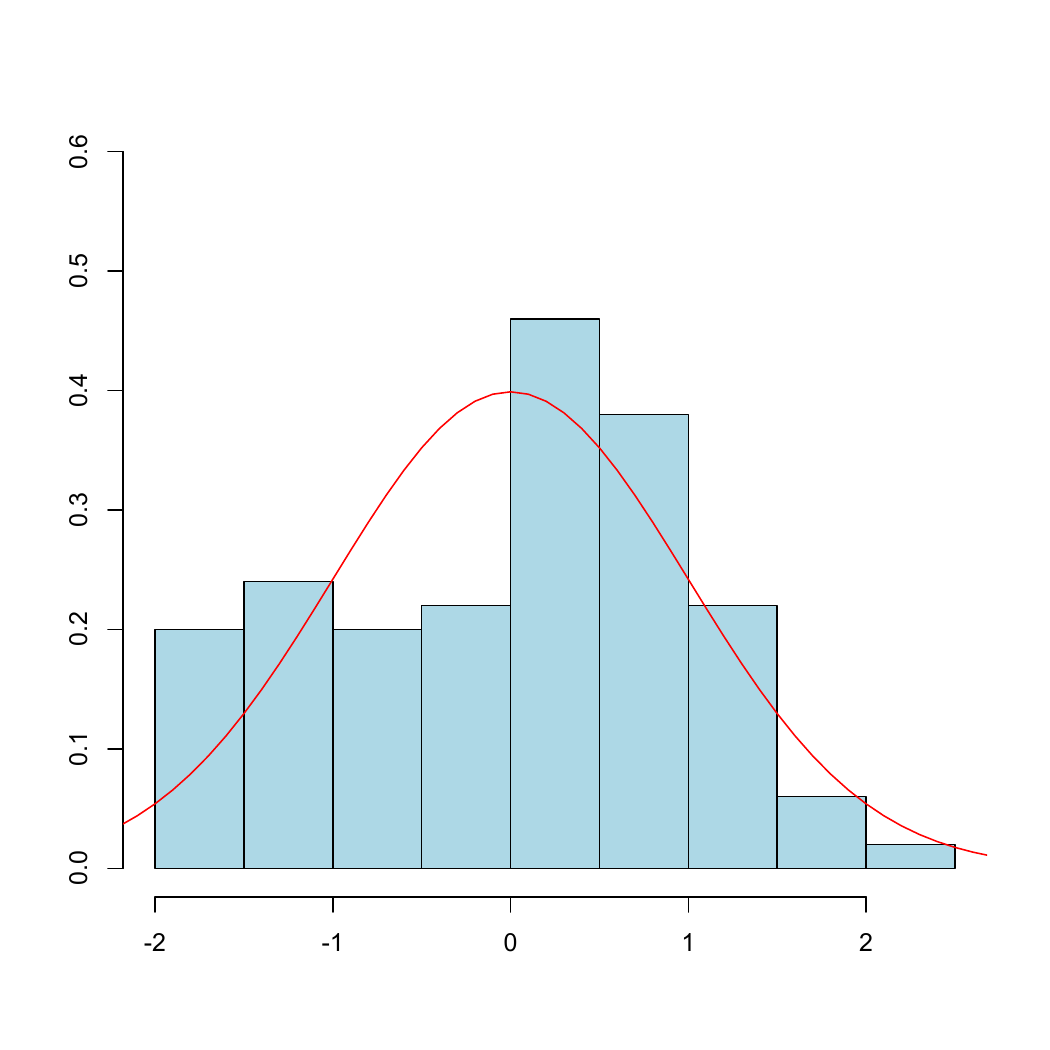} 
	\caption{Histogram of  $(\alpha,\beta)$ $\sqrt{n}$-normalised estimators distribution under model (S1). First row $n=5,000$, second row $n=10,000$, third row $n=20,000$. First column $\hat \alpha_n$, second column $\hat \beta_n$, third column  $\tilde \alpha_n$, fourth column $\tilde \beta_n$.}
		\label{fig:normalityS1}
	\end{center}
\end{figure}
\newpage
\begin{figure}[!h]
	\begin{center}
	        \includegraphics[scale=0.2]{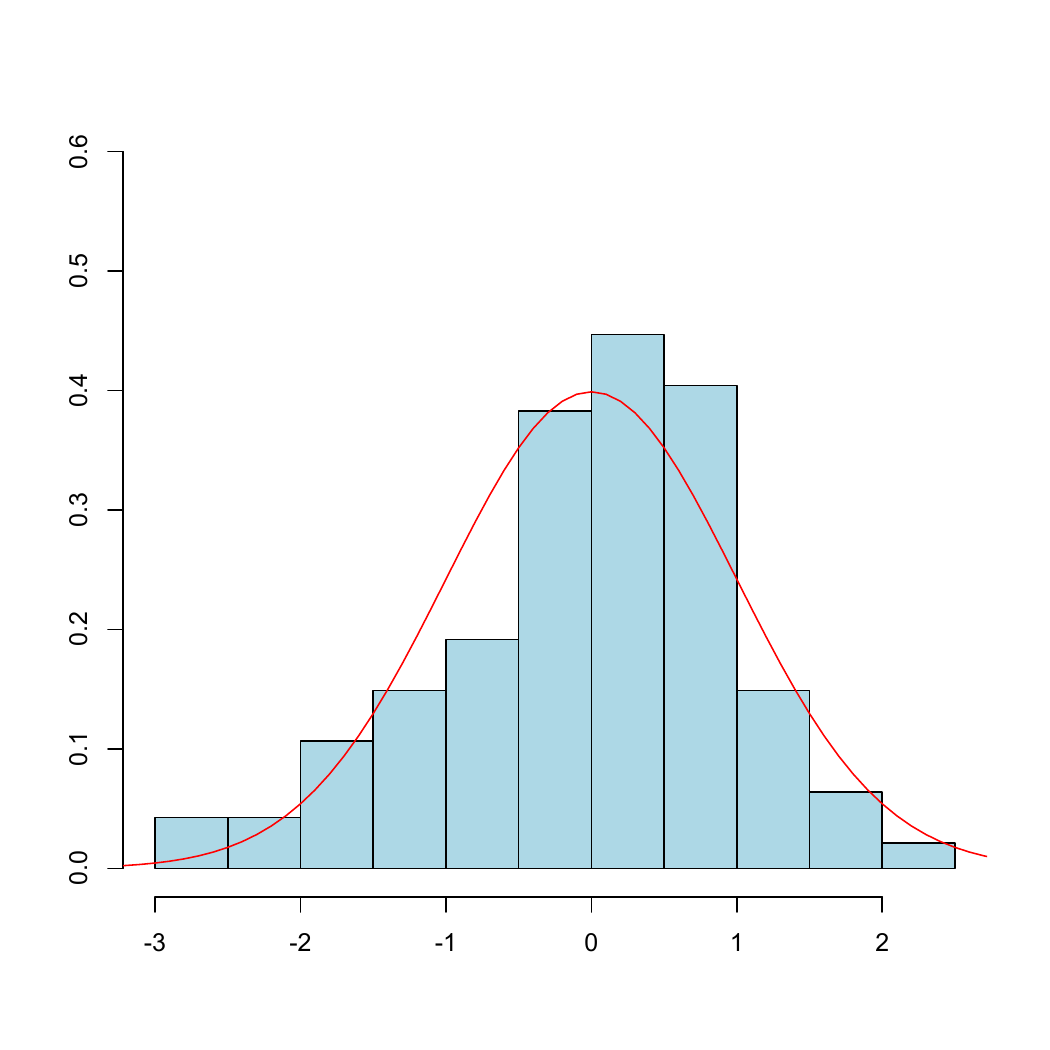}
		\includegraphics[scale=0.2]{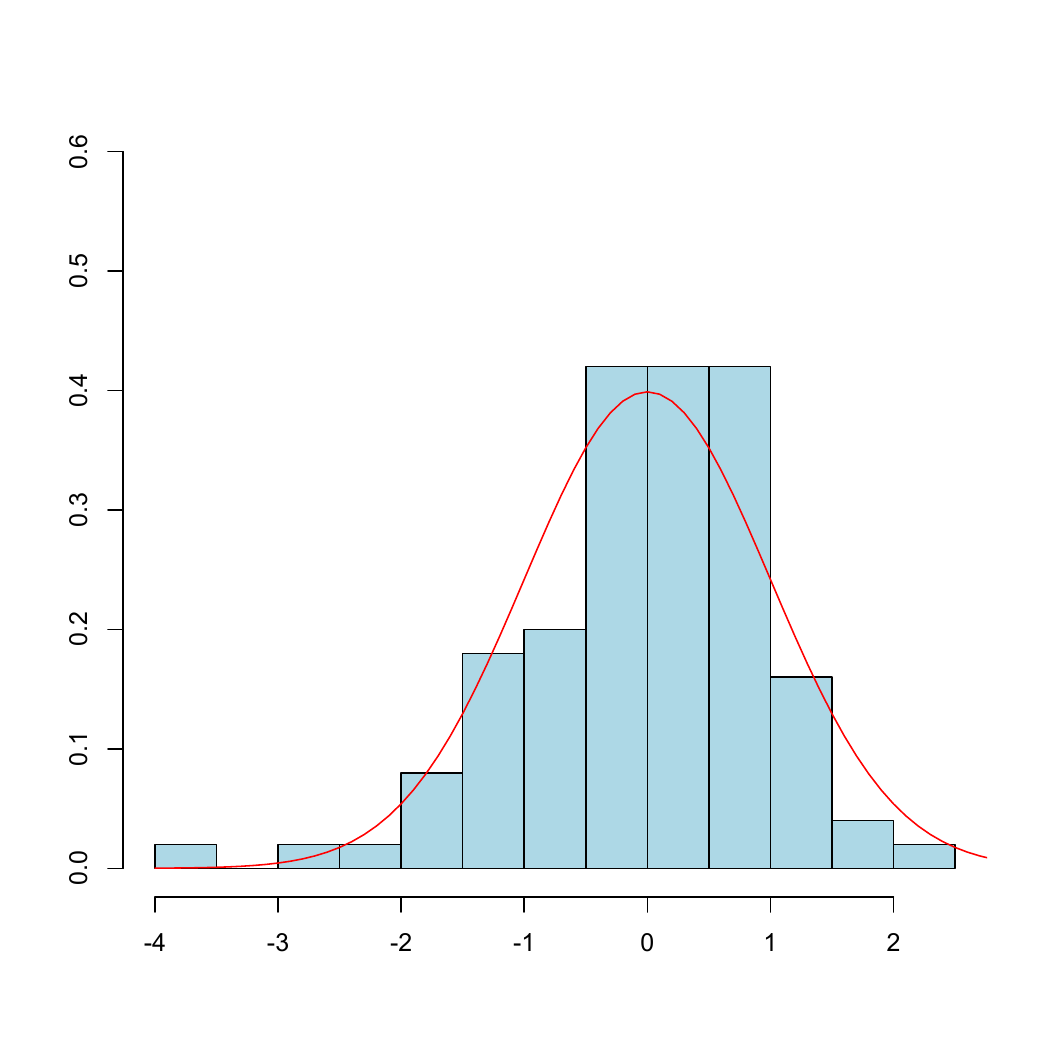} 
		 \includegraphics[scale=0.2]{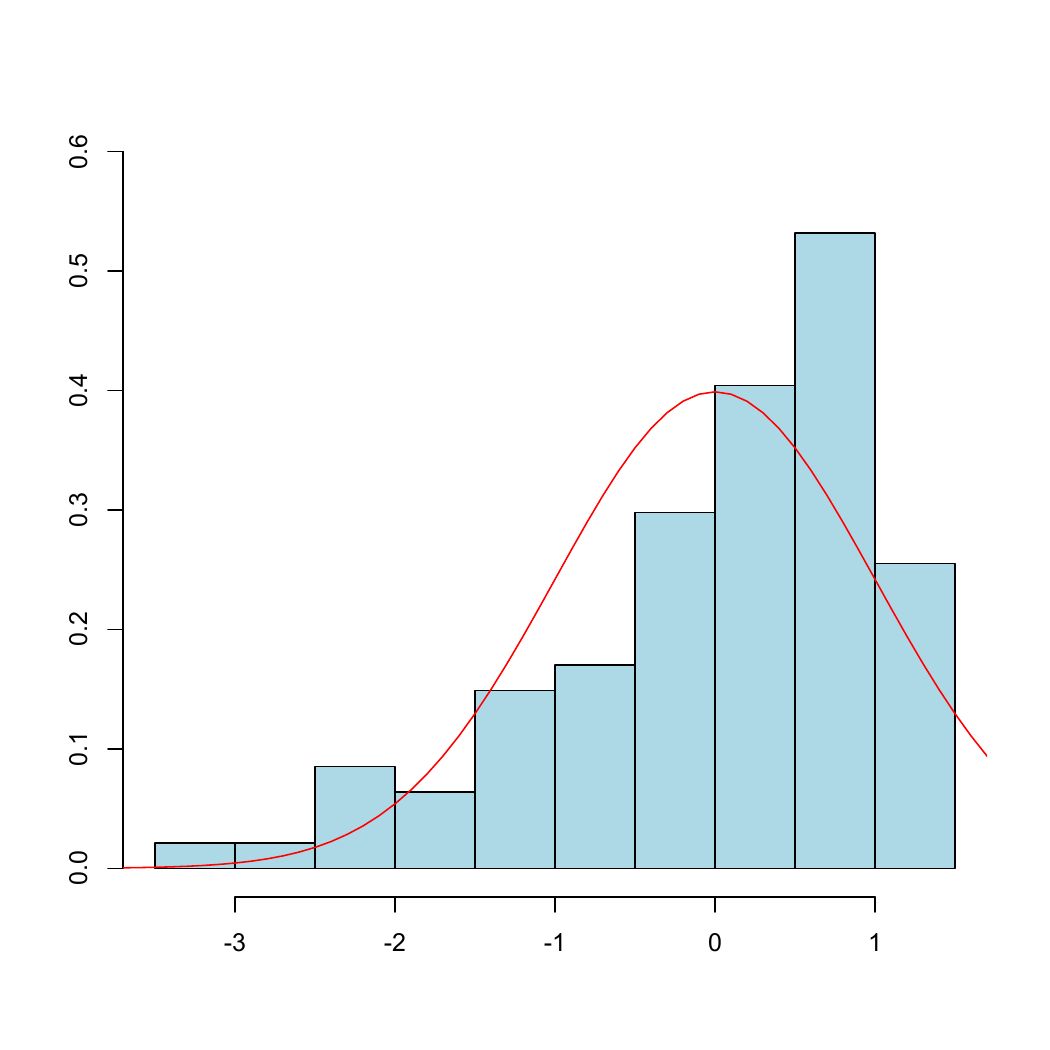}
		\includegraphics[scale=0.2]{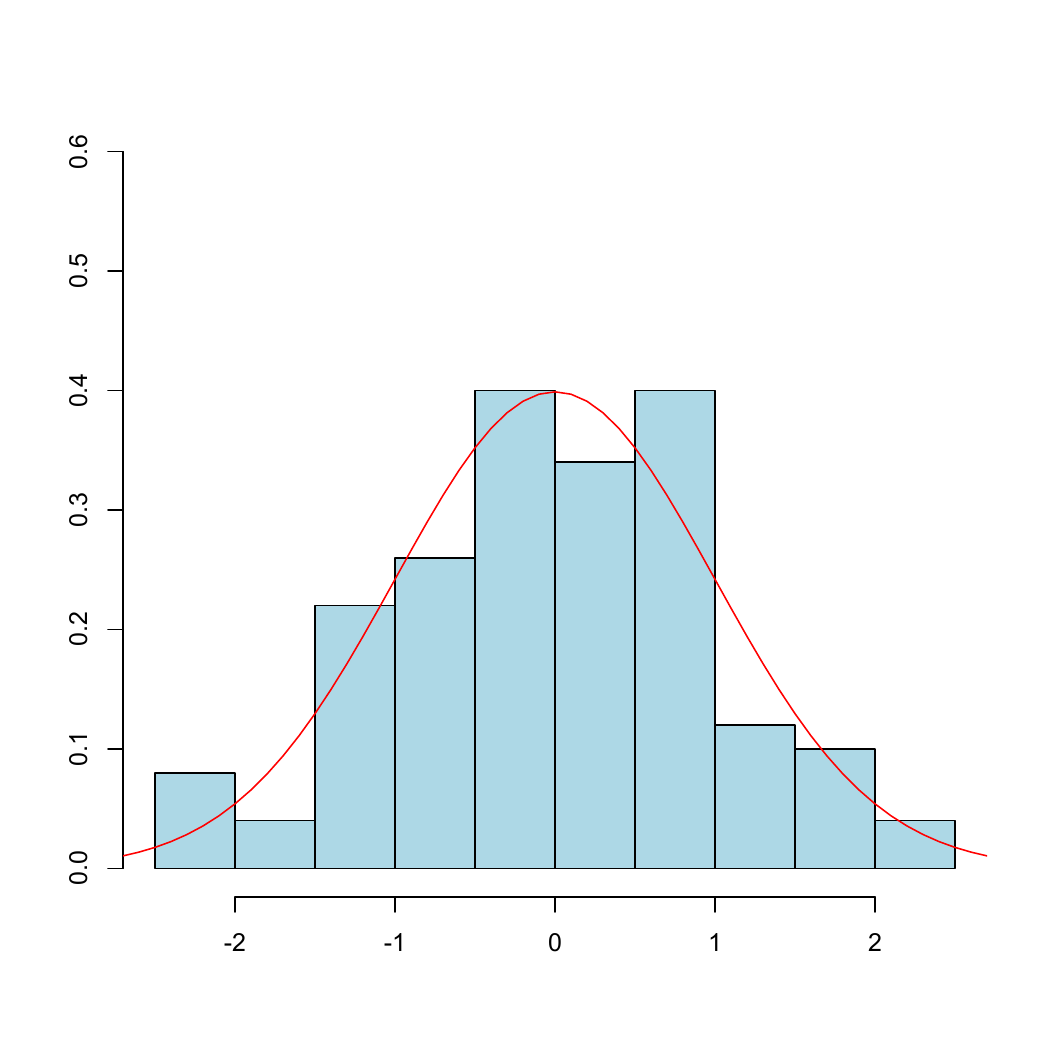} 
		\\
	        \includegraphics[scale=0.2]{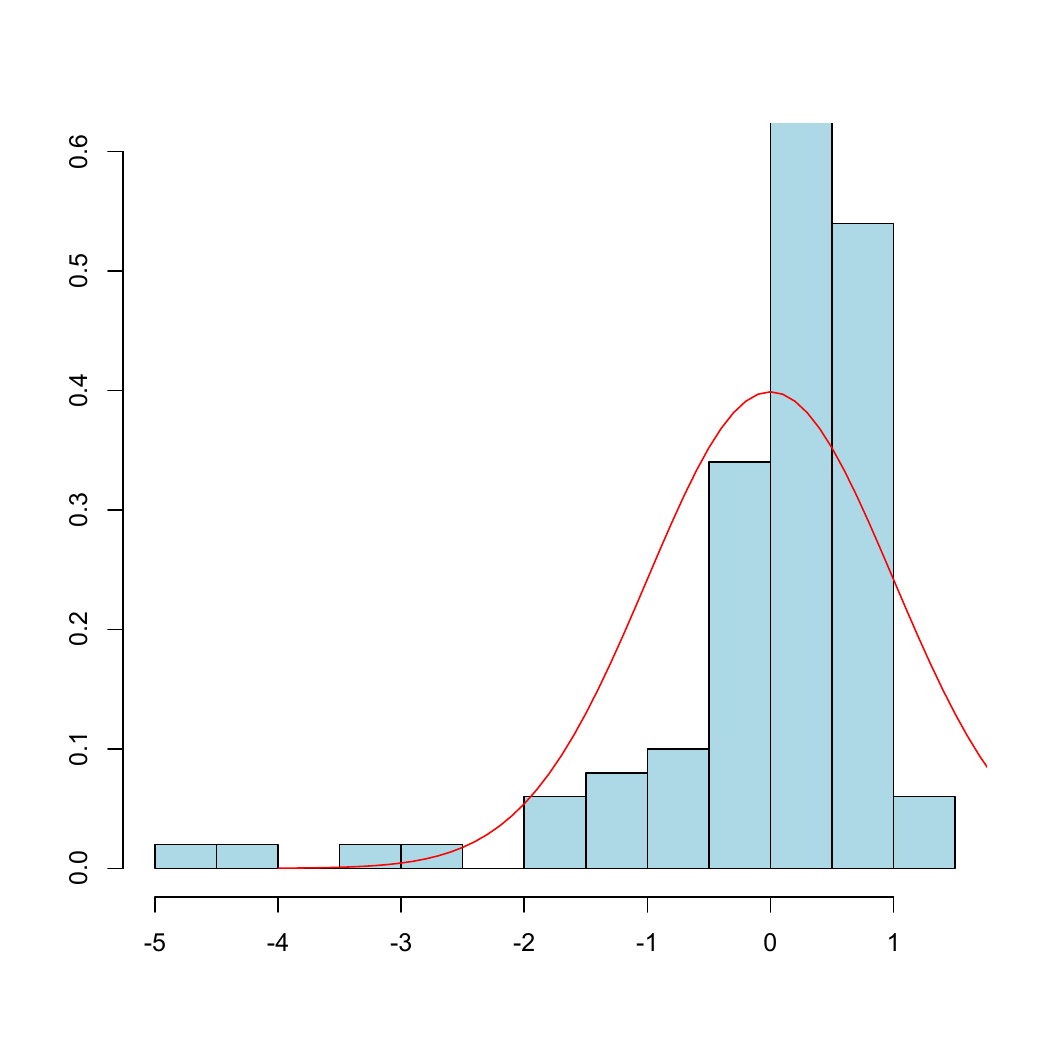}
		\includegraphics[scale=0.2]{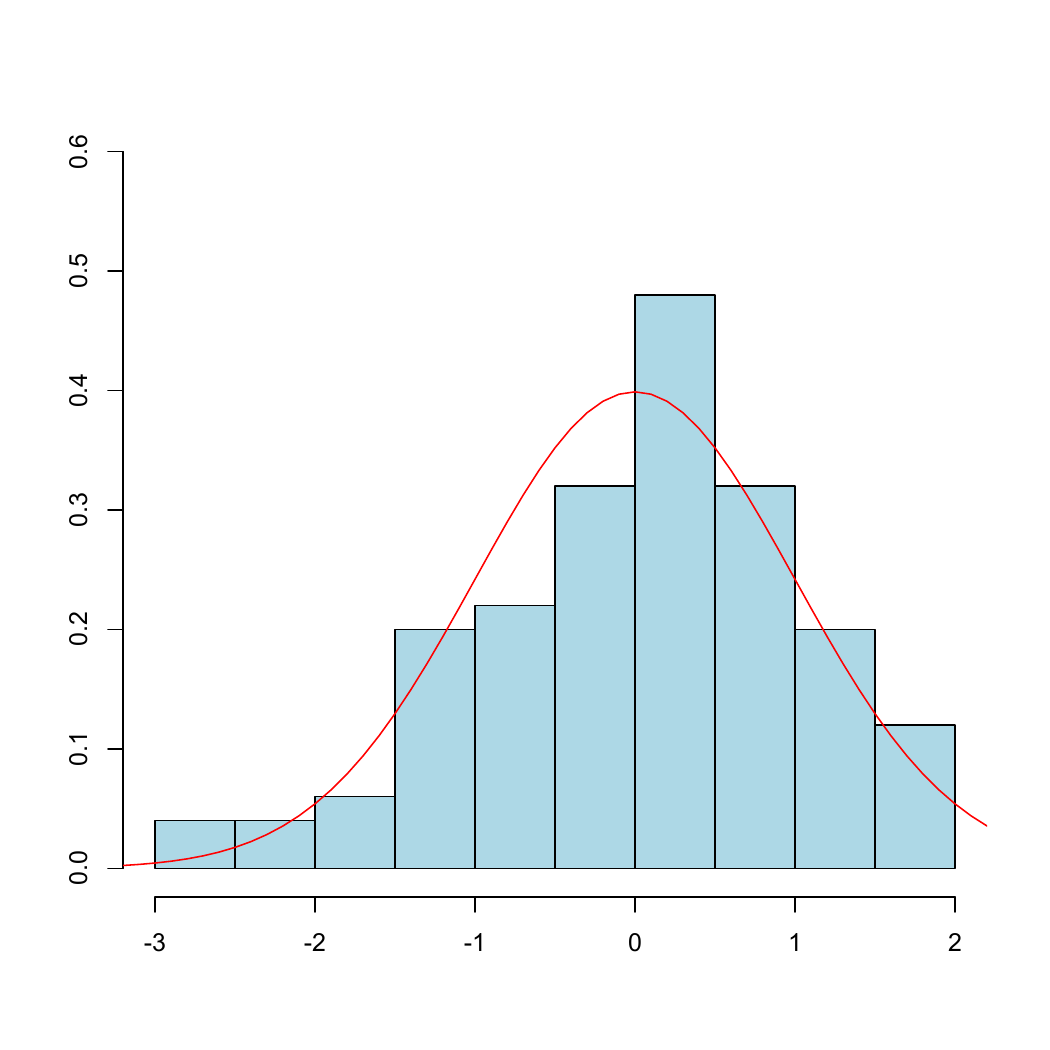} 
		\includegraphics[scale=0.2]{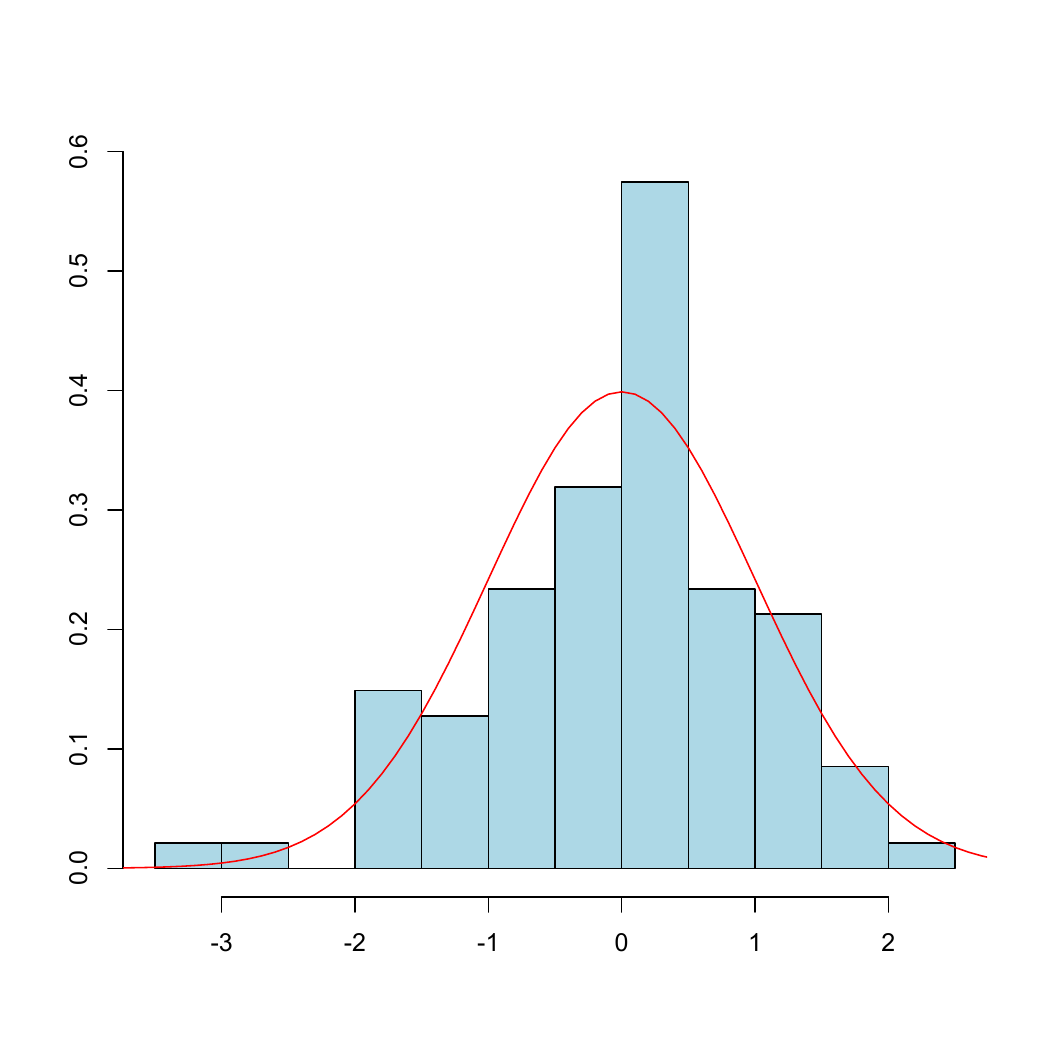}
		\includegraphics[scale=0.2]{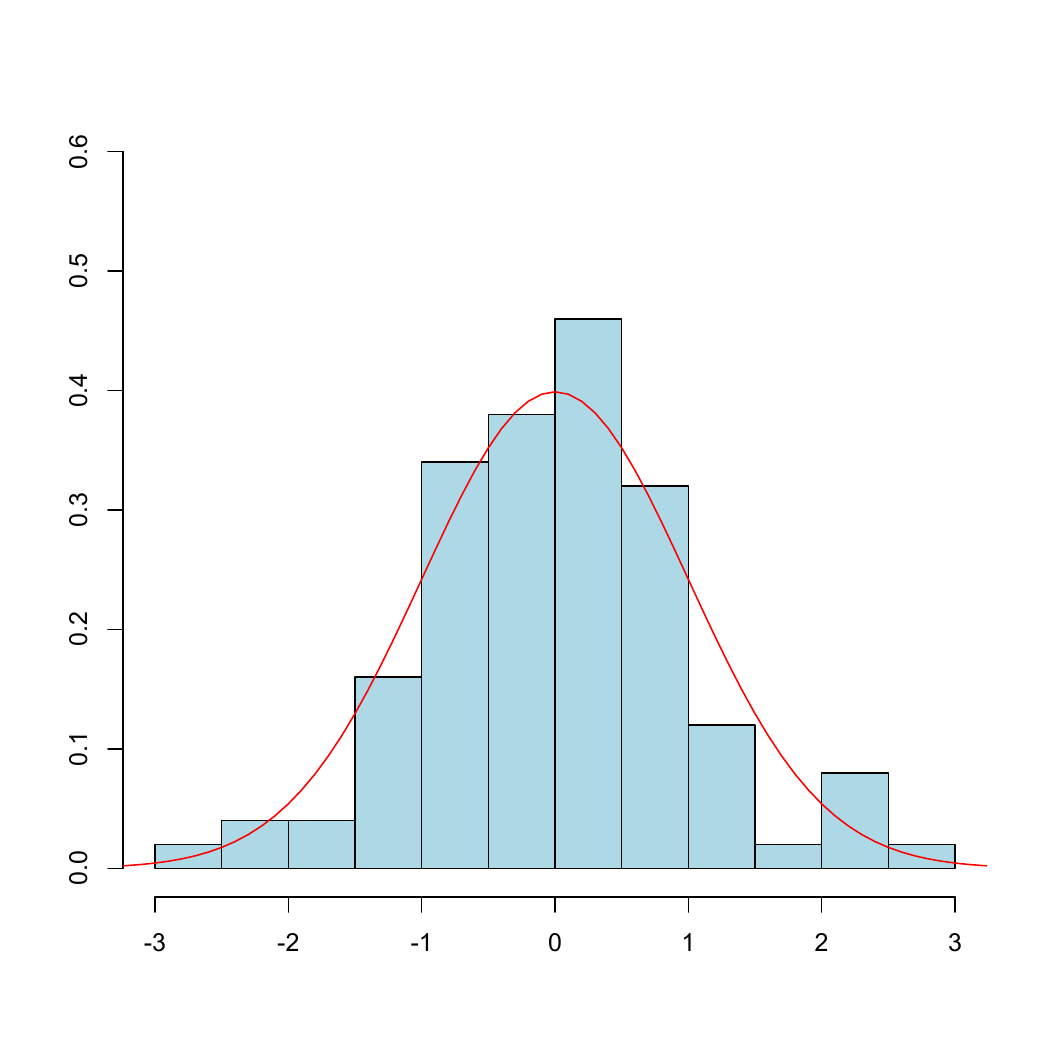}
		\\
		\includegraphics[scale=0.2]{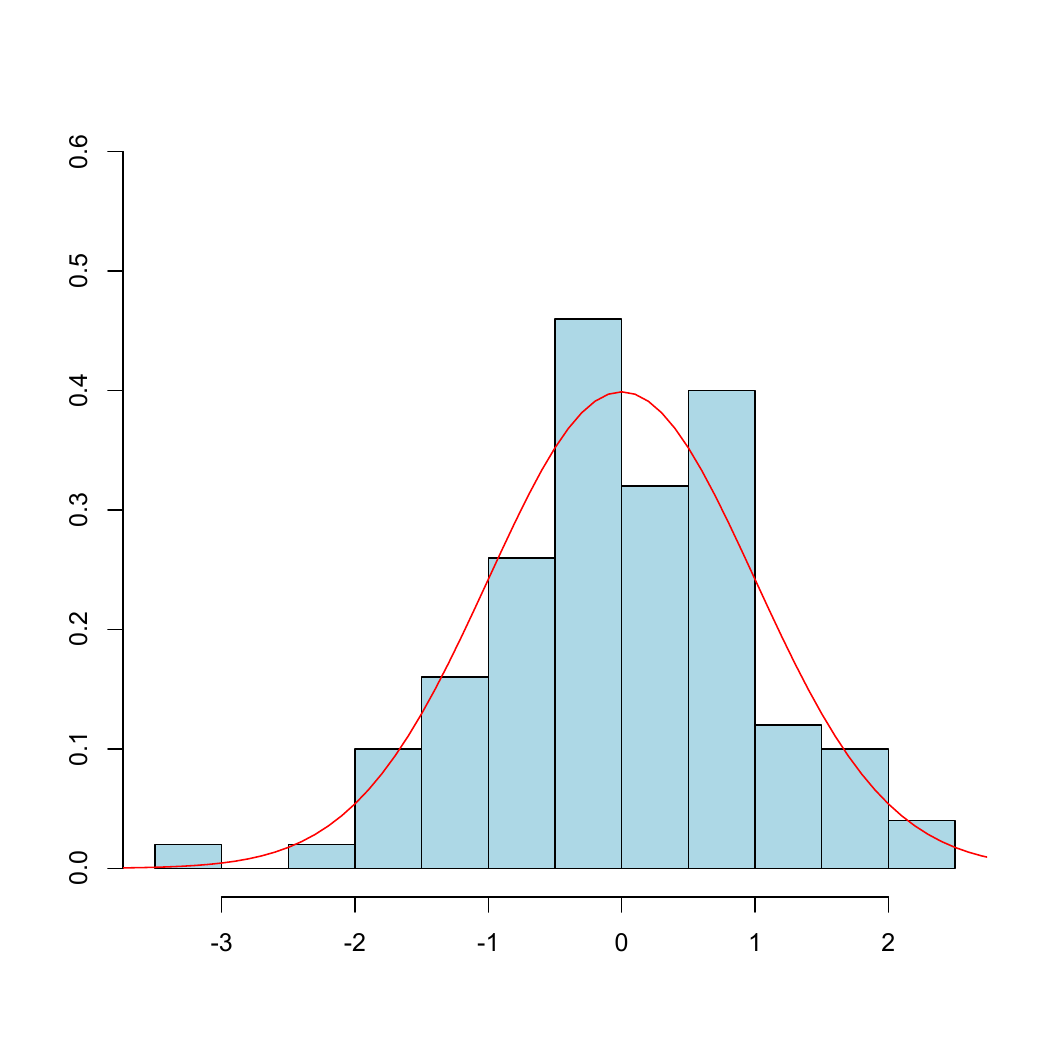}
		\includegraphics[scale=0.2]{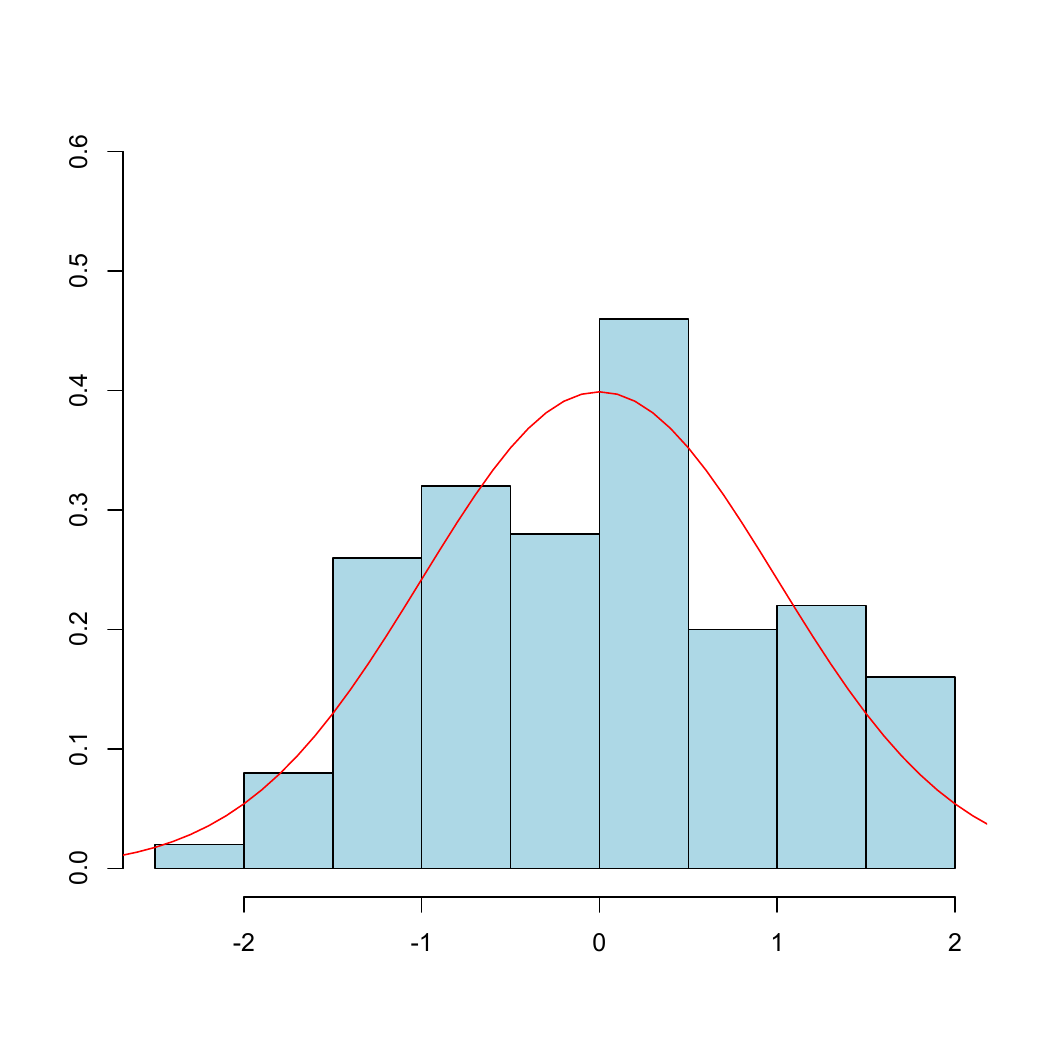} 
		\includegraphics[scale=0.2]{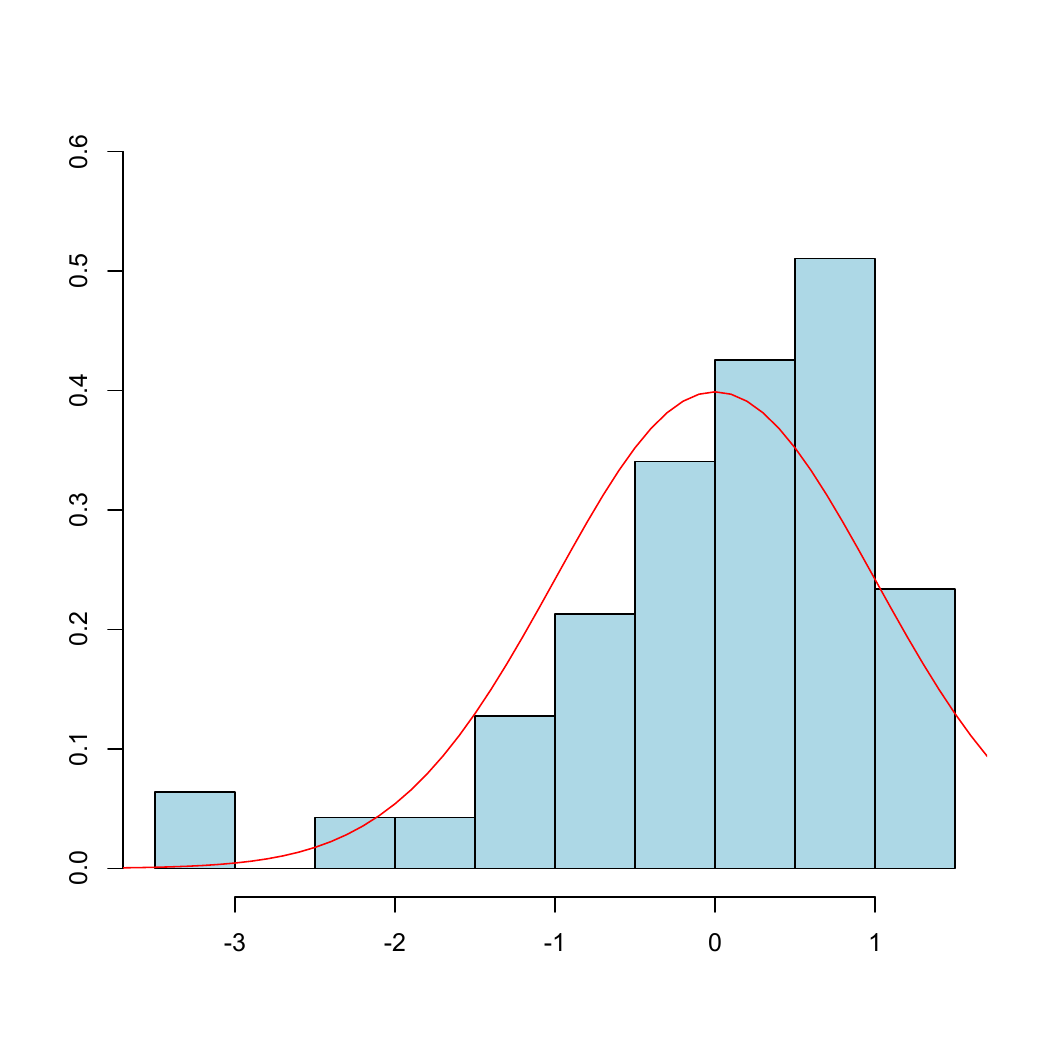}
		\includegraphics[scale=0.2]{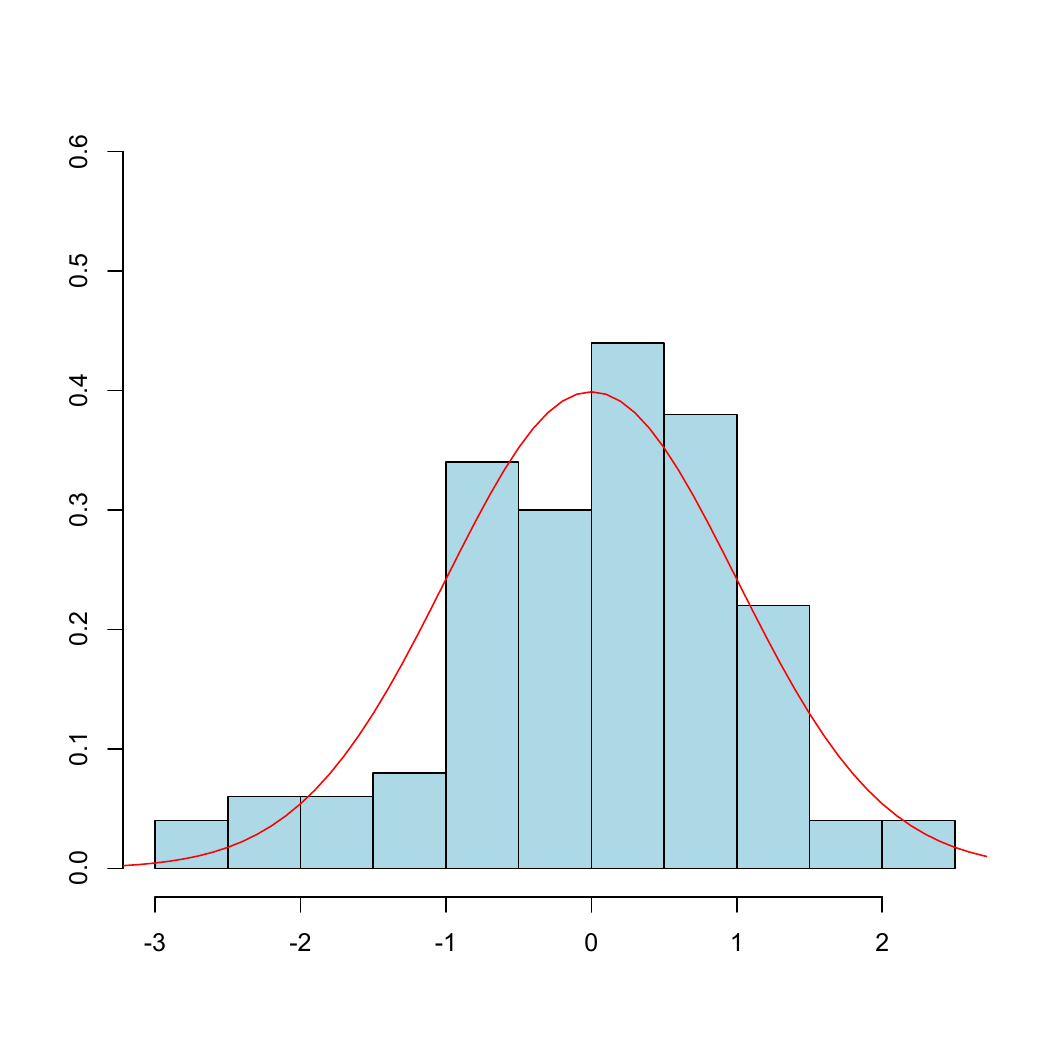} 
	\caption{Histogram of  $(\alpha,\beta)$ $\sqrt{n}$-normalised estimators distribution under model (S2). First row $n=5,000$, second row $n=10,000$, third row $n=20,000$. First column $\hat \alpha_n$, second column $\hat \beta_n$, third column  $\tilde \alpha_n$, fourth column $\tilde \beta_n$.}
		\label{fig:normalityS2}
	\end{center}
\end{figure}
\newpage
\begin{figure}[!h]
	\begin{center}
	        \includegraphics[scale=0.2]{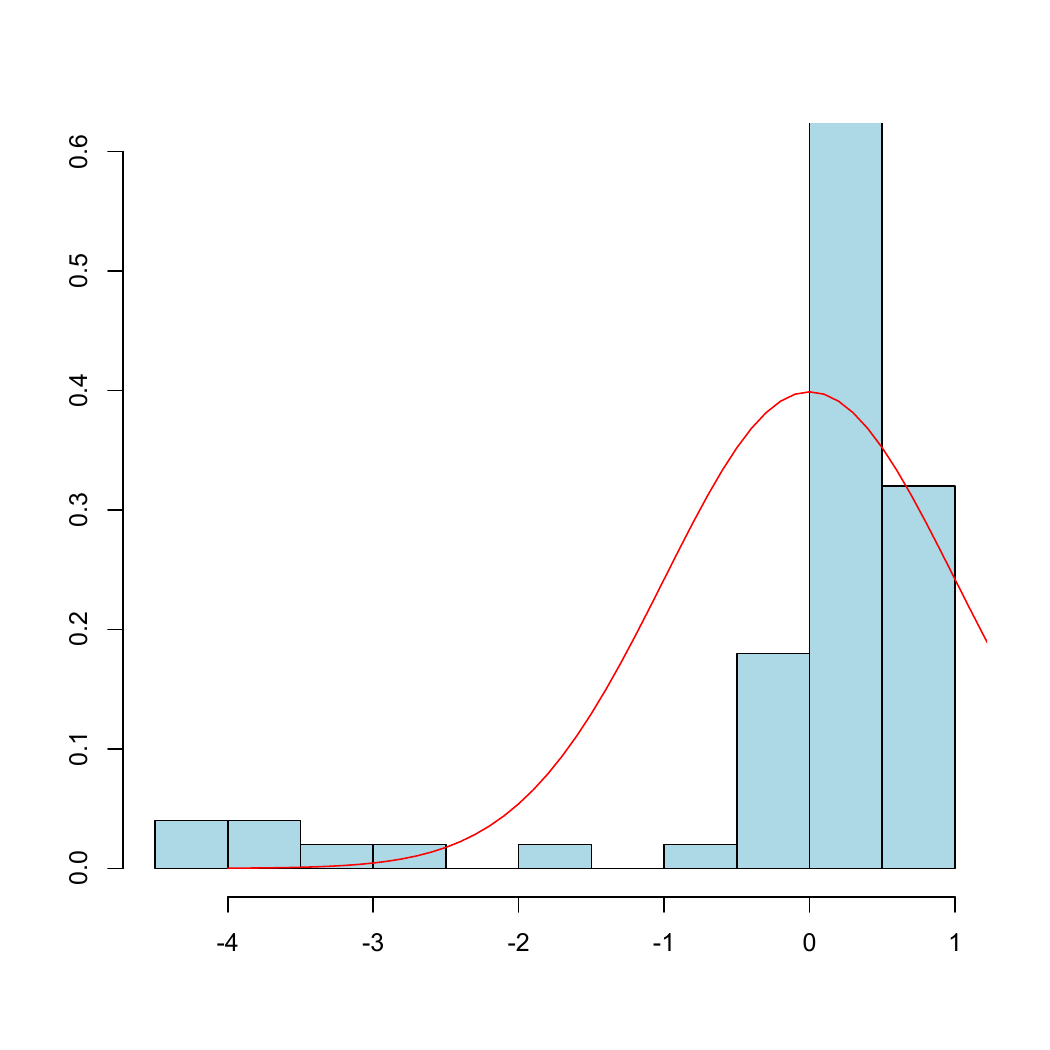}
		\includegraphics[scale=0.2]{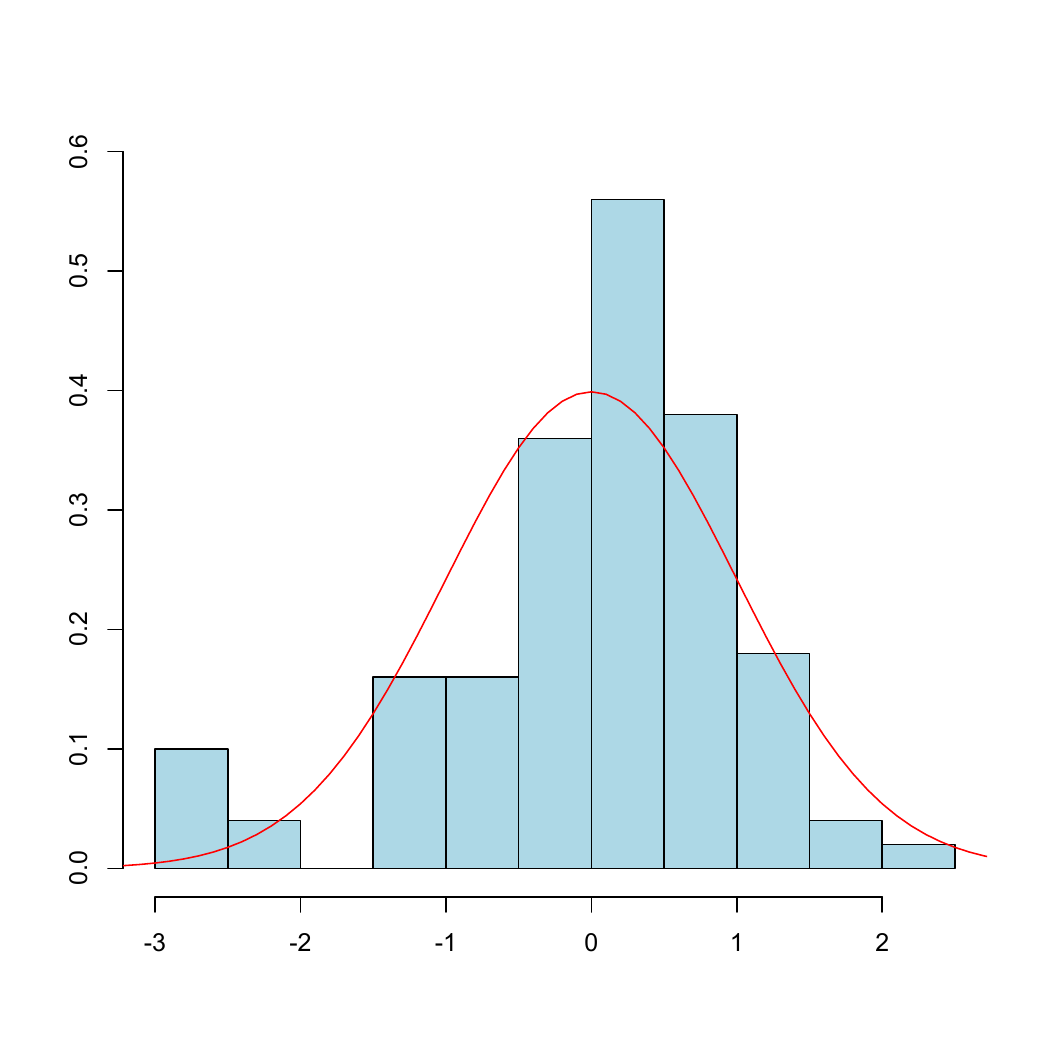} 
		 \includegraphics[scale=0.2]{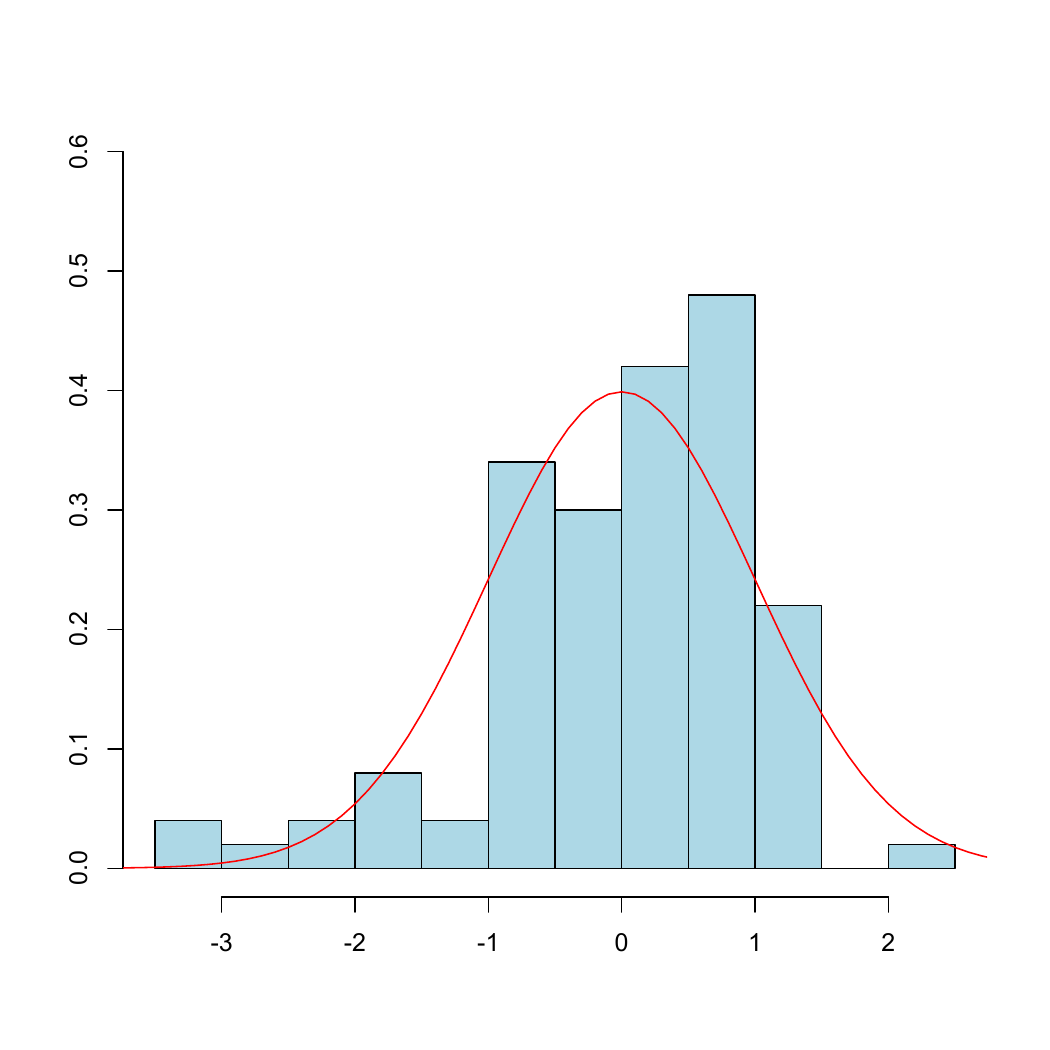}
		\includegraphics[scale=0.2]{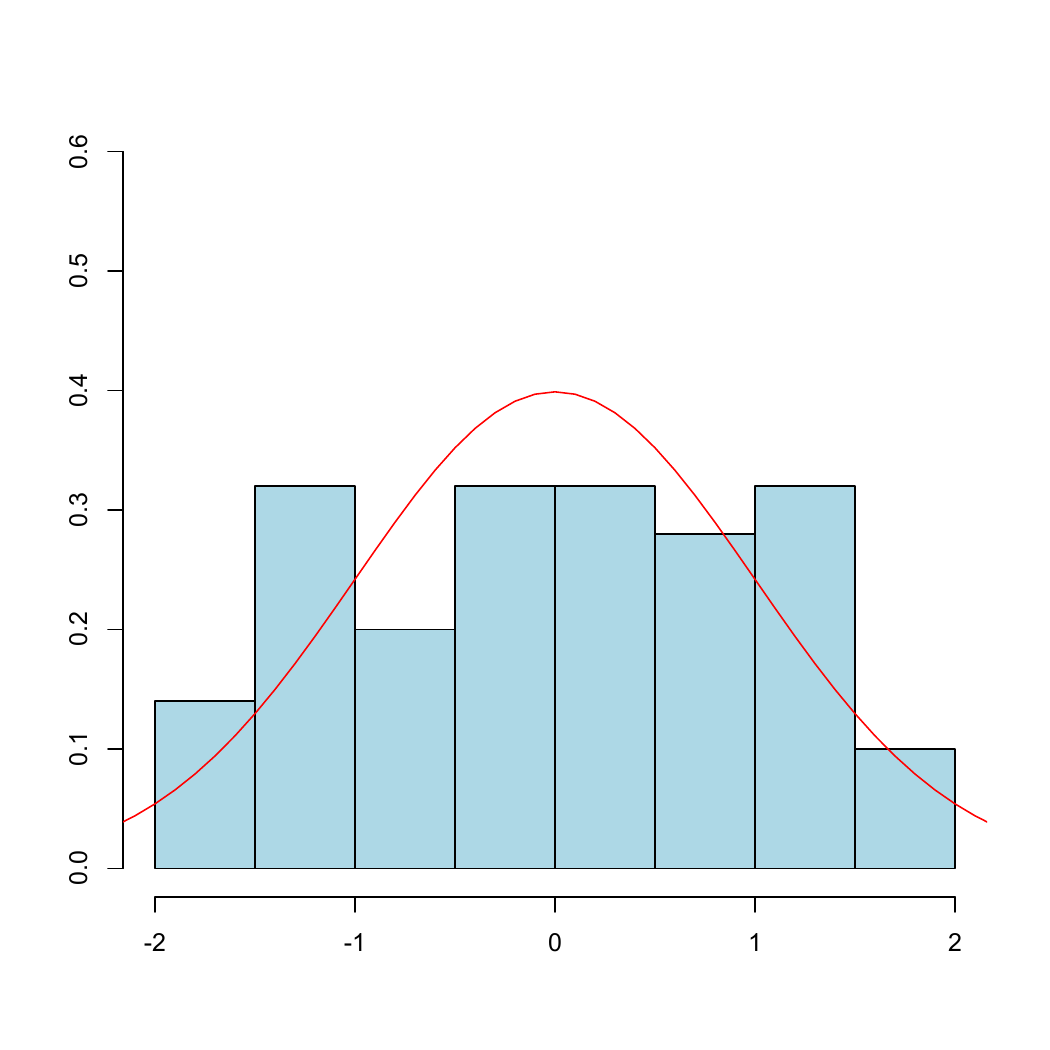} 
		\\
	        \includegraphics[scale=0.2]{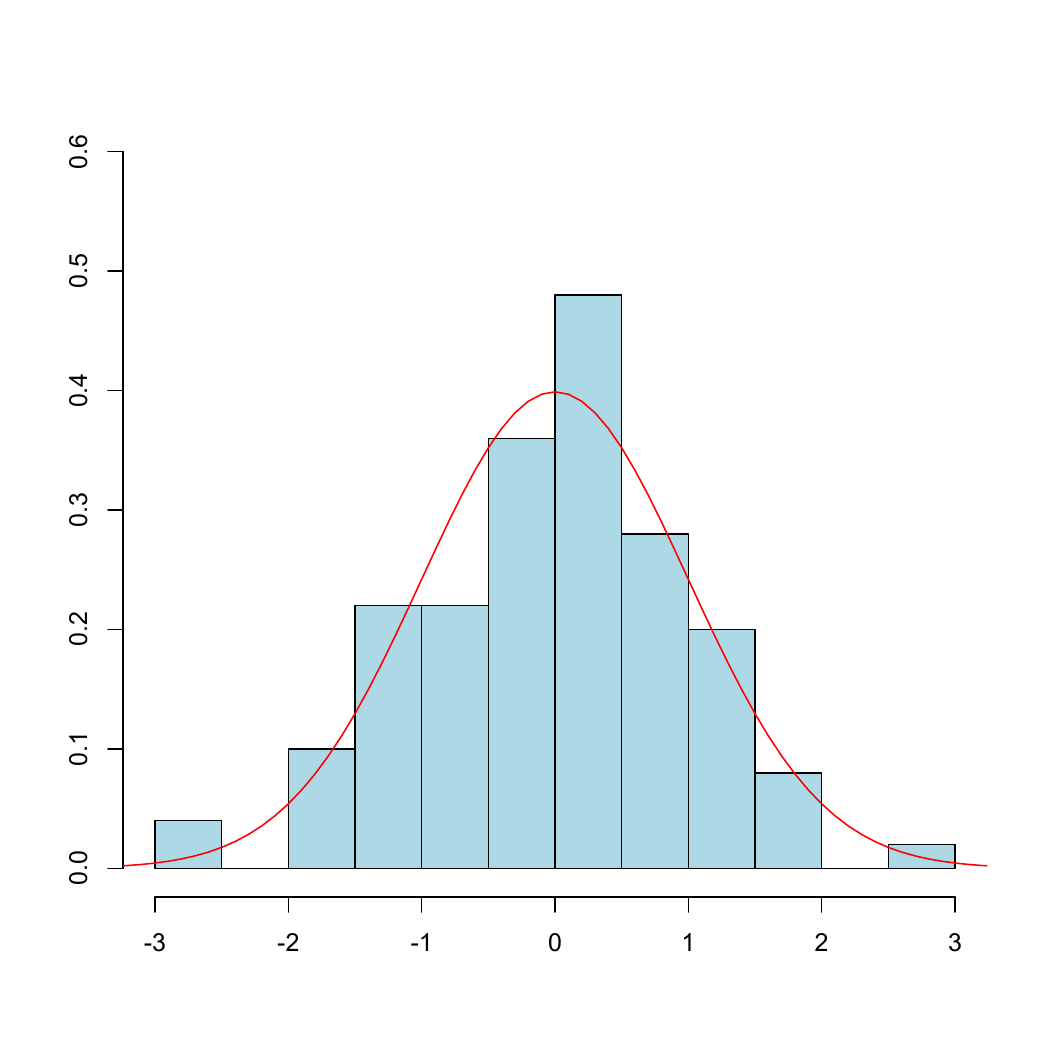}
		\includegraphics[scale=0.2]{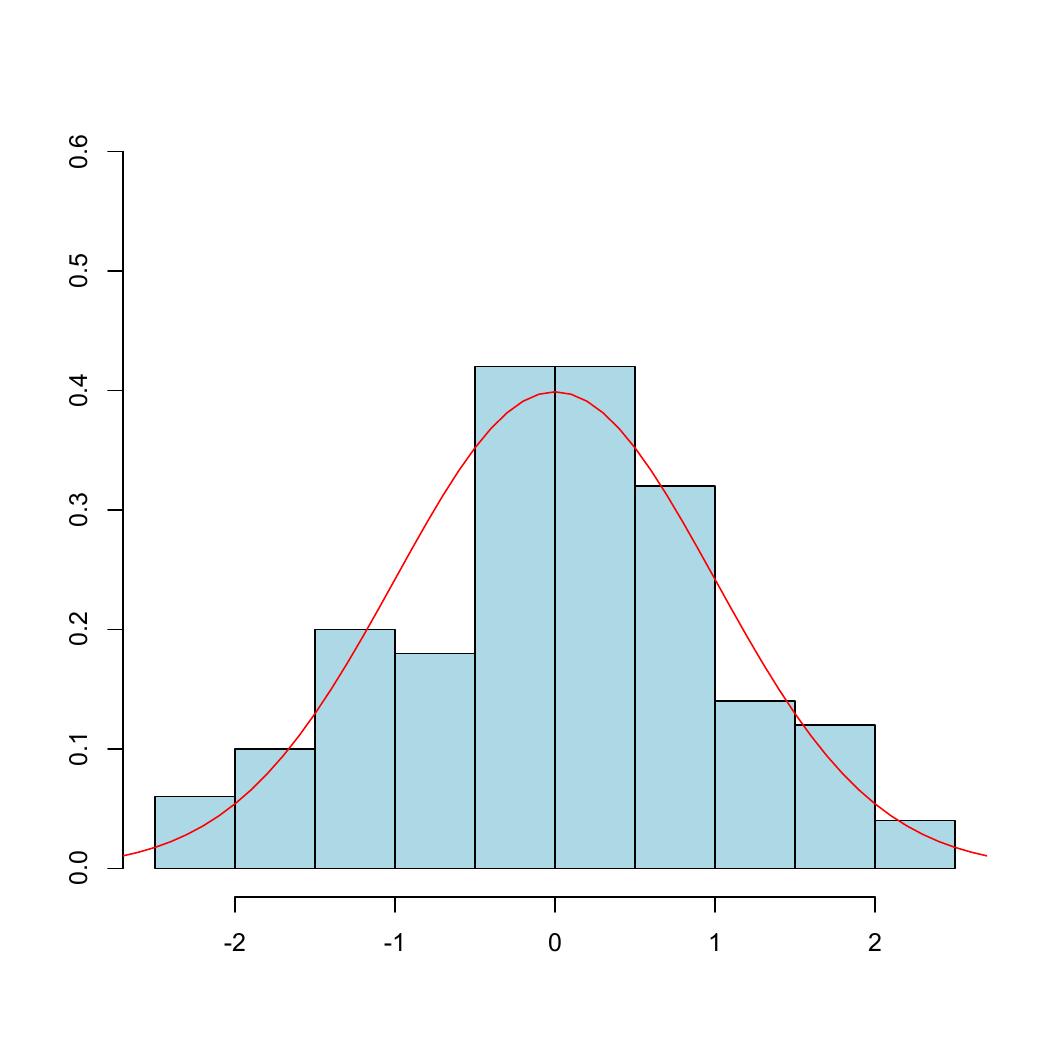} 
		\includegraphics[scale=0.2]{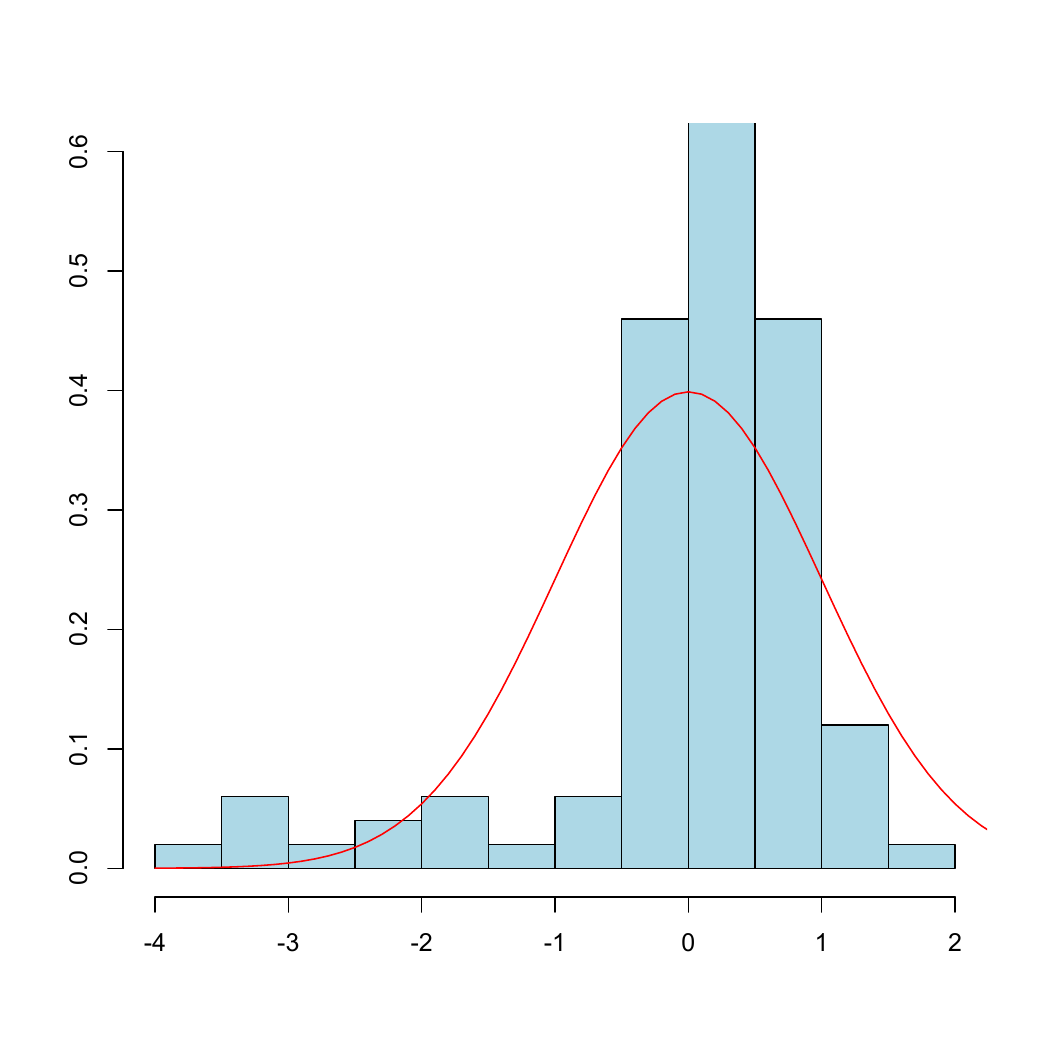}
		\includegraphics[scale=0.2]{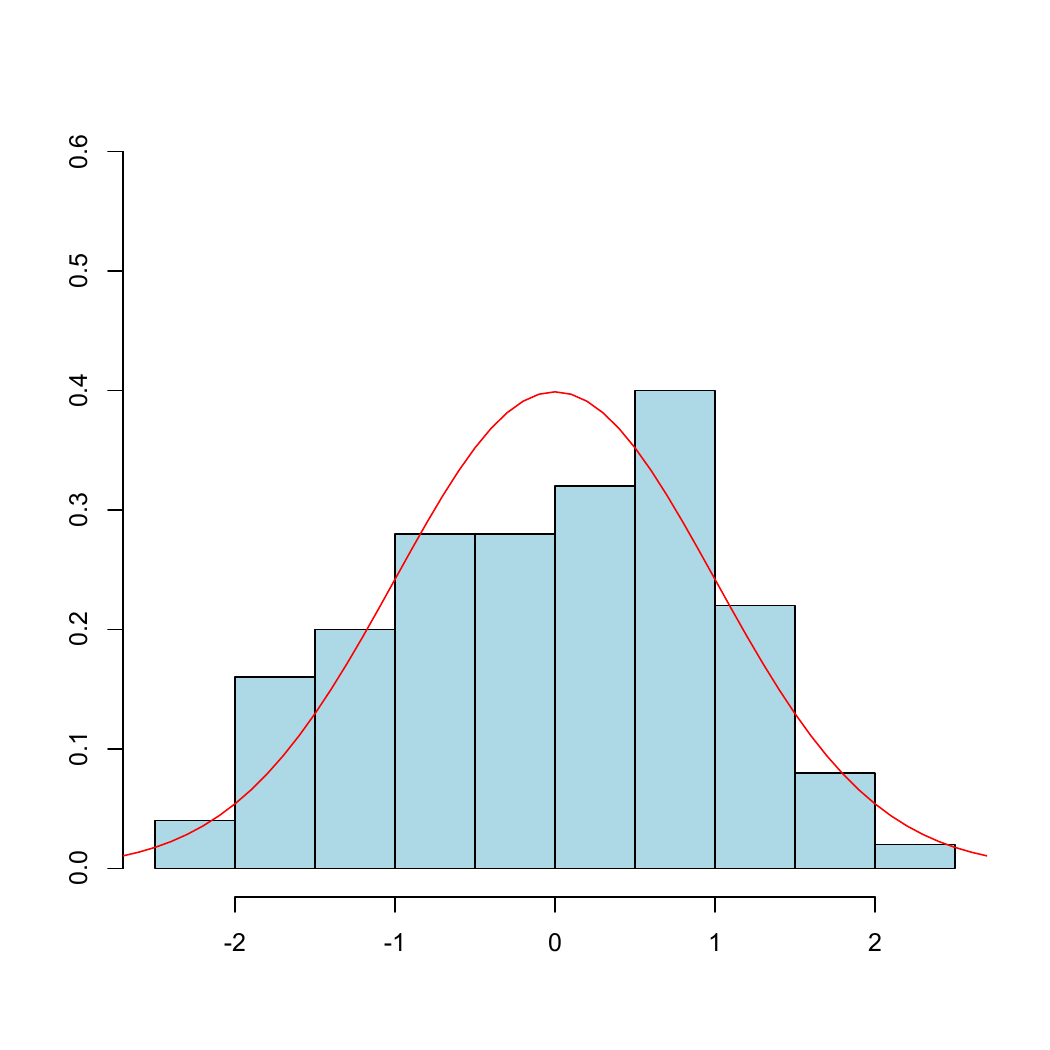}
		\\
		\includegraphics[scale=0.2]{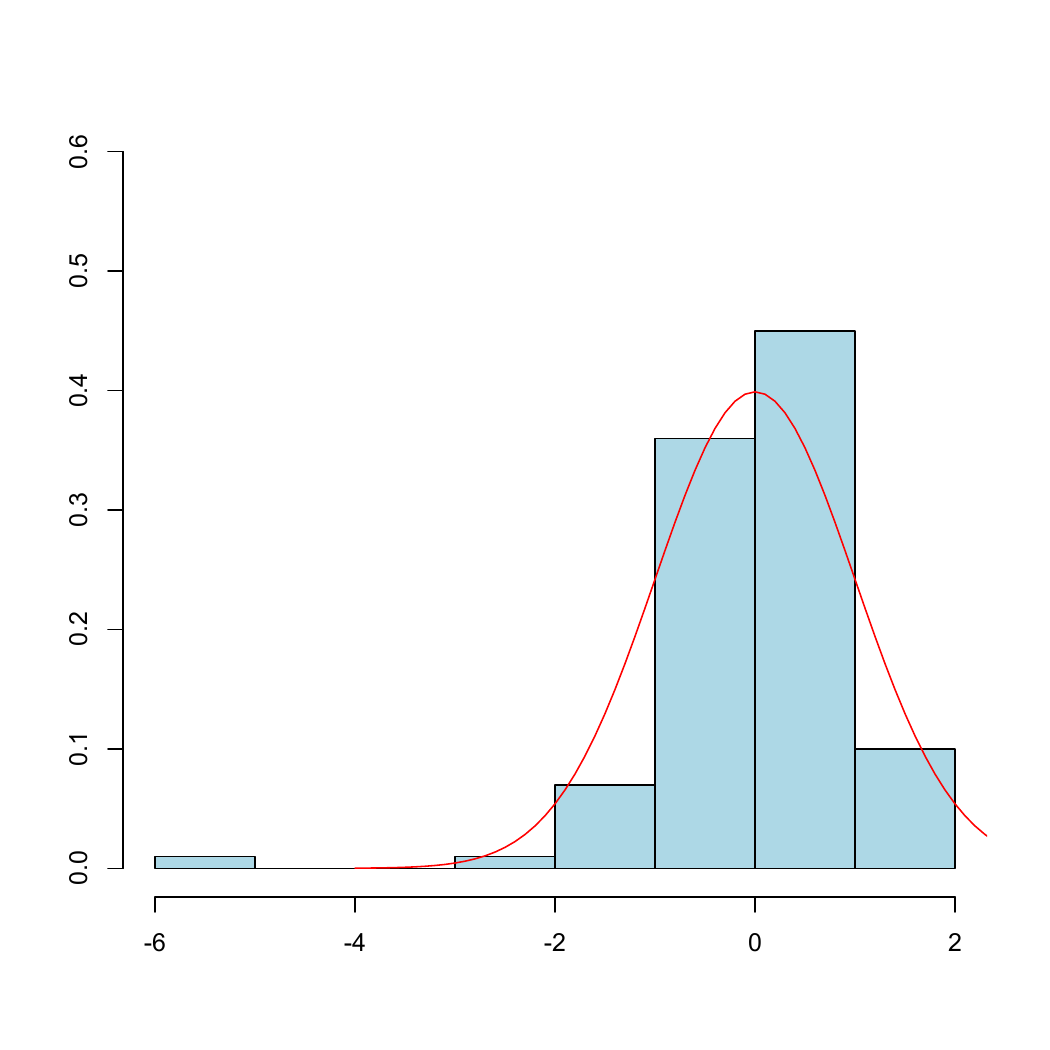}
		\includegraphics[scale=0.2]{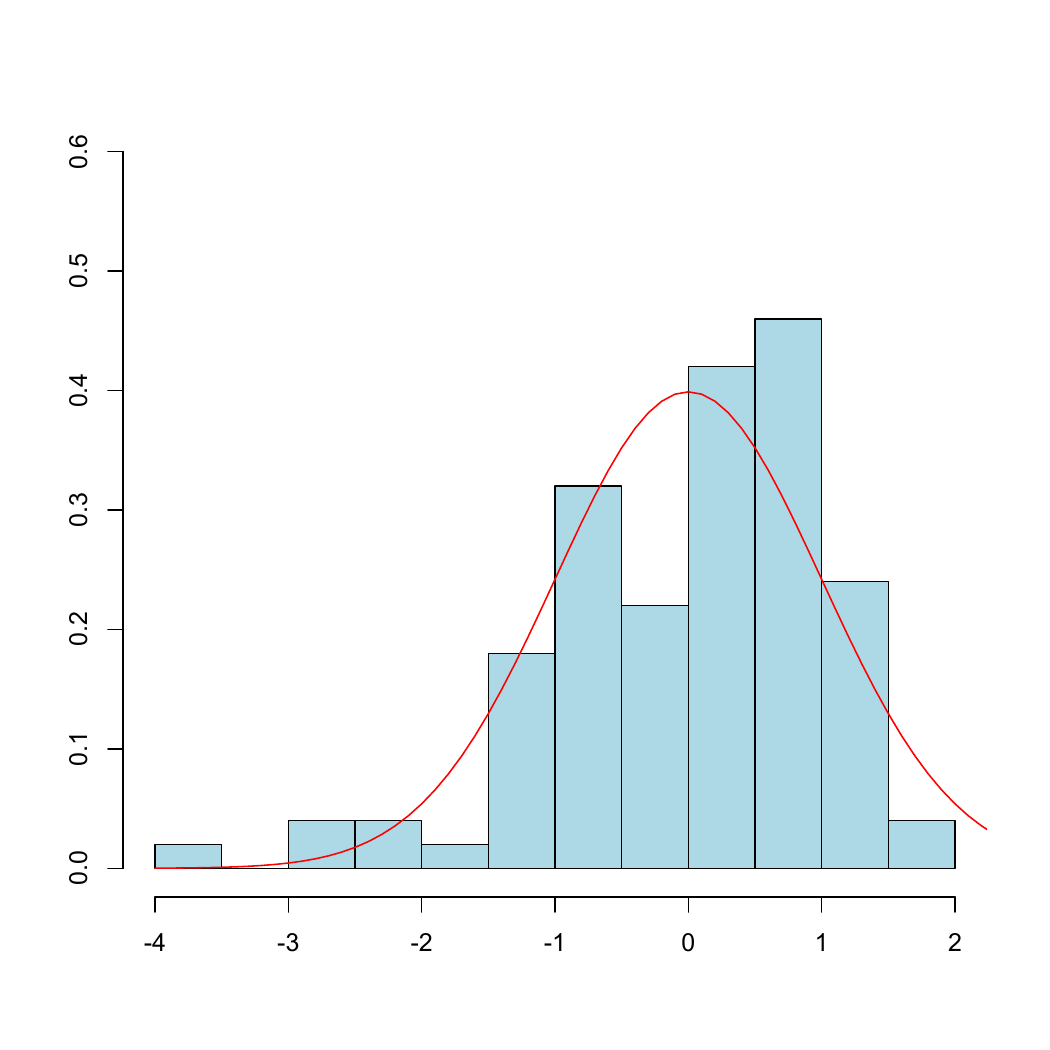} 
		\includegraphics[scale=0.2]{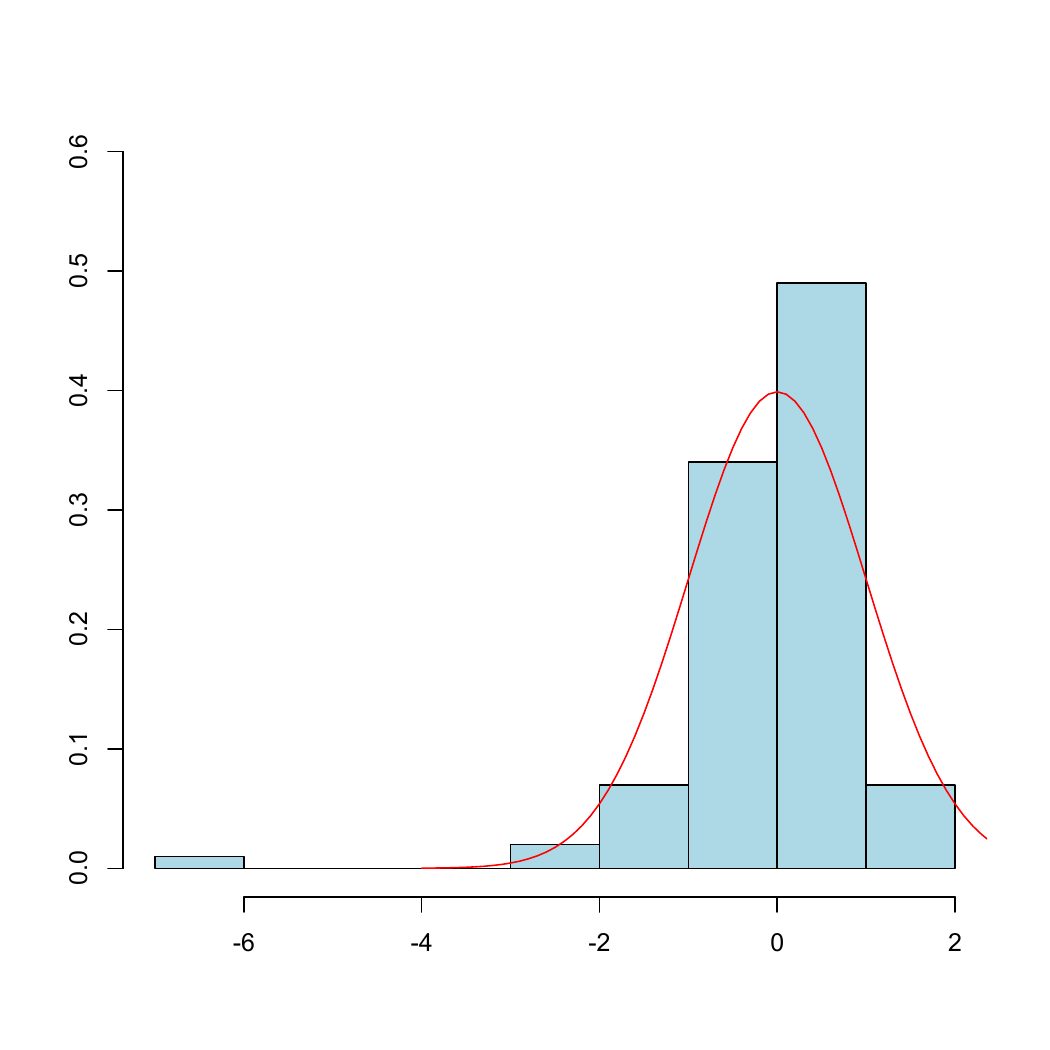}
		\includegraphics[scale=0.2]{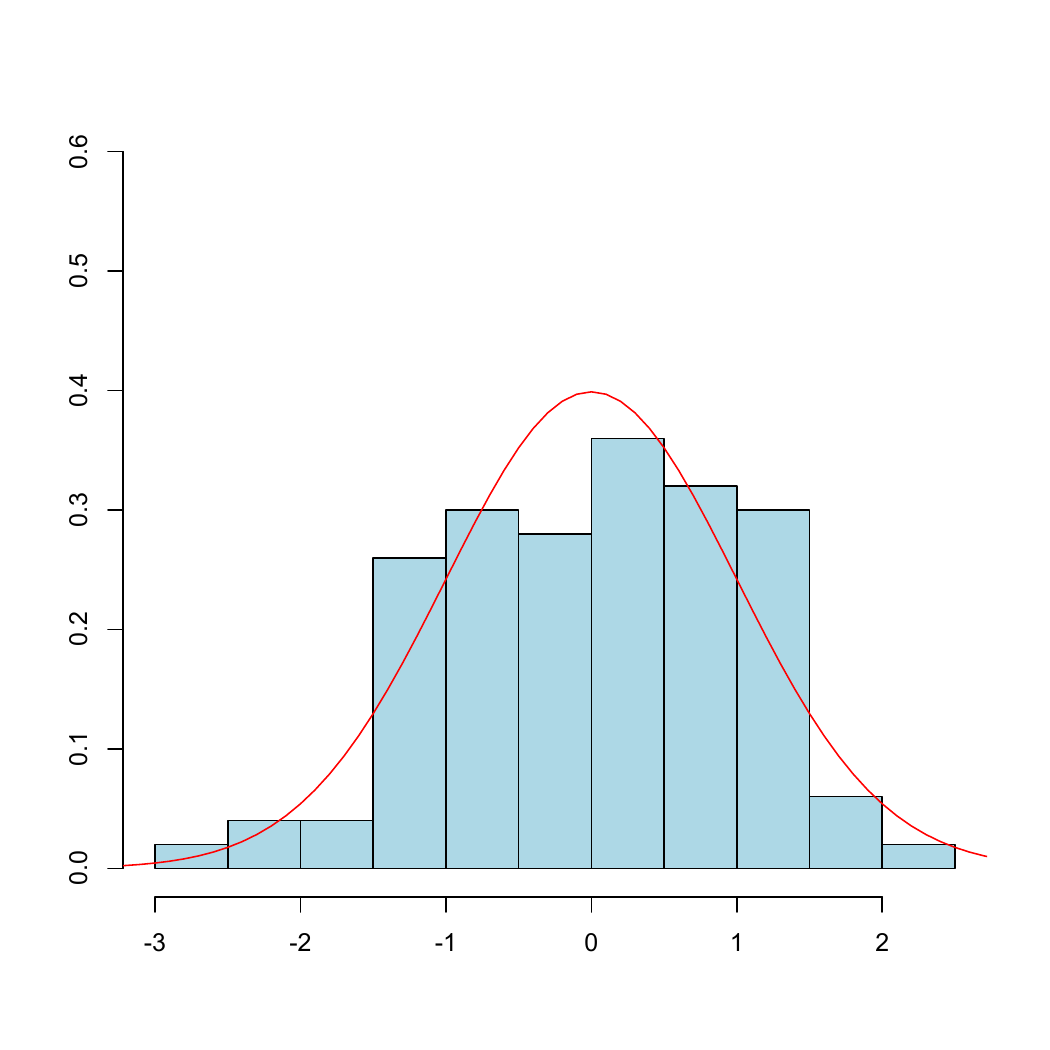} 
	\caption{Histogram of  $(\alpha,\beta)$ $\sqrt{n}$-normalised estimators distribution under model (S3). First row $n=5,000$, second row $n=10,000$, third row $n=20,000$. First column $\hat \alpha_n$, second column $\hat \beta_n$, third column  $\tilde \alpha_n$, fourth column $\tilde \beta_n$.}
		\label{fig:normalityS3}
	\end{center}
\end{figure}
\newpage
	
		% \usepackage{multirow}
		\begin{figure}[!h]
	\begin{center}
	        \includegraphics[scale=0.2]{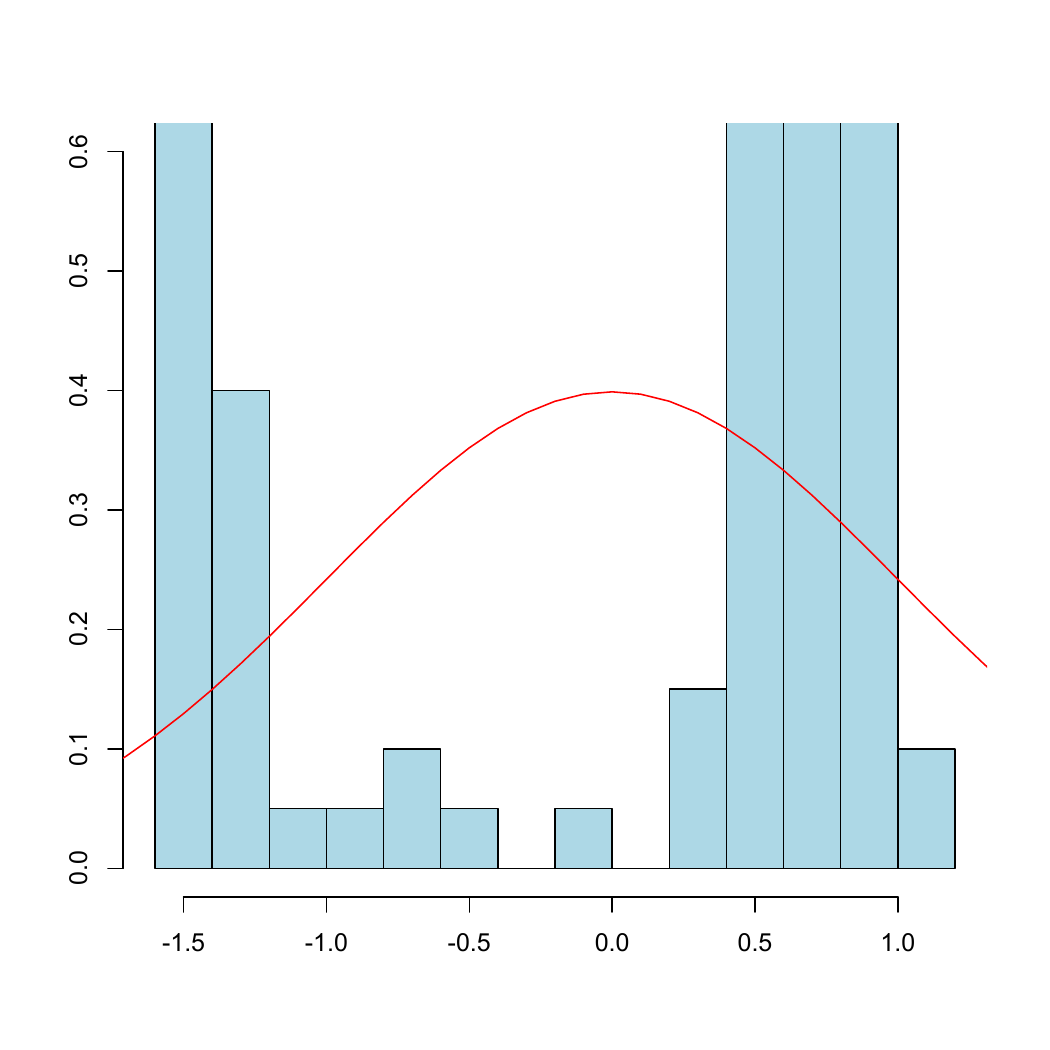}
		\includegraphics[scale=0.2]{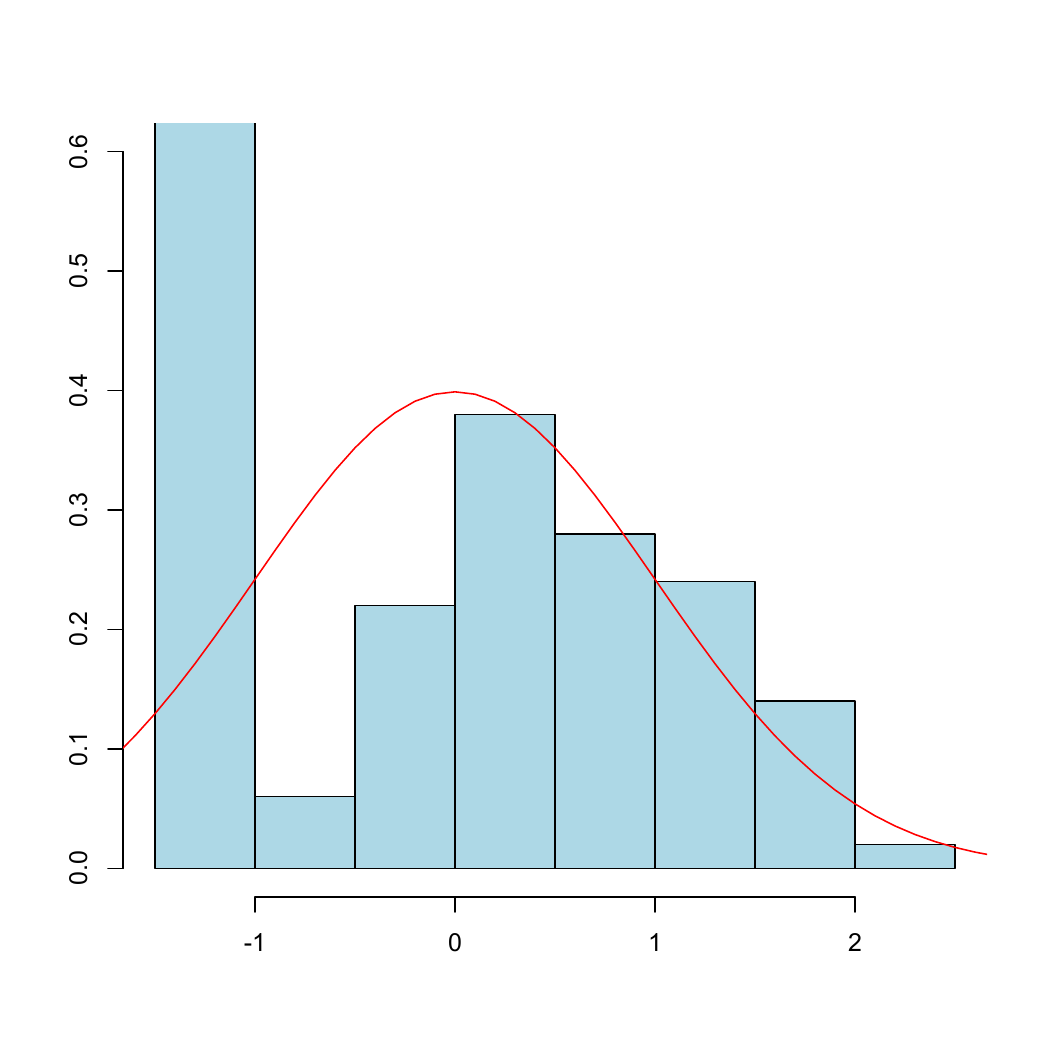} 
		 \includegraphics[scale=0.2]{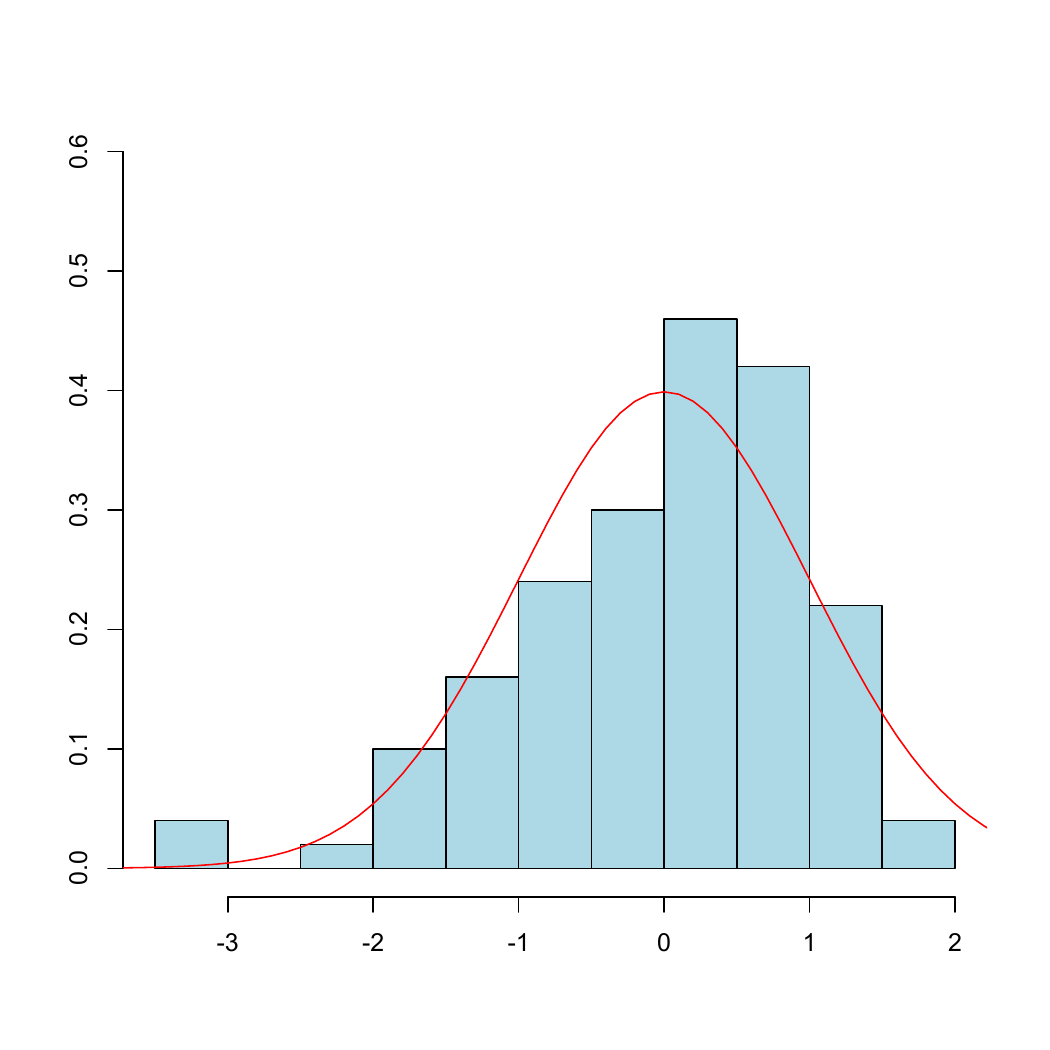}
		\includegraphics[scale=0.2]{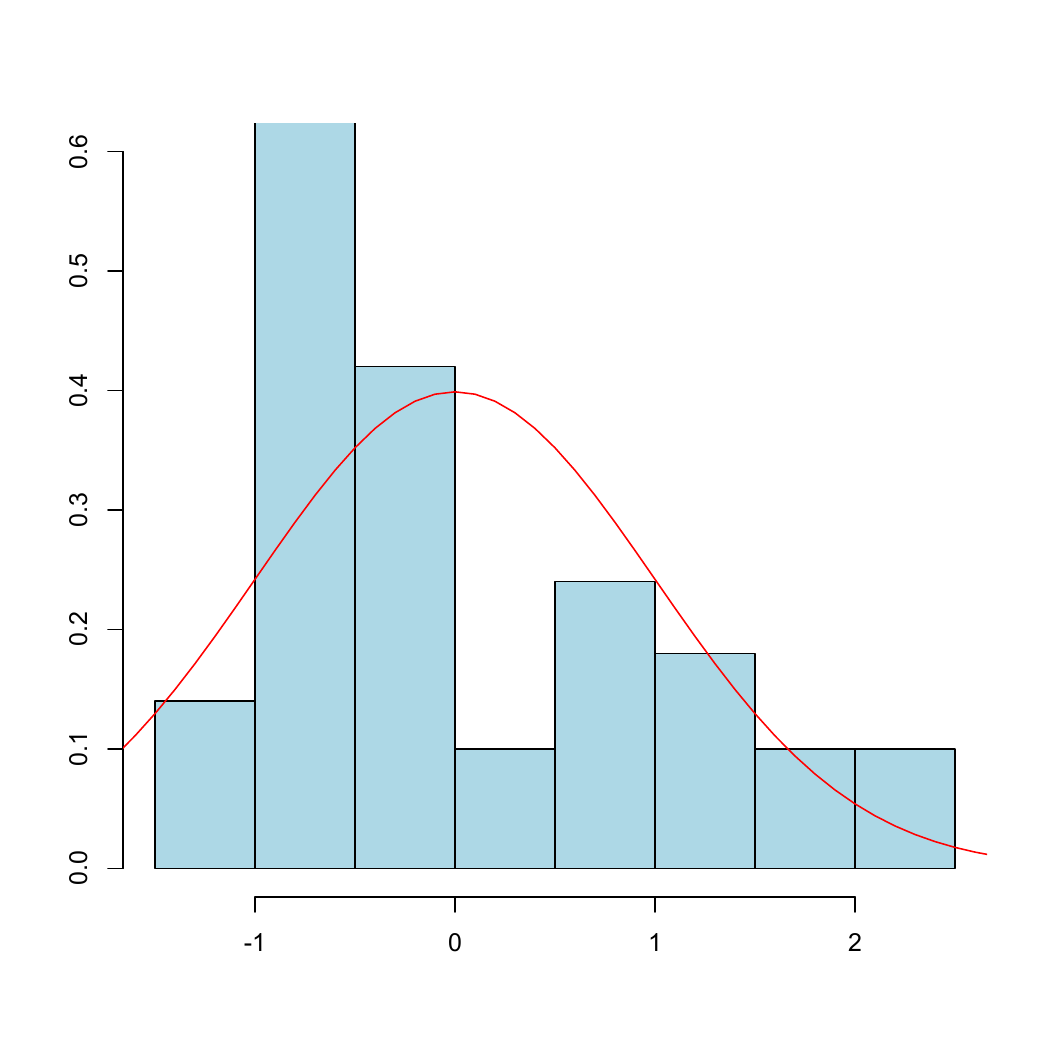} 
		\\
	        \includegraphics[scale=0.2]{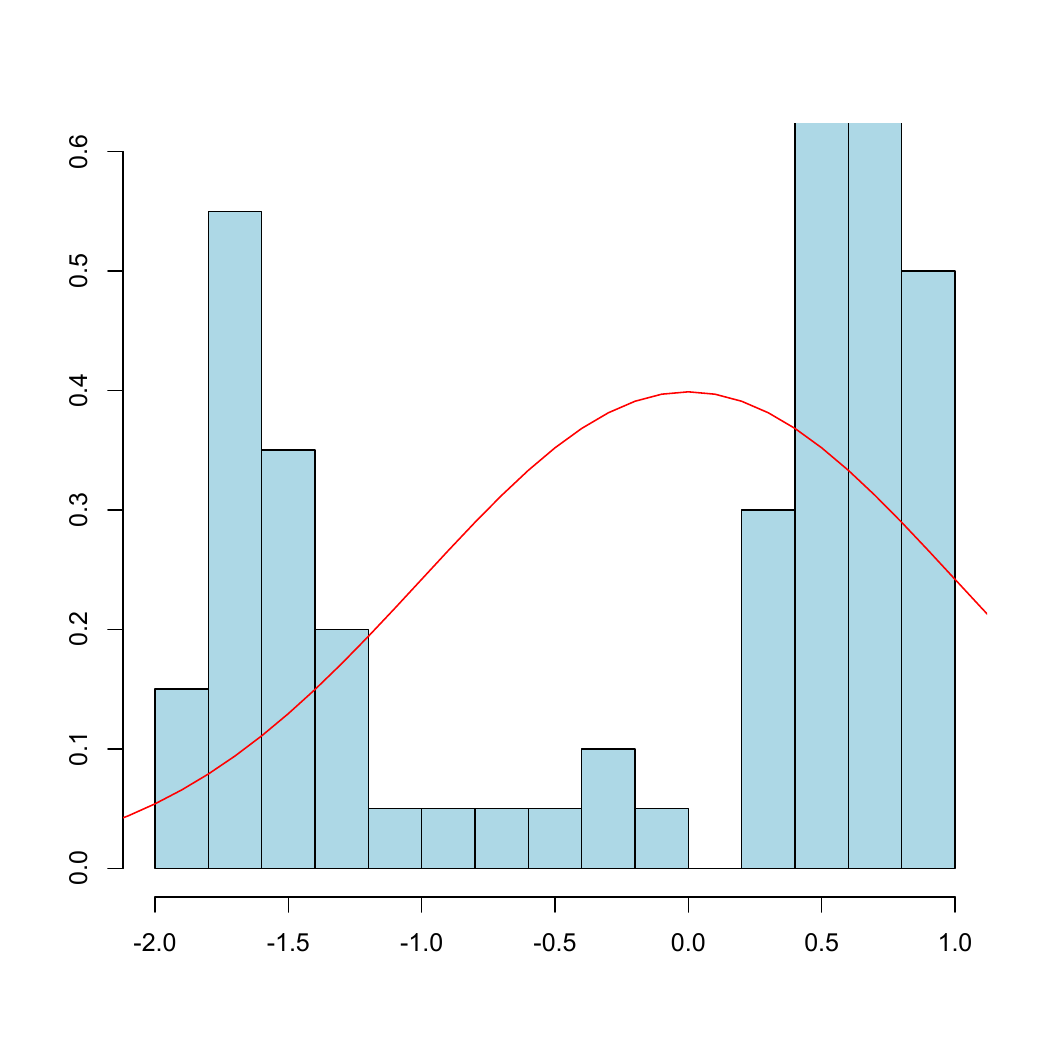}
		\includegraphics[scale=0.2]{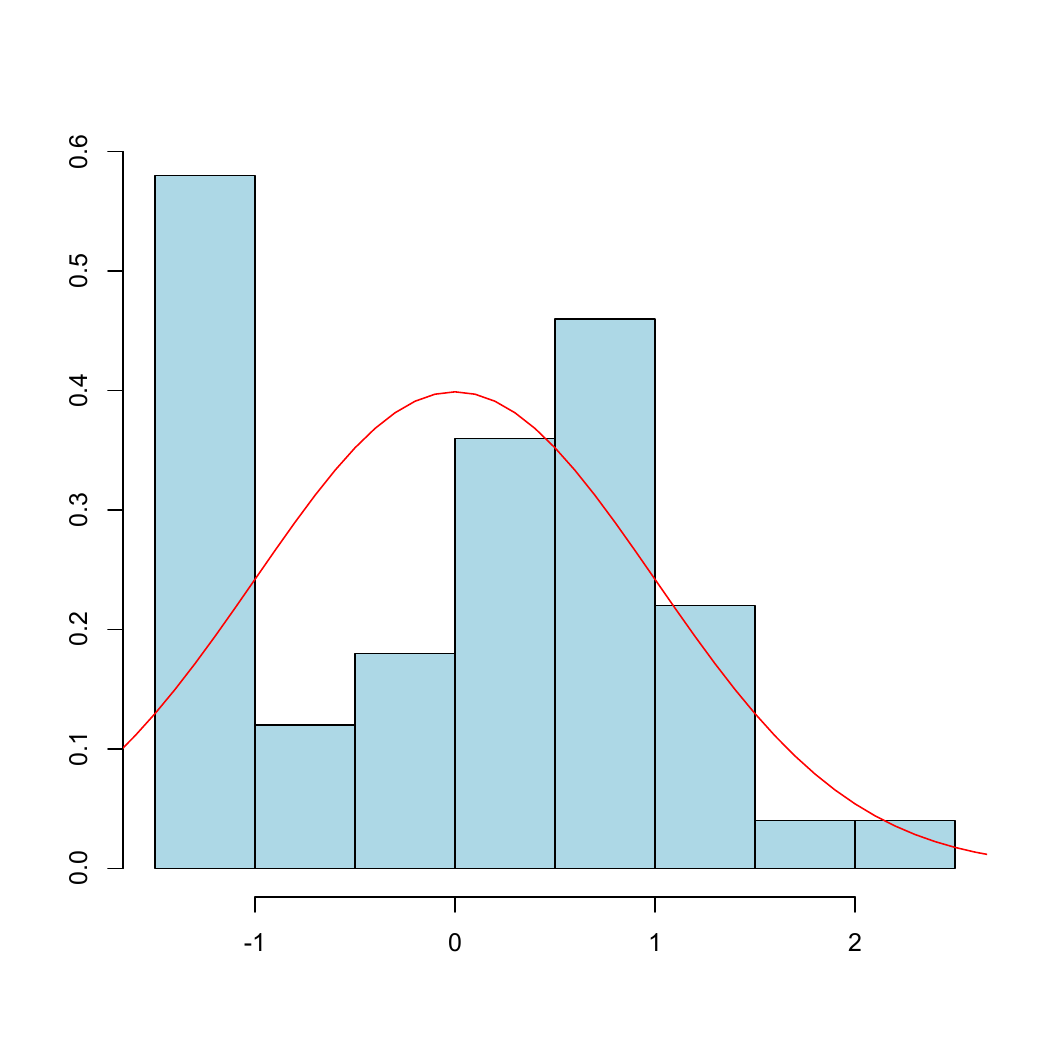} 
		\includegraphics[scale=0.2]{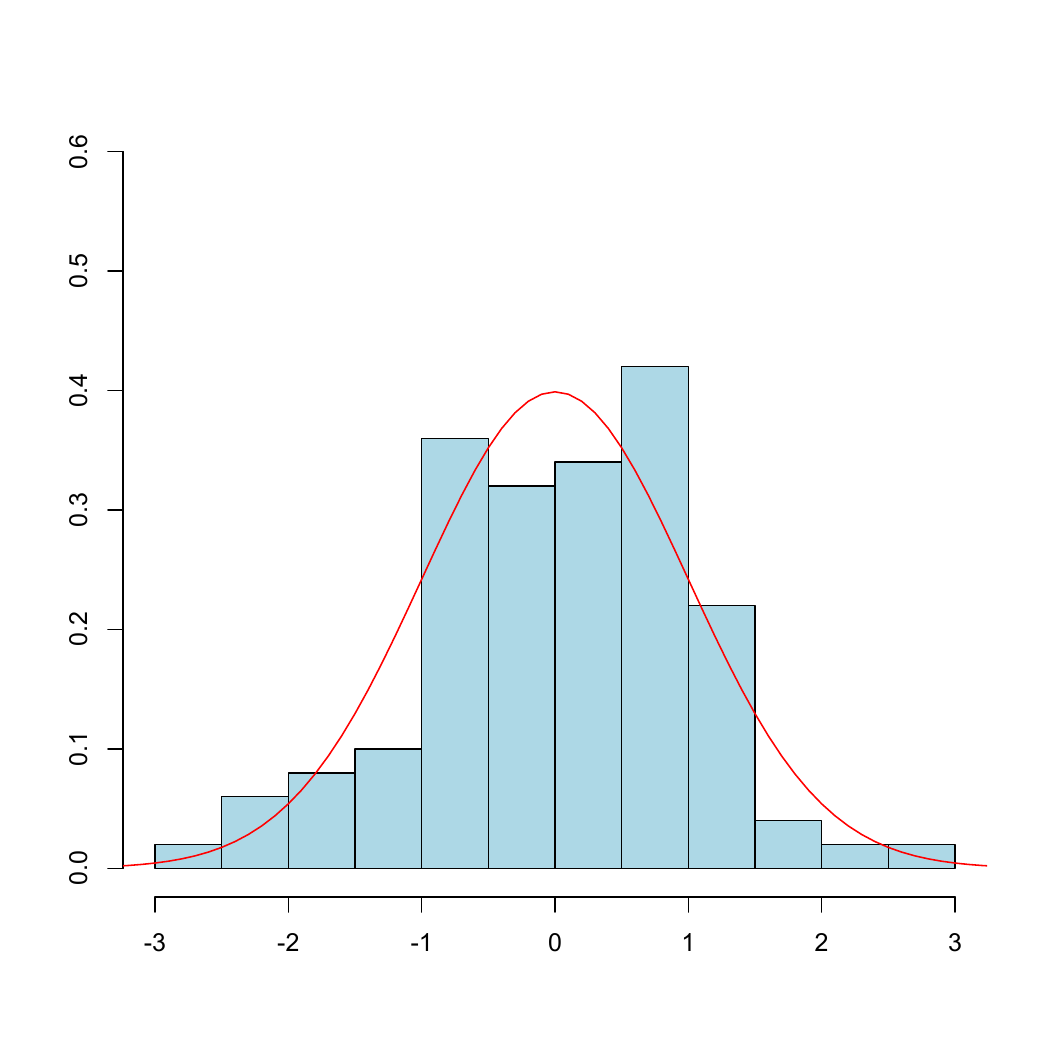}
		\includegraphics[scale=0.2]{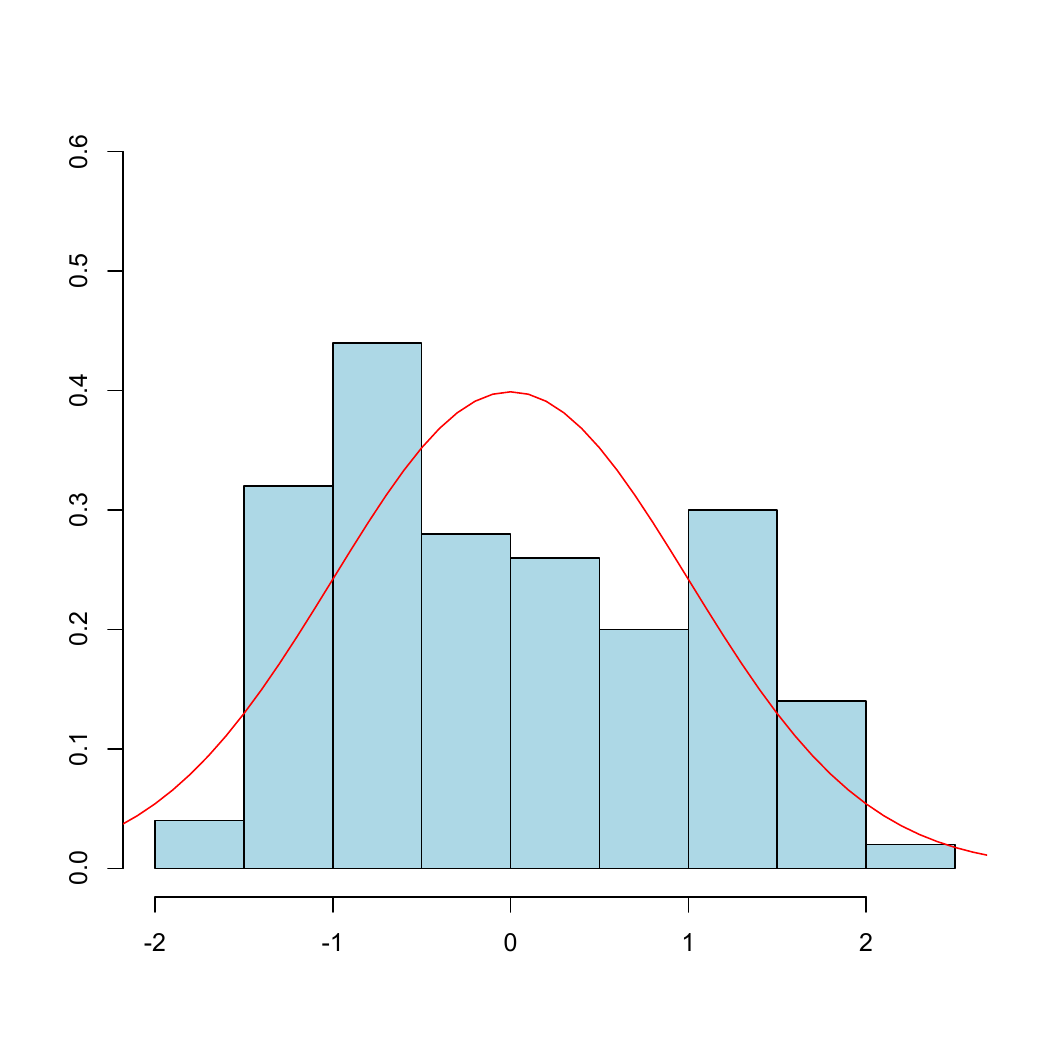}
		\\
		\includegraphics[scale=0.2]{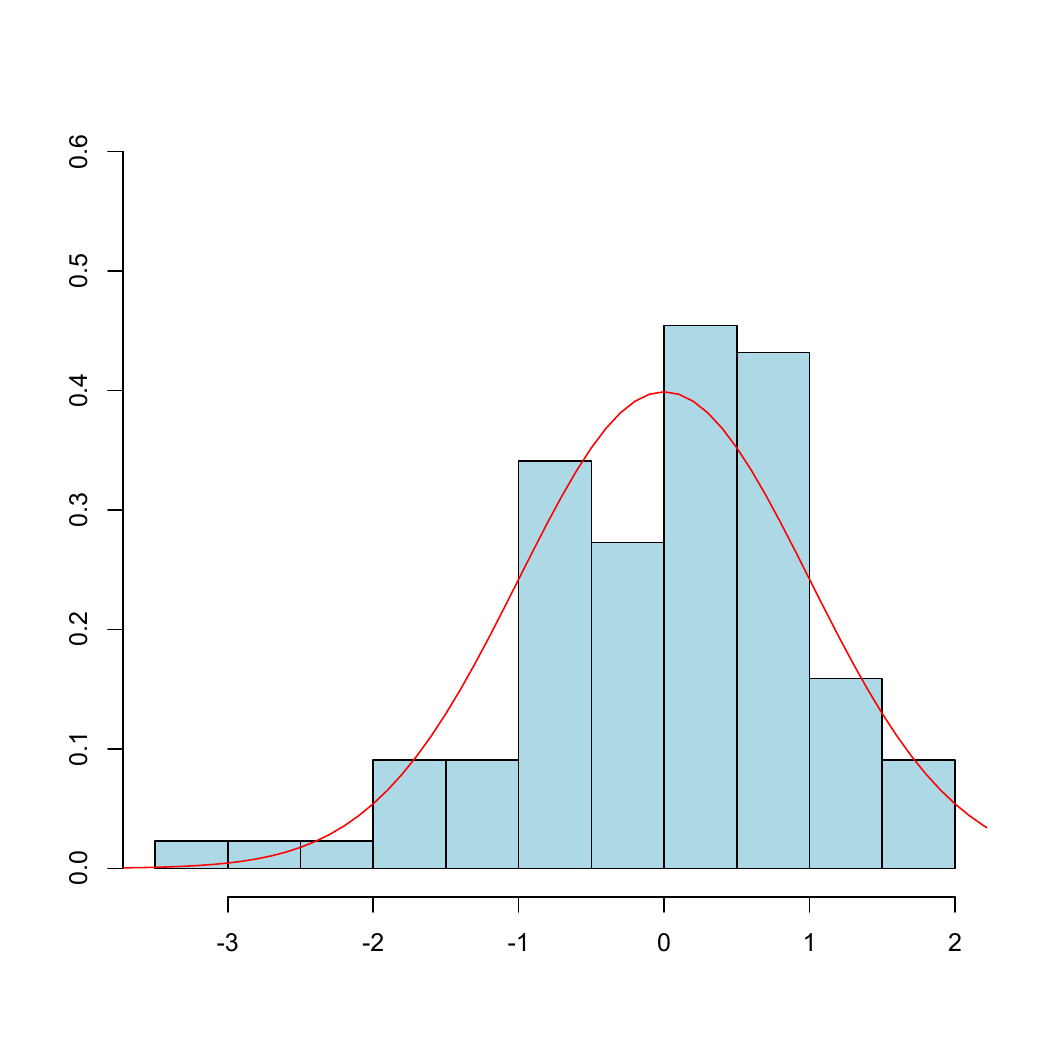}
		\includegraphics[scale=0.2]{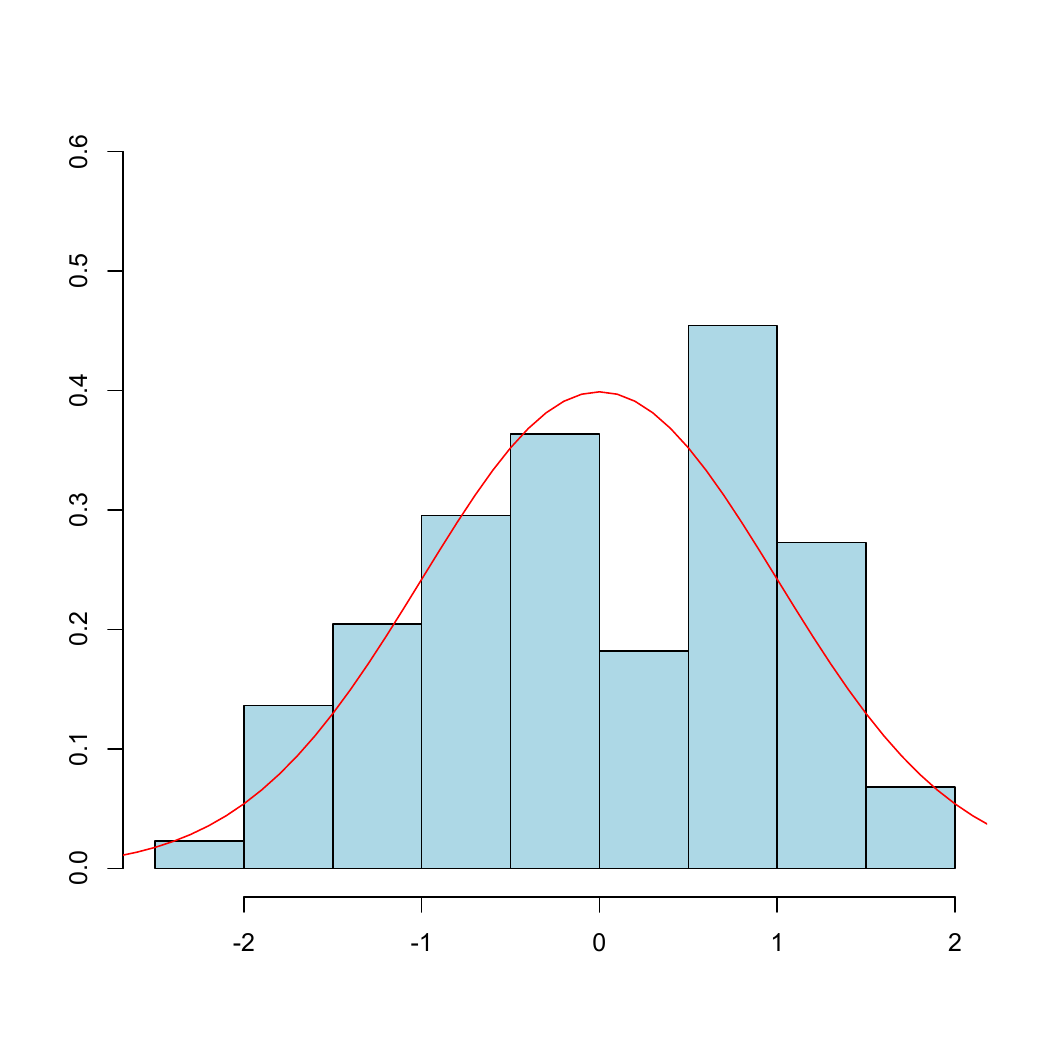} 
		\includegraphics[scale=0.2]{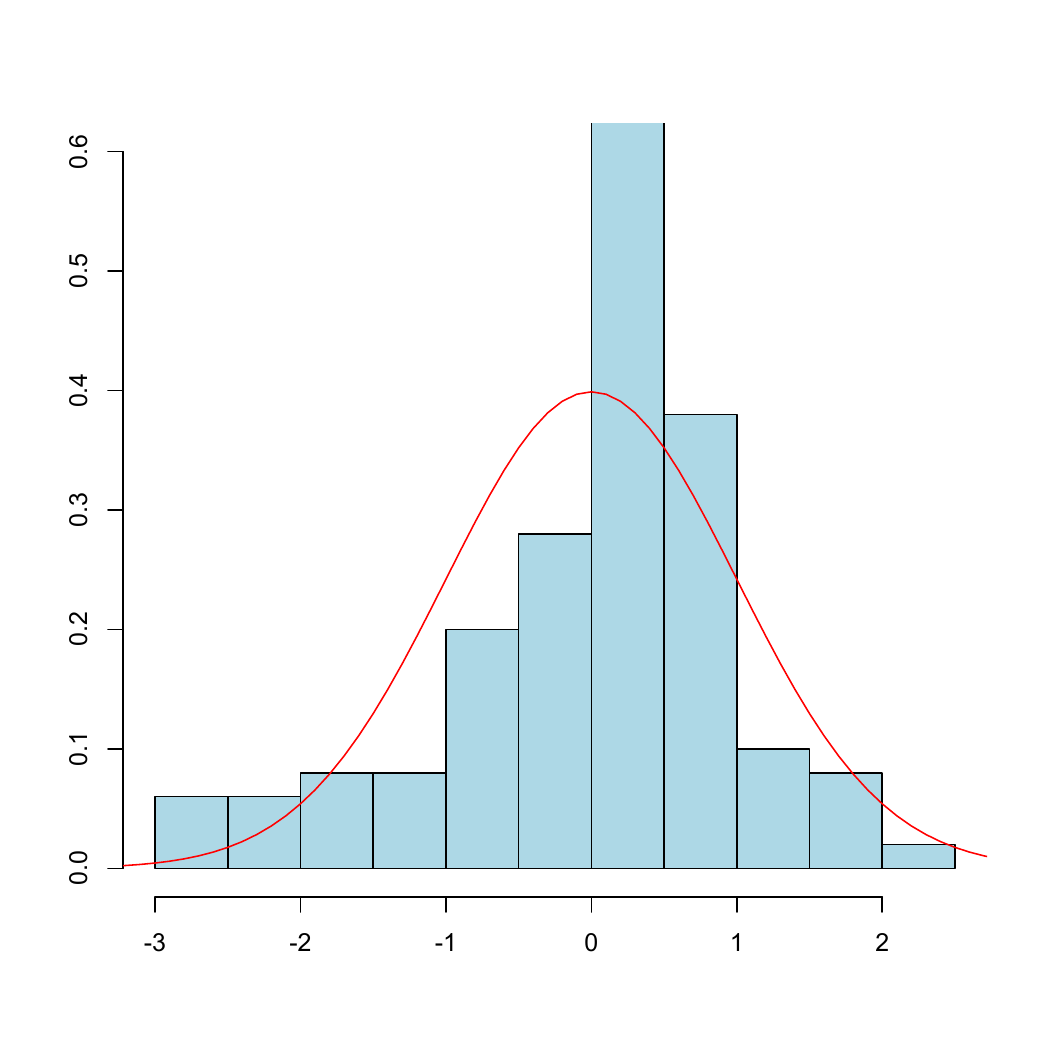}
		\includegraphics[scale=0.2]{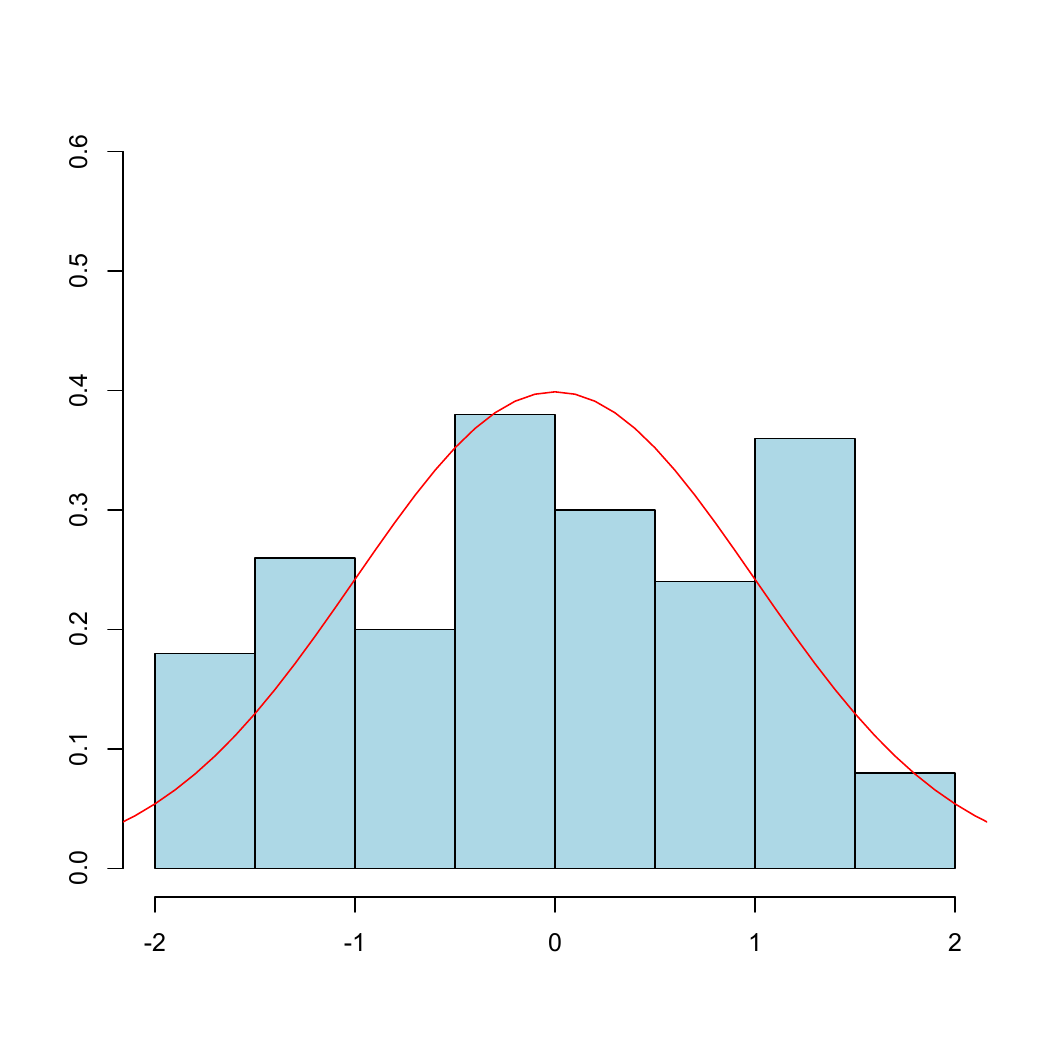} 
	\caption{Histogram of  $(\alpha,\beta)$ $\sqrt{n}$-normalised estimators distribution under model (S4). First row $n=5,000$, second row $n=10,000$, third row $n=20,000$. First column $\hat \alpha_n$, second column $\hat \beta_n$, third column  $\tilde \alpha_n$, fourth column $\tilde \beta_n$.}
		\label{fig:normalityS4}
	\end{center}
\end{figure}

\end{document}